\documentclass[10pt]{amsart}
\usepackage{mathtools}
\usepackage[mathscr]{eucal}
\usepackage{graphics}
\usepackage[all]{xy}
\usepackage{caption}
\usepackage[labelformat=simple]{subcaption}
\usepackage[usenames,dvipsnames]{color}
\usepackage{graphicx}
\usepackage[text={5.58in,8.5in},centering,letterpaper,dvips]{geometry}
\usepackage{amsfonts}
\usepackage{epsf}
\usepackage{amssymb}
\usepackage{amsmath}
\usepackage{amscd}
\usepackage{pdfpages}
\usepackage{fancyhdr}
\usepackage{setspace}
\usepackage[all]{xy}
\usepackage{verbatim}
\usepackage{enumerate}
\usepackage{multirow}
\usepackage{pgf}
\usepackage[latin1]{inputenc}
\usepackage[colorlinks=true, urlcolor=NavyBlue, linkcolor=NavyBlue, citecolor=NavyBlue]{hyperref}

\theoremstyle{theorem}
\newtheorem{theorem}{Theorem}[section]
\newtheorem{proposition}[theorem]{Proposition}
\newtheorem{lemma}[theorem]{Lemma}
\newtheorem{corollary}[theorem]{Corollary}
\newtheorem{conjecture}[theorem]{Conjecture}

\makeatletter
\newtheorem*{rep@theorem}{\rep@title}
\newcommand{\newreptheorem}[2]{%
\newenvironment{rep#1}[1]{%
 \def\rep@title{#2 \ref{##1}}%
 \begin{rep@theorem}}%
 {\end{rep@theorem}}}
\makeatother

\newreptheorem{theorem}{Theorem}
\newreptheorem{lemma}{Lemma}
\newreptheorem{question}{Question}
\newreptheorem{corollary}{Corollary}
\newreptheorem{proposition}{Proposition}

\theoremstyle{definition}
\newtheorem{definition}[theorem]{Definition}
\newtheorem{example}[theorem]{Example}
\newtheorem{question}[theorem]{Question}
\newtheorem{remark}[theorem]{Remark}
\newtheorem{remarks}[theorem]{Remarks}

\newcommand{\Z}{\mathbb{Z}}
\newcommand{\C}{\mathbb{C}}
\newcommand{\N}{\mathbb{N}}

\newcommand{\R}{\mathbb{R}}
\newcommand{\T}{\mathbb T}
\newcommand{\PP}{\mathbb P}
\newcommand{\RP}{\mathbb{RP}}
\newcommand{\CP}{\mathbb{CP}}

\newcommand{\Int}{\text{Int}}

\newcommand{\Aa}{\mathcal A}

\newcommand{\Dd}{\mathcal D}
\newcommand{\Ee}{\mathcal E}
\newcommand{\Tt}{\mathcal T}
\newcommand{\Ll}{\mathcal L}
\newcommand{\Mm}{\mathcal M}
\newcommand{\Ff}{\mathcal F}

\newcommand{\Hh}{\mathcal H}

\newcommand{\DD}{\mathfrak D}

\newcommand{\Id}{\text{Id}}

\newcommand{\wh}{\widehat}

\newcounter{saveenum}

\begin{document}

\rhead{\thepage}
\lhead{\author}
\thispagestyle{empty}


\raggedbottom
\pagenumbering{arabic}
\setcounter{section}{0}

\renewcommand\thesubfigure{(\Alph{subfigure})}


\title{Filling braided links with trisected surfaces}

\author{Jeffrey Meier}
\address{Department of Mathematics, Western Washington University}
\email{jeffrey.meier@wwu.edu}
\urladdr{\href{http://jeffreymeier.org}{http://jeffreymeier.org}}


\begin{abstract}
	We introduce the concept of a bridge trisection of a neatly embedded surface in a compact four-manifold, generalizing previous work with Alexander Zupan in the setting of closed surfaces in closed four-manifolds.  Our main result states that any neatly embedded surface $\Ff$ in a compact four-manifold $X$ can be isotoped to lie in bridge trisected position with respect to any trisection $\T$ of $X$.  A bridge trisection of $\Ff$ induces a braiding of the link $\partial\Ff$ with respect to the open-book decomposition of $\partial X$ induced by $\T$, and we show that the bridge trisection of $\Ff$ can be assumed to induce any such braiding.
	
	We work in the general setting in which $\partial X$ may be disconnected, and we describe how to encode bridge trisected surface diagrammatically using shadow diagrams.  We use shadow diagrams to show how bridge trisected surfaces can be glued along portions of their boundary, and we explain how the data of the braiding of the boundary link can be recovered from a shadow diagram. Throughout, numerous examples and illustrations are given.  We give a set of moves that we conjecture suffice to relate any two shadow diagrams corresponding to a given surface.
	
	We devote extra attention to the setting of surfaces in $B^4$, where we give an independent proof of the existence of bridge trisections and develop a second diagrammatic approach using tri-plane diagrams.  We characterize bridge trisections of ribbon surfaces in terms of their complexity parameters.  The process of passing between bridge trisections and band presentations for surfaces in $B^4$ is addressed in detail and presented with many examples.
\end{abstract}

\maketitle

\section{Introduction}\label{sec:Intro}

The philosophy underlying the theory of trisections is that four-dimensional objects can be decomposed into three simple pieces whose intersections are well-enough controlled that all of the four-dimensional data can be encoded on the two-dimensional intersection of the three pieces, leading to new diagrammatic approaches to four-manifold topology. Trisections were first introduced for four-manifolds by Gay and Kirby in 2016~\cite{GayKir_16_Trisecting-4-manifolds}.  A few years later, the theory was adapted to the setting of closed surfaces in four-manifolds by the author and Zupan~\cite{MeiZup_17_Bridge-trisections-of-knotted,MeiZup_18_Bridge-trisections-of-knotted}. The present article extends the theory to the general setting of neatly embedded surfaces in compact four-manifolds, yielding two diagrammatic approaches to the study of these objects: one that applies in general and one that applies when we restrict attention to surfaces in~$B^4$.

To introduce bridge trisections of surfaces in $B^4$, we must establish some terminology.  First, let $H$ be a three-ball $D^2\times I$, equipped with a critical-point-free Morse function $D^2\times I\to I$.  Let $\Tt\subset H$ be a neatly embedded one-manifold such that the restriction of the Morse function to each component of $\Tt$ has either one critical point or none.  If there are $b$ components with one critical point and $v$ with none, we call $(H,\Tt)$ a \emph{$(b,v)$--tangle}.
Next, let $Z$ be a four-ball $B^3\times I$, equipped with a critical-point-free Morse function $B^3\times I\to I$.  Let $\Dd\subset Z$ be a collection of neatly embedded disks such that the restriction of the Morse function to each component of $\Dd$ has either one critical point or none.  If there are $c$ components with one critical point and $v$ with none, we call $(Z,\Dd)$ a \emph{$(c,v)$--disk-tangle}.
Finally, let $\T_0$ denote the \emph{standard trisection} of $B^4$ -- i.e., the decomposition $B^4 = Z_1\cup Z_2\cup Z_3$ in which, for each $i\in\Z_3$, the $Z_i$ are four-balls, the pairwise intersections $H_i = Z_{i-1}\cap Z_i$ are three-balls, and the common intersection $\Sigma = Z_1\cap Z_2\cap Z_3$ is a disk.

A neatly embedded surface $\Ff\subset B^4$ is in \emph{$(b,\bold c,v)$--bridge position} with respect to $\T_0$ if the following hold for each $i\in\Z_3$:
\begin{enumerate}
	\item $\Ff\cap Z_i$ is a $(c_i,v)$--disk-tangle, where $\bold c = (c_1,c_2,c_3)$; and
	\item $\Ff\cap H_i$ is a $(b,v)$--tangle.
\end{enumerate}
A definition very similar to this one was introduced independently in~\cite{BlaCamTay_20_Kirby-Thompson-distance-for-trisections}.

The trisection $\T_0$ induces the open-book decomposition of $S^3 =\partial B^4$ whose pages are the disks and whose binding is $\partial\Sigma$.
Let $\Ll = \partial\Ff$, and let $\beta_i = S^3\cap\Dd$.  Then $\Ll = \beta_1\cup\beta_2\cup\beta_3$ is braided about $\partial \Sigma$ with index $v$.  Having outlined the requisite structures, we can state our existence result for bridge trisections of surfaces in the four-ball.

\begin{reptheorem}
{thm:four-ball}
	Let $\T_0$ be the standard trisection of $B^4$, and let $\Ff\subset B^4$ be a neatly embedded surface with $\Ll = \partial \Ff$.  Fix an index $v$ braiding $\wh\beta$ of $\Ll$.  Suppose $\Ff$ has a handle decomposition with $c_1$ cups, $n$ bands, and $c_3$ caps.  Then, for some $b\in\N_0$, $\Ff$ can be isotoped to be in $(b,\bold c;v)$--bridge trisected position with respect to $\T_0$, such that $\partial\Ff = \wh\beta$, where $c_2=b-n$.
\end{reptheorem}

Explicit in the above statement is a connection between the complexity parameters of a bridge trisected surface and the numbers of each type of handle in a Morse decomposition of the surface.  An immediate consequence of this correspondence is the fact that a ribbon surface admits a bridge trisection where $c_3=0$.  It turns out that this observation can be strengthened to give the following characterization of ribbon surfaces in $B^4$. Again, $\bold c = (c_1,c_2,c_3)$, and we set $c=c_1+c_2+c_3$.

\begin{reptheorem}{thm:ribbon}
	Let $\T_0$ be the standard trisection of $B^4$, and let $\Ff\subset B^4$ be a neatly embedded surface with $\Ll = \partial \Ff$.  Let $\wh\beta$ be an index $v$ braiding $\Ll$.
	Then, the following are equivalent.
	\begin{enumerate}
		\item $\Ff$ is ribbon.
		\item $\Ff$ admits a $(b,\bold c;v)$--bridge trisection filling $\wh\beta$ with $c_i=0$ for some $i$.
		\item $\Ff$ admits a $(b,0;v+c)$--bridge trisection filling a Markov perturbation $\wh\beta^+$ of $\wh\beta$.
	\end{enumerate}
\end{reptheorem}

A bridge trisection turns out to be determined by its spine -- i.e., the union $(H_1,\Tt_1)\cup(H_2,\Tt_2)\cup(H_3,\Tt_3)$, and each tangle $(H_i,\Tt_i)$ can be faithfully encoded by a planar diagram.  It follows that any surface in $B^4$ can be encoded by a triple of planar diagrams whose pairwise unions are planar diagrams for split unions of geometric braids and unlinks.  We call such triples \emph{tri-plane diagrams}.

\begin{repcorollary}
{coro:tri-plane}
	Every neatly embedded surface in $B^4$ can be described by a tri-plane diagram.
\end{repcorollary}

In Section~\ref{sec:tri-plane}, we show how to read off the data of the braiding of $\Ll$ induced by a bridge trisection from a tri-plane for the bridge trisection, and we describe a collection of moves that suffice to relate any two tri-plane diagrams corresponding to a given bridge trisection.  The reader concerned mainly with surfaces in $B^4$ can focus their attention on Sections~\ref{sec:four-ball} and~\ref{sec:tri-plane}, referring to the more general development of the preliminary material given in Section~\ref{sec:general} when needed.

Having summarized the results of the paper that pertain to the setting of $B^4$, we now describe the more general setting in which $X$ is a compact four-manifold with (possibly disconnected) boundary and $\Ff\subset X$ is a neatly embedded surface.
To account for this added generality, we must expand the definitions given earlier for the basic building blocks of a bridge trisection.  For ease of exposition, we will not record the complexity parameters, which are numerous in this setting; Section~\ref{sec:general} contains compete details.

Let $H$ be a compressionbody $(\Sigma\times I)\cup(\text{3--dimensional 2--handles})$, where $\Sigma=\partial_+H$ is connected and may have nonempty boundary, while $P = \partial_-H$ is allowed to be disconnected but cannot contain two-sphere components.  We work relative to the obvious Morse function on $H$.  Let $\Tt\subset H$ be a neatly embedded one-manifold such that the restriction of the Morse function to each component of $\Tt$ has either one critical point or none.  We call $(H,\Tt)$ a \emph{trivial tangle}.
Let $Z$ be a four-dimensional compressionbody $(P\times I\times I)\cup(\text{4--dimensional 1--handles})$, where $P$ is as above.  We work relative to the obvious Morse function on $Z$.  Let $\Dd\subset Z$ be a collection of neatly embedded disks such that the restriction of the Morse function to each component of $\Dd$ has either one critical point or none.  We call $(Z,\Dd)$ a \emph{trivial disk-tangle}.

Let $X$ be a compact four-manifold, and let $\Ff\subset X$ be a neatly embedded surface.  A \emph{bridge trisection} of $(X,\Ff)$ is a decomposition
$$(X,\Ff) = (Z_1,\Dd_1)\cup(Z_2,\Dd_2)\cup(Z_3,\Dd_3)$$
such that, for each $i\in\Z_3$,
\begin{enumerate}
	\item $(Z_i,\Dd_i)$ is a trivial disk-tangle.
	\item $(H_i,\Tt_i) = (Z_{i-1},\Dd_{i-1})\cap(Z_i,\Dd_i)$ is a trivial tangle.
\end{enumerate}
We let $(\Sigma,\bold x) = \partial(H_i,\Tt_i)$.
The underlying trisection $X = Z_1\cup Z_2\cup Z_3$ induces an open-book decomposition on each component of $Y=\partial X$, and we find that the bridge trisection of $\Ff$ induces a braiding of $\Ll = \partial \Ff$ with respect to these open-book decompositions.  Given this set-up, our general existence result can now be stated.

\begin{reptheorem}{thm:general}
	Let $\T$ be a trisection of a four-manifold $X$ with $\partial X = Y$, and let $(B,\pi)$ denote the open-book decomposition of $Y$ induced by $\T$.  Let $\Ff$ be a neatly embedded surface in $X$; let $\Ll = \partial \Ff$; and fix a braiding $\wh\beta$ of $\Ll$ about $(B,\pi)$.  Then, $\Ff$ can be isotoped to be in bridge trisected position with respect to $\T$ such that $\partial \Ff = \wh\beta$.  If $\Ll$ already coincides with the braiding $\beta$, then this isotopy can be assumed to restrict to the identity on $Y$.
\end{reptheorem}

If $H$ is not a three-ball, then $(H,\Tt)$ cannot be encoded as a planar diagram, as before.  However, $H$ is determined by a collection of curves $\alpha\subset \Sigma\setminus\nu(\bold x)$, and $\Tt$ is determined by a collection of arcs $\Tt^*$ in $\Sigma$, where the arcs of $\Tt^*$ connect pairs of points of $\bold x$.  We call the data $(\Sigma,\alpha,\Tt^*)$, which determine the trivial tangle $(H,\Tt)$, a \emph{tangle shadow}.
A triple of tangle shadows that satisfies certain pairwise-standardness conditions is called a \emph{shadow diagram}.  Because bridge trisections are determined by their spines, we obtain the following corollary.

\begin{repcorollary}{coro:shadow_describe}
	Let $X$ be a smooth, orientable, compact, connected four-manifold, and let $\Ff$ be a neatly embedded surface in $X$.  Then, $(X,\Ff)$ can be described by a shadow diagram.
\end{repcorollary}

A detailed development of shadow diagrams is given in Section~\ref{sec:shadow}, where it is described how to read off the data of the braiding of $\Ll$ induced by a bridge trisection from a shadow diagram corresponding to the bridge trisection.  Moves relating shadow diagrams corresponding to a fixed bridge trisection are given.  Section~\ref{sec:gluing} discusses how to glue two bridge trisected surfaces so that the result is bridge trisected, as well as how these gluings can be carried out with shadow diagrams.

Section~\ref{sec:class} gives some basic classification results, as well as a handful of examples to add to the many examples included throughout Sections~\ref{sec:four-ball}--\ref{sec:gluing}.  The proof of the main existence result, Theorem~\ref{thm:general}, is delayed until Section~\ref{sec:gen_proof}, though it requires only the content of Section~\ref{sec:general} to be accessible.  In Section~\ref{sec:stabilize}, we discuss stabilization and perturbation operations that we conjecture are sufficient to relate any two bridge trisections of a fixed surface.  A positive resolution of this conjecture would give complete diagrammatic calculi for studying surfaces via tri-plane diagrams and shadow diagrams.

\subsection*{Acknowledgements}

The author is deeply grateful to David Gay and Alexander Zupan for innumerable provocative and enlightening discussions about trisections over the last few years.  The author would like to thank Juanita Pinz\'on-Caicedo and Maggie Miller for helpful suggestions and thoughts throughout this project.  This work was supported in part by NSF grant DMS-1933019.

\section{Preliminaries}\label{sec:general}

In this section, we give a detailed development of the ingredients required throughout the paper, establishing notation conventions as we go.  This section should probably be considered as prerequisite for all the following sections, save for Sections~\ref{sec:four-ball} and~\ref{sec:tri-plane}, which pertain to the consideration of surfaces in the four-ball.  The reader interested only in this setting may be able to skip ahead, referring back to this section only as needed.

\subsection{Some conventions}
\label{subsec:conventions}
\ 

Unless otherwise noted, all manifolds and maps between manifolds are assumed to be smooth, and manifolds are compact.
The central objects of study here all have the form of a \emph{manifold pair} $(M,N)$, by which we mean that $N$ is \emph{neatly embedded} in $M$ in the sense that $\partial N\subset\partial M$ and $N\pitchfork \partial M$.  Throughout, $N$ will usually have codimension two in $M$.
In any event, we let $\nu(N)$ denote the interior of a tubular neighborhood of $N$ in $M$.
If $M$ is oriented, we let $\overline{(M,N)}$ denote the pair $(M,N)$ with the opposite orientation and we call it the \emph{mirror} of $(M,N)$.
We use the symbol $\sqcup$ to denote either the disjoint union or the split union, depending on the context.
For example, writing $(M_1,N_1)\sqcup(M_2,N_2)$ indicates $M_1\cap M_2 = \emptyset$.
On the other hand, $(M,N_1\sqcup N_2)$ indicates that $N_1$ and $N_2$ are \emph{split} in $M$, by which we usually mean there are disjoint, codimension zero balls $B_1$ and $B_2$ in $M$ (not necessarily neatly embedded) such that $N_i\subset\Int B_i$ for each $i\in\{1,2\}$.

\subsection{Lensed cobordisms}
\label{subsec:Lensed}
\ 

Given compact manifold pairs $(M_1,N_1)$ and $(M_2,N_2)$ with $\partial(M_1,N_1)\cong\partial(M_2,N_2)$ nonempty, we normally think of a cobordism from $(M_1,N_1)$ to $(M_2,N_2)$ as a manifold pair $(W,Z)$, where 
$$\partial(W,Z) = \left((M_1,N_1)\sqcup\overline{(M_2,N_2)}\right)\cup\left(\partial(M_1,N_1)\times I\right).$$
Thus, there is a cylindrical portion of the boundary.  Consider the quotient space $(W',Z')$ of $(W,Z)$ obtained via the identification $(x,t)\sim(x,t')$ for all $x\in\partial M_1$ and $t,t'\in I$.  The space $(W',Z')$ is diffeomorphic to $(W,Z)$, but we have
$$\partial(W',Z') = (M_1,N_1)\cup_{\partial(M_1,N_1)}\overline{(M_2,N_2)}.$$
We refer to $(W',Z')$ as a \emph{lensed cobordism}.  An example of a lensed cobordism is the submanifold $W'$ co-bounded by two Seifert surfaces for a knot $K$ in $S^3$ that are disjoint in their interior.  If $W = M_1\times I$, then we call $W'$ a \emph{product lensed cobordism}.  An example of a product lensed cobordism is the submanifold $W'$ co-bounded by two pages of an open-book decomposition on an ambient manifold $X$.  See Figure~\ref{fig:lensed_tangle} below for examples of lensed cobordisms between surfaces that contain 1--dimensional cobordisms as neat submanifolds.

We offer the following two important remarks regarding our use of lensed cobordisms.  

\begin{remark}
\label{rmk:lensed}
	Throughout this article, we will be interested in cobordisms between manifolds with boundary.  For this reason, lensed cobordisms are naturally well-suited for our purposes.  However at times, we will be discussing cobordisms between closed manifolds (e.g. null-cobordisms).  In this case, lensed cobordisms do not make sense.  We request that the reader remember to drop the adjective `lensed', upon consideration of such cases.  For example, if $(M,N)$ is any manifold pair with $N\subset\Int(M)$ closed, then for the product lensed cobordism $(M,N)\times I$, we have that $M\times I$ is lensed, but $N\times I$ is not.
\end{remark}

\begin{remark}
\label{rmk:no_Morse}
	Lensed cobordisms do not admit Morse functions where $(M_1,N_1)$ and $(M_2,N_2)$ represent distinct level sets, since $(M_1,N_1)\cap(M_2,N_2)\not=\emptyset$.  However, the manifold pair
	$$(W'',Z'') = (W',Z')\setminus\nu(\partial(M_1,N_1))$$
	does admit such a function and is trivially diffeomorphic to $(W',Z')$: We think of $(W'',Z'')$ as being formed by `indenting' $(W',Z')$ by removing $\nu(\partial (M_1,N_1))$.  Note that there is a natural identification of $(W'',Z'')$ with the original (ordinary) cobordism $(W,Z)$.  Since a generic Morse function on the cobordism $W''$ will not have critical points on its boundary, there is no loss of information here. We will have this modification in mind when we consider Morse functions on lensed cobordisms $(W',Z')$, which we will do throughout the paper.
	This subtlety illustrates that lensed cobordisms are unnatural in a Morse-theoretic approach to manifold theory, but we believe they are more natural in a trisection-theoretic approach.
\end{remark}

\subsection{Compressionbodies}
\label{subsec:Compression}
\ 

Given a surface $\Sigma$ and a collection $\alpha$ of simple closed curves on $\Sigma$, let $\Sigma^\alpha$ denote the surface obtained by surgering $\Sigma$ along $\alpha$.  Let $H$ denote the three-manifold obtained by attaching a collection $\frak h_\alpha$ of three-dimensional 2--handles to $\Sigma\times[-1,1]$ along $\alpha\times\{1\}$, before filling in any resulting sphere components with balls.  As discussed in Remark~\ref{rmk:lensed}, in the case that $\Sigma$ has nonempty boundary, we quotient out by the vertical portion of the boundary and view $H$ as a lensed cobordism from $\partial_+H = \Sigma$ to $\partial_-H=\Sigma^\alpha$.  Considering $H$ as an oriented manifold yields the following decomposition:
$$\partial H_\alpha = \partial_+H\cup_{\partial(\partial_+H)}\overline{\partial_-H}.$$

The manifold $H_\alpha$ is called a \emph{(lensed) compressionbody}.  A collection $\Dd$ of disjoint, neatly embedded disks in a compressionbody $H$ is called a \emph{cut system} for $H$ if $H\setminus\nu(\Dd)\cong (\partial_-H)\times I$ or $H\setminus\nu(\Dd)\cong B^3$, according with whether $\partial(\partial_+H)=\partial(\partial_-H)$ is nonempty or empty.
A collection of essential, simple closed curves on $\partial_+H$ is called a \emph{defining set of curves} for $H$ if it is the boundary of a cut system for $H$.

In order to efficiently discuss compressionbodies $H$ for which $\partial_-H$ is disconnected, we will introduce the following terminology.

\begin{definition}
	Given $m\in\N_0$, an \emph{ordered partition} of $m$ is a sequence $\bold m = (m_1,\ldots, m_n)$ such that $m_j\in\N_0$ and $\sum m_j = m$.  We say that such an ordered partition is of type $(m,n)$.  If $m_j>0$ for all $j$, then the ordered partition is called \emph{positive} and is said to be of type $(m,n)^+$.  If $m_j = m'$ for all $j$, then the ordered partition is called \emph{balanced}.
\end{definition}

Let $\Sigma_g$ denote the closed surface of genus $g$, and let $\Sigma_{g,f}$ denote the result of removing $f$ disjoint, open disks from $\Sigma_g$.   A surface $\Sigma$ with $n>1$ connected components is called \emph{ordered} if there is an ordered partition  $\bold p = (p_1,\ldots, p_n)$ of $p\in\N_0$ and a positive ordered partition $\bold f = (f_1,\ldots, f_n)$ of $f\in\N$ such that 
$$\Sigma = \Sigma_{p_1,f_1}\sqcup\cdots\sqcup\Sigma_{p_n,f_n}.$$
We denote such an ordered surface by $\Sigma_{\bold p, \bold f}$, and we consider each $\Sigma_{p_j,f_j}\subset\Sigma_{\bold p,\bold f}$ to come equipped with an ordering of its $f_j$ boundary components, when necessary.  Note that we are requiring each component of the \emph{disconnected} surface $\Sigma_{\bold p,\bold f}$ to have boundary.

\begin{figure}[h!]
\begin{subfigure}{.45\textwidth}
  \centering
  \includegraphics[width=.8\linewidth]{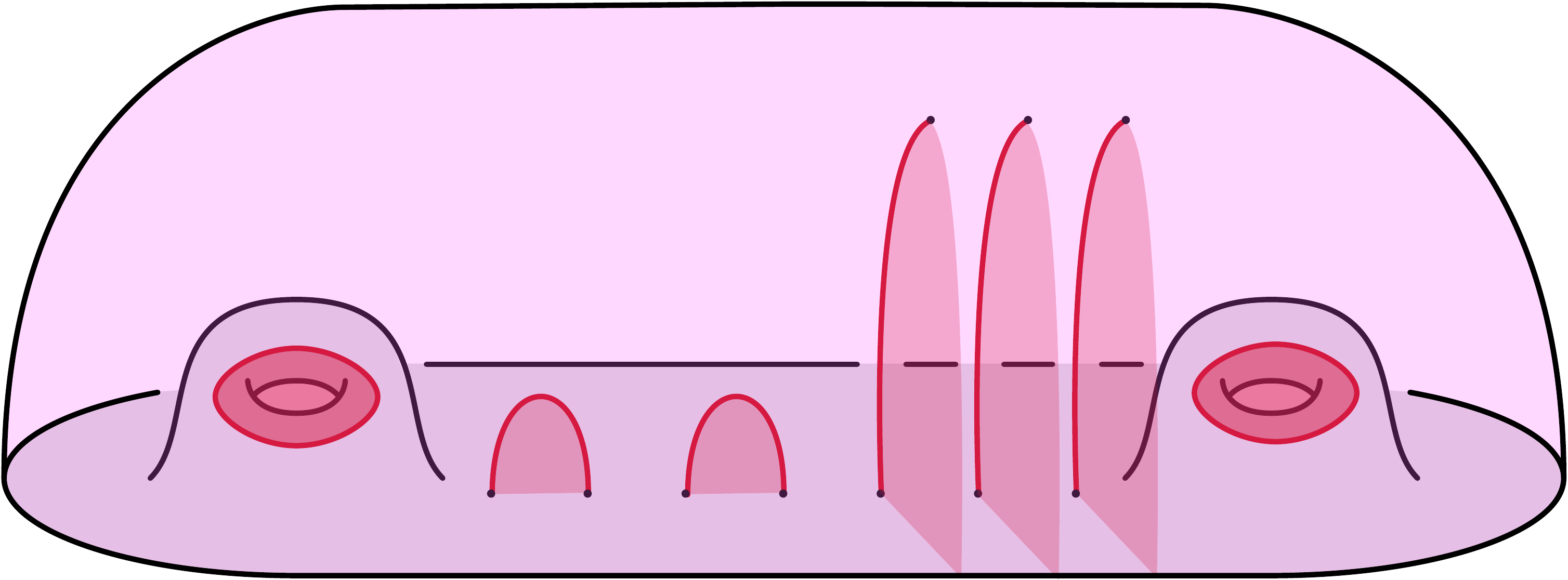}
  \caption{$(H_{2,0,1},\Tt_{2,3})$}
  \label{fig:lensed_tangle1}
\end{subfigure}%
\begin{subfigure}{.275\textwidth}
  \centering
  \includegraphics[width=.8\linewidth]{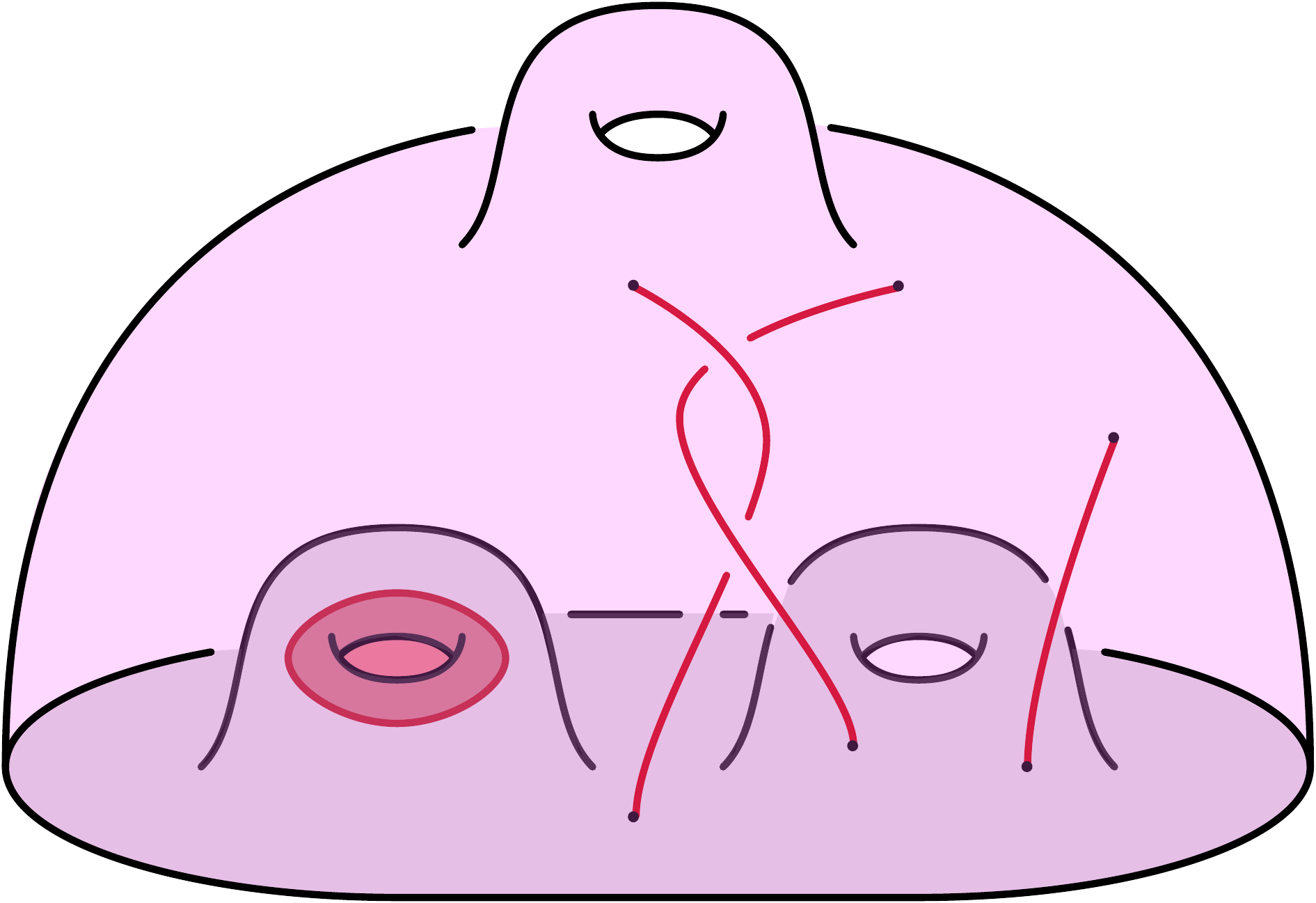}
  \caption{$(H_{2,1,1},\Tt_{0,3})$}
  \label{fig:lensed_tangle2}
\end{subfigure}%
\begin{subfigure}{.275\textwidth}
  \centering
  \includegraphics[width=.8\linewidth]{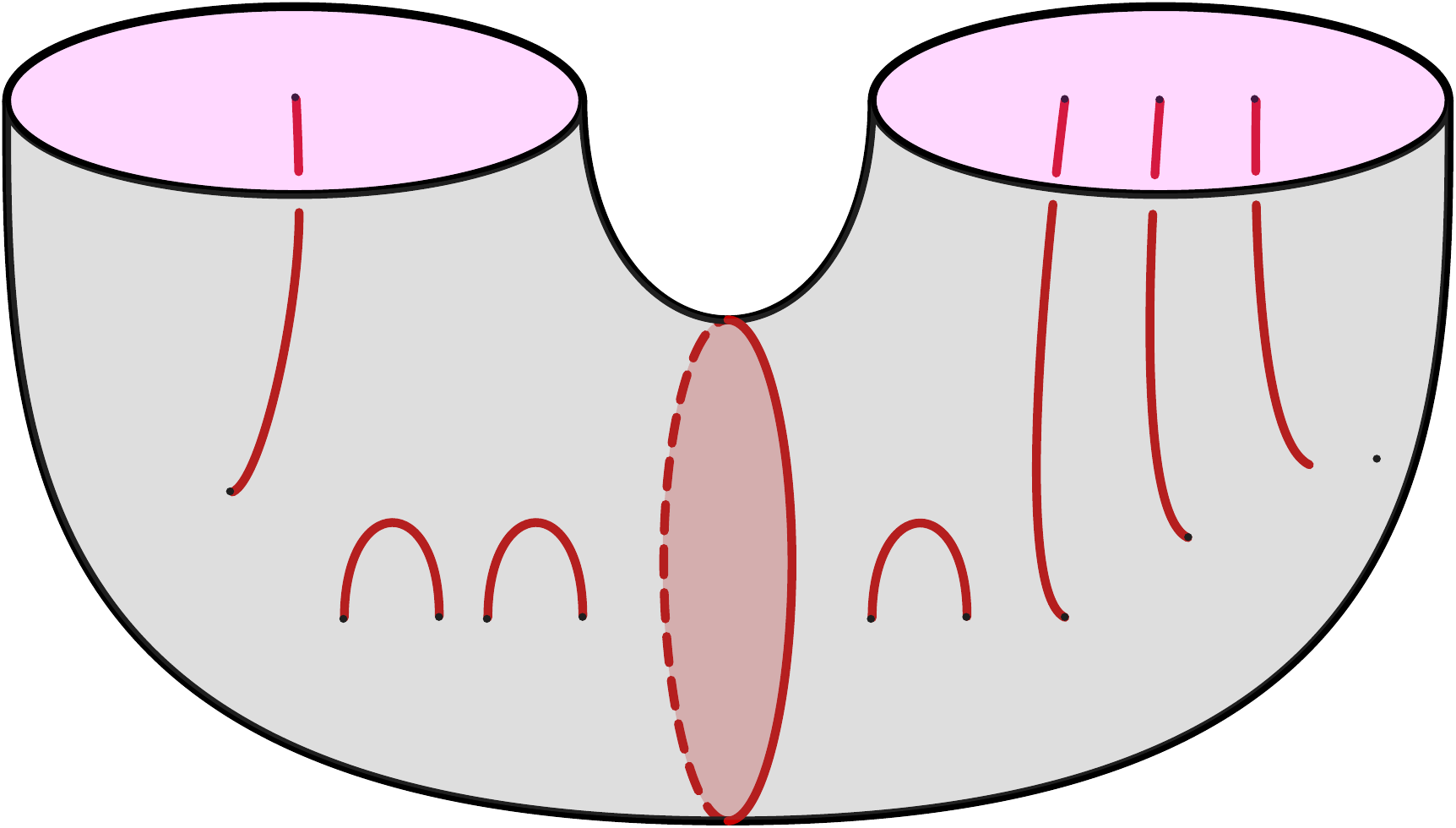}
  \caption{$(H_{0,(0,0),(1,1)},\Tt_{3,(1,3)})$}
  \label{fig:lensed_tangle3}
\end{subfigure}%
\caption{Three examples of trivial tangles inside lensed compressionbodies.}
\label{fig:lensed_tangle}
\end{figure}

Let $H_{g,\bold p,\bold f}$ denote the lensed compressionbody satisfying 
	\begin{enumerate}
		\item $\partial_+ H_{g,\bold p,\bold f} = \Sigma_{g,f}$, and
		\item $\partial_- H_{g,\bold p,\bold f} = \Sigma_{\bold p,\bold f}$.
	\end{enumerate}
If $\alpha$ is a defining set for such a compressionbody, then $\alpha$ consists of  $(n-1)$ separating curves and $(g-p)$ non-separating curves.  See Figure~\ref{fig:lensed_tangle} for three examples of lensed compressionbodies, ignoring for now the submanifolds.
Let $H_{p_j,f_j}$ denote the product lensed cobordism from $\Sigma_{p_j,f_j}$ to itself, and let
$$H_{\bold p,\bold f} = \sqcup_{j=1}^\infty H_{p_j,f_j}.$$
We refer to $H_{\bold p,\bold f}$ as a \emph{spread}.

A lensed compressionbody $H$ admits a Morse function $\Phi\colon H\to [-1,3]$, which, as discussed in Remark~\ref{rmk:no_Morse}, is defined on $H\setminus\nu(\partial(\partial_+H))$, such that $\Phi(\partial_+H)=-1$, $\Phi(\partial_-H) = 3$, and $\Phi$ has $(n-1)+(g-p)$ critical points, all of index two, and all lying in $\Phi^{-1}(2)$.  We call such a $\Phi$ a \emph{standard} Morse function for $H$.

For a positive natural number $I$, we let $\bold x_I\subset \Sigma_{g,f}$ denote a fixed collection of $I$ marked points.

\subsection{Heegaard splittings and Heegaard-page splittings}
\label{subsec:Heegaard}
\ 

Let $M$ be an orientable three-manifold.  A \emph{Heegaard splitting} of $M$ is a decomposition 
$$M = H_1\cup_\Sigma\overline H_2,$$
where $\Sigma\subset M$ is a neatly embedded surface $\Sigma_{g,f}$, and each $H_i$ is a lensed compressionbody $H_{g,\bold p, \bold f}$ with $\partial_+H_i = \Sigma$.  It follows that 
$$\partial M = \overline{\partial_-H_1}\cup_{\partial\Sigma}\partial_-H_2.$$
We denote the Heegaard splitting by $(\Sigma; H_1, H_2)$, and we call it a $(g;\bold p, \bold f)$--splitting, in reference to the relevant parameters.  Note that our notion of Heegaard splitting restricts to the usual notion when $M$ is closed, but is different from the usual notion when $M$ has boundary.  Our Heegaard splittings are a special type of sutured manifold decomposition.  Since each of the $H_i$ is determined by a collection $\alpha_i$ of curve on $\Sigma$, the Heegaard splitting, including $M$ itself, is determined by the triple $(\Sigma; \alpha_1, \alpha_2)$, which is called a \emph{Heegaard diagram} for $M$.

\begin{remark}
\label{rmk:ordering1}
	Note that we have defined Heegaard splittings so that the two compressionbodies are homeomorphic, since this is the only case we will be interested in.  Implicit in the set-up are matching orderings of the components of the $\partial_-H_i$ in the case that $|\partial_-H_i|>1$.  This will be important when we derive a Heegaard-page structure from a Heegaard splitting below.  See also Remark~\ref{rmk:ordering2}
\end{remark}

A Heegaard splitting $(\Sigma; H_1, H_2)$ with $H_i\cong H_{g, \bold p, \bold f}$ is called \emph{$(m,n)$--standard}  if there are cut systems $\Dd_i = \{D_i^l\}_{l=1}^{n-1+g-p}$ for the $H_i$ such that 
\begin{enumerate}
	\item For $1\leq l\leq n-1$, we have $\partial D_1^l = \partial D_2^l$, and this curve is separating;
	\item For $n\leq l\leq m+n-1$,  we have $\partial D_1^l = \partial D_2^l$, and this curve is non-separating; and
	\item For $m+n\leq l,l'  \leq g-p$, we have $|\partial D_1^l \cap \partial D_2^{l'}|$ given by the Kronecker delta $\delta_{l,l'}$, and the curves $\partial D_1^l$ and $\partial D_2^l$ are non-separating.
\end{enumerate}
A Heegaard diagram $(\Sigma; \alpha_1, \alpha_2)$ is called \emph{$(m,n)$--standard} if $\alpha_i = \partial\Dd_i$ for cut systems $\Dd_i$ satisfying these three properties.  See Figure~\ref{fig:bridge_double1} for an example. In a sense, a standard Heegaard splitting is a ``stabilized double''.  The following lemma makes this precise.

\begin{lemma}\label{lemma:std_Heeg}
	Let $(\Sigma; H_1, H_2)$ be a $(m,n)$--standard Heegaard splitting with $H_i\cong H_{g,\bold p,\bold f}$.  Then, 
	$$(\Sigma; H_1, H_2) = \left(\mathop\#\limits_{j=1}^n((\Sigma')^j; (H'_1)^j, (H'_2)^j)\right)\#(\Sigma''; H''_1, H''_2),$$
	where $(H'_1)^j\cong(H'_2)^j\cong H_{p_j,f_j}$, for each $j = 1, \ldots, n$, and $(\Sigma'';H_1'',H_2'')$ is the standard genus $g-p$ Heegaard surface for $\#^m(S^1\times S^2)$.
\end{lemma}

\begin{proof}
	Consider the $n$ regions of $\Sigma$ cut out by the $n-1$ separating curves that bound in each compressionbody.  After a sequence of handleslides, we can assume that all of the non-separating curves of the $\alpha_i$ are contained in one of these regions.  Once this is arranged, there is a separating curve $\delta$ in $\Sigma\setminus\nu(\alpha_1\cup\alpha_2)$ that cuts off a subsurface $\Sigma''$ such that $\Sigma''$ has only one boundary component (the curve $\delta$) and $g(\Sigma'') = g-p$.  Since $\delta$ bounds in each of $H_1$ and $H_2$, we have that $(\Sigma;H_1,H_2) = (\Sigma';H_1',H_2')\#_\delta(\Sigma'';H_1'',H_2'')$, such that the latter summand is the standard splitting of $\#^m(S^1\times S^2)$, as claimed.  The fact that the regions of $\Sigma'$ cut out by the separating curves that bound in both handlebodies contain no other curves of the $\alpha_i$ means that these curves give the connected sum decomposition
	$$(\Sigma'; H_1', H_2') = \left(\mathop\#\limits_{j=1}^n((\Sigma')^j; (H'_1)^j, (H'_2)^j)\right)$$
	that is claimed.
\end{proof}

Let $H_1$ and $H_2$ be two copies of $H_{g, \bold p, \bold f}$, and let $h\colon\partial_+H_1\to \partial_+H_2$ be a diffeomorphism.  Let $Y$ be the closed three-manifold obtained as the union of $H_1$ and $H_2$ along their boundaries such that $\partial_+H_1$ and $\partial_+H_2$ are identified via $h$ and $\partial_-H_1$ and $\partial_-H_2$ are identified via the identity on $\partial_-H_{g, \bold p, \bold f}$. The manifold $Y$ is called a \emph{Heegaard double} of $H_{g, \bold p, \bold f}$ along $h$.  
We say that a Heegaard double $Y$ is \emph{$(m,n)$--standard} if the Heegaard splitting $(\Sigma; H_1, H_2)$ is $(m,n)$--standard. 
Let $Y_{g,\bold p,\bold f}$ denote the Heegaard double of a standard Heegaard splitting whose compressionbodies are $H_{g,\bold p, \bold f}$.  The uniqueness of $Y_{g,\bold p,\bold f}$ is justified by the following lemma, which is proved with slightly different terminology as Corollary~14 of~\cite{CasGayPin_18_Diagrams-for-relative-trisections}.

\begin{lemma}
\label{lem:HeegDouble}
	Let $M = H_1\cup_\Sigma\overline H_2$ be a standard Heegaard splitting with $H_i\cong H_{g,\bold p,\bold f}$.  Then there is a unique (up to isotopy rel-$\partial$) diffeomorphism $\Id_{(M,\Sigma)}\colon \partial_-H_1\to\partial_-H_2$ such that the identification space $M/_{x\sim\Id_{(M,\Sigma)}(x)}$, where $x\in\partial_-H_1$, is diffeomorphic to the standard Heegaard double $Y_{g,\bold p,\bold f}$.
\end{lemma}

We now identify the total space of a standard Heegaard double. Let $\Id_{p_j,f_j}\colon\Sigma_{p_j,f_j}\to\Sigma_{p_j,f_j}$ be the identity map, and let $M_{\Id_{p_j,f_j}}$ be the total space of the abstract open-book $(\Sigma_{p_j,f_j},\Id_{p_j,f_j})$.
See Subsection~\ref{subsec:OBD}, especially Example~\ref{ex:Id_obd}, for definitions and details regarding open-book decompositions.

\begin{lemma}\label{lemma:k-value}
	There is a decomposition
	$$ Y_{g,\bold p, \bold f} = \left(\mathop\#_{j=1}^nM_{\Id_{p_j,f_j}}\right)\#(\#^m(S^1\times S^2)),$$
	such that $\Sigma$ restricts to a page in each of the first $n$ summands and to a Heegaard surface in the last summand. Moreover,
	$$M_{\Id_{p_j,f_j}}\cong\#^{2p_j+f_j-1}(S^1\times S^2),$$
	so $Y_{g,\bold p, \bold f}\cong \#^k(S^1\times S^2)$, with $k = 2p+f-n+m$.
\end{lemma}

\begin{proof}
	Consider the abstract open-book $(\Sigma_{p_j,f_j}, \Id_{p_j,f_j})$, and let $M_{\Id_{p_j,f_j}}$ denote the total space of this abstract open-book.  Pick two pages, $P_1$ and $P_2$, of the open-book decomposition of $M_{\Id_{p_j,f_j}}$, and consider the two lensed cobordisms co-bounded thereby.
	Each of these pieces is a handlebody of genus $2p_j+f_j-1$, since it is diffeomorphic to $H_{p_j,f_j}$.  A collection of arcs decomposing the page into a disk give rise to a cut system for either handlebody, but these cut systems have the same boundary.  The object described is a genus $2p_j+f_j-1$ (symmetric) Heegaard splitting for $\#^{2p_j+f_j-1}(S^1\times S^2)$.
	The rest of the proof follows from Lemma~\ref{lemma:std_Heeg}.
\end{proof}

\begin{figure}[h!]
\begin{subfigure}{.5\textwidth}
  \centering
  \includegraphics[width=.8\linewidth]{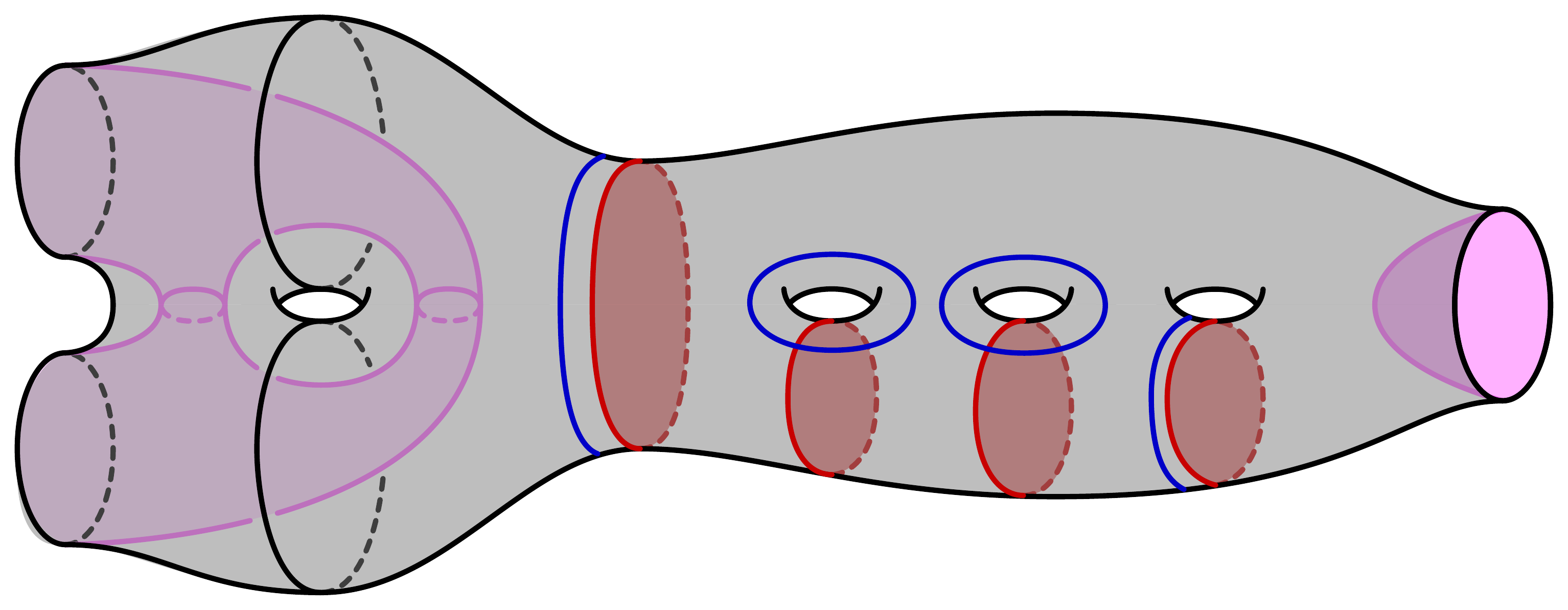}
  \caption{}
  \label{fig:bridge_double1}
\end{subfigure}%
\begin{subfigure}{.5\textwidth}
  \centering
  \includegraphics[width=.8\linewidth]{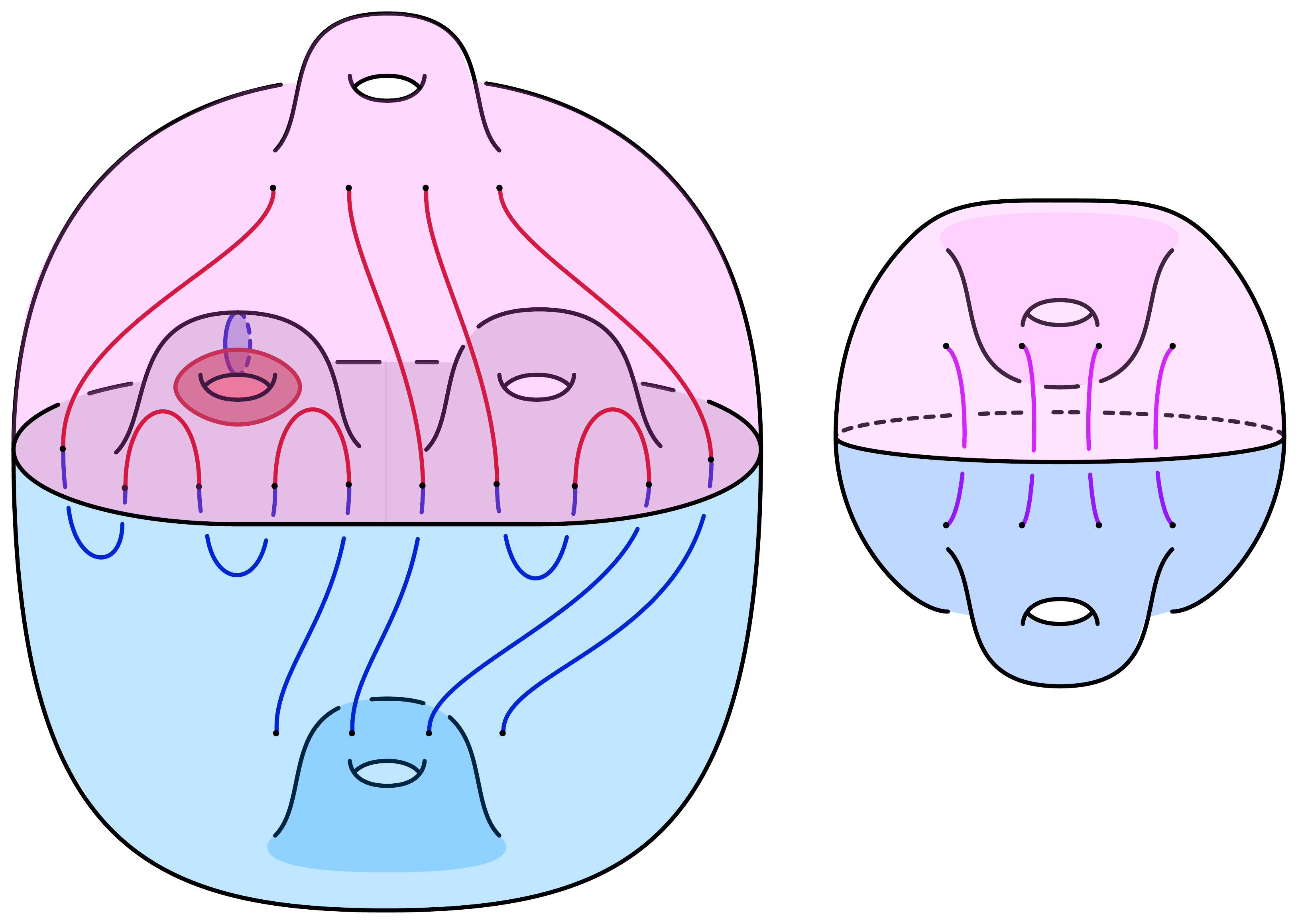}
  \caption{}
  \label{fig:bridge_double2}
\end{subfigure}%
\caption{(A) A $(1,2)$--standard Heegaard diagram for the standard Heegaard double $Y_{4,(1,0),(2,1)}$.  (B) A schematic showing the standard Heegaard double $Y_{2,1,1}$, containing a $(3,4)$--bridge splitting for an unlink.  The unlink has no flat components and four vertical components.}
\label{fig:bridge_double}
\end{figure}

Let $Y$ be a standard Heegaard double. We consider the lensed compressionbodies $H_1$ and $H_2$ as embedded submanifolds of $Y$ in the following way, which is a slight deviation from the way they naturally embed in the Heegaard double.  For $i=1,2$, let $P_i^j$ denote the result of a slight isotopy of $\partial_-H_i^j$ into $H_i$ along the product structure induced locally by the lensed cobordism structure of $H_i$.  Let $Y_1^j$ denote lensed product cobordism co-bounded by $P_1^j$ and $P_2^j$.  In this way, we think of the Heegaard double $Y$ as divided into three regions: $H_1$, $H_2$, and $\sqcup_jY_1^j$, each of whose connected components is a lensed compressionbody.  The union of $H_1$ and $H_2$ along their common, southern boundary, which we denote by $\Sigma$ is a standard Heegaard splitting, and each $Y_1^j$ is the product lensed cobordism $H_{p_j,f_j}$.  See Figure~\ref{fig:bridge_double2}, as well as Figure~\ref{fig:ortn_conv}, for a schematic illustration of this structure. We call this decomposition a \emph{(standard) Heegaard-page structure} and note that it is determined by the Heegaard splitting data $(\Sigma,H_1,H_2)$, by Lemma~\ref{lem:HeegDouble}.

\subsection{Trivial tangles}
\label{subsec:Tangles}
\ 

A \emph{tangle} is a pair $(H,\Tt)$, where $H$ is a compressionbody and $\Tt$ is a collection of neatly embedded arcs in $H$, called \emph{strands}. Let $\Phi$ be a standard Morse function for $H$.  After an ambient isotopy of $\Tt$ rel-$\partial$, we can assume that $\Phi$ restricts to $\Tt$ to give a Morse function $\Phi|_\Tt\colon\Tt\to[-1,3]$ such that each local maximum of $\Tt$ maps to $1\in[-1,3]$ and each local minimum maps to $0\in[-1,3]$.  We have arranged that $\Phi$ be self-indexing on $H$ and when restricted to $\Tt$.

A strand $\tau\subset \Tt$ is called \emph{vertical} if $\tau$ has no local minimum or maximum with respect to $\Phi|_\Tt$ and is called \emph{flat strand} if $\tau$ has a single local extremum, which is a maximum.  Note that vertical strands have one boundary point in each of $\partial_+H$ and $\partial_-H$, while flat strands have both boundary points in $\partial_+H$.  A tangle $\Tt$ is called \emph{trivial} if it is isotopic rel-$\partial$ to a tangle all of whose strands are vertical or flat.
Such a tangle with $b$ flat strands and $v$ vertical strands is called an \emph{$(b,v)$--tangle}, with the condition that it be trivial implicit in the terminology.
More precisely, if $H\cong H_{g,\bold p,\bold f}$, then we have an ordered partition of the vertical strands determined by which component $\Sigma_{p_j,b_j}$ of $\partial_-H\cong\Sigma_{\bold p,\bold f}$ contains the top-most endpoint of each vertical strand, and we can more meticulously describe $\Tt$ as an \emph{$(b,\bold v)$--tangle}.
See Figure~\ref{fig:lensed_tangle} for three examples of trivial tangles in lensed compressionbodies.

\begin{remark}
\label{rmk:br_strand}
	In this paper, any tangle $(H,\Tt)$ with $\partial_+ H$ disconnected will not contain flat strands.  Moreover, such an $H$ will always be a spread $H\cong H_{\bold p,\bold v}$.  Therefore, we will never partition the flat strands of $\Tt$.
\end{remark}

There is an obvious model tangle $(H_{g,\bold p,\bold f},\Tt_{b,\bold v})$ that is a lensed cobordism from $(\Sigma_{g,\bold f},\bold x_{2b+v})$ to $(\Sigma_{\bold p,\bold f},\bold y_{\bold v})$ in which the first $2b$ points of $\bold x_{2b+v}$ are connected by slight push-ins of arcs in $\Sigma_{g,\bold f}$, and the final $v$ rise vertically to $\Sigma_{\bold p,\bold f}$, as prescribed by the standard height function on $H_{g,\bold p,\bold f}$ and the ordered partitions.  The points $\bold x_{2b+v}$ are called \emph{bridge points}. A pair $(H,\Tt)$ is determined up to diffeomorphism by the parameters $g$, $b$, $\bold p$, $\bold f$, and $\bold v$, and we refer to any tangle with these parameters as a $(g,b;\bold p,\bold f,\bold v)$--tangle.  Note that this diffeomorphism can be assumed to be supported near $\partial_+H$ and can be understood as a braiding of the bridge points $\bold x_{2b+v}$.  For this reason, we consider trivial tangles up to isotopy rel-$\partial$, and we think of each such tangle as having a fixed identification of the subsurface $(\Sigma_{g,b},\bold x_{2b+v})$ of its boundary.

Let $\tau$ be a strand of a trivial tangle $(H,\Tt)$.  Suppose first that $\tau$ is flat.  A \emph{bridge semi-disk for $\tau$} is an embedded disk $D_\tau\subset H$ satisfying $\partial D_\tau = \tau\cup\tau^*$, where $\tau^*$ is an arc in $\partial_+H$ with $\partial\tau^* = \partial \tau$, and $D_\tau\cap\Tt = \tau$.  The arc $\tau^*$ is called a \emph{shadow} for $\tau$.  Now suppose that $\tau$ is vertical.  A \emph{bridge triangle for $\tau$} is an embedded disk $D_\tau\subset H$ satisfying $\partial D_\tau = \tau\cup\tau^*\cup\tau^-$, where $\tau^*$ (respectively, $\tau^-$) is an arc in $\partial_+H$ (respectively, $\partial_-H$) with one endpoint coinciding with an endpoint of $\tau$ and the other endpoint on $\partial(\partial_+H)$, coinciding with the other endpoint of $\tau^-$ (respectively, $\tau^*$), and $D_\tau\cap\Tt =\tau$. 

\begin{remark}
	Note that the existence of a bridge triangle for a vertical strand $\tau$ requires that $\partial_-H$ have boundary; there is no notion of a bridge disk for a vertical strand in a compressionbody co-bounded by closed surfaces.  In this paper, if $\partial_+H$ is ever closed, $H$ will be a handlebody and will not contain vertical strands, so bridge semi-disks and triangles will always exist for trivial tangles that we consider.
\end{remark}

Given a trivial tangle $(H,\Tt)$, a \emph{bridge disk system for $\Tt$} is a collection $\Delta$ of disjoint disks in $H$, each component of which is a bridge semi-disk or triangle for a strand of $\Tt$, such that $\Delta$ contains precisely one bridge semi-disk or triangle for each strand of $\Tt$.

\begin{lemma}
	Let $(H,\Tt)$ be a trivial tangle such that either $\partial_+H$ has nonempty boundary or $\Tt$ contains no vertical strands.  Then, there is a bridge disk system $\Delta$ for $\Tt$.
\end{lemma}

\begin{proof}
	There is a diffeomorphism from $(H,\Tt)$ to $(H_{g,\bold p,\bold f},\Tt_{b,\bold v})$, as discussed above.  This latter tangle has an obvious bridge disk system:  The `slight push-in' of each flat strand sweeps out a disjoint collection of bridge semi-disks for these strands, while the points $x\in\bold x_{2b+v}$ corresponding to vertical strands can be connected to $\partial \Sigma_{g,\bold f}$ via disjoint arcs, the vertical traces of which are disjoint bridge triangles for the vertical strands.  Pulling back this bridge system to $(H,\Tt)$ using the inverse diffeomorphism completes the proof.
\end{proof}

We will refer to a $(0,\bold v)$--tangle as a \emph{vertical $\bold v$--tangle} and to a $(b,0)$--tangle as a \emph{flat $b$--tangle}.  In the case that $\Tt$ is a vertical tangle in a spread $H\cong H_{\bold p,\bold f}$, we call $\Tt$ a \emph{$\bold v$--thread} and call the pair $(H,\Tt)$ a \emph{$(\bold p,\bold f,\bold v)$--spread}.  Note that a $(p,f,v)$--spread is simply a lensed geometric (surface) braid; in particular, a $(0,1,v)$--spread is a lensed geometric braid $(D^2\times I,\beta)$.

\subsection{Bridge splittings}
\label{subsec:Bridge}
\ 

Let $K$ be a neatly embedded, one-manifold in a three-manifold $M$.  A \emph{bridge splitting} of $K$ is a decomposition 
$$(M,K) = (H_1,\Tt_1)\cup_{(\Sigma,\bold x)}\overline{(H_2,\Tt_2)},$$
where $(\Sigma; H_1, H_2)$ is a Heegaard splitting for $M$ and $\Tt_i\subset H_i$ is a trivial tangle.  If $\Tt_1$ is a trivial $(b,\bold v)$--tangle, then we require that $\Tt_2$ be a trivial $(b,\bold v)$--tangle, and we call the decomposition a \emph{$(g,\bold p,\bold f; b, \bold v)$--bridge splitting}. A one-manifold $K\subset M$ is in \emph{$(b,\bold v)$--bridge position} with respect to a Heegaard splitting of $M$ if $K$ intersects the compressionbodies $H_i$ as a  $(b,\bold v)$--tangle.

\begin{remark}
\label{rmk:ordering2}
	As  we have assumed a correspondence between the components of the $\partial_-H_i$ (see Remark~\ref{rmk:ordering1}, we can require that the partitions of the vertical strands of the $\Tt_i$ respect this correspondence.  This is the sense in which both $\Tt_i$ are $(b,\bold v)$--tangles.  This will be important when we turn a bridge splitting into a bridge-braid decomposition below.
\end{remark}

Consider the special case that $M$ is the trivial lensed cobordism between $\partial_-H_1$ and $\partial_-H_2$ and $K\subset M$ is a $v$--braid -- i.e., isotopic rel-$\partial$ so that it intersects each level surface of the trivial lensed cobordism transversely.  (Note that the $\partial_-H_i$ are necessarily connected, since $\Sigma$ is.)  If each of $\Tt_i=H_i\cap K$ is a trivial $(g,b;p,f,v)$--tangle, we call the union $(H_1,\Tt_1)\cup_{\Sigma,\bold x}\overline{(H_2,\Tt_2)}$ an \emph{$b$--perturbing of a $v$--braid}.

More generally, we say that a bridge splitting is \emph{standard} if the underlying Heegaard splitting $M = H_1\cup_\Sigma\overline{H_2}$ is standard (as defined in Subsection~\ref{subsec:Heegaard} above) and there are collections of bridge semi-disks $\Delta_i$ for the flat strands of the tangles $\Tt_i$ whose corresponding shadows $\Tt_i^*$ have the property that $\Tt_1^*\cup_\bold x\Tt_2^*$ is an embedded collection of polygonal arcs and curves.  As a consequence, if $(M,K)$ admits a standard bridge splitting, then $K$ is the split union of a an unlink (with one component corresponding to each polygonal curve of shadow arcs) with a braid (with one strand corresponding to each polygonal arc of shadows arcs).  As described in Lemma~\ref{lemma:std_Heeg}, the ambient manifold $M$ is a connected sum of copies of surfaces cross intervals and copies of $S^1\times S^2$.

Let $(H_1,\Tt_1)$ and $(H_2,\Tt_2)$ be two copies of the model tangle $(H_{g,\bold p,\bold f},\Tt_{b,\bold v})$, and let
$$h\colon\partial_+(H_1,\Tt_1)\to\partial_+(H_2,\Tt_2)$$
be a diffeomorphism.  Let $(Y,L)$ be the pair obtained as the union of $(H_1,\Tt_1)$ and $(H_2,\Tt_2)$, where the boundaries $\partial_+(H_i,\Tt_i)$ are identified via $h$ and the boundaries $\partial_-(H_i,\Tt_i)$ are identified via the identity map of $\partial_-(H_{g,\bold p,\bold f},\Tt_{b,\bold v})$.  We call the pair $(Y,L)$ a \emph{bridge double} of $(H_{g,\bold p,\bold f},\Tt_{b,\bold v})$ along $h$. Note that a component of $L$ can be referred to as \emph{flat} or \emph{vertical} depending on whether or not is is disjoint from $\partial_-H_i$.  We say that the bridge double is \emph{standard} if
\begin{enumerate}
	\item the bridge splitting $(H_1,\Tt_1)\cup_{(\Sigma,\bold x)}\overline{(H_2,\Tt_2)}$ is standard, and
	\item $L$ has exactly $v$ vertical components.  In other words, each component of $L$ hits $\partial_-H_i$ exactly once or not at all.
\end{enumerate}
Let $(Y_{g,\bold p,\bold f},L_{b,\bold v})$ denote the bridge double of a standard bridge splitting with $(H_i,\Tt_i)\cong (H_{g,\bold p,\bold f},\Tt_{b,\bold v})$.  The uniqueness of the \emph{standard bridge double} $(Y_{g,\bold p,\bold f},L_{b,\bold v})$ is given by the following lemma, which generalizes Lemma~\ref{lem:HeegDouble} above.

\begin{lemma}
\label{lem:BridgeDouble}
	Let $(M,K)=(H_1,\Tt_1)\cup_{(\Sigma,\bold x)}\overline{(H_2,\Tt_2)}$ be a standard bridge splitting with $(H_i,\Tt_i)\cong (H_{g,\bold p,\bold f},\Tt_{b,\bold v})$.  Then there is a unique (up to isotopy rel-$\partial$) diffeomorphism $\Id_{(M,K,\Sigma)}\colon \partial_-(H_1,\Tt_1)\to\partial_-(H_2,\Tt_2)$ such that the identification space $(M,K)/_{x\sim\Id_{(M,K,\Sigma)}(x)}$, where $x\in\partial_-(H_1,\Tt_1)$, is diffeomorphic to the standard bridge double $(Y_{g,\bold p,\bold f},L_{b,\bold v})$.
\end{lemma}

\begin{proof}
	Let $(M,K)$ be a standard bridge splitting.  Suppose $(Y,L)$ is the bridge double obtained via the gluing map $\Id_{(M,\Sigma)}\colon\partial_-H_1\to\partial_-H_2$, which is determined uniquely up to isotopy rel-$\partial$ by Lemma~\ref{lem:HeegDouble}.  The claim that must be justified is that $\Id_{(M,\Sigma)}$ is unique up to isotopy rel-$\partial$ when considered as a map of pairs $\partial_-(H_1,\bold y_1)\to\partial_-(H_2,\bold y_2)$
	
	Criterion (2) of a standard bridge double above states that $K$ must close up to have $v$ vertical components, where $v$ is the number of vertical strands in the splitting $(M,K)$.  It follows that $\Id_{(M,\Sigma)}$ restricts to the identity permutation as a map $\bold y_1\to\bold y_2$ -- i.e. the end of a vertical strand in $\bold y_1$ must get matched with the end of the same strand in $\bold y_2$.
	
	Let $(M,K)^\circ$ denote the pair obtained by deperturbing the vertical arcs of $K$ so that they have no local extrema, then removing tubular neighborhoods of them.
	Note that $(M,K)^\circ$ is a standard bridge splitting (of the flat components of $K$) of type $(g,\bold p,\bold f';b',0)$.
	The restriction $\Id_{(M,\Sigma)}^\circ$ to $(\partial_-H_1)^\circ$ is the identity on $\partial(\partial_-H_1)^\circ$, so we can apply Lemma~\ref{lem:HeegDouble} to conclude that $\Id_{(M,\Sigma)}^\circ$ is unique up to isotopy rel-$\partial$.
	Since $\Id_{(M,\Sigma)}^\circ$ extends uniquely to a map $\Id_{(M,\Sigma,K)}$ of pairs, as desired, we are done.	
\end{proof}

Finally, consider a standard bridge double $(Y_{g,\bold p,\bold f},L_{b,\bold v})$, and recall the Heegaard-page structure on $Y_{g,\bold p,\bold f}$.  This induces a structure on $L$ that we call a \emph{bridge-braid structure}.  In particular, we have
\begin{enumerate}
	\item $\Tt_i = L\cap H_i$ is a $(b,\bold v)$--tangle, and
	\item $\beta_1^j = L\cap Y_1^j$ is a $v_j$--braid.
\end{enumerate}

\subsection{Disk-tangles}
\label{subsec:DiskTangles}
\ 

Let $Z_k$ denote the four-dimensional 1--handlebody $\natural^k(S^1\times B^3)$.  Given nonnegative integers $p$, $f$, $m$, and $n$ such that $k=2p+f-1+m$ and ordered partitions $\bold p$ and $\bold f$ of $p$ and $f$ of length $n$, there is a natural way to think of $Z_k$ as a lensed cobordism from the spread $Y_1 = H_{\bold p,\bold f}$ to the $(m,n)$--standard Heegaard splitting $(\Sigma;H_1,H_2) = (\Sigma_{g,\bold f}; H_{g,\bold f},H_{g,\bold f}$).  Starting with $Y_1\times[0,1]$, attach $m+n-1$ four-dimensional 1--handles to $Y_1\times\{1\}$ so that the resulting four-manifold is connected.  The three-manifold resulting from this surgery on $Y_1\times\{1\}$ is $H_1\cup_\Sigma\overline{H_2}$, and the induced structure on $\partial Z_k$ is that of the standard Heegaard-page structure on $Y_{g;\bold p,\bold f}$. With this extra structure in mind, we denote this distinguished copy by $Z_k$ by $Z_{g,k;\bold p,\bold f}$.

A \emph{disk-tangle} is a pair $(Z,\Dd)$ where $Z\cong Z_k$ and $\Dd$ is a collection of neatly embedded disks.  A disk-tangle is called \emph{trivial} if $\Dd$ can be isotoped rel-$\partial$ to lie in $\partial Z$.

\begin{proposition}\label{prop:triv_disks}
	Let $\Dd$ and $\Dd'$ be trivial disk-tangles in $Z$.  If $\partial \Dd =\partial \Dd'$, then $\Dd$ and $\Dd'$ are isotopic rel-$\partial$ in $Z$.
\end{proposition}

\begin{proof}
	Then case when $Z\cong B^4$ is a special case of a more general result of Livingston~\cite{Liv_82_Surfaces-bounding-the-unlink}, and is also proved in~\cite{Kam_02_Braid-and-knot}.  See~\cite{MeiZup_18_Bridge-trisections-of-knotted} for the general case.
\end{proof}

A trivial disk-tangle $(Z,\Dd)$ inherits extra structure along with $Z_{g,k;\bold p,\bold f}$,  since we can identify $\partial\Dd$ with an unlink $L$ in standard $(b,\bold v)$--bridge position in $Y_{g;\bold p,\bold f}$.  In this case,  a disk $D\subset\Dd$ is called \emph{vertical} (resp., \emph{flat}) if it corresponds to a vertical (resp., flat) component of $L$.   With this extra structure in mind, we call a trivial disk-tangle a \emph{$(c,\bold v)$--disk-tangle} and denote it by $\Dd_{c,\bold v}$.  Note that $\Dd_{c,\bold v}$ is a tangle of $c+v$ disks. We call the pair $(Z_{g,k;\bold p,\bold f},\Dd_{c;\bold v})$ a \emph{$(g,k,c;\bold p,\bold f,\bold v)$--disk-tangle}.  Note that Proposition~\ref{prop:triv_disks} respects this extra structure, since part of the hypothesis was that the two disk systems have the same boundary.  See Figure~\ref{fig:disk_tangle} for a schematic illustration.

\begin{figure}[h!]
	\centering
	\includegraphics[width=.4\textwidth]{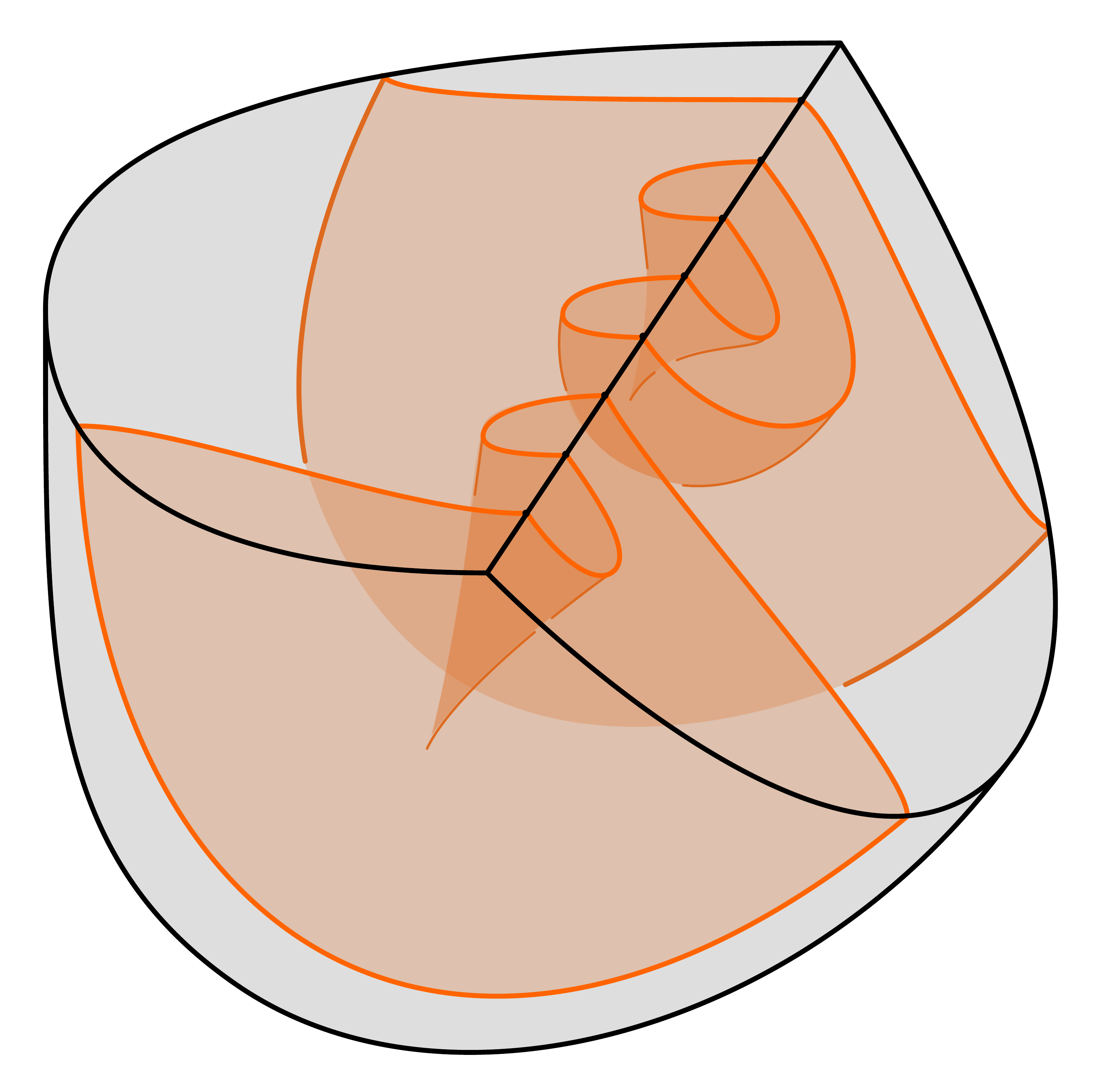}
	\caption{A schematic of the disk-tangle $\Dd_{1,2}$, which contains one flat component and two vertical components.  Note that the 3--component unlink on the boundary is in $(3,2)$--bridge position with respect to the standard Heeggaard double $Y_{0,0,1}$ for the 3--sphere.}
	\label{fig:disk_tangle}
\end{figure}

The special structure on $Z_{g,k;\bold p, \bold f}$ described above induces a special Morse function $\Phi\colon Z \to \R$ with $m+n-1$ critical points, all of which are index one.  The next lemma characterizes trivial disk-tangles with respect to this standard Morse function.

\begin{lemma}\label{lem:one_min}
	Let $Z = Z_{g,k;\bold p, \bold f}$, and let $\Dd\subset Z$ be a collection of neatly embedded disks with $\partial\Dd\cap Y_1$ a $\bold v$--thread.  Suppose the restriction $\Phi_\Dd$ of $\Phi$ to $\Dd$ has $c$ critical points, each of which is index zero.  Then $\Dd$ is a $(c,\bold v)$--disk-tangle for some ordered partition $\bold v$ of $v=|\Dd|-c$.
\end{lemma}

\begin{proof}
	We parameterize $\Phi\colon Z\to\R$ so that $\Phi(Z) = [0,1.5]$, $\Phi^{-1}(0) = Y_1\setminus\nu(P_1\cup P_2)$, $\Phi^{-1}(1.5) = (H_1\cup_\Sigma\overline{H_2})\setminus\nu(\overline{P_1}\cup P_2)$, and $\Phi(x) = 0.5$ for each critical point $x\in Z$ of $\Phi$.

	Let $\Gamma$ denote the cores of the 1--handles of $Z$.  By a codimension argument, we can assume, after a small perturbation of $\Phi$ that doesn't introduce any new critical points, that $\Dd$ is disjoint from a neighborhood $\nu(\Gamma)\cup Y_1\times[0,1]$.  Thus, we can assume that $\Phi_\Dd(x) = 1.0$ for any critical point $x\in\Dd$ of $\Phi_\Dd$.	

	First, note that $0\leq c\leq |\Dd|$; each connected component of $\Dd$ can have at most one minimum, since $\Phi_\Dd$ has no higher-index critical points. Let $\{D_i\}_{i=1}^c\subset\Dd$ denote the sub-collection of disks in $\Dd$ that contain the index zero critical points of $\Phi_\Dd$.  We claim that $D=\cup_{i=1}^cD_i$ is a $(c,0)$--disk-tangle.  We will now proceed to construct the required boundary-parallelism.

	Consider the moving picture of the intersection $D_{\{t\}}$ of $D$ with the cross-section $Z_{\{t\}} = \Phi^{-1}(1+t)$ for $t\in[0,0.5]$.  This movie shows the birth of a $c$--component unlink $L$ from $c$ points at time $t=0$, followed by an ambient isotopy of $L$ as $t$ increases.  Immediately after the birth, say $t=\epsilon$, we have that the sub-disks $D_{[1,1+\epsilon]} = D\cap\Phi^{-1}([1,1+\epsilon])$ of $D$ are clearly boundary-parallel to a spanning collection of disks $E_\epsilon$ for $L_\epsilon = D_{\{1+\epsilon\}}$.  Now, we simply push this spanning collection of disks $E_\epsilon$ along through the isotopy taking $L_\epsilon$ to $\partial D$.  Because this isotopy is ambient, the traces of the disks of $E_\epsilon$ are disjoint, thus they provide a boundary parallelism for $D$, as desired.

	It remains to see that the collection $D''$ of disks in $\Dd$ containing no critical points of $\Phi_\Dd$ are also boundary parallel.  Note however, that they will not be boundary parallel into $\Phi^{-1}(1.5)$, as before.
	
	Let $\beta = D''\cap Y_1$; by hypothesis, $(Y_1,\beta)$ is a $(\bold p, \bold f, \bold v)$--spread, i.e., $Y_1$ is a product lensed bordism (a spread) $H_{\bold p,\bold f}$ and $\beta$ is a vertical $\bold v$--tangle (a $\bold v$--thread) therein. Similar to before, we can assume that $D''$ is disjoint from a small neighborhood of the cores of the 1--handles.
	
	Since $D''$ contains no critical points, it is vertical in the sense that we can think of it as the trace of an ambient isotopy of $\beta$ in $Y_1$ as $t$ increases from $t=0$ to $t=0.5$, followed by the trace of an ambient isotopy of $\beta$ in $H_1\cup_\Sigma\overline{H_2}$ between $t=0.5$ and $t=1.5$.  The change in the ambient space is not a problem, since $D''$ is disjoint form the cores $\Gamma$ of the 1--handles, hence these isotopies are supported away from the four-dimensional critical points.
	
	If $\Delta$ is any choice of bridge triangles for $\beta$ in $Y_1$, then the trace of $\Delta$ under this isotopy gives a boundary-parallelism of $D''$, as was argued above.  We omit the details in this case.  
\end{proof}

Note that the assumption that $\beta$ be a thread was vital in the proof, as it gave the existence of~$\Delta$.  If $\beta$ contained knotted arcs, the vertical disk sitting over such an arc would not be boundary parallel.  Similarly, if $\beta$ contained closed components, the vertical trace would be an annulus, not a disk.  The converse to the lemma is immediate, hence it provides a characterization of trivial disk-tangles.

We next show how a standard bridge splitting can be uniquely extended to a disk-tangle. The following lemma builds on portions of~\cite[Section~4]{CasGayPin_18_Diagrams-for-relative-trisections}.

\begin{lemma}
\label{lem:LP}
	Let $(M,K) = (H_1,\Tt_1)\cup_{(\Sigma,\bold x)}\overline{(H_2,\Tt_2)}$ be a standard $(g,\bold p, \bold f;b, \bold v)$--bridge splitting.  There is a unique (up to diffeomorphism rel-$\partial$) pair $(Z,\Dd)$, diffeomorphic to $(Z_{g,k;\bold p,\bold f},\Dd_{c,\bold v})$, such that the bridge double structure on $\partial(Z,\Dd)$ is the bridge double of $(M,K)$.
\end{lemma}

\begin{proof}
	By Lemma~\ref{lem:BridgeDouble}, there is a unique way to close $(M,K)$ up and obtain its bridge double $(Y,L)$.  By Laudenbach-Poenaru~\cite{LauPoe_72_A-note-on-4-dimensional}, there is a unique way to cap off $Y\cong\#^k(S^1\times S^2)$ with a copy of $Z$ of $Z_k$.  By Proposition~\ref{prop:triv_disks}, there is a unique way to cap off $L$ with a collection $\Dd$ of trivial disks.  Since these choice are unique (up to diffeomorphism rel-$\partial$ and isotopy rel-$\partial$, respectively), the pair $(Z,\Dd)$ inherit the correct bridge double structure on its boundary, as desired.
\end{proof}

\subsection{Open-book decompositions and braidings of links}
\label{subsec:OBD}
\ 

We follow Etnyre's lecture notes~\cite{Etn_04_Lectures-on-open} to formulate the definitions of this subsection.
Let $Y$ be a closed, orientable three-manifold.  An \emph{open-book decomposition} of $Y$ is a pair $(B,\pi)$, where $B$ is a link in $M$ (called the \emph{binding}) and $\pi\colon Y\setminus B \to S^1$ is a fibration such that $P_\theta = \pi^{-1}(\theta)$ is a non-compact surface (called the \emph{page}) with $\partial P_\theta = B$.  Note that it is possible for a given link $B$ to be the binding of non-isotopic (even non-diffeomorphic) open-book decomposition of $Y$, so the projection data $\pi$ is essential in determining the decomposition.

An \emph{abstract open-book} is a pair $(P,\phi)$, where $P$ is an oriented, compact surface with boundary, and $\phi\colon P \to P$ is a diffeomorphism (called the \emph{monodromy}) that is the identity on a collar neighborhood of $\partial P$.
An abstract open-book $(P,\phi)$ gives rise to a closed three-manifold, called the \emph{model manifold}, with an open-book decomposition in a straight-forward way.  Define
$$Y_\phi = (P\times_\phi S^1) \cup \left(\bigsqcup_{|\partial P|} S^1\times D^2\right),$$
where $P\times_\phi S^1$ denotes the mapping torus of $\phi$, and $Y_\phi$ is formed from this mapping torus by capping off each torus boundary component with a solid torus such that each $p\times_\phi S^1$ gets capped off with a meridional disk for each $p\in\partial P$.
(Note that $p\times_\phi S^1 = p\times S^1$ by the condition on $\phi$ near the boundary of $P$.)
Our convention is that $P\times_\phi S^1 = P\times[0,1]/_{(x,1)\sim(\phi(x),0)}$ for all $x\in P$.

If we let $B_\phi$ denote the cores of the solid tori used to form $Y_\phi$, then we see that $Y_\phi\setminus B_\phi$ fibers over $S^1$, so we get an open-book decomposition $(B_\phi,\pi_\phi)$ for $Y_\phi$. 
Conversely, an open-book decomposition $(B,\pi)$ of a three-manifold $M$ gives rise to an abstract open-book $(P_\pi,\phi_\pi)$ in the obvious way such that $(Y_{\phi_\pi},B_{\phi_\pi})$ is diffeomorphic to $(M,B)$.

We now recall an important example which appeared in Lemma~\ref{lemma:k-value}.

\begin{example}
\label{ex:Id_obd}
	Consider the abstract open-book $(P,\phi)$, where $P = \Sigma_{p,f}$ is a compact surface of genus $p$ with $f$ boundary components and $\phi\colon P\to P$ is the identity map.  the total space $Y_\phi$ of this abstract open-book is diffeomorphic to $\#^{2p+f-1}(S^1\times S^2)$.  To see this, simply note that the union of half of the pages gives a handlebody of genus $2p+f-1$; since the monodromy is the identity, $Y_\phi$ is the symmetric double of this handlebody.
\end{example}

Harer described a set of moves that suffice to pass between open-book decompositions on a fixed three-manifold~\cite{Har_82_How-to-construct-all-fibered-knots}.  These include Hopf stabilization and destabilization, as well as a certain double-twisting operation, which was known to be necessary in order to change the homotopy class of the associated plane field.  (Harer's calculus was recently refined in~\cite{PieZud_18_Special-moves-for-open}.)
In fact, Giroux and Goodman proved that two open-book decompositions on a fixed three-manifold have a common Hopf stabilization if and only if the associated plane fields are homotopic~\cite{GirGoo_06_On-the-stable-equivalence-of-open}.  For a trisection-theoretic account of this story, see~\cite{CasIslMil_19_The-relative-L-invariant-of-a-compact}.

Having introduced open-book decompositions, we now turn our attention to braided links.
Suppose that $\Ll\subset Y$ is a link and $(B,\pi)$ is an open-book decomposition on $Y$.  We say that $\Ll$ is \emph{braided with respect to $(B,\pi)$} if $\Ll$ intersects each page of the open-book transversely. We say that $(Y,\Ll)$ is equipped with the structure of an \emph{open-book braiding}. The \emph{index} of the braiding is the number of times that $\Ll$ hits a given page.  By the Alexander Theorem~\cite{Ale_20_Note-on-Riemann-spaces} and the generalization due to Rudolph~\cite{Rud_83_Constructions-of-quasipositive-knots}, any link can be braided with respect to any open-book in any three-manifold.

An \emph{abstract open-book braiding} is a triple $(P,\bold y,\phi)$, where $P$ is an oriented, compact surface with boundary, $\bold y\subset P$ is a collection of points, and $\phi\colon (P,\bold y)\to (P,\bold y)$ is a diffeomorphism.  As with abstract open-books, this data gives rise to a manifold pair $(Y_\phi,\Ll_\phi)$, called the \emph{model open-book braiding} of the abstract open-book braiding, where $Y_\phi$ has an open-book structure with binding $B_\phi$ and projection $\pi_\phi$ and $\Ll_\phi$ is braided with respect to $(B_\phi,\pi_\phi)$.  More precisely,
$$(Y_\phi,\Ll_\phi) = (P,\bold y)\times_\phi S^1 = (P,\bold y)\times[0,1]/_{(x,0)\sim(\phi(x),1)}$$
for all $x\in P$.
Conversely, a braiding of $\Ll$ about $(B,\pi)$ gives rise in the obvious way to an abstract open-book braiding $(P_\pi,\phi_\pi)$ such that $(Y_{\phi_\pi},\Ll_{\phi_\pi})$ is diffeomorphic to $(Y,\Ll)$.

By the Markov Theorem~\cite{Mar_35_Uber-die-freie-Aquivalenz} or its generalization to closed 3--manifolds~\cite{Sko_92_Closed-braids-in-3-manifolds,Sun_93_The-Alexander-and-Markov-theorems}, any two braidings of $\Ll$  with respect to a fixed open-book decomposition of $Y$ can be related by an isotopy that preserves the braided structure, except at finitely many points in time at which the braiding is changed by a \emph{Markov stabilization or destabilization}.  We think of a Markov stabilization in the following way.
Let $J$ be a meridian for a component of the binding $B$ of the open-book decomposition on $Y$, and let $\frak b$ be a band connecting $\Ll$ to $J$ such that the core of $\frak b$ is contained in a page of the open-book decomposition and such that the link $\Ll' = \Ll_\frak b$ resulting from the resolution of the band is braided about $(B,\pi)$.  We say that $\Ll'$ is obtained from $\Ll$ via a \emph{Markov stabilization}, and we call the inverse operation \emph{Markov destabilization}.  (Markov destabilization can be thought of as attaching a vertical band to $\Ll'$ such that resolving the band has the effect of splitting off from $\Ll'$ a meridian for a binding component.) See Figure~\ref{fig:Markov}.

\begin{figure}[h!]
	\centering
	\includegraphics[width=.8\textwidth]{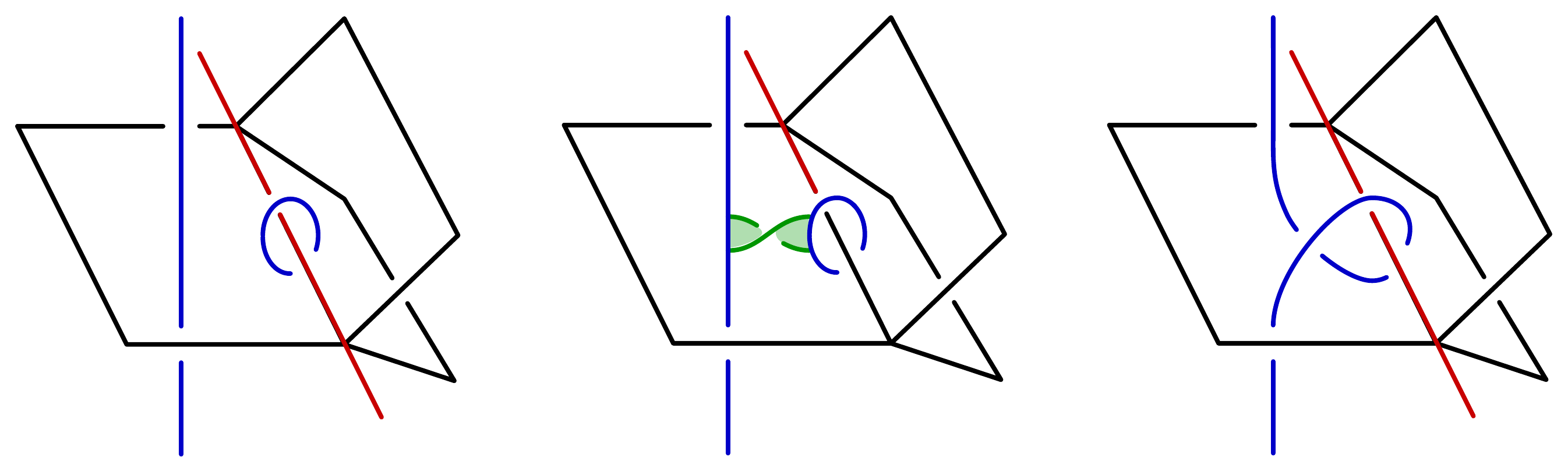}
	\caption{Markov stabilization, depicted as the banding of a braid to a meridian of the binding.}
	\label{fig:Markov}
\end{figure}

Suppose that $Y = Y^1\sqcup\cdots\sqcup Y^n$ is the disjoint union of closed three-manifolds such that each $Y^j$ is equipped with an open-book decomposition $(B^j,\pi^j)$.  Suppose that $\Ll = \Ll^1\sqcup\cdots\Ll^n$ is a link such that $Y^j\subset Y^j$ is braided about $(B^j,\pi^j)$.  We say that $\Ll$ has \emph{multi-index} $\bold v = (v^1,\ldots,v^n)$ if $\Ll^j$ has index $v^j$.  We allows the possibility that $\Ll^j = \emptyset$ for any given $j$.

\begin{remark}
\label{rmk:braid_ortn}
	If $Y$ is oriented, and we pick orientations on $\Ll$ and on a page $P$ of $(B,\pi)$, then we can associate a sign to each point of $\Ll\cap P$.  By definition, if $\Ll$ is a knot, then each such point will have identical sign; more generally, connected components of $\Ll$ have this property.  If the orientations of the points $\Ll\cap P$ all agree, then we say that the braiding is \emph{coherently oriented}.  If the orientation of these points disagree across components of $\Ll$, then we say that the braiding is \emph{incoherently oriented}.
	
	Our reason for considering incoherently oriented braidings is that sometimes a bridge trisection of a surface will induce a braiding of the boundary link that is incoherently oriented once the surface is oriented.  A simple example of this, the annulus bounded by the $(2,2n)$--torus link, will be explored in Examples~\ref{ex:2-braid} and~\ref{ex:T24coherent}. Even though some bridge trisections induce incoherently oriented braidings on the boundary link, it is always possible to find a bridge trisection of a surface so that the induced braiding is coherently oriented.
\end{remark}

\subsection{Formal definitions}
\label{subsec:Formal}
\ 

Finally, we draw on the conventions laid out above to give formal definitions.

\begin{definition}\label{def:Trisection}
	Let $X$ be an orientable, connected four-manifold, and let
	$$Y = \partial X = Y^1\sqcup\cdots\sqcup Y^n,$$
	where $Y^j$ is a connected component of $\partial X$ for each $j=1,\ldots, n$. Let $g$, $k^*$, $p$, and $f$ be non-negative integers, and let $\bold k$, $\bold p$, and $\bold f$ be ordered partitions of type $(k^*,3)$, $(p,n)$, and $(b,n)^+$, respectively.
	
	A \emph{$(g,\bold k;\bold p,\bold f)$--trisection} $\T$ of $X$ is a decomposition $X = Z_1\cup Z_2\cup Z_3$ such that, for all $j=1,\ldots, n$ and all $i\in\Z_3$,
	\begin{enumerate}
		\item $Z_i \cong Z_{g,k_i;\bold p,\bold f}$, 
		\item $Z_i\cap Z_{i+1} \cong H_{g;\bold p,\bold f}$,
		\item $Z_1\cap Z_2\cap Z_3 \cong \Sigma_{g,b}$, and
		\item $Z_i\cap Y^j \cong H_{\bold p,\bold f}$.
	\end{enumerate}
	
	The four-dimensional pieces $Z_i$ are called \emph{sectors}, the three-dimensional pieces $H_i = Z_i\cap Z_{i-1}$ are called \emph{arms}, and the central surface $\Sigma = Z_1\cap Z_2\cap Z_3$ is called the \emph{core}.  
	If $k_1 = k_2 = k_3 = k$, then $\T$ is described as a \emph{$(g,k;\bold p,\bold f)$--trisection} and is called \emph{balanced}.  Otherwise, $\T$ is called \emph{unbalanced}. Similarly, if either of the ordered partitions $\bold p$ and $\bold f$ are balanced, we replace these parameters with the integers $p/n$ and/or $f/n$, respectively.
	The parameter $g$ is called the \emph{genus} of $\T$.
	The surfaces $P_i^j = H_i\cap Y^j$ are called \emph{pages}, and their union is denoted $P_i$.  The lensed product cobordisms $Y_i^j = Z_i\cap Y^j$ are called \emph{spreads}, and their union is denoted $Y_i$.  The links $B^j = \Sigma\cap Y^j$ are called \emph{bindings}, and their union is $B=\partial \Sigma$.
	
	If $X$ is oriented, we require that the orientation on $Z_i$ induces the oriented decompositions
	$$\partial Z_i = H_i\cup Y_i\cup \overline H_{i+1}, \hspace{.25in} \partial H_i=\Sigma\cup_B\overline{P_i}, \hspace{.25in} \text{\ and\ } \hspace{.25in} \partial Y_i=P_i\cup_B\overline{P_{i+1}}.$$
	See Figure~\ref{fig:ortn_conv} (below) for a schematic illustrating these conventions.

\end{definition}

\begin{remarks}\ 
	\begin{enumerate}
		\item If $X$ is closed, then $n=0$, $Y = \emptyset$, and $\T$ is a trisection as originally introduced by Gay and Kirby~\cite{GayKir_16_Trisecting-4-manifolds} and generalized slightly in~\cite{MeiSchZup_16_Classification-of-trisections-and-the-Generalized}.
		\item If $X$ has a single boundary component, then $n=1$, and $\T$ is a relative trisection as first described  in~\cite{GayKir_16_Trisecting-4-manifolds} and later developed in~\cite{Cas_16_Relative-trisections-of-smooth}, where gluing of such objects was studied, and in~\cite{CasGayPin_18_Diagrams-for-relative-trisections}, where the diagrammatic aspect to the theory was introduced. The general case of multiple boundary components was recently developed in~\cite{Cas_17_Trisecting-smooth-4--dimensional}.
		\item Since $Y^j = Y^j_1\cup Y^j_2\cup Y^j_3$, with each $Y^j_i\cong H_{p_j,b_j}$, it follows that $Y^j$ admits an open-book decomposition where $P_i^j$ is a page for each $i\in\Z_3$ and $B^j$ is the binding.  This open-book decomposition is determined by $\T$, and the monodromy can be explicitly calculated from a relative trisection diagram~\cite{CasGayPin_18_Diagrams-for-relative-trisections}.
		\item Note that the triple $(\Sigma, P_i, P_{i+1})$ defines the standard Heegaard double structure on $\partial Z_i\cong Y_{g;\bold p,\bold f}$.  It follows from Lemma~\ref{lemma:k-value} that $k_i = 2p+f-n+m_i$, where $(\Sigma;H_i,H_{i+1})$ is an $(m_i,n)$--standard Heegaard splitting.  We call $m_i$ the \emph{interior} complexity of $Z_i$. Notice that $g$ is bounded below by $m_i$ and $p$, but not by $f$ nor $k_i$.
	\end{enumerate}
\end{remarks}

\begin{definition}\label{def:Bridge}
	Let $\T$ be a trisection of a four-manifold $X$. Let $\Ff$ be a neatly embedded surface in $X$.  Let $b$, $c^*$, and $v$ be non-negative integers, and let $\bold c$ and $\bold v$ be ordered partitions of type $(c^*,3)$ and $(v,n)$, respectively.  The surface $\Ff$ is in \emph{$(b,\bold c;\bold v)$--bridge trisected position} with respect to $\T$ (or is \emph{$(b,\bold c; \bold v)$--bridge trisected} with respect to $\T$) if, for all $i\in\Z_3$,
	\begin{enumerate}
		\item $\Dd_i = Z_i\cap \Ff$ is a trivial $(c_i;\bold v)$--disk-tangle in $Z_i$, and 
		\item $\Tt_i = H_i\cap\Ff$ is a trivial $(b;\bold v)$--tangle in $H_i$.
	\end{enumerate}
	
	The disk components of the $\Dd_i$ are called \emph{patches}, and the $\Tt_i$ are called \emph{seams}.  Let
	$$\Ll = \partial \Ff = \Ll^1\sqcup\cdots\sqcup\Ll^n,$$
	where $\Ll^j = \Ll\cap Y^j$ is the link representing the boundary components of $\Ff$ that lie in $Y^j$.  The pieces $\beta^j_i = \Ll^j\cap Z_i$ comprising the $\Ll_i$ are called \emph{threads}.
	
	If $\Ff$ is oriented, we require that the induced orientation of $\Dd_i$ induces the oriented decomposition
	$$\partial \Dd_i = \Tt_i\cup\beta_i\cup\overline\Tt_{i+1}.$$
	See Figure~\ref{fig:ortn_conv} (below) for a schematic illustrating these conventions.
	
	The induced decomposition $\T_\Ff$ given by
	$$(X,\Ff) =  (Z_1, \Dd_1)\cup(Z_2, \Dd_2)\cup(Z_3, \Dd_3),$$
	is called a \emph{$(g,\bold k,b,\bold c;\bold p, \bold f, \bold v)$--bridge trisection} of $\Ff$ (or of the pair $(X,\Ff)$).
	If $\T$ is balanced and $c_i= c$ for each $i\in\Z_3$, then $\T_\Ff$ is described as a \emph{$(g,k,b,c;\bold p,\bold f,\bold v)$--bridge trisection} and is called \emph{balanced}.  Otherwise, $\T_\Ff$ is called \emph{unbalanced}.  Similarly, if the partition $\bold v$ is balanced, we replace this parameter with the integer $v/n$.
	The parameter $b$ is called the \emph{bridge number} of $\T_\Ff$.
\end{definition}

\begin{remarks} \ 
	\begin{enumerate}
		\item If $X$ is a closed four-manifold, then $n=0$, $\Ll = \emptyset$, and $\Ff$ is a closed surface in $X$. If $g=0$, we recover the notion of bridge trisections originally introduced in~\cite{MeiZup_17_Bridge-trisections-of-knotted}, while the more general case of arbitrary $g$ is treated in in~\cite{MeiZup_18_Bridge-trisections-of-knotted}.
		\item Note that if $\Ll\cap Y^j = \emptyset$ for some $j=1,\ldots, n$, then $\Ll^j = \emptyset$.  Equivalently, $v_j = 0$.  If $\Ll^j$ is not empty, then we have
		$$\Ll^j = \beta^j_1\cup\beta^j_2\cup\beta^j_3.$$
		If follows that $\Ll^j$ is braided with index $v_j$ with respect to the open-book decomposition $(B^j,P^j_i)$ on $Y^j$ induced by $\T$.
		\item The link $L_i = \partial\Dd_i$ is in $(b,\bold v)$--bridge position with respect to the standard Heegaard double structure on $\partial Z_i$.
		\item The surface $\Ff$ has a cellular decomposition consisting of $(2b+4v)$ 0--cells, $3v$ of which lie in the pages of $\partial X$; $(3b+6v)$ 1--cells, $3v$ of which lie in the spreads of $\partial X$; and $(c_1+c_2+c_3+3v)$ 2--cells, $3v$ of which are vertical patches.  It follows that the Euler characteristic of $\Ff$ is given as
		$$\chi(\Ff) = c_1 + c_2 + c_3 + v - b.$$
		\item Note that $c_i\geq b$, but that $v$ is independent of $b$ and the $c_i$.
	\end{enumerate}
\end{remarks}

We conclude this section with a key fact about bridge trisections.  We refer to the union
$$(H_1,\Tt_2)\cup(H_2,\Tt_2)\cup(H_3,\Tt_3)$$
as the \emph{spine} of the bridge trisection $\T$.  Two bridge trisections $\T$ and $\T'$ for pairs $(X,\Ff)$ and $(X,\Ff')$ \emph{diffeomorphic} if there is a diffeomorphism $\Psi\colon (X,\Ff)\to (X',\Ff')$ such that $\psi(Z_i,\Dd_i) = (Z_i',Dd_i')$ for all $i\in\Z_3$.

\begin{proposition}
\label{prop:spine}
	Two bridge trisections are diffeomorphic if and only if their spines are diffeomorphic.
\end{proposition}

\begin{proof}
	If $\Psi$ is a diffeomorphism of bridge trisections $\T$ and $\T'$, then the restriction of $\Psi$ to the spine of $\T$ is a diffeomorphism onto the spine of $\T'$.  Conversely, suppose $\Psi$ is a diffeomorphism from the spine of $\T$ to the spine of $\T'$ -- i.e., $\Psi(H_i,\Tt_i) = (H_i',\Tt_i')$ for all $i\in\Z_3$.  By Lemma~\ref{lem:LP}, $\Psi$ there is an extension of $\Psi$ across $(Z_i,\Dd_i)$ that is uniquely determined up to isotopy fixing $(H_1,\Tt_i)\cup_{(\Sigma,\bold x)}\overline{(H_{i+1},\Tt_{i+1})}$ for each $i\in\Z_3$.  It follows that that $\Psi$ extends to a diffeomorphism bridge trisections, as desired.
\end{proof}

In light of this, we find that the four-dimensional data of a bridge trisection is determined by the three-dimensional data of its spine, a fact that will allow for the diagrammatic development of the theory in Sections~\ref{sec:tri-plane} and~\ref{sec:shadow}.

\begin{corollary}
\label{coro:spine}
	A bridge trisection is determined uniquely by its spine.
\end{corollary}

\section{The four-ball setting}\label{sec:four-ball}

In this section, we restrict our attention to the study of surfaces in the four-ball.  Moreover, we work relative to the standard genus zero trisection.  These restrictions allow for a cleaner exposition than the general framework of Section~\ref{sec:general} and give rise to a new diagrammatic theory for surfaces in this important setting.

\subsection{Preliminaries and a precise definition}\label{subsec:Special}\ 

Here, we revisit the objects and notation introduced in Section~\ref	{sec:general} with the setting of $B^4$ in mind, culminating in a precise definition of a bridge trisection of a surface in $B^4$.

Let $H$ denote the three-ball, and let $B$ denote an equatorial curve on $\partial H$, which induces the decomposition
$$\partial H = \partial_+ H \cup_B \partial_- H$$
of the boundary sphere into two hemispheres. We think of $H$ as being swept out by disks: smoothly isotope $\partial_+ H$ through $H$ to $\partial_- H$.
(Compare this description of $H$ with the notion of a lensed cobordism from Subsection~\ref{subsec:Lensed} and the development for a general compressionbody in Subsection~\ref{subsec:Compression}.)

A trivial tangle is a pair $(H,\Tt)$ such that $H$ is a three-ball and $\Tt\subset H$ is a neatly embedded 1--manifold with the property that $\Tt$ can be isotoped until the restriction $\Phi_\Tt$ of the above Morse function to $\Tt$ has no minimum and at most one maximum on each component of $\Tt$.  In other words, each component of $\Tt$ is a neatly embedded arc in $H$ that is either \emph{vertical} (with respect to the fibering of $H$ by disks) or parallel into $\partial_+ H$.  The latter arcs are called \emph{flat}.  We consider trivial tangles up to isotopy rel-$\partial$.  If $\Tt$ has $v$ vertical strands and $b$ flat strands, we call the pair $(H,\Tt)$ an $(b,v)$--tangle.  This is a special case of the trivial tangles discussed in Subsection~\ref{subsec:Tangles}.

Let $H_1$ and $H_2$ be three-balls, and consider the union $H_1\cup_\Sigma \overline H_2$, where $\Sigma = \partial_+ H_1 = \partial_+ \overline H_2$. We consider this union of as a subset of the three-sphere $Y$ so that $B = \partial\Sigma$ is an unknot and $\Sigma$, $\partial_- H_1$, and $\partial_- H_2$ are all disjoint disk fibers meeting at $B$.  Let $Y_1$ denote 
$$Y\setminus\Int(H_1\cup_\Sigma \overline H_2),$$
and notice that $Y_1$ is simply an interval's worth of disk fibers for $B$, just like the $H_i$.  We let $Y$ denote the three-sphere with this extra structure, which we call the \emph{standard Heegaard double} (cf. Subsection~\ref{subsec:Heegaard}).

\begin{figure}[h!]
	\centering
	\includegraphics[width=.25\textwidth]{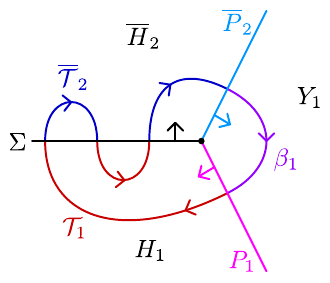}
	\caption{A schematic illustration of a standard Heegaard double, with orientation conventions for the constituent pieces of $\partial Z_1$ indicated.}
	\label{fig:ortn_conv}
\end{figure}

An unlink $L\subset Y$ is in \emph{$(b,v)$--bridge position} with respect the standard Heegaard double structure if $L\cap H_i$ is a $(b,v)$--tangle, $L$ is transverse to the disk fibers of $Y_1$, and each component of $L$ intersects $Y_1$ in at most one arc.  The $v$ components of $L$ that intersect $Y_1$ are called \emph{vertical}, while the other $b$ components are called \emph{flat}. 

Let $Z$ denote the four-ball, with $\partial Z = Y$ regarded as the standard Heegaard double.  A trivial disk-tangle is a pair $(Z,\Dd)$ such that $Z$ is a four-ball and $\Dd$ is a collection of neatly embedded disks, each of which is parallel into $\partial Z$.  Note that the boundary $\partial \Dd$ is an unlink.  If $\partial \Dd$ is in $(b,v)$--bridge position in $Y = \partial Z$, then the disk components of $\Dd$ are called \emph{vertical} and \emph{flat} in accordance with their boundaries.  A \emph{$(c,v)$--disk-tangle} is a trivial disk-tangle with $c$ flat components and $v$ vertical components.

\begin{definition}\label{def:Bridge2}
	Let $\Ff$ be a neatly embedded surface in $B^4$, and let $\T_0$ be the standard genus zero trisection of $B^4$.  Let $b$ and $v$ be non-negative integers, and let $\bold c = (c_1, c_2, c_3)$ be an ordered triple of non-negative integers.  The surface $\Ff$ is in \emph{$(b,\bold c; v)$--bridge trisected position} with respect to $\T_0$ (or is \emph{$(b,\bold c; v)$--bridge trisected} with respect to $\T_0$) if, for all $i\in\Z_3$,
	\begin{enumerate}
		\item $\Dd_i = Z_i\cap \Ff$ is a trivial $(c_i, v)$--disk-tangle in the four-ball $Z_i$, and 
		\item $\Tt_i = H_i\cap\Ff$ is a trivial $(b, v)$--tangle in the three-ball $H_i$.
	\end{enumerate}
	
	The disk components of the $\Dd_i$ are called \emph{patches}, and the $\Tt_i$ are called \emph{seams}.  Let $\Ll = \partial \Ff$.  The braid pieces $\beta_i = \Ll\cap Z_i$ are called \emph{threads}.
	
	If $\Ff$ is oriented, we require that the induced orientation of $\Dd_i$ induces the oriented decomposition
	$$\partial \Dd_i = \Tt_i\cup\beta_i\cup\overline \Tt_{i+1}.$$
	
	The induced decomposition $\T_\Ff$ given by
	$$(X,\Ff) =  (Z_1, \Dd_1)\cup(Z_2, \Dd_2)\cup(Z_3, \Dd_3),$$
	is called a \emph{$(b,\bold c, v)$--bridge trisection} of $\Ff$ (or of the pair $(X,\Ff)$).
	If $\T_\Ff$ is balanced and $c_1 = c_2 = c_3 = c$, then $\T_\Ff$ is a \emph{$(b,c,v)$--bridge trisection} and is called \emph{balanced}.  Otherwise, $\T_\Ff$ is called \emph{unbalanced}. 
\end{definition}

\subsection{Band presentations}\label{subsec:band_pres}\ 

Let $M$ be a three-manifold, and let $J$ be a neatly embedded one-manifold in $M$.  Let $\frak b$ be a copy of $I\times I$ embedded in $M$, and denote by $\partial_1\frak b$ and $\partial_2\frak b$ the portions of $\partial \frak b$ corresponding to $I\times\{-1,1\}$ and $\{-1,1\}\times I$, respectively.  We call such a $\frak b$ a \emph{band} for $J$ if $\Int(\frak b)\subset M\setminus J$ and $\partial \frak b\cap J = \partial_1\frak b$.  The arc of $\frak b$ corresponding to $\{0\}\times I$ is called the \emph{core} of $\frak b$.

Let $J_\frak b$ denote the one-manifold obtained by \emph{resolving} the band $\frak b$:
$$J_\frak b = (J\setminus\partial_1\frak b)\cup\partial_2\frak b.$$
The band $\frak b$ for $J$ gives rise to a \emph{dual band} $\frak b^*$ that is a band for $J_\frak b$, so $\partial_1 \frak b^* = \partial_2 \frak b$ and $\partial_2\frak b^* = \partial_1\frak b$.  Note that, as embedded squares in $M$, we have $\frak b = \frak b^*$, though their cores are perpendicular.
More generally, given a collection $\frak b$ of disjoint bands for $J$, we denote by $J_\frak b$ the \emph{resolution} of all the bands in $\frak b$.  As above, the collection $\frak b^*$ of dual bands is a collection of bands for $J_\frak b$.

\begin{definition}[\textbf{\emph{band presentation}}]
\label{def:band_pres}
	A \emph{band presentation} is a 2--complex in $S^3$ defined by a triple $(\Ll,U,\frak b)$ as follows:
	\begin{enumerate}
		\item $\Ll\subset S^3$ is a link;
		\item $U$ is a split unlink in $S^3\setminus\nu(\Ll)$; and
		\item $\frak b$ is a collection of bands for $\Ll\sqcup U$ such that $U'=(\Ll\sqcup U)_\frak b$ is an unlink.
	\end{enumerate}
	If $U$ is the empty link, then we write $(\Ll,\frak b)$ and call the encoded 2--complex in $S^3$ a \emph{ribbon presentation}.
		
	We consider two band presentations to be \emph{equivalent} if they are ambient isotopic as 2--complexes in $S^3$.  Given a fixed link $\Ll\subset S^3$, two band presentations $(\Ll,U_1,\frak b_1)$ and $(\Ll,U_2,\frak b_2)$ are \emph{equivalent rel-$\Ll$} if they are equivalent via an ambient isotopy that preserves $\Ll$ set-wise.  (In other words, $\Ll$ is fixed, although the attaching regions of $\frak b$ are allowed to move along $\Ll$.)
\end{definition}

Band presentations encode smooth, compact, neatly embedded surfaces in $B^4$ in a standard way.  Before explaining this, we first fix some conventions that will be useful later.  (Here, we follow standard conventions, as in~\cite{KawShiSuz_82_Descriptions-on-surfaces-in-four-space.,Kaw_96_A-survey-of-knot-theory,MeiZup_17_Bridge-trisections-of-knotted,MeiZup_18_Bridge-trisections-of-knotted}.)

Let $h\colon B^4\to [0,4]$ be a standard Morse function on $B^4$ -- i.e., $h$ has a single critical point, which is definite of index zero and given by $h^{-1}(0)$, while $h^{-1}(4) = \partial B^4 = S^3$.  For any compact submanifold $X$ of $B^4$ and any $0\leq t < s\leq 4$, let $X_{[t,s]}$ denote $X\cap h^{-1}\left([t,s]\right)$ and let $X_{\{t\}} = X\cap h^{-1}(t)$.  For example, $B^4_{[t,s]} = h^{-1}[t,s]$.  Similarly, for any compact submanifold $Y$ of $B^4_{\{t\}}$ and any $0\leq r<s\leq 4$, let $Y[r,s]$ denote the vertical cylinder obtained by pushing $Y$ along the gradient flow across the height interval $[r,s]$, which we call a \emph{gradient product}.  We extend these notions in the obvious way to open intervals and singletons in $[0,4]$.

Now we will show how, given a band presentation $(\Ll,U,\frak b)$, we can construct the \emph{realizing surface} $\Ff_{(\Ll,U,\frak b)}$: a neatly embedded surface in $B^4$ with boundary $\Ll$.  Start by considering $(\Ll,U,\frak b)$ as 2--complex in $B^4_{\{2\}}\cong S^3$, and consider the surface $\Ff$ with the following properties:
\begin{enumerate}
	\item $\Ff_{(3,4]} = \Ll(3,4]$;
	\item $\Ff_{\{3\}} = \Ll\{3\}\sqcup D$, where $D$ is a collection of spanning disks for the unlink $U\{3\}\subset B^4_{\{3\}}\cong S^3$;
	\item $\Ff_{(2,3)} = (\Ll\sqcup U)(2,3)$;
	\item $\Ff_{\{2\}} = (\Ll\sqcup U)\cup\frak b$;
	\item $\Ff_{(1,2)} = U'(1,2)$;
	\item $\Ff_{\{1\}} = D'$, where $D'$ is a collection of spanning disks for the unlink $U'\subset B^4_{\{1\}}\cong S^3$; and
	\item $\Ff_{[0,1)} = \emptyset$.
\end{enumerate}
Note that the choices of spanning disks $D$ and $D'$ are unique up to perturbation into $B^4_{(3,3+\epsilon)}$ and $B^4_{(1,1-\epsilon)}$, respectively, by Proposition~\ref{prop:triv_disks}.  Note also that $\partial \Ff = \Ff\cap B^4_{\{4\}} = \Ll\{4\}$.  

\begin{proposition}\label{prop:band_pres}
	Every neatly embedded surface $\Ff$ with $\partial \Ff = \Ll$ is isotopic rel-$\partial$ to a realizing surface $\Ff_{(\Ll,U,\frak b)}$ for some band presentation $(\Ll,U,\frak b)$.  If $\Ff$ has a handle-decomposition with respect to the standard Morse function on $B^4$ consisting of $c_1$ cups, $n$ bands, and $c_3$ caps, then $(\Ll,U,\frak b)$ can be assumed to satisfy $|U|=c_3$, $|\frak b|=n$, and $|U'|=c_1$.
\end{proposition}

\begin{proof}
	Given $\Ff$, we can assume after a minor perturbation that the restriction $h_\Ff$ of a standard height function $h\colon B^4\to [0,4]$ is Morse.  After re-parametrizing the codomain of $h$, we can assume that the critical points of $h_\Ff$ are contained in $h^{-1}\left((1.5,2.5)\right)$.  For each index zero critical point $x$ of $h_\Ff$, we choose a vertical strand $\omega$ connecting $x$ to $B^4_{\{1\}}$.  (Here, vertical means that $\omega_{\{t\}}$ is a point or empty for each $t\in[1,2.5]$.)  By a codimension count, $\omega$ is disjoint from $\Ff$, except at $x$.  We can use a small regular neighborhood of $\omega$ to pull $x$ down to $B^4_{\{1\}}$.  Repeating, we can assume that the index zero critical points of $h_\Ff$ lie in $B^4_{\{1\}}$.  By a similar argument, we achieve that the index two critical points of $h_\Ff$ lie in $B^4_{\{3\}}$ and that the index one critical points of $h_\Ff$ lie in $B^4_{\{2\}}$.
	
	Next, we perform the standard flattening of the critical points:  For each critical point $x$ of index~$i$, find a small disk neighborhood $N$ of $x$ in $\Ff$, and isotope $\Ff$ so that $N$ lies flat in $B^4_{\{i+1\}}$.  Near critical points of index zero or two, $\Ff$ now resembles a flat-topped or flat-bottomed cylinder; for index one critical points, $N$ is now a flat square.  Let $\frak b'$ denote the union of the flat, square neighborhoods of the index one critical points in $B^4_{\{2\}}$. 
	
	So far, we have achieved properties (2), (4), (6), and (7) of a realizing surface.  Properties (1), (3), and (5) say that $\Ff$ should be a gradient product on the intervals $(3,4]$, $(2,3)$, and $(1,2)$, respectively. The products $\Ff_{(3,4]}$ and $\Ll(3,4]$ (for example) agree at $\Ff_{\{4\}} = \Ll\{4\}$, but may disagree in $B^4_{\{t\}}$ for $t\in(3,4)$.  This issue can be addressed by a ``combing-out'' process.
	
	For each $t\in[1,4]$, we can choose ambient isotopies $G_t\colon [0,1]\times B^4_{\{t\}}\to B^4_{\{t\}}$ such that
	\begin{enumerate}
		\item $G_4(s,x) = x$ for all $s\in[0,1]$ and $x\in B^4_{\{4\}}$;
		\item $G_t(0,x) = x$ for all $t\in[1,4]$ and $x\in B^4_{\{t\}}$;
		\item $G_t(1,\Ff_{\{t\}}) = \Ll\{t\}$ for all $t\in(3,4]$, where we now let $\Ll = \Ff_{\{4\}}$;
		\item $G_t(1,\Ff_{\{t\}}) = (\Ll\sqcup U)\{t\}$ for all $t\in(2,3)$, where we now let $\Ll\sqcup U = G_3(\Ff_{\{3\}}\setminus\Int D)$;
		\item $G_t(1,\Ff_{\{t\}}) = U'\{t\}$ for all $t\in(1,2)$, where we now let $U' = G_2(\partial\Ff_{[0,2)})$; and
		\item $G_t$ is smoothly varying in $t$.
	\end{enumerate}
	
	After applying the family $G_t$ of ambient isotopies to $\Ff_{[1,4]}$, we have properties (1), (3), and (5), as desired. However, the ambient isotopies $G_t$ have now altered $\Ff_{\{t\}}$ for $t = 1, 2, 3$.  For example, the disks $D$ and $D'$ have been isotoped around in their respective level sets; but, clearly, properties (2), (4), (6), and (7) are still satisfied.  We remark that, if desired, we can choose $G_t$ so that (a) the disks of $D$ end up contained in small, disjoint 3--balls and either (b) the disks of $D'$ have the same property or (c) the bands $\frak b$ have the same property.  However, we cannot always arrange (a), (b), \emph{and} (c) if we want $\Ff_{(1,2)}$ to be a gradient product.
	
	With a slight abuse of notation, we now let $\Ll = \Ll\{2\}$, $U = U\{2\}$, and $\frak b = G_2(\frak b')$. (The only abuse is which level set of the now-gradient-product portion $\Ll[2,4]$ of $\Ff$ should be denoted by $\Ll$.) In the end, we have that $\Ff$ is the realizing surface of the band presentation $(\Ll,U,\frak b)$.
	
	With regards to the second claim of the proposition, assume that $\Ff$ has $c_1$ cups, $n$ bands, and $c_3$ caps once it is in Morse position.  Each cap gives rise to a component of $U$, while each cup gives rise to a component of $U'$.  The numbers of bands, cups, and caps are constant throughout the proof.
\end{proof}

Examples of a band presentations are shown below in Figures~\ref{fig:f81},~\ref{fig:steve1}, and~\ref{fig:square7}. However, each of these is a ribbon presentation.  Throughout the rest of the paper, we will work almost exclusively with ribbon presentation.  To emphasize the generality of Definition~\ref{def:band_pres}, we give in Figure~\ref{fig:band_pres}  a non-ribbon band presentation, where the black unknot is $\Ll$ and the orange unknot is $U$.  Note that a non-ribbon band presentation $(\Ll,U,\frak b)$ for a surface $\Ff$ can always be converted to a ribbon presentation $(\Ll',\frak b)$ for a surface $\Ff'$ by setting $\Ll'=\Ll\sqcup U$.  The ribbon surface $\Ff'$ is obtained from the non-ribbon surface $\Ff$ by puncturing at each maxima and dragging the resulting unlink to the boundary. 

\begin{figure}[h!]
	\centering
	\includegraphics[width=.4\textwidth]{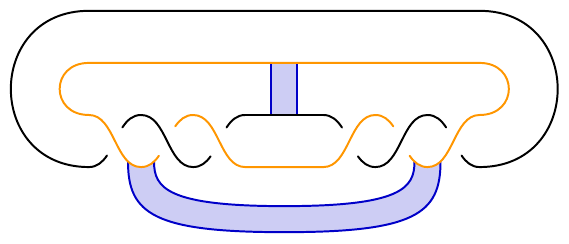}
	\caption{A band presentation for the punctured spun trefoil, considered as a neatly embedded disk in $B^4$ with unknotted boundary.}
	\label{fig:band_pres}
\end{figure}

\subsection{Bridge-braiding band presentations}\label{subsec:braiding_presentations}\ 

Recall the standard Heegaard-double decomposition $Y=Y_{0,0,1}$ of $S^3$ that was introduced in Subsection~\ref{subsec:Heegaard} and revisited in Subsection~\ref{subsec:Special}, which is a decomposition of $S^3$ into three trivial lensed cobordisms (three-balls), $H_1$, $H_3$, and $Y_3$, which meet along disk pages  $H_1\cap \overline H_3 = \Sigma$ and $H_i\cap Y_3 = P_i$ whose boundary is the unknotted braid axis $B$ in $S^3$. The choice to use $H_3$ instead of $H_2$ will ensure that the labelings of our pieces agree with our conventions for the labeling of the pieces of a bridge trisection; cf. the proof of Proposition~\ref{prop:band_to_bridge} below.

\begin{definition}[\textbf{\emph{bridge-braided}}]
\label{def:bridge-braided}
	A band presentation $(\Ll,U,\frak b)$, considered with respect to the standard Heegaard-page decomposition $Y_{0,0,1}$ of $S^3$, is called \emph{$(b,\bold c; v)$--bridge-braided} if the following conditions hold.
	\begin{enumerate}
		\item $\beta_3 = \Ll\cap Y_3$ is a $v$--braid;
		\item $\Ll\cap(H_1\cup_\Sigma\overline H_3)$ is an $b'$--perturbing of a $v$--braid;
		\item $U$ is in $b''$--bridge position with respect to $\Sigma$;
		\item $\frak b\cap\Sigma$ is precisely the cores $y_*$ of $\frak b$, which are embedded in $\Sigma$;
		\item There is a bridge system $\Delta$ for the trivial tangle $\Tt_3 = H_3\cap(\Ll\cup U)$ whose shadows $\Delta_*$ have the property that $\Delta_*\cup y_*$ is a collection of embedded arcs in $\Sigma$; and
		\item $U' = (\Ll\cup U)_\frak b$ is a $(c_1+v)$--component unlink that is in standard $(b,v)$--bridge position with respect to $Y_{0,0,1}$. (Hence, $U'$ consists of $c_1$ flat components and $v$ vertical components.)
	\end{enumerate}
	Here, $b = b' + b''$, $c_3 = |U|$, $c_2 =b-|\frak b|$, and $c_1 = |U'|-v$.  Let $\wh\beta$ denote the index $v$ braiding of $\Ll$ given by $\beta_3\cup\Tt_1\cup\overline\Tt_3$.  In reference to this added structure, we denote the bridge-braided band presentation by $(\wh\beta,U,\frak b)$. If $U=\emptyset$, so $(\Ll,\frak b)$ is a ribbon presentation, we denote the corresponding bridge-braiding by $(\wh\beta,\frak b)$.
\end{definition}

\begin{proposition}\label{prop:to_BBB_realizing}
	Let $\Ff\subset B^4$ be a surface with $\partial\Ff = \Ll$, and let $\wh\beta$ be an index $v$ braiding of $\Ll$.  There is a bridge-braided band presentation $(\wh\beta,U,\frak b)$ such that $\Ff = \Ff_{(\wh\beta,U,\frak b)}$. If $\Ff$ has a handle-decomposition with respect to the standard Morse function on $B^4$ consisting of $c_1$ cups, $n$ bands, and $c_3$ caps, then $(\wh\beta,U,\frak b)$ can be assumed to be $(b,(c_1,b-(n+v),c_3);v)$--bridge-braided, for some $b\in\N$.
\end{proposition}

\begin{proof}
	Consider $\Ff\subset B^4$ with $\partial\Ff = \Ll$. By Proposition~\ref{prop:band_pres}, we can assume (after an isotopy rel-$\partial$) that $\Ff = \Ff_{(\Ll,U,\frak b)}$ for some band presentation $(\Ll,U,\frak b')$.   We assume that $|U|=c_3$, $|\frak b'|=n$, and $|(\Ll\sqcup U)_{\frak b'}|=c_1$. By Alexander's Theorem~\cite{Ale_20_Note-on-Riemann-spaces}, there is an ambient isotopy $G_4\colon I\times B^4_{\{4\}}\to B^4_{\{4\}}$ taking $\partial \Ff$ to $\wh\beta$.  As in the proof of Proposition~\ref{prop:band_pres}, there is a family $G_t$ of ambient isotopies extending $G_4$ across $B^4$.  This results in the ``combing-out'' of Alexander's isotopy $G_4$, with the final effect that $\Ff$ is the realizing surface of the (not-yet-bridge-braided) band presentation $(\wh\beta,U,\frak b')$.  Henceforth, we consider the 2--complex corresponding to $(\wh\beta,U,\frak b')$ to be living in $B^4_{\{2\}}$, as in Proposition~\ref{prop:band_pres}.
	
	We have already obtained properties (1) and (2) towards a bridge-braided band presentation; although, presently $b'=0$.  (This will change automatically once we begin perturbing the bridge surface $\Sigma$ relative to $\wh\beta$ and $U$.)  By an ambient isotopy of $B^4_{\{2\}}$ that is the identity in a neighborhood of $\wh\beta$, we can move $U$ to lie in bridge position with respect to $\Sigma$, realizing property (3).  (Again, the bridge index $b''$ of this unlink will change during what follows.)  Since this ambient isotopy was supported away from $\wh\beta$ it can be combed-out (above and below) via a family of isotopies that are supported away from the gradient product $\wh\beta[2,4]$; so $\Ff$ is still the realizing surface.
	
	Next, after an ambient isotopy that fixes $\wh\beta\sqcup U$ set-wise (and point-wise near $\Sigma$), we can arrange that $\frak b'$ lies in $H_1\cup_\Sigma\overline H_3$. (Think of the necessity of sliding the ends of $\frak b'$ along $\beta_3$ to extract it from $Y_3$, while isotoping freely the unattached portion of $\frak b'$ to the same end.) This time, we need only comb-out towards $h^{-1}(0)$.  Using the obvious Morse function associated to $(H_1\cup_\Sigma \overline H_3)\setminus\nu(B)$, we can flow $\frak b'$, in the complement of $\wh\beta\sqcup U$, so that the cores of the bands lie as an immersed collection of arcs $y$ in $\Sigma\setminus\nu(\bold x)$.  At this point, we can perturb the bridge surface $\Sigma$ relative to $\wh\beta\sqcup U$ to arrange that the cores $y$ be embedded in $\Sigma$.  For details as to how this is achieved, we refer the reader to Figure~10 (and the corresponding discussion starting on page~17) of~\cite{MeiZup_17_Bridge-trisections-of-knotted}.  Now that the cores $y_*$ of $\frak b'$ are embedded in $\Sigma$, we can further perturb $\Sigma$ relative to $\wh\beta\sqcup U$ (as in Figure~11 of~\cite{MeiZup_17_Bridge-trisections-of-knotted}) to achieve that $\frak b'\cap\Sigma$ is precisely the cores of $\frak b'$.  Thus, we have that the bands $\frak b'$ satisfy property (4).  A further perturbation of $\Sigma$ relative to $\wh\beta\sqcup U$ produces, for each band $\upsilon$ of $\frak b'$, a dualizing bridge disk $\Delta_\upsilon$, as required by property (5). (See Figure~12 of~\cite{MeiZup_17_Bridge-trisections-of-knotted}.)
	
	However, at this point it is possible that the $c_1$--component unlink $U'' = (\widehat\beta\sqcup U)_{\frak b'}$ is \emph{not} in standard $(b,v)$--bridge position; more precisely, it is possible that components of $U''$ intersect $Y_3$ is more than one strand.
	On the other hand, we automatically have that $U''\cap Y_3$ is a $v$--braid, since the band resolutions changing $\Ll\cup U$ into $U''$ were supported away from $Y_3$.  Moreover, we know that $U''\cap H_i$ is a $(b,v)$--tangle; this follows from the proof of Lemma~3.1 of~\cite{MeiZup_17_Bridge-trisections-of-knotted}.
		
	Thus, we must modify $U''$ in order to obtain an unlink in standard position.
	To do so, we will produce a new collection $\frak b''$ of bands such that $U' = U''_{\frak b''}$ is a $(c_1+v)$--component unlink in $(b,v)$--bridge position. We call the bands $\frak b''$ \emph{helper bands}. We will then let $\frak b = \frak b'\sqcup\frak b''$, and the proof will be complete.
	
	Since $(Y_3,\beta_3)$ is a $v$--braid, there is a collection of bridge triangles $\Delta$ for $\beta_3$.  Let $\omega = \Delta\cap(P_1\cup_B\overline P_3)$.  Let $\frak b''$ denote the collection of $v$ bands whose core are the arcs $\omega$ and that are framed by the two-sphere $P_1\cup_B\overline P_3$.  By a minor isotopy that fixes $U''$ set-wise (and point-wise away from a neighborhood of $\partial\omega$), we consider $\frak b''$ as lying in the interior of $H_1\cup_\Sigma\overline H_3$.  Thus, $\frak b''$ is a collection of bands for $\Tt_1\cup_\bold x\overline\Tt_3$.  See Figure~\ref{fig:helpers} for two simple examples.		

\begin{figure}[h!]
\begin{subfigure}{.5\textwidth}
  \centering
  \includegraphics[width=.8\linewidth]{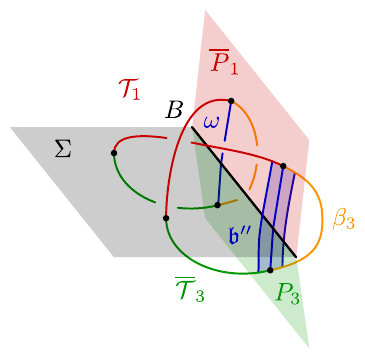}
  \caption{}
  \label{fig:helpers1}
\end{subfigure}%
\begin{subfigure}{.5\textwidth}
  \centering
  \includegraphics[width=.8\linewidth]{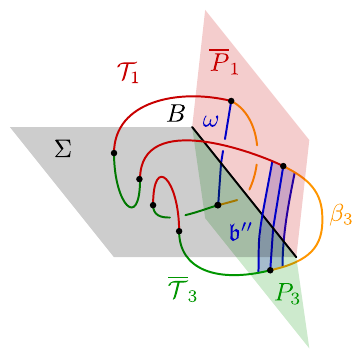}
  \caption{}
  \label{fig:helpers2}
\end{subfigure}
\caption{Adding extra bands to ensure that $U'$ is in standard $(b,v)$--bridge position.}
\label{fig:helpers}
\end{figure}

	Let $U' = U''_{\frak b''}$.  Let $J$ denote the components of $U'$ containing the strands of $\beta_3$.  Since the helper bands $\frak b''$ were created from the bridge triangles of $\Delta$, we find that $J$ bound a collection of $v$ disjoint meridional disks for $B$.  In particular, $J$ is a $v$--component unlink in $v$--braid position with respect to $B$.  Let $K = U'\setminus J$, and note that $K$ is isotopic (disregarding the Heegaard double structure) to the unlink $U''$.  It follows that $K$ is a $c_1$--component unlink in bridge position with respect to $\Sigma$.  Therefore, $U'$ is a $(c_1+v)$--component unlink in standard $(b,v)$--bridge position, as required by property (6) of Definition~\ref{def:bridge-braided}.
	
	Now, to wrap up the construction, we let $\frak b = \frak b'\cup\frak b''$.  While we have arranged the bands of $\frak b'$ are in the right position with respect to the Heegaard splitting, we must now repeat the process of perturbing the bridge splitting in order to level the helper bands $\frak b''$.  The end result is that the bands of $\frak b$ satisfy properties (4) and (5) of Definition~\ref{def:bridge-braided}.  In the process, we have not changed the fact that properties (1)--(3) and (6) are satisfied, though we may have further increased the parameters $b'$ and $b''$ (and, thus, $b=b'+b''$) during this latest bout of perturbing.
	
	We complete the proof by noting that $|U|=c_3$, $|U'|=c_1+v$, and $|\frak b|=n+v$.
\end{proof}

\begin{remark}
\label{rmk:helpers}
	A key technical step in the proof of Proposition~\ref{prop:to_BBB_realizing} was the addition of the so-called \emph{helper bands} $\frak b''$ to the original set $\frak b'$ of bands that were necessary to ensure that $U'$ was in standard position. In the proof, $\frak b''$ consisted of $v$ bands; in practice, one can make do with a subset of these $v$ bands.
	This can be seen in the two simple examples of Figure~\ref{fig:helpers}, where the addition of only one band (in each example) suffices to achieve standard bridge position.  In Figure~\ref{fig:helpers1}, the addition of the single band shown transforms an unknot component of $U''$ that is in 2--braid position into a pair of 1--braids (one of which is perturbed) in the link $U'$.  In Figure~\ref{fig:helpers2}, an unknot component that is not braided at all is transformed to the same result.  In each of these examples, the addition of a second band corresponding to the second arc of $\omega$ would be superfluous.
	
	From a Morse-theoretic-perspective, the helper bands correspond to cancelling  pairs of minima and saddles: the minima are the meridional disks bounded by $J$. Using more bands from $\frak b''$ than is strictly necessary results in a surface with more minima (and bands) than are actually required to achieve the desired bridge-braided band presentation.  Below, when we convert the bridge-braided band presentation to a bridge trisection, we will see that the superfluous bands and minima have the effect that the bridge trisection produced is perturbed -- see Section~\ref{sec:stabilize}.  Another way of thinking about the helper bands is that they ensure that the trivial disk-tangle $\Dd_1$ in the resulting bridge trisection has enough vertical patches.
\end{remark}

Before proving that a bridge-braided band presentation can be converted to a bridge trisection, we pause to give a few examples illustrating the process of converting a band presentation into a bridge-braided band presentation.

\begin{example}
\label{ex:F81}
	\textbf{(Figure-8 knot Seifert surface)} Figure~\ref{fig:f81} shows a band presentation for the genus one Seifert surface for the figure-8 knot, together with a gray dot representing an unknotted curve about which the knot will be braided; this braiding is shown in Figure~\ref{fig:f82}.
	Note that the resolution of the bands at this point would yield a unknot (denoted $U''$ in the proof of Proposition~\ref{prop:to_BBB_realizing}) that is in 3--braid position.  Thus, at least two helper bands are need.  In Figure~\ref{fig:f83} we have attached three helper bands, as described in the proof of Proposition~\ref{prop:to_BBB_realizing}.  Note that the cores of these bands are simultaneously parallel to the arcs one would attach to form the braid closure, and the disks exhibiting this parallelism correspond to the bridge triangles in the proof.
	In Figure~\ref{fig:f84}, all five bands have been leveled so that they are framed by the bridge sphere, intersecting it only in their cores.  In addition, each band is dualized by a bridge disk for $\Tt_3$.  Three of these bridge disks are obvious. The remaining two are only slightly harder to visualize; one can choose relatively simple disks corresponding to any two of the three remaining flat arcs.

\begin{figure}[h!]
\begin{subfigure}{.15\textwidth}
  \centering
  \includegraphics[width=.6\linewidth]{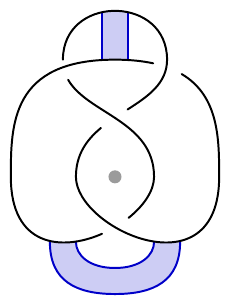}
  \caption{}
  \label{fig:f81}
\end{subfigure}%
\begin{subfigure}{.2\textwidth}
  \centering
  \includegraphics[width=.5\linewidth]{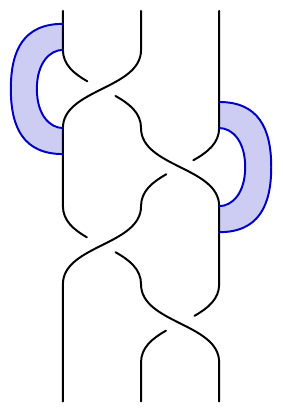}
  \caption{}
  \label{fig:f82}
\end{subfigure}%
\begin{subfigure}{.2\textwidth}
  \centering
  \includegraphics[width=.9\linewidth]{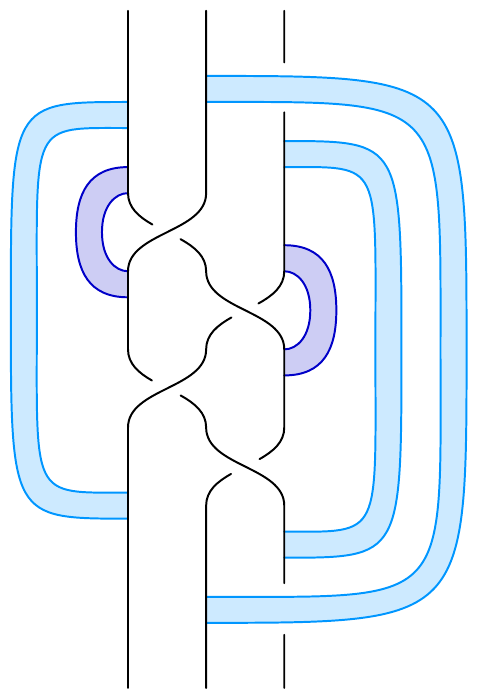}
  \caption{}
  \label{fig:f83}
\end{subfigure}%
\begin{subfigure}{.45\textwidth}
  \centering
  \includegraphics[width=.8\linewidth]{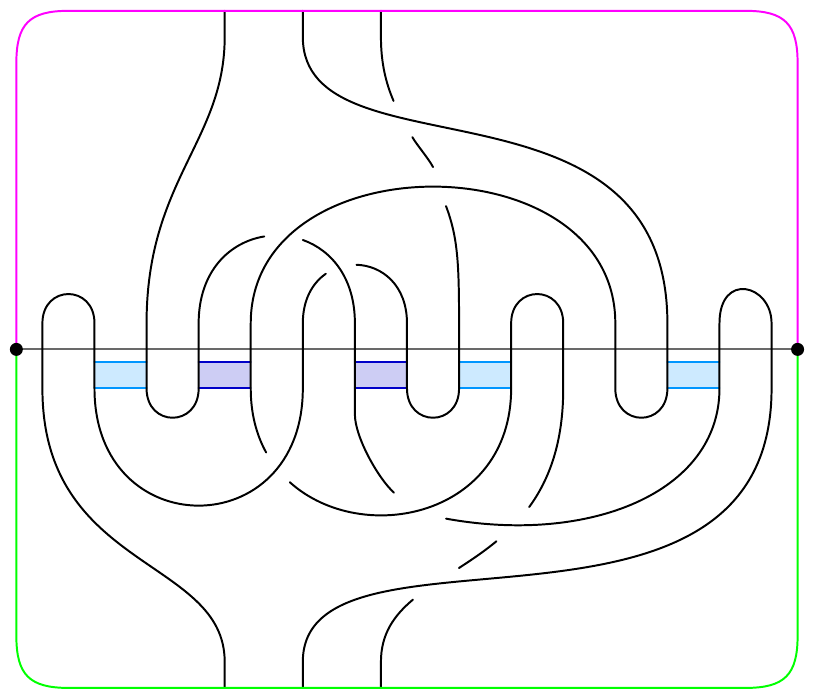}
  \caption{}
  \label{fig:f84}
\end{subfigure}
\par\vspace{5mm}
\begin{subfigure}{\textwidth}
  \centering
  \includegraphics[width=\linewidth]{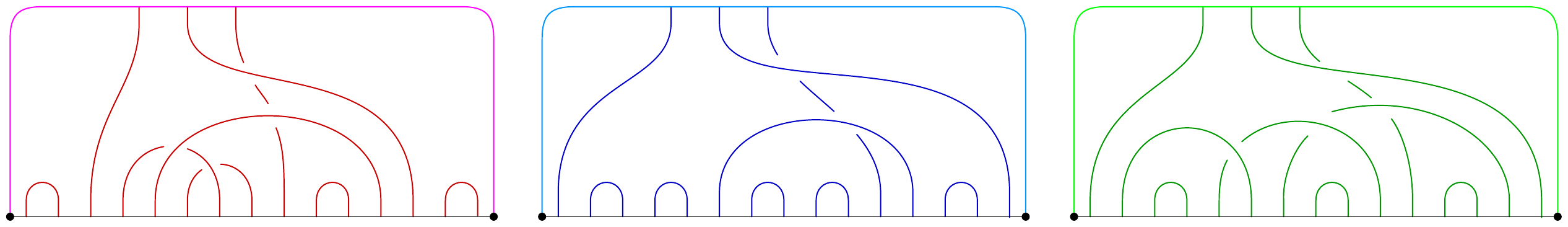}
  \caption{}
  \label{fig:f85}
\end{subfigure}
\par\vspace{5mm}
\begin{subfigure}{\textwidth}
  \centering
  \includegraphics[width=\linewidth]{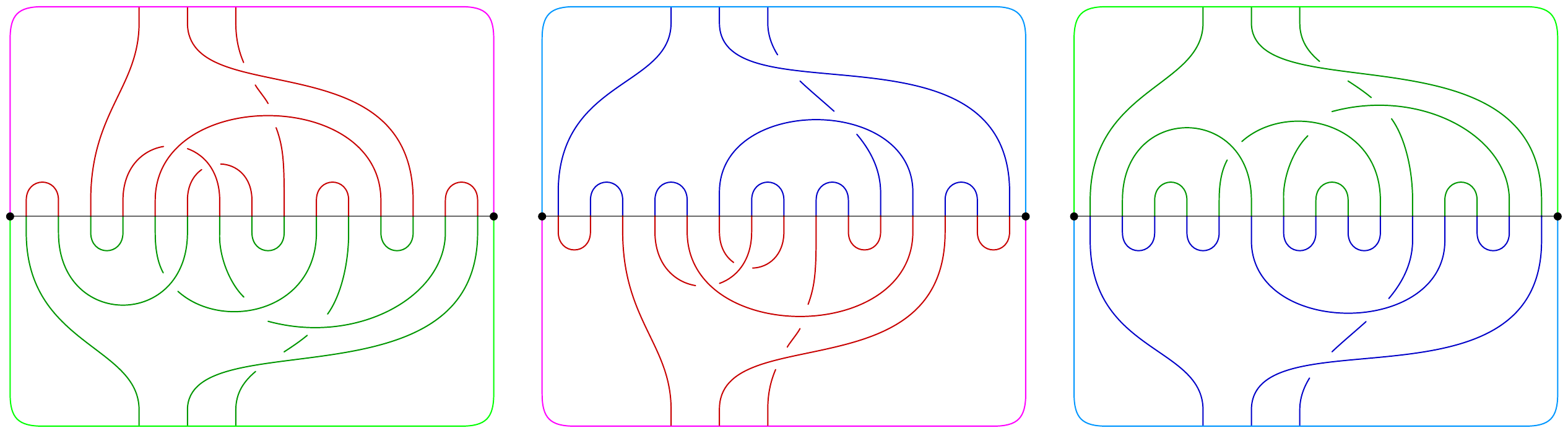}
  \caption{}
  \label{fig:f86}
\end{subfigure}
\caption{(A)--(D) The process of converting a band presentation for the genus one Seifert surface for the figure-8 knot into a bridge-braided band presentation. (E) A tri-plane diagram corresponding to the bridge-braided band presentation of~(D).  See Figure~\ref{fig:fig8} for a second instantiation of this example.}
\label{fig:f8}
\end{figure}
			
	Figure~\ref{fig:f85} shows a tri-plane diagram for the bridge trisection that can be obtained from the bridge-braided band presentation given in Figure~\ref{fig:f84} according to Proposition~\ref{prop:band_to_bridge}. (See Section~\ref{sec:tri-plane} for precise details regarding tri-plane diagrams.)
	Figure~\ref{fig:f86} shows the pairwise unions of the seams of this bridge trisection.  Relevant to the present discussion is the fact that the second two unions each contain a closed, unknotted component.  The fact that the red-blue union contains such a component is related to the fact that we chose to use three helper bands, when two would suffice.
	The fact that the green-blue union contains such a component is related to the fact that the bridge splitting in Figure~\ref{fig:f84} is excessively perturbed.  We leave it as an exercise to the reader to deperturb the bridge splitting of Figure~\ref{fig:f84} to obtain a simpler bridge-braided band presentation.
\end{example}

\begin{example}
\label{ex:f82}
	\textbf{(Figure-8 knot Seifert surface redux)} As discussed in Remark~\ref{rmk:helpers}, it is often not necessary to append $v$ helper bands. The frames of Figure~\ref{fig:fig8} are analogous to those of Figure~\ref{fig:f8}, with the main change being that only two of the three helper bands are utilized. The two inner most bands from Figure~\ref{fig:f83} have been chosen, and they have each been slid once over the original bands from Figure~\ref{fig:fig82} to make subsequent picture slightly simpler.  			
			
\begin{figure}[h!]
\begin{subfigure}{.2\textwidth}
  \centering
  \includegraphics[width=.64\linewidth]{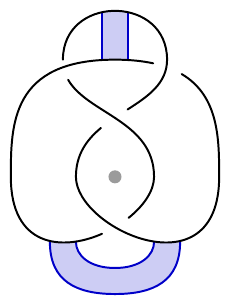}
  \caption{}
  \label{fig:fig81}
\end{subfigure}%
\begin{subfigure}{.2\textwidth}
  \centering
  \includegraphics[width=.72\linewidth]{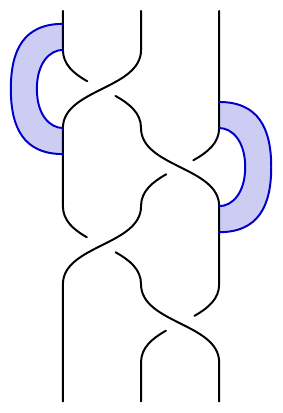}
  \caption{}
  \label{fig:fig82}
\end{subfigure}%
\begin{subfigure}{.2\textwidth}
  \centering
  \includegraphics[width=.72\linewidth]{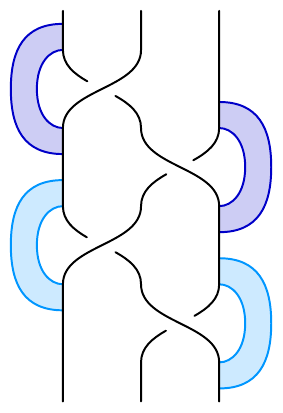}
  \caption{}
  \label{fig:fig83}
\end{subfigure}%
\begin{subfigure}{.4\textwidth}
  \centering
  \includegraphics[width=.7\linewidth]{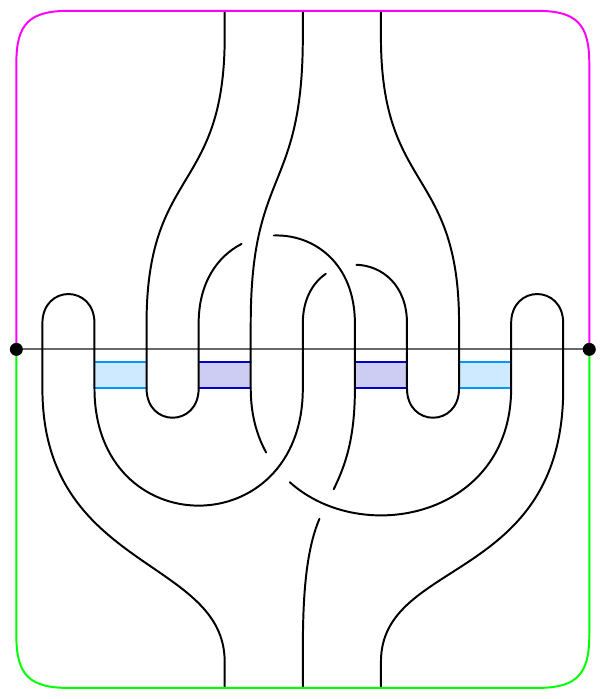}
  \caption{}
  \label{fig:fig84}
\end{subfigure}
\par\vspace{5mm}
\begin{subfigure}{\textwidth}
  \centering
  \includegraphics[width=.8\linewidth]{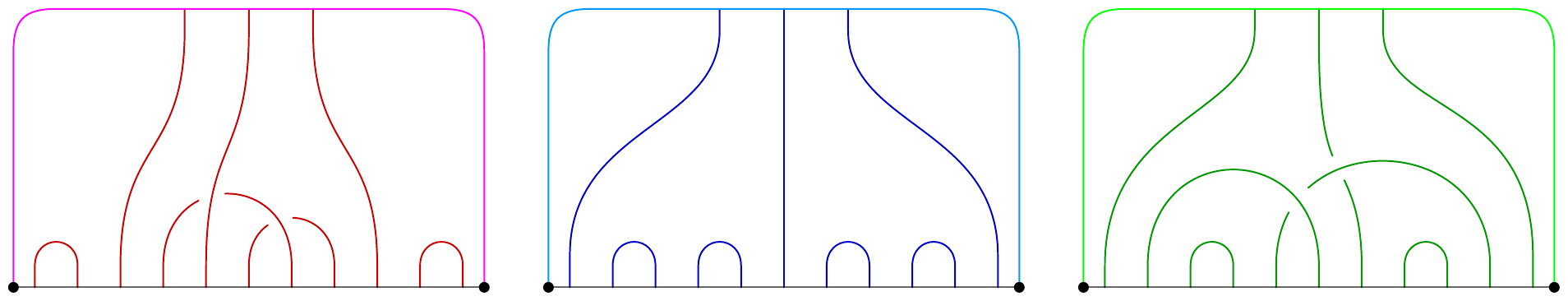}
  \caption{}
  \label{fig:fig85}
\end{subfigure}
\par\vspace{5mm}
\begin{subfigure}{\textwidth}
  \centering
  \includegraphics[width=.8\linewidth]{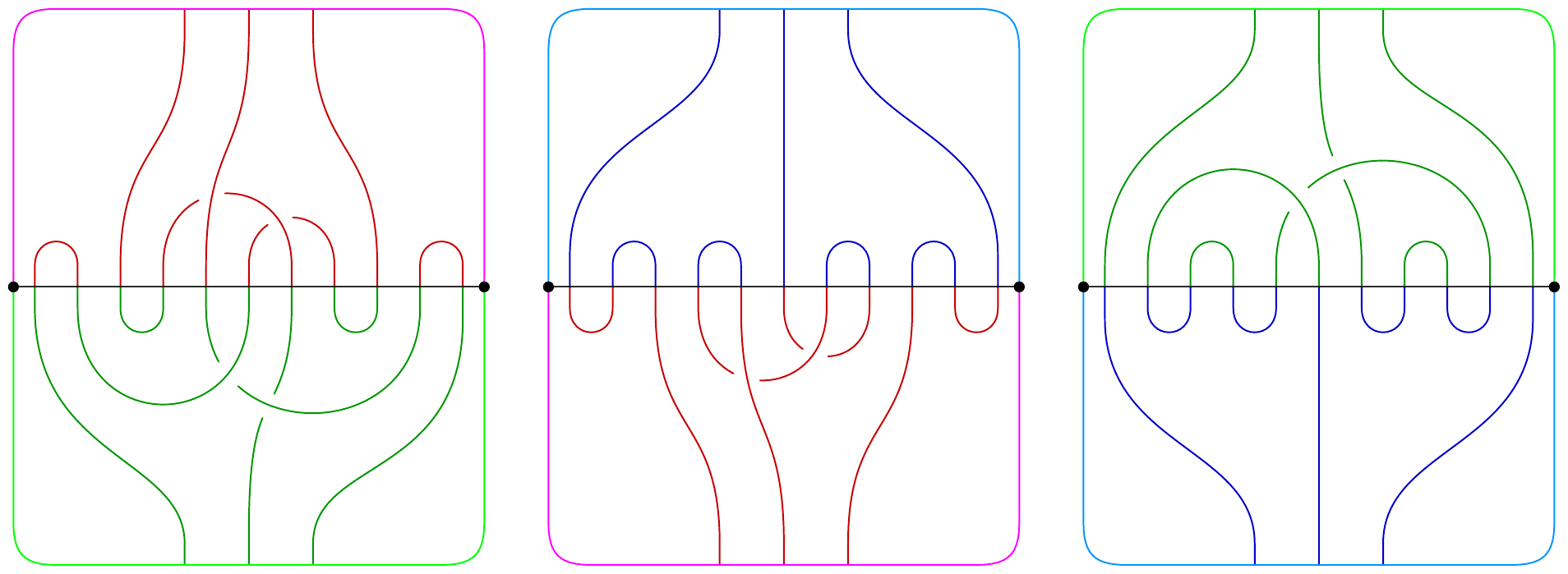}
  \caption{}
  \label{fig:fig86}
\end{subfigure}
\caption{(A)--(D) The process of converting a band presentation for the genus one Seifert surface for the figure-8 knot into a bridge-braided band presentation. (E) A tri-plane diagram corresponding to the bridge-braided band presentation of~(D). See Figure~\ref{fig:f8} for another instantiation of this example.}
\label{fig:fig8}
\end{figure}

	Since fewer bands are included, the bridge splitting required to level and dualize them is simpler. In this case, the perturbing in Figure~\ref{fig:fig84} is minimal.
	In light of these variations, we see in Figure~\ref{fig:fig86} that the pairwise unions of the seams of the bridge trisection contain no closed components, implying the bridge trisection is not perturbed -- see Section~\ref{sec:stabilize}.
\end{example}

\begin{example}
\label{ex:steve}				
	\textbf{(Stevedore knot ribbon disk)} Figure~\ref{fig:steve1} shows a band presentation for a ribbon disk for the stevedore knot, together with a gray dot representing an unknotted curve about which the knot is braided in Figure~\ref{fig:steve2}.  Note that the result of resolving the band in Figure~\ref{fig:steve2} is a 4--braiding of the 2--component unlink, with each component given by a 2--braid.  Thus, at least two helper bands are required to achieve bridge-braided band position in this example; Figure~\ref{fig:steve3} shows two such bands that suffice.  (See Remark~\ref{rmk:helpers2} below.)

\begin{figure}[h!]
\begin{subfigure}{.15\textwidth}
  \centering
  \includegraphics[width=.8\linewidth]{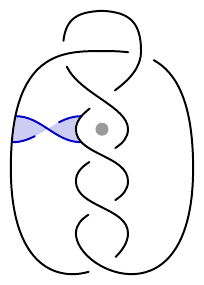}
  \caption{}
  \label{fig:steve1}
\end{subfigure}%
\begin{subfigure}{.225\textwidth}
  \centering
  \includegraphics[width=.48\linewidth]{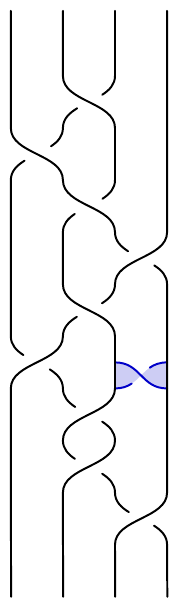}
  \caption{}
  \label{fig:steve2}
\end{subfigure}%
\begin{subfigure}{.225\textwidth}
  \centering
  \includegraphics[width=.8\linewidth]{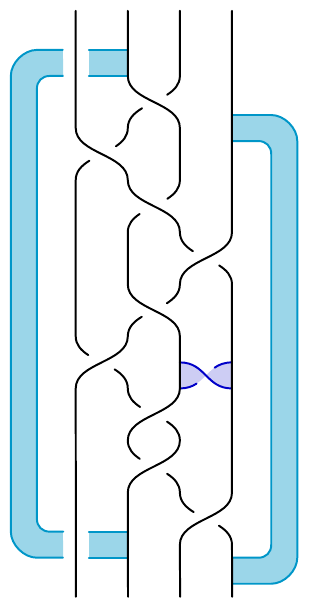}
  \caption{}
  \label{fig:steve3}
\end{subfigure}%
\begin{subfigure}{.4\textwidth}
  \centering
  \includegraphics[width=.7\linewidth]{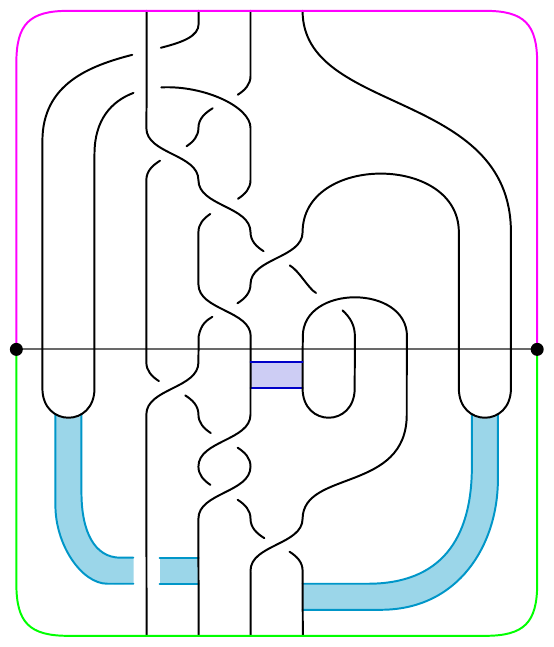}
  \caption{}
  \label{fig:steve4}
\end{subfigure}
\par\vspace{5mm}
\begin{subfigure}{\textwidth}
  \centering
  \includegraphics[width=.9\linewidth]{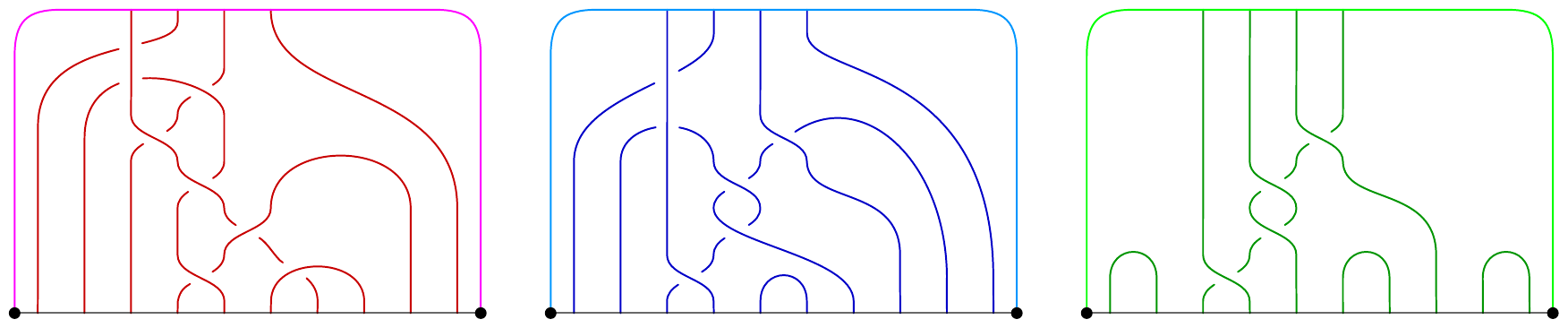}
  \caption{}
  \label{fig:steve5}
\end{subfigure}
\par\vspace{5mm}
\begin{subfigure}{\textwidth}
  \centering
  \includegraphics[width=.9\linewidth]{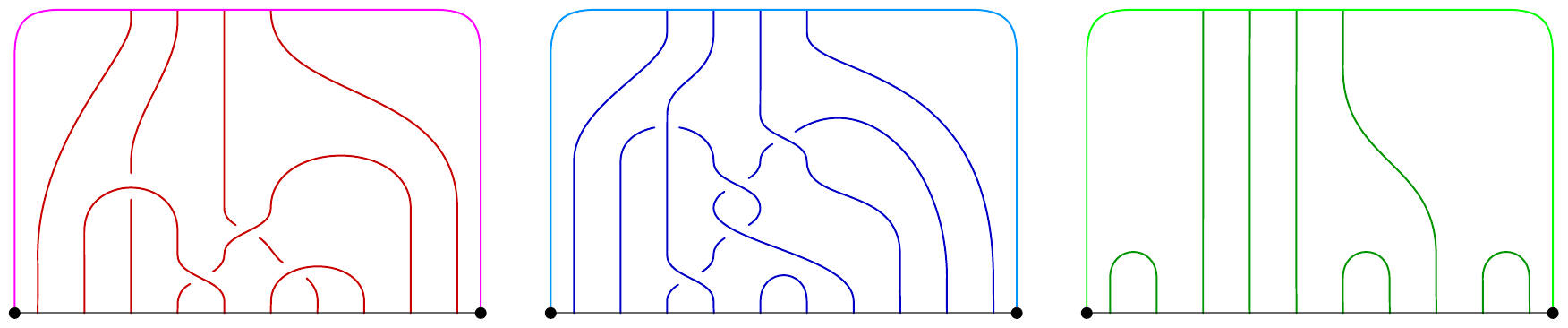}
  \caption{}
  \label{fig:steve6}
\end{subfigure}
\caption{(A)--(D) The process of converting a band presentation for a ribbon disk for the stevedore knot into a bridge-braided band presentation. (E) A tri-plane diagram corresponding to the bridge-braided band presentation of~(D). (F) A second tri-plane diagram, obtained from the first via a sequence of tri-plane moves.}
\label{fig:steve}
\end{figure}
	
	Figure~\ref{fig:steve4} gives a bridge-braided band presentation for the ribbon disk, with the caveat that the helper bands do not appear to be leveled as shown.  However, we claim that such a leveling is possible:  First, note that the left helper band can be isotoped so that its core lies in the bridge sphere without self-intersection.  Depending on how one chooses to do this, the core may intersect the core of the dark blue band (the original fission band for the ribbon disk).  However, since this latter band is dualized by a bridge disk for $\Tt_3$, there is an isotopy pushing the helper band off the fission band.  At this point, the left helper band and the fission band are both level, disjoint, and dualized by bridge disks.  Now, we note that the right helper band can be isotoped so that its core lies in the bridge sphere without self-intersection. To do this, however, we must slide the right helper band over the fission band so that their endpoints (attaching regions) are disjoint.  Again, the core may intersect the cores of the other two bands, but since the other two bands are each dualized by bridge disks, we may push the core of the right helper band off the cores of the other two bands.  The end result is that all three bands lies in the required position.
		
	Figure~\ref{fig:steve5} shows a tri-plane diagram for the bridge trisection corresponding to the bridge-braided band position from Figure~\ref{fig:steve4}.  It is worth observing that it was not necessary to carry out the leveling of the bands described in the previous paragraph; it suffices simply to know that it can be done.  Had we carried out the leveling described above, the result would have been a tri-plane diagram that could be related to the one given by a sequence of interior Reidemeister moves. Figure~\ref{fig:steve6} shows a tri-plane diagram that is related to the tri-plane diagram of Figure~\ref{fig:steve5} by tri-plane moves.  See Section~\ref{sec:tri-plane} for details regarding these moves.
\end{example}

\begin{remark}
\label{rmk:helpers2}
	There is a subtle aspect to Figure~\ref{fig:steve3} that is worth pointing out.  Suppose instead that the left helper band were chosen to cross over the braid in the two places where it crosses under.  It turns out that this new choice is still a helper band but would fail to result in a bridge-braided band position.
	To be precise, let $\Tt$ denote the braid in Figure~\ref{fig:steve3}, which we think of as a 4--stranded tangle, and let $\frak b$ denote this new choice of bands -- i.e., three bands that are identical to the ones shown in Figure~\ref{fig:steve3}, except that the left helper band passes above $\Tt$ in two places, rather than under.  The resolution $\Tt_\frak b$ is a new 4--stranded tangle.  Regardless of any concerns about bridge position that could be alleviated by perturbing $\Tt$, it is necessary that $\Tt_\frak b$ be a 4--braid.  However, this is not the case in this example.  In fact, $\Tt_\frak b$ is not even a trivial tangle!  The reader can check that $\Tt_\frak b$ is the split union of two trivial arcs, together a 2--stranded tangle $\Tt'$ that has a closure to the square knot.
	
	So, the ``helper bands'' of the $\frak b$ presently being considered are not actually helper bands in the sense that they don't transform $U''$ into an unlink $U'$ in standard position, as required.  Of course, by the proof of Proposition~\ref{prop:to_BBB_realizing}, we know that we can augment $\frak b$ by adding two more helper bands, resulting in a total of five bands, so that the result can be bridge-braided.  On the other hands, Figure~\ref{fig:steve} shows that it is possible to achieve a bridge-braided band position with fewer than four helper bands; comparison of Figures~\ref{fig:f8} and~\ref{fig:fig8} gives another example of this. Precisely when this is possible and precisely how one chooses a more efficient set of helper bands of this sort is not clear; we pose the following question.
\end{remark}

\begin{question}
	Does there exist a surface $\Ff$ in $B^4$ such that every $(b,v)$--bridge braided band presentation of $\Ff$ requires $v$ helper bands?
\end{question}

Such a surface would have the property that every bridge trisection contains some flat patches.  For this reason, it cannot be ribbon, due to the results of Subsection~\ref{subsec:Ribbon} below. 

Having discussed in detail the above examples, we now return our attention to the goal of bridge trisecting surfaces.

\begin{proposition}\label{prop:band_to_bridge}
	Let $\Ff\subset B^4$ be the realizing surface for a $(b,\bold c;v)$--bridge-braided band presentation $(\wh\beta,U,\frak b)$.  Then, $\Ff$ admits a $(b,\bold c;v)$--bridge trisection $\Tt_{(\wh\beta,U,\frak b)}$.
\end{proposition}

\begin{proof}
	As in Proposition~\ref{prop:band_pres}, we imagine that the 2--complex $\Ll\cup U\cup \frak b$ corresponding to the  bridge-braided band presentation $(\wh\beta, \frak b, U)$ is lying in the level set $B^4_{\{2\}}$, which inherits the Heegaard double structure $(H_1,H_3,Y_3)$. Assume that $\Ff$ is the corresponding realizing surface.  We modify this 2--complex so that the bands $\frak b$ lie in the interior of $H_3$, rather than centered on $\Sigma$.
	
	Let $\epsilon>0$, and assume that the resolution of the bands $\frak b$ for $\Ll\cup U$ occurs in $H_3(2-\epsilon,2)$.  So, $\Ff\{2\} = \Ll\cup U$, while $\Ff\{2-\epsilon\} = U'$.  Let $(P_3^+,\bold x_3^+)$ denote a slight push-off of $(P_3,\bold x_3)$ into $(H_3,\Tt_3)$.  Let $(H^-_{13},\beta^-_{13})$ denote the corresponding contraction of $(Y_3,\beta_3)$, and let $(H^+_3,\Tt^+_3)$ denote the corresponding expansion of $(H_3,\Tt_3)$.  In other words, we remove a (lensed) collar of $P_3$ from $Y_3$ and add it to $H_3$.
	
	We will now describe the pieces of a bridge trisection for $\Ff$. Figure~\ref{fig:Morse_to_tri} serves as a guide to the understanding these pieces. Define: 
	\begin{enumerate}
		\item $(\Sigma',\bold x') = (\Sigma,\bold x)\{2\}\cup B[2,4]$;
		\item $(H_1',\Tt_1') = (H_1,\Tt_1)\{2\} \cup (P_1,\bold x_1)[2,4]$;
		\item $\overline{(H_2',\Tt_2')} = (\Sigma,\bold x)[2-\epsilon,2] \cup (H_3^+,\Tt_3^+)\{2-\epsilon\} \cup (P_3^+,\bold x_3^+)[2-\epsilon,4]$;
		\item $(H_3',\Tt_3') =  (H_3,\Tt_3)\{2\} \cup (P_3,\bold x_3)[2,4] $;
		\item $(Z_1',\Dd_1') = (B^4,\Ff)_{[0,2-\epsilon]}\cup((H_1,\Tt_1)[2-\epsilon,2])\cup(Y_3^-,\beta_3^-)[2-\epsilon,2]$;
		\item $(Z_2',\Dd_2') = ((B^4,\Ff)_{[2-\epsilon,2]}\cap H_3^+[2-\epsilon,2]) \cup ((Y_3\setminus\Int(Y_3^-),\beta_3\setminus\Int(\beta_3^-))[2,4])$; and
		\item $(Z_3',\Dd_3') = (B^4,\Ff)_{[2,4]}\cap(H_1\cup_\Sigma\overline H_3)[2,4]$.
	\end{enumerate}
	
\begin{figure}[h!]
	\centering
	\includegraphics[width=.5\textwidth]{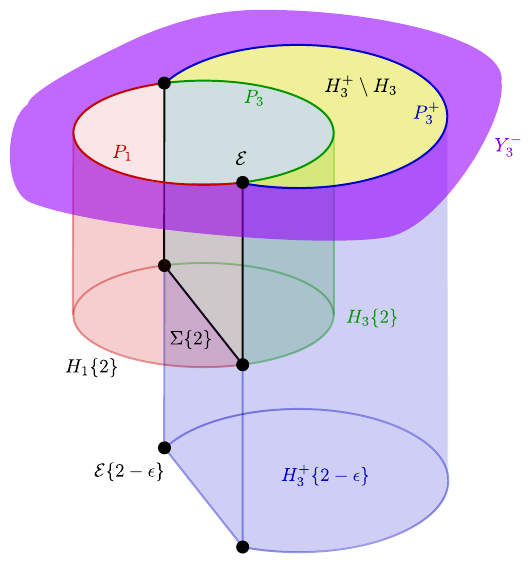}
	\caption{A schematic illustrating how to obtain a bridge trisection from a bridge-braided band presentation; codimension two objects are not shown.}
	\label{fig:Morse_to_tri}
\end{figure}

It is straight-forward to verify that the pairs (1)-(7) have the right topology, except in the case of (3) and (6), where slightly more care is needed.  For (3), the claim is that $(H_2',\Tt_2') \cong (H_2, (\Tt_2)_\frak b)$ is a trivial $(b,v)$--tangle.  For (6), the claim is that the trace $(Z_2',\Dd_2')$ of this band attachment is a trivial $(c_2,v)$--disk-tangle.  Both of these claims follow from the fact that each band of $\frak b$ is dualized by a bridge disk for $\Tt_3$; this is essentially Lemma~3.1 of~\cite{MeiZup_17_Bridge-trisections-of-knotted}.  Finally, it only remains to verify that the pieces (1)-(7) intersect in the desired way.  This is straight-forward to check, as well.
\end{proof}

\begin{remark}
\label{rmk:or1}
	Care has been taken to track the orientations throughout this section so that the orientations of the pieces of the bridge trisection produced in Proposition~\ref{prop:band_to_bridge} agree with the orientation conventions given in Subsection~\ref{subsec:Formal}.  For example, the union $H_1\cup_\Sigma\overline H_3$ appearing in the bridge-braided band presentation set-up of Definition~\ref{def:bridge-braided} gets identified with a portion of $B^4\{2\}$ in the proof of Proposition~\ref{prop:band_to_bridge}, where it is oriented as the boundary of $B^4[0,2]$.  This agrees with the convention that $\partial Z_1 = H_3\cup_\Sigma H_1\cup Y_3$, so $\partial(Z_2\cup Z_3) = Y_1\cup H_1\cup_\Sigma \overline H_3\cup Y_2$.  See Figure~\ref{fig:two-thirds}.
\end{remark}

\begin{proposition}\label{prop:bridge_to_band}
	If $\Ff$ admits a $(b,\bold c;v)$--bridge trisection, then $\Ff = \Ff_{(\wh\beta, U,\frak b)}$ for some $(b,\bold c;v)$--bridge-braided band presentation $(\wh\beta,U,\frak b)$.
\end{proposition}

\begin{proof}
	Suppose $\Ff$ is in bridge position with respect to $\T_0$.  Consider link $L_3 = \beta_3\cup\Tt_3\cup\overline\Tt_1 = \partial \Dd_3$. Let $\overline L$ denote the vertical components of $L_3\setminus\Int(\beta_3) = \Tt_3\cup_\bold x\overline\Tt_1$, and let $U$ denote the flat components.  Then we have $\partial\Dd_3 = \overline L\cup\beta_3\cup U$; in particular, $\overline L$ is parallel to $\beta_3$ (as oriented tangles) through the vertical disks of $\Dd_3$.  Let $\Ll$ be the closed one-manifold given by
	$$\beta_1\cup\beta_2\cup \overline L.$$
	By the above reasoning, $\Ll$ is boundary parallel to the boundary braid $\beta_1\cup\beta_2\cup\beta_3 = \wh\beta = \partial\Ff$ via the vertical disks of $\Dd_3$.
	
\begin{figure}[h!]
	\centering
	\includegraphics[width=.3\textwidth]{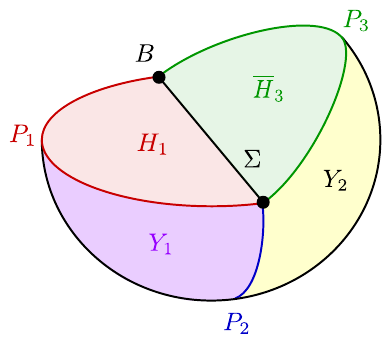}
	\caption{Two-thirds of a trisection, with induced orientations on the boundary.}
	\label{fig:two-thirds}
\end{figure}
	
	Let $Y = Y_1\cup H_1 \cup \overline H_3\cup Y_2$ and note that $Y$ has the structure of a standard Heegaard-double decomposition $(H_1, H_3, Y_1\cup Y_2)$ on $S^3 = \partial(Z_1\cup Z_2)$ and is oriented as the boundary of $Z_1\cup Z_2$, which induces the opposite orientations on the 3--balls $H_1$ and $H_3$ as does $Z_3$.  See Figure~\ref{fig:two-thirds}. It will be with respect to this structure that we produce a bridge-braided band presentation for $\Ff$.    Note that $\Ll\cap (Y_1\cup Y_2)$ is already a $v$--braid, giving condition (1) of the definition of a bridge-braided band presentation.  Similarly, conditions (2) and (3) have been met given the position of $\overline L\cup U$ with respect to the Heegaard splitting $H_1\cup_\Sigma \overline H_3$.
	
	Next, we must produce the bands $\frak b$.  This is done in the same way as in Lemma~3.3 of~\cite{MeiZup_17_Bridge-trisections-of-knotted}.  We consider the bridge splitting $(H_2,\Tt_2)\cup_{(\Sigma,\bold x)}\overline{(H_3,\Tt_3)}$, which is standard -- i.e., the union of a perturbed braid and a bridge spilt unlink. Choose shadows $\Tt_2^*$ and $\Tt_3^*$ on $\Sigma$ for these tangles.  Note that we choose shadows only for the flat strands in each tangle, not for the vertical strands.  Because the splitting is standard, we may assume that $\Tt_2^*\cup\Tt_3^*$ is a disjoint union of $c_2$ simple closed curves $C_1,\ldots,C_{c_2}$, together with some embedded arcs, in the interior of $\Sigma$.  For each closed component $C_i$, choose a shadow $\bar \tau_i^*\subset(\Tt_2^*\cap C_i)$.  Let 
	$$\omega^* = \Tt_2^*\setminus\left(\bigcup_{i=1}^{c_2}\bar\tau_i^*\right).$$
	In other words, $\omega^*$ consists of the shadow arcs of $\Tt_2^*$, less one arc for each closed component of $\Tt_2^*\cup\Tt_3^*$. Note that $|\omega^*| = b-c_2$.
	
	The arcs of $\omega^*$ will serve as the cores of the bands $\frak b$ as follows. Let $\frak b = \omega^*\times I$, where the interval is in the vertical direction with respect to the Heegaard splitting $H_1\cup_{\overline\Sigma} \overline H_3$.  In other words, $\frak b$ is a collection of rectangles with vertical edges lying on $\overline L\cup U$ and a horizontal edge in each of $H_1$ and $\overline H_3$ that is parallel through $\frak b$ to $\omega^*$.  We see that condition (4) is satisfied.
	
	Note that the arcs $\omega^*$ came from chains of arcs in $\Tt_2^*\cup\Tt_3^*$, so each one is adjacent to a shadow arc in $\Tt_3^*$.  This is obvious in the case of the closed components, since each such component must be an even length chain of shadows alternating between $\Tt_2^*$ and $\Tt_3^*$.  Similarly, each non-closed component consists of alternating shadows. This follows from the fact that these arcs of shadows correspond to vertical components of $\overline L$, each of which must have the same number of bridges on each side of $\Sigma$. These adjacent shadow arcs in $\Tt_3^*$ imply that $\frak b$ is dual to a collection of bridge disks for $\Tt_3$, as required by condition (5).
	
	Finally, let $U' = \Ll_\frak b$, which should be thought of as lying in $H_1\cup \overline H_2\cup\beta_1$.  In fact, $U' = \Tt_1\cup\overline \Tt_2\cup\beta_1$, so it is the standard link $L_{c_1,w}$ in the standard Heegaard-double structure on $\partial Z_1$.  Thus, (6) is satisfied, and the proof is complete.
\end{proof}

The following example illustrates the proof of Proposition~\ref{prop:bridge_to_band}.

\begin{example}
\label{ex:square}
	Figure~\ref{fig:square1} shows a tri-plane diagram for a surface that we will presently determine to be the standard ribbon disk for the square knot, as described by the band presentation in Figure~\ref{fig:square7}.  The first step to identifying the surface is to identify the boundary braid.  In the proof of Proposition~\ref{prop:bridge_to_band}, this was done by considering the union $\beta_1\cup\beta_2\cup\overline L$.  Diagrammatically, this union can be exhibited by the following three part process: (1) Start with the cyclic union
	$$\Tt_1\cup\overline\Tt_3\cup\Tt_3\cup\overline\Tt_2\cup\Tt_2\cup\overline\Tt_1$$
	of the seams of the bridge trisection; see Figure~\ref{fig:square3}. (2) Discard any components that are not braided; there are no such components in the present example, though there would be if this process were repeated with the tri-plane diagram in Figure~\ref{fig:f85} -- a worthwhile exercise.  (3) Straighten out (deperturb) near the intersections $\Tt_3\cap\overline\Tt_2$ and $\Tt_2\cap\overline\Tt_1$; see Figure~\ref{fig:square4}.  If we continued straightening out near $\Tt_1\cup\overline\Tt_3$, we would obtain a braid presentation for the boundary link; see Subsection~\ref{subsec:tri-plane_braid} for a discussion relating to this point.  Presently, however, it suffices to consider the 1--manifold $\beta_1\cup\beta_2\cup\overline L$ shown in Figure~\ref{fig:square4}, which we know to be isotopic (via the deperturbing near $\Tt_1\cap\overline\Tt_3$) to the boundary braid.

\begin{figure}[h!]
\centering
\begin{tabular}{ccccc}
\multirow{3}{*}{
\begin{subfigure}{.21\textwidth}
\setcounter{subfigure}{2}
\centering
\includegraphics[width=.9\linewidth]{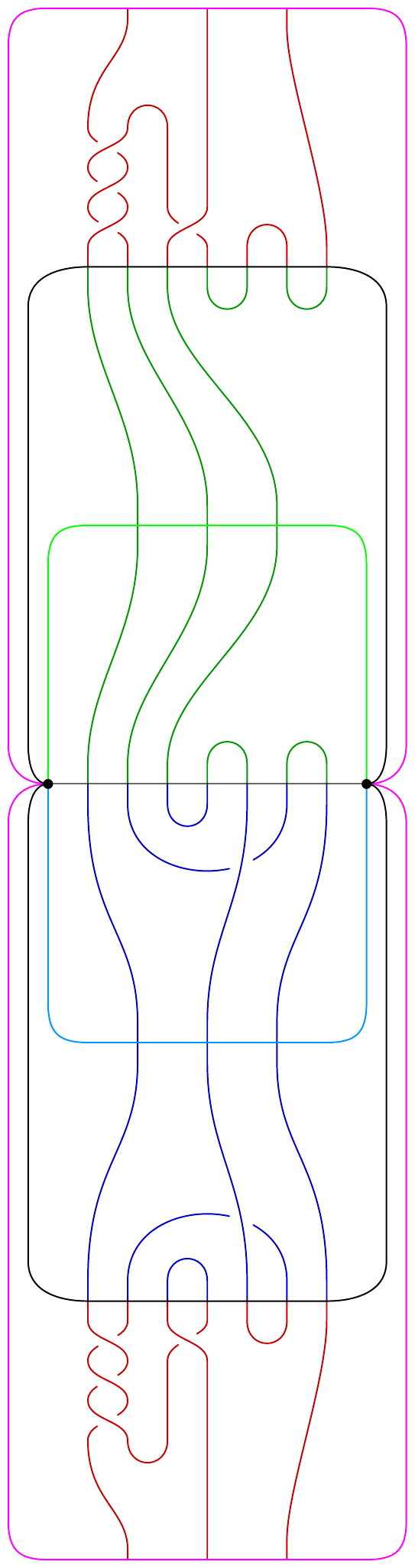}
\caption{}
\label{fig:square3}
\end{subfigure}%
}
&
\multirow{3}{*}{
\begin{subfigure}{.21\textwidth}
\centering
\includegraphics[width=.9\linewidth]{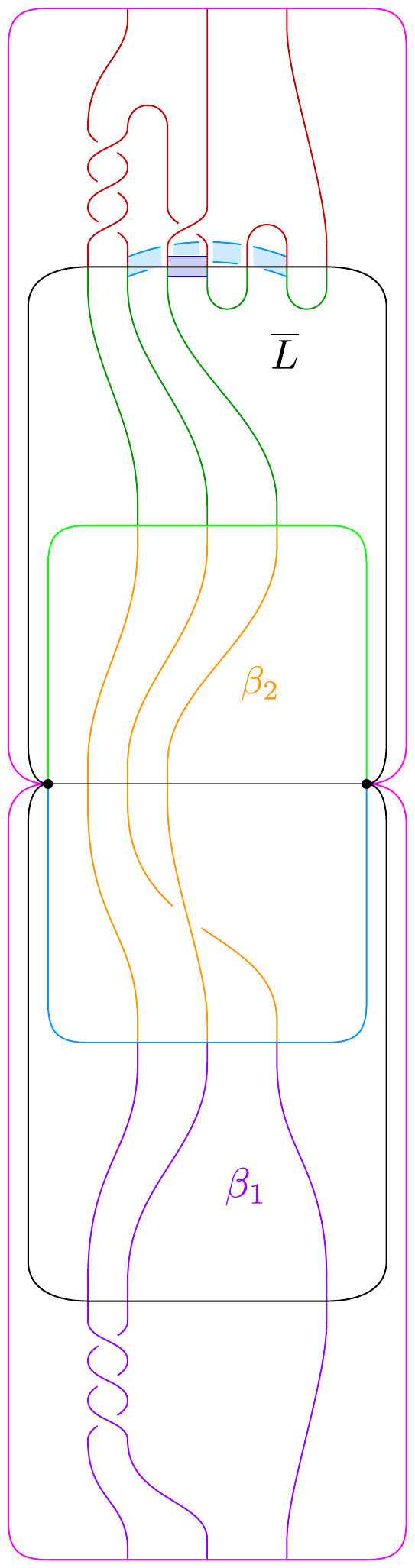}
\caption{}
\label{fig:square4}
\end{subfigure}%
}
&
\multicolumn{3}{c}{
\begin{subfigure}{.48\textwidth}
\setcounter{subfigure}{0}
\centering
\includegraphics[width=\linewidth]{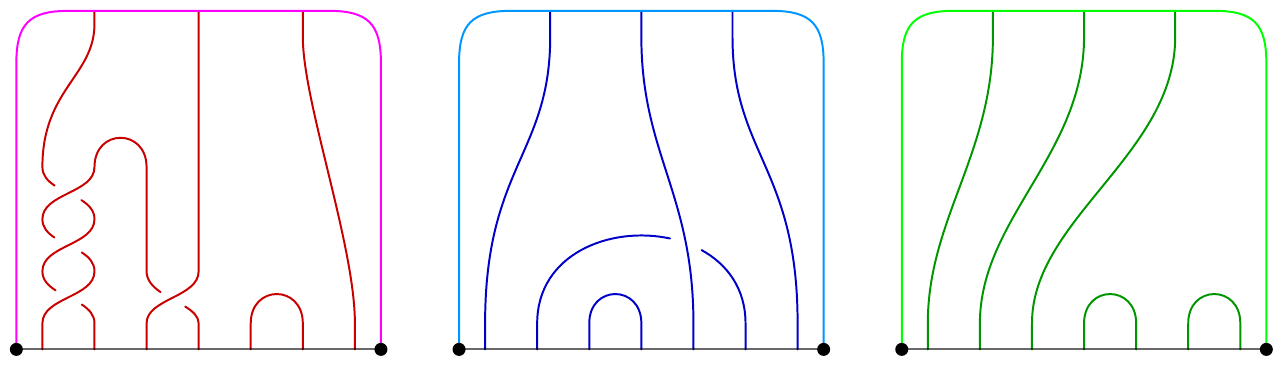}
\caption{}
\label{fig:square1}
\end{subfigure}%
}
\\[.75in]
& &
\multicolumn{3}{c}{
\begin{subfigure}{.48\textwidth}
\centering
\includegraphics[width=.5\linewidth]{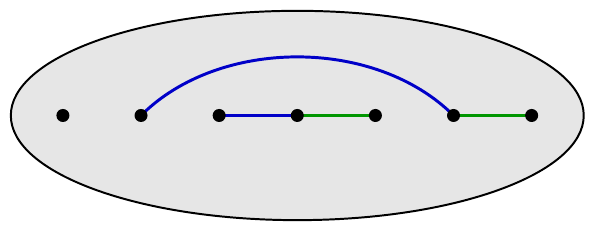}
\caption{}
\label{fig:square2}
\end{subfigure}%
}
\\[1.5in]
& &
\begin{subfigure}{.15\textwidth}
\setcounter{subfigure}{4}
\centering
\includegraphics[width=.9\linewidth]{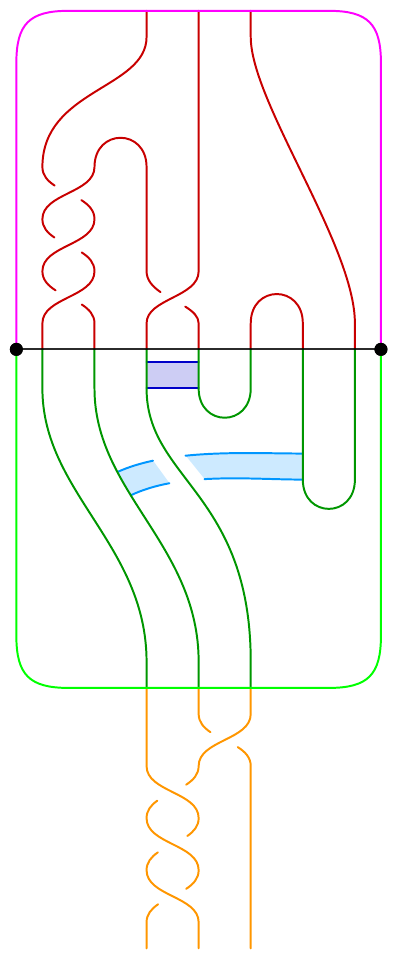}
\caption{}
\label{fig:square5}
\end{subfigure}%
&
\begin{subfigure}{.15\textwidth}
\centering
\includegraphics[width=.6\linewidth]{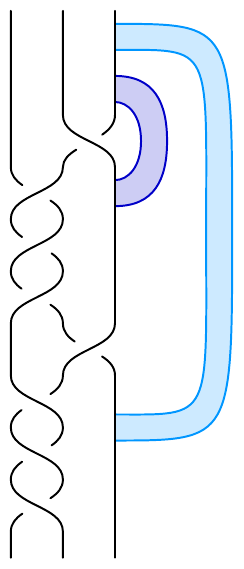}
\caption{}
\label{fig:square6}
\end{subfigure}%
&
\begin{subfigure}{.15\textwidth}
\centering
\includegraphics[width=.6\linewidth]{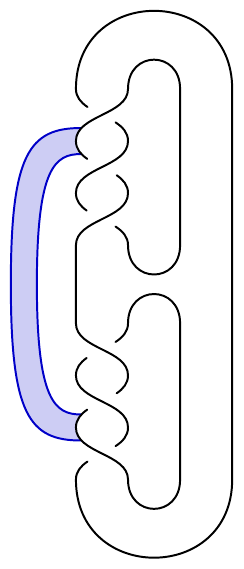}
\caption{}
\label{fig:square7}
\end{subfigure}%
\par\vspace{2mm}
\end{tabular}   
\caption{The process of converting the tri-plane diagram (A) into a bridge-braided band presentation (E) in order to identify the underlying surface, which in this case can be seen to be the standard ribbon disk for the square knot (G).}
\label{fig:square}
\end{figure}
	
	Having identified the boundary braid, we must identify a set of bands that will exhibit a bridge-braided band presentation corresponding to the original bridge trisection.  Following the proof of Proposition~\ref{prop:bridge_to_band}, these bands will come from a subset of the shadows $\Tt_2^*$.  To this end, shadows for the tangles $\Tt_2$ and $\Tt_3$ are shown in Figure~\ref{fig:square2}.  If there are closed components, one shadow of $\Tt_2^*$ is discarded from each such component.  In the present example, this step is not necessary; again, consider repeating this exercise with the tri-plane diagram from Figure~\ref{fig:f85}.  So, the set $\omega_*$ of the cores of the bands we are looking for, is precisely the blue shadows of Figure~\ref{fig:square2}.  In Figure~\ref{fig:square4} these shadows have been thickened vertically into bands that are framed by the bridge sphere $\Tt_1\cap\overline\Tt_3$.  In Figure~\ref{fig:square5}, this picture has been simplified, and the bands have been perturbed into $\overline\Tt_3$.  In Figure~\ref{fig:square6}, the bridge splitting structure has been forgotten, and the boundary braid is clearly visible.  At this point, we see that one band (light blue) is a helper band and can be discarded.  At last, Figure~\ref{fig:square7}, we recover an efficient band presentation for the surface originally described by the tri-plane diagram of Figure~\ref{fig:square1}.
\end{example}

\begin{example}
\label{ex:mono}
	\textbf{(2--stranded torus links)}  Figure~\ref{fig:mono1} shows a tri-plane diagram corresponding to a bridge trisection of the M\"obius band bounded in $S^3$ by the $(2,3)$--torus knot; see Figure~\ref{fig:mono4} for the band presentation.  However, this example could be generalized by replacing the four half-twists in the first diagram $\PP_1$ with $n$ half-twists for any $n\in\Z$, in which case the surface described would be the annulus (respectively, the M\"obius band) bounded by the $(2,n)$--torus link when $n$ is even (respectively, the $(2,n)$--torus knot when $n$ is odd).

\begin{figure}[h!]
\begin{subfigure}{.4\textwidth}
  \centering
  \includegraphics[width=.9\linewidth]{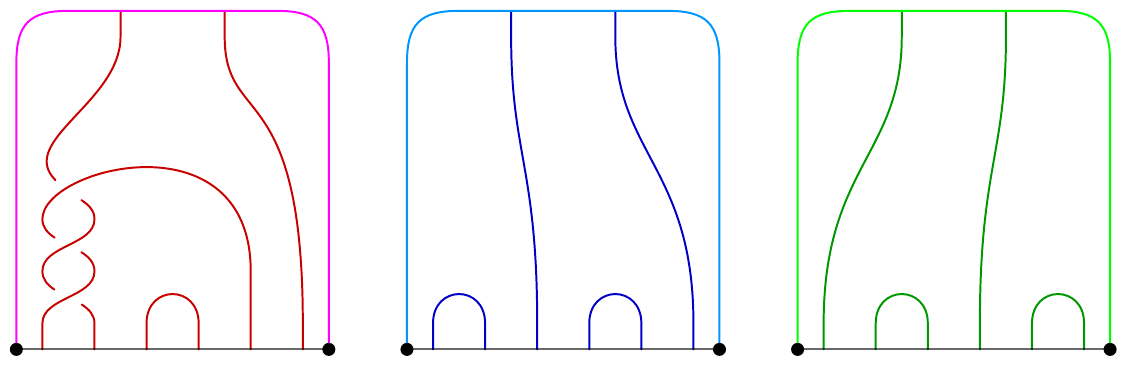}
  \caption{}
  \label{fig:mono1}
\end{subfigure}%
\begin{subfigure}{.2\textwidth}
  \centering
  \includegraphics[width=.8\linewidth]{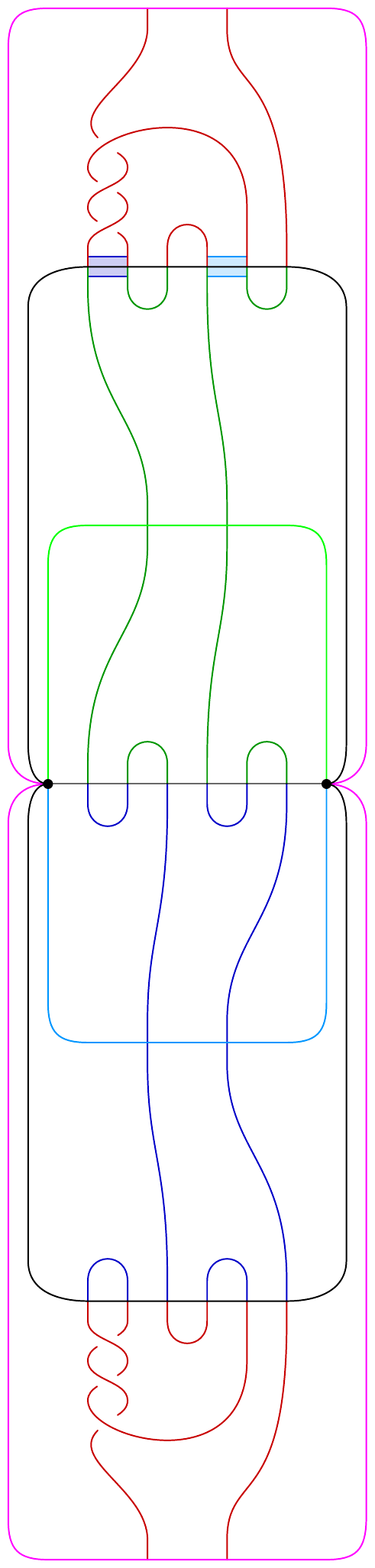}
  \caption{}
  \label{fig:mono2}
\end{subfigure}%
\begin{subfigure}{.2\textwidth}
  \centering
  \includegraphics[width=.8\linewidth]{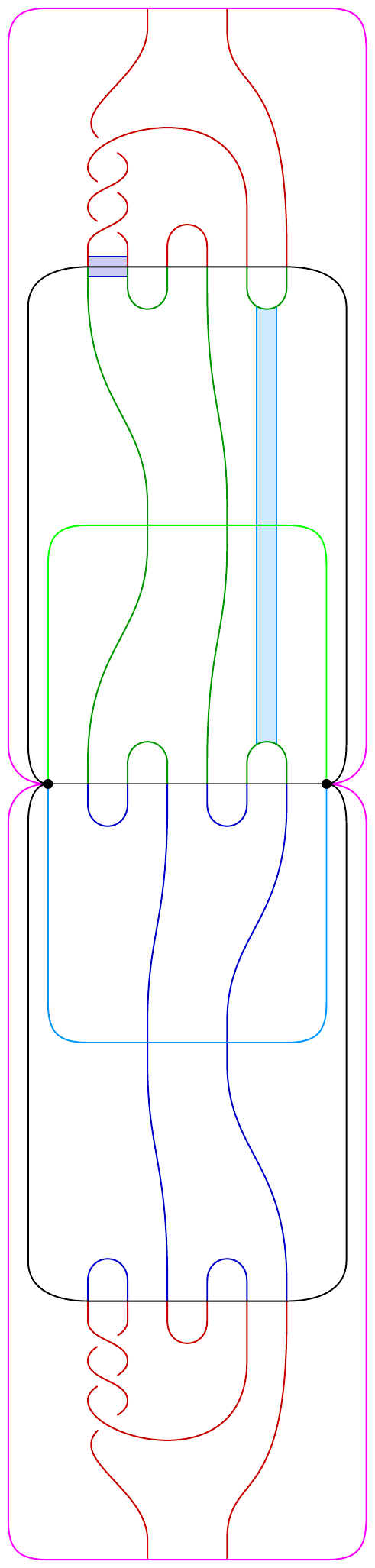}
  \caption{}
  \label{fig:mono3}
\end{subfigure}%
\begin{subfigure}{.2\textwidth}
  \centering
  \includegraphics[width=.8\linewidth]{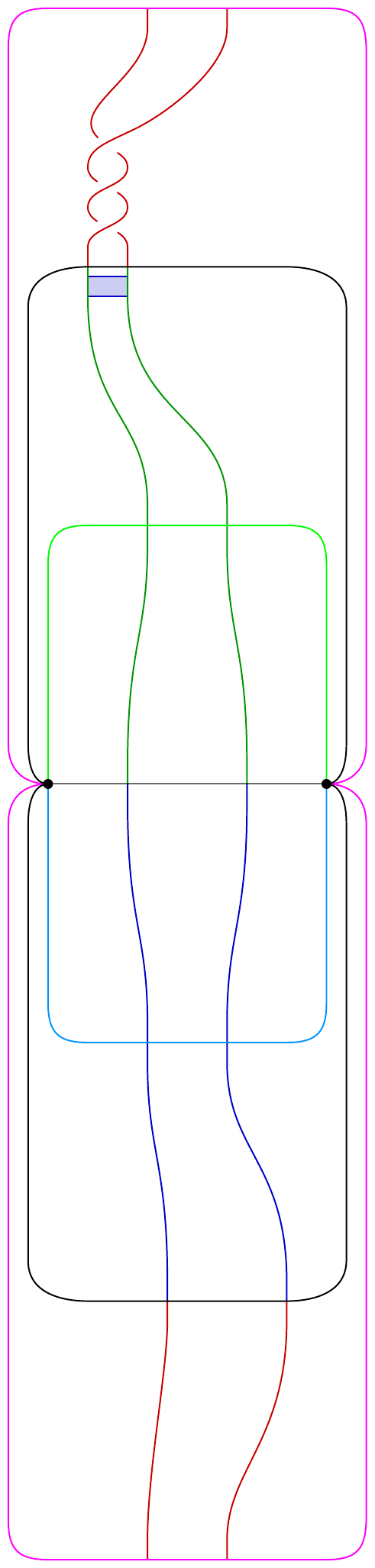}
  \caption{}
  \label{fig:mono4}
\end{subfigure}%
\caption{Recovering the boundary braid (D) from a tri-plane diagram (A), with bands tracked.  The surface described is the M\"obius band bounded by the right-handed trefoil in $S^3$.}
\label{fig:mono}
\end{figure}
	
	In any event, Figures~\ref{fig:mono2}--\ref{fig:mono4} give cross-sections of the bridge trisected surface with concentric shells of $B^4$, as described in Example~\ref{ex:monodromy} below.  In this example, we also track the information about bands encoded in the tri-plane diagram; cf. Figure~\ref{fig:square} and Example~\ref{ex:square}.  In slight contrast to the square knot examples, the shadows of $\Tt_2$ are quite simple, so the bands are easy to include.  In Figure~\ref{fig:mono3}, it becomes apparent that the right band (light blue) is a helper band and can be disregarded.
	
	A shadow diagrammatic analysis of this example is given in Example~\ref{ex:Mob_sh}.
\end{example}

\begin{theorem}
\label{thm:four-ball}
	Let $\T_0$ be the standard trisection of $B^4$, and let $\Ff\subset B^4$ be a neatly embedded surface with $\Ll = \partial \Ff$.  Fix an index $v$ braiding $\wh\beta$ of $\Ll$.  Suppose $\Ff$ has a handle decomposition with $c_1$ cups, $n$ bands, and $c_3$ caps.  Then, for some $b\in\N_0$, $\Ff$ can be isotoped to be in $(b,\bold c;v)$--bridge trisected position with respect to $\T_0$, such that $\partial\Ff = \wh\beta$, where $c_2=b-n$.
\end{theorem}

\begin{proof}
	By Proposition~\ref{prop:to_BBB_realizing}, $\Ff = \Ff_{(\wh\beta,U,\frak b)}$ for some bridge-braided band presentation $(\wh\beta,U,\frak b)$ of type $(b,\bold c;v)$.  By Proposition~\ref{prop:band_to_bridge}, $\Ff$ admits a bridge trisection of the same type.
\end{proof}

\subsection{Bridge-braided ribbon surfaces}
\label{subsec:Ribbon}
\ 

By construction, a $(b,\bold c;v)$--bridge-braided ribbon presentation $(\wh\beta,\frak b)$ will have $c_3=0$.  The next lemma shows that this fact can be used to systematically decrease the number $c_1$ of components of the unlink $U'$, at the expense of increasing the index $v$ of the braid $\wh\beta$.

\begin{lemma}\label{lem:decrease_c1}
	If $\Ff$ is the realizing surface for a $(b,(c_1,c_2,0);v)$--bridge-braided ribbon presentation $(\wh\beta,\frak b)$ with $c_1>0$, then $\Ff$ is the realizing surface for a $(b,(c_1-1,c_2,0);v+1)$--bridge-braided ribbon presentation $(\wh\beta^+,\frak b)$, where $\wh\beta^+$ is a Markov perturbation of $\wh\beta$.  The Markov perturbation can be assumed to be positive.
\end{lemma}

\begin{proof}
	Suppose that $(\wh\beta, \frak b)$ is a bridge-braided ribbon presentation with respect to the standard Heegaard double structure $(H_1,H_3,Y_3)$ on $S^3$, as in Definition~\ref{def:bridge-braided}.
	We orient $\wh\beta$ so that it winds counterclockwise about the braid axis $B=\partial \Sigma$.
	This induces an orientation on the arcs of $\overline L=\Tt_1\cup_\bold x\overline \Tt_3$, which induces an orientation on the bridge points $\bold x$:
	A bridge point $x\in\bold x$ is \emph{positive} if an oriented arc of $\overline L$ passes from $H_1$ to $\overline H_3$ through $x$.  Since $c_3=0$, every point of $\bold x$ can be oriented in this way.
	
	Recall from the proof of Proposition~\ref{prop:band_to_bridge} that we can perturb the bands of $\frak b$, which originally intersect $\Sigma$ in their core arcs, into the interior of $H_3$ so that they may be thought of as bands for the tangle $\Tt_3$.  Let $\Tt_2 = (\Tt_3)_\frak b$, and let $L' = \Tt_1\cup_\bold x\overline\Tt_2$.
	
	Utilizing the assumption that $c_1>0$, let $J$ be a flat component of $U'$.  Let $x$ be a positive point of $L\cap\Sigma$ so that $x\in J$. See Figure~\ref{fig:Markov_drag1}. Such a point exists, since $J$ contains a flat arc of $\Tt_1$, and the endpoints of this arc have differing signs. We perturb $\Sigma$ at $x$ to produce a new bridge splitting $\Tt_1'\cup_{\bold x'}\overline\Tt_3'$, which we consider as $\Tt_i' = \Tt_i\cup\tau_i$, where $\tau_i$ is the new flat strand near $x$.  If $\Delta_i$ was a bridge system for $\Tt_i$, then $\Delta_i' = \Delta_i\cup D_i$ is a bridge system for $\Tt_i'$, where $D_i$ is a bridge semi-disk for $\tau_i$. See Figure~\ref{fig:Markov_drag2}, and note that there may or may not be a band attached to $\Tt_3$ near $x$.

\begin{figure}[h!]
\begin{subfigure}{.33\textwidth}
  \centering
  \includegraphics[width=.9\linewidth]{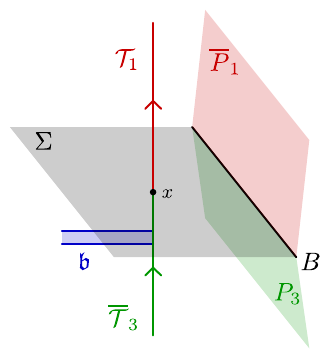}
  \caption{}
  \label{fig:Markov_drag1}
\end{subfigure}%
\begin{subfigure}{.33\textwidth}
  \centering
  \includegraphics[width=.9\linewidth]{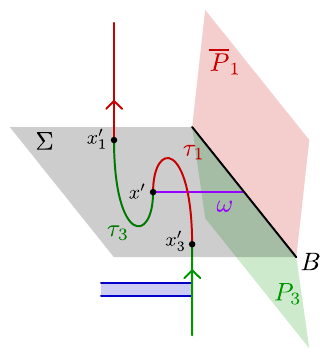}
  \caption{}
  \label{fig:Markov_drag2}
\end{subfigure}%
\begin{subfigure}{.33\textwidth}
  \centering
  \includegraphics[width=.9\linewidth]{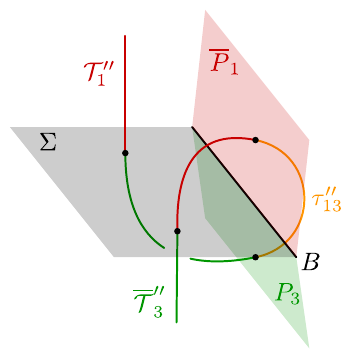}
  \caption{}
  \label{fig:Markov_drag3}
\end{subfigure}
\caption{Modifying a bridge trisection of a ribbon surface to remove a flat patch at the expense of Markov-stabilizing the boundary braid.}
\label{fig:Markov_drags}
\end{figure}

	Now, we have that $x' = \tau_1\cap\tau_3$ is negative.  Let $x'_i = \partial\tau_i\setminus x$ denote the positive points introduced by this perturbation.  Let $\lambda = \tau_1\cup_x\tau_3$.  Note that we can assume there is no band of $\frak b$ incident to either $\tau_i$.
	The bridge splitting $\Tt_1'\cup_\bold x\overline\Tt_3'$ is perturbed at $x'$.  We will swap this perturbation for a Markov perturbation by dragging the point $x'$ towards and through the boundary $B$ of $\Sigma$.  Let $\omega$ be an embedded arc in $\Sigma$ connecting $x'$ to $B$ such that $\Int(\omega)\cap \bold x = \emptyset$.  Since $\omega$ is dualized by each of the two small bridge semi-disks $D_i\subset \Delta_i'$, we can assume that $\Int(\omega)\cap\Delta_i' = \emptyset$.
	
	Change $(\wh\beta,\frak b)$ by an ambient isotopy that is supported in a tubular neighborhood of $\omega$ and that pushes $x'$ along $\omega$ towards and past $B$.  This is a finger move of $\lambda$ along $\omega$. (Note that the surface $\Ff$ is locally a product of $\lambda$ near $x'$.) Let $\lambda'$ denote the end result of this finger move; i.e., a portion of $\lambda$ has been pushed out of $H_1\cup_\Sigma\overline H_3$ into $Y_3$.  Let $\tau_i'' = \lambda'\cap H_i$.  Let $\tau_{13}'' = \lambda'\cap Y_3$.  Let $D_i''$ denote the bridge triangle resulting from applying the ambient isotopy to $D_i$.  We see immediately that $\tau_i''$ are vertical, and that $\Delta_i'' = (\Delta_i'\setminus D_i)\cup D_i''$ is a bridge system for $\Tt_i'' = (\Tt_i'\setminus\tau_i)\cup\tau_i''$.  It's also clear that $\tau_{13}''$ is a vertical strand in $Y_3$.  We make the following observations, with an eye towards Definition~\ref{def:bridge-braided}:
	\begin{enumerate}
		\item $\beta_3'' = \beta_3\cup\tau_{13}''$ is a $(v+1)$--braid.
		\item $L'' = \Tt_1''\cup_{\bold x''}\overline\Tt_3''$ is a perturbing of a $(v+1)$--braid.
		\item We still have $c_1=0$; the $\Tt_i''$ are $(b-v-1)$--perturbings of $(v+1)$--braids.
		\item The bands $\frak b$ can still be isotoped to intersect $\Sigma$ in their cores.
		\item The bride disks $\Delta_3''$ dualize the bands $\frak b$.
		\item $(\wh\beta)_\frak b$ has one fewer flat component.
	\end{enumerate}
	
	Thus, we have verified that conditions (1)-(6) of Definition~\ref{def:bridge-braided} are still satisfied, with the only relevant differences being that each tangle has an additional vertical strand and the flat component $J$ of $U'$ is now vertical. It follows that we have produced a bridge-banded ribbon presentation $(\wh\beta^+,\frak b)$ for $\Ff$, where $\wh\beta^+$ is a Markov perturbation of $\wh\beta$.
\end{proof}

\begin{remark}\label{rmk:why_no_U}
	The hypothesis that $c_3=0$ is the above lemma was necessary to ensure that the process described in the proof resulted in a Markov perturbation of the boundary. If $c_3>0$, then it is possible that each point $x\in\bold x\cap J$ lies on a (flat) component of $U$.  If the proof were carried out in this case, it would have the effect of changing the link type from $\Ll$ to the split union of $\Ll$ with an unknot on the boundary of $\Ff$.  This is reflective of the general fact that a non-ribbon $\Ff$ with boundary $\Ll$ can be thought of as a ribbon surface for the split union of $\Ll$ with an unlink.
\end{remark}

Recall that $\bold c$ is an ordered partition of type $(c,3)$ for some $c\in\N_0$; in particular, $c=c_1+c_2+c_3$.

\begin{lemma}\label{lem:decrease_all_c}
	If $\Ff$ is the realizing surface for a $(b,\bold c;v)$--bridge-braided band presentation $(\wh\beta,U,\frak b)$ with $c_i=0$ for some $i$, then $\Ff$ is the realizing surface for a $(b,0;w+c)$--bridge-braided ribbon presentation $(\wh\beta^{++},\frak b'')$, where $\wh\beta^{++}$ is a Markov perturbation of $\wh\beta$. 
\end{lemma}

\begin{proof}
	Suppose $\Ff$ is the realizing surface for a $(b,\bold c;v)$--bridge-braided band presentation $(\wh\beta,U,\frak b)$ with $c_i=0$ for some $i$.  By Proposition~\ref{prop:band_to_bridge}, $\Ff$ admits a $(b,(c_1,c_2,c_3);v)$--bridge trisection filling $\wh\beta$.  By re-labeling the pieces, we can assume that $c_3=0$.  By Proposition~\ref{prop:bridge_to_band}, this gives us a $(b,(c_1,c_2,0);v)$--bridge-braided ribbon presentation $(\wh\beta,\frak b')$.  Note that while the braid type $\wh\beta$ hasn't changed, the bands $\frak b$ may have, and the intersection of $\wh\beta$ with the pieces of the standard Heegaard-double decomposition may have, as well.  Nonetheless, we can apply Lemma~\ref{lem:decrease_c1} iteratively to decrease $c_1$ to zero, at the cost of Markov-perturbing $\wh\beta$ into a $(w+c_1)$--braid $\wh\beta^+$.
	
	Passing back to a $(b,(0,c_2,0);w+c_1)$--bridge trisection filling $\wh\beta^+$ via Proposition~\ref{prop:band_to_bridge}, re-labelling, and applying Proposition~\ref{prop:bridge_to_band}, we extract a $(b,(c_2,0,0);w+c_1)$--bridge-braided ribbon presentation $(\wh\beta^+,\frak b'')$.  Again, the bands and the precise bridge splitting may have changed.  However, a second application of Lemma~\ref{lem:decrease_c1} allows us to decrease $c_2$ to zero, at the cost of Markov perturbing $\wh\beta^+$ into a $(w+c_1+c_2)$--braid $\wh\beta^{++}$.  Note that we have Markov perturbed a total of $c = c_1+c_2$ times.
\end{proof}

\begin{theorem}
\label{thm:ribbon}
	Let $\T_0$ be the standard trisection of $B^4$, and let $\Ff\subset B^4$ be a neatly embedded surface with $\Ll = \partial \Ff$.  Let $\wh\beta$ be an index $v$ braiding $\Ll$.
	Then, the following are equivalent.
	\begin{enumerate}
		\item $\Ff$ is ribbon.
		\item $\Ff$ admits a $(b,\bold c;v)$--bridge trisection filling $\wh\beta$ with $c_i=0$ for some $i$.
		\item $\Ff$ admits a $(b,0;v+c)$--bridge trisection filling a Markov perturbation $\wh\beta^+$ of $\wh\beta$.
	\end{enumerate}
\end{theorem}

\begin{proof}
	Assume (1).  Since $\Ff$ is ribbon, it admits a $(b,(c_1,c_2,0);v)$--bridge-braided ribbon presentation $(\wh\beta,\frak b)$, by Proposition~\ref{prop:to_BBB_realizing}.  By Proposition~\ref{prop:band_to_bridge}, this can be turned into a $(b,(c_1,c_2,0);v)$--bridge trisection filling $\wh\beta$, which implies (2).
	
	Assume (2).  The bridge trisection filling $\wh\beta$ with $c_i=0$ for some $i$ gives a bridge-braided ribbon presentation $(\wh\beta,\frak b')$ with $c_i=0$ for the same $i$.  By Lemma~\ref{lem:decrease_all_c}, there is a $(b,0;w+c)$--bridge-braided ribbon presentation $(\wh\beta^{++},\frak b'')$ for $\Ff$, where $\wh\beta^{++}$ is a Markov perturbation of $\wh\beta$.  By Proposition~\ref{prop:band_to_bridge}, this gives a $(b,0;w+c)$--bridge trisection of $\Ff$ filling $\wh\beta^{++}$.  This implies (3), where $\wh\beta^{++}$ is denoted by $\wh\beta^+$ for simplicity.
	
	Assume (3).  The $(b,0;w')$--bridge trisection filling $\wh\beta^+$ gives rise to a bridge-braided ribbon presentation $(\wh\beta^+,\frak b'')$ of the same type, by Proposition~\ref{prop:bridge_to_band}, such that $\Ff = \Ff_{(\wh\beta^+,\frak b'')}$.  However, a band presentation of a surface is precisely a handle-decomposition of the surface with respect to the standard Morse function on $B^4$.  It follows that $\Ff$ can be built without caps; hence, $\Ff$ is ribbon, and (1) is implied.
	
	Note for completeness that (2) can be seen to imply (1) by the argument immediately above, and that (3) implies (2) trivially.
	
\end{proof}

\section{Tri-plane diagrams}
\label{sec:tri-plane}

A significant feature of the theory of trisections (broadly construed) is that it gives rise to new diagrammatic representations for four-dimensional objects (manifolds and knotted surfaces therein).  In this section, we describe the diagrammatic theory for bridge trisections of surfaces in the four-ball.  Recall the notational set-up of Subsection~\ref{subsec:Special}.

Let $(H,\Tt)$ be a tangle with $H\cong B^3$.  Let $E\subset H$ be a neatly embedded disk with $\partial \Tt\subset \partial E$.  By choosing a generic projection of $H$ onto $E$, we can represent $(H,\Tt)$ by a \emph{tangle diagram}.
In the case that $H\cong B^3$, the lensed cobordism structure on $(H,\Tt)$ discussed in Subsection~\ref{subsec:Compression} can be thought of as inducing the hemispherical decomposition of $\partial H\cong S^2$.  So, we refer to $\partial_+H$ and $\partial_-H$ as the \emph{southern} and \emph{northern} boundaries.  This induces a decomposition of $\partial E$ into a \emph{northern arc} and a \emph{southern arc}.
See Figure~\ref{fig:moves} for examples of $(1,2)$--tangle diagrams.

\begin{definition}
	A \emph{$(b,\bold c;v)$--tri-plane diagram} is a triple $\PP = (\PP_1,\PP_2,\PP_3)$ such that $\PP_i$ is a $(b,v)$--tangle diagram and the union $\PP_i\cup\overline{\PP_i}$ is a tangle diagram for a split union of a $v$--braid with a $c_i$--component unlink.  (Note that $\overline{\PP_i}$ is the diagram $\PP_i$ with crossing information reversed.) The southern arcs (and the $2b+v$ points $\bold x$ that they contain) are assumed to be identified.  We denote the $v$ points contained in the northern arc of $\PP_i$ by $\bold y_i$; the three northern arcs are not identified.
\end{definition}

A tri-plane diagram describes a bridge trisected surface in the following way.  Let $(H_i,\Tt_i)$ be tangles corresponding to the tangle diagrams $\PP_i$.  Then the triple of tangle diagrams can be thought of as describing the union
$$(H_1,\Tt_1)\cup(H_2,\Tt_2)\cup(H_3,\Tt_3)$$
of these tangles, where $(H_i,\Tt_i)\cap\overline{(H_{i+1},\Tt_i)} = (\Sigma,\bold x)$.
This explains the identification of the souther portions of the tangle diagrams in the definition.
Now, by definition, each union $(H_i,\Tt_i)\cup\overline{(H_{i+1},\Tt_{i+1})}$ is a the split union of a braid with an unlink of $c_i$ components inside a 3--ball.  By Lemma~\ref{lem:LP}, there is a unique way to glom on to this 3--ball a $(c_i,v)$--disk-tangle  $(Z_i,\Dd_i)$, where $Z_i\cong B^4$.  Therefore, the union
$$(Z_1,\Dd_1)\cup(Z_2,\Dd_2)\cup(Z_3,\Dd_3)$$
is a bridge trisected surface in $B^4$.  In the next section, we will prove the existence of bridge trisections for surfaces in $B^4$.  Since bridge trisections are determined by their spine (Corollary~\ref{coro:spine}, this gives the following.

\begin{corollary}
\label{coro:tri-plane}
	Every neatly embedded surface in $B^4$ can be described by a tri-plane diagram.
\end{corollary}

\begin{proof}
	By Theorem~\ref{thm:four-ball}, every such surface in $B^4$ can be put in bridge position with respect to the genus zero trisection $\T_0$.  The corresponding bridge trisection is determined by its spine
	$$(H_1,\Tt_1)\cup(H_2,\Tt_2)\cup(H_3,\Tt_3).$$
	This spine can be represented by a tri-plane diagram by choosing a triple of disks $E_i\subset H_i$ whose boundaries agree and choosing generic projections $H_i\twoheadrightarrow E_i$ that induce tangle diagrams for the $\Tt_i$.
\end{proof}

The union $E_i\cup E_2\cup E_3$ of disks that appeared in the proof above is called a \emph{tri-plane} for the bridge trisection. We consider bridge trisections up to ambient isotopy, and an ambient isotopy of a bridge trisection can change the induced tri-plane diagram.  These changes can manifest in following three ways, which we collectively call \emph{tri-plane moves}.  See Figure~\ref{fig:moves} for an illustration of each move.

\begin{figure}[h!]
\begin{subfigure}{\textwidth}
  \centering
  \includegraphics[width=.5\linewidth]{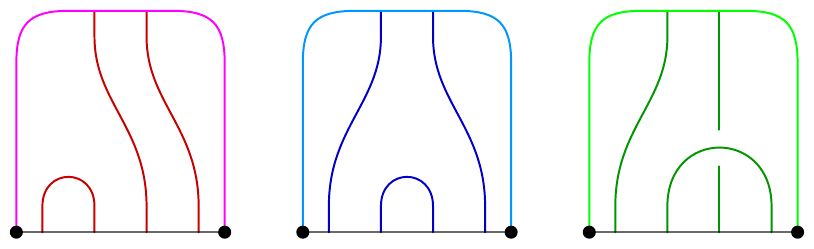}
  \caption{}
  \label{fig:b=11}
\end{subfigure}
\par\vspace{5mm}
\begin{subfigure}{\textwidth}
  \centering
  \includegraphics[width=.5\linewidth]{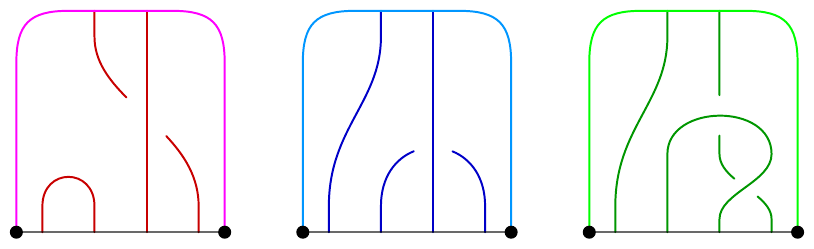}
  \caption{}
  \label{fig:b=13}
\end{subfigure}
\par\vspace{5mm}
\begin{subfigure}{\textwidth}
  \centering
  \includegraphics[width=.5\linewidth]{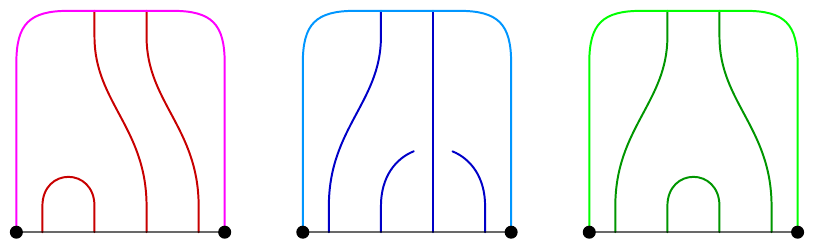}
  \caption{}
  \label{fig:b=12}
\end{subfigure}
\caption{(A) A tri-plane diagram. (B) The result of applying to (A) a bridge sphere braid transposition at the third and fourth bridge points.  (C) The result of applying to (B) a page braid transposition in the fist tangle and a Reidemeister move in the third tangle.}
\label{fig:moves}
\end{figure}

An \emph{interior Reidemeister move} on $\PP$ is a Reidemeister move that is applied to the interior of one of the tangle diagrams $\PP_i$.  Interior Reidemeister moves correspond to ambient isotopies of the surface that are supported away from $\partial B^4$ and away from the core surface $\Sigma$. They also reflect the inherent indeterminacy of choosing a tangle diagram to represent a given tangle.

A \emph{core (braid) transposition} is performed as follows:
Pick a pair of adjacent bridge points $x,x'\in\bold x$, recalling that $x$ and $x'$ are (identified) points in the southern arc of each of the three tangles diagram.
Apply a braid transposition to all three tangle diagrams that exchanges $x$ and $x'$.
This introduces a crossing in each tangle diagram; the introduced crossing should have the same sign in each diagram.
Bridge sphere braiding corresponds to ambient isotopies of the surface that are supported in a neighborhood of the core surface $\Sigma$.
Note that this gives an action of the braid group $\Mm(D^2,\bold x)$ on the set of tri-plane diagrams.

A \emph{page (braid) transposition} is performed as follows
Pick a pair of adjacent points $y,y'\in\bold y_i$ in the northern arc of one of the tangle diagrams.
Apply a braid transposition to this tangle diagram that exchanges $y$ and $Y'$.
In contrast to a core transposition, the braid transposition is only applied simultaneously to one diagram.
Page transpositions correspond to ambient isotopies of the surface that are supported near $\partial B^4$.

Interior Reidemeister moves and core transpositions featured in the theory of bridge trisections of closed surfaces in the four-sphere described in~\cite{MeiZup_17_Bridge-trisections-of-knotted}.  See, in particular, Lemma~7.4 for more details.

\begin{proposition}
\label{prop:tri-plane_moves}
	Suppose $\PP$ and $\PP'$ are tri-plane diagrams corresponding to isotopic bridge trisections.  Then $\PP$ and $\PP'$ can be related by a finite sequence tri-plane moves.
\end{proposition}

\begin{proof}[Proof of Proposition~\ref{prop:tri-plane_moves}]
	As in the proof of Lemma~7.4 of~\cite{MeiZup_17_Bridge-trisections-of-knotted}, it suffices to assume that the two tri-plane diagrams $\PP$ and $\PP'$ are induced by different choices of tri-planes $E_1\cup E_2\cup E_3$ and $E_1'\cup E_2'\cup E_3'$ for a fixed bridge trisection.  Further more, we can switch perspective and assume that the tri-planes agree, but the seams of the tangles do not.  Thus, assume we have a fixed $\Ee =E_1\cup E_2\cup E_3$ within $\Hh = H_1\cup H_2\cup H_3$ and that we have two sets of seams $\Tt=\Tt_1\cup\Tt_2\cup\Tt_3$ and $\Tt'=\Tt_1'\cup\Tt_2'\cup\Tt_3'$ determining a pair of isotopic spines in $B^4$. 
	
	Note that the southern endpoints of the $\Tt_i$ and the $\Tt_i'$ are both contained in the southern arc $\partial E_i\cap\Int(B^4)$, while all the northern endpoints are contained in the northern arc $\partial E_i\cap \partial B^4$.  Without loss of generality, we assume the northern (resp., southern) endpoints of $\Tt_i$ agree with the northern (resp., southern) endpoints of $\Tt_i'$ for each $i$.
	
	As in the proof of Lemma~7.4 of~\cite{MeiZup_17_Bridge-trisections-of-knotted}, if $f_t$ is an ambient isotopy of $\Hh$ such that $f_0$ is the identity and $f_1(\Tt) =\Tt'$, then $f_t$ induces a loop in the configuration space of the bridge points $\bold x=\Tt\cap\Sigma$.  In this setting, $f_t$ also induces, for each $i\in\Z_3$, a loop in the configuration space of the points $\bold y\Tt_i\cap \partial_- H_i$ in the disk $\partial_- H_i$.
	
	We write $f_t$ as $f_t^\Sigma\cup f_t^1\cup f_t^2\cup f_t^3\cup f_t'$, where $f_t^\Sigma$ agrees with $f_t$ in a small neighborhood of $\Sigma$ and is the identity outside of a slightly larger neighborhood of $\Sigma$; $f_t^i$ agrees with $f_t$ in a small neighborhood of $\partial_- H_i$ and is the identity outside a slightly larger neighborhood of $\partial_- H_i$; and $f_t'$ is supported away from the small neighborhoods of $\Sigma$ and $\partial_- H_i$.
	Since these can be isolated to a single region near $\partial_- H_i$ for some $i$, they are independent of each other.

	Since $f_t^\Sigma$ corresponds to a braiding of the bridge points $\bold x$ there are tri-plane diagrams $\PP$ and $\PP^\Sigma$ corresponding to $\Tt$ and $\Tt^\Sigma = f_1^\Sigma(\Tt)$ that are related by a sequence of core transpositions.  Continuing, there is a tri-plane diagram $\PP''$ corresponding to $\Tt''=(f_1^1\cup f_1^2\cup f_1^3)(\Tt^\Sigma)$ that is related to $\PP^\Sigma$ by a sequence of interior Reidemeister moves.  Finally, the tri-plane diagram $\PP'$ corresponds to $f_1'(\Tt'')$ and is related to $\PP''$ by a sequence of page transpositions.  In total, $\PP$ and $\PP'$ are related by a sequence of tri-plane moves, as desired.
\end{proof}

\subsection{Recovering the boundary braid from a tri-plane diagram}
\label{subsec:tri-plane_braid}
\ 

We now describe how to recover the boundary braid $(S^3,\Ll) = \partial(B^4,\Ff)$ from the data of a tri-plane diagram for $(B^4,\Ff)$.  This process is illustrated in the example of the Seifert surface for the figure-8 knot in Figure~\ref{fig:monodromy}; cf. Figure~\ref{fig:f8} for more details regarding this example.  See also Figure~\ref{fig:square} for another example.

Let $\PP = (\PP_1,\PP_2,\PP_3)$ be a tri-plane diagram for a surface $(B^4,\Ff)$.  Let $\Ee = (E_1,E_2,E_3)$ denote the underlying tri-plane.  Let $\partial_-E_i$ and $\partial_+E_i$ denote the northern and southern boundary arcs of these disks, respectively, and let $S^0_i = \partial_-E_i\cap\partial_+E_i$ their 0--sphere intersections.
Recall that, diagrammatically, the arcs $\partial_+E_i$ correspond to the core surface $\Sigma$ of the trisection, which is a disk, and the 0--spheres $S^0_i$ correspond to the unknot $B = \partial\Sigma$, which we think of as the binding of an open-book decomposition of $S^3$ with three disk pages given by the $P_i$.
Recall that $\Sigma$ is isotopic rel-$\partial$ to each of the $\PP_i$ via the arms $H_i$.

With this in mind, consider the planar link diagram $\circ\widehat\PP$ obtained as follows.  First, form the cyclic union
$$\PP_1\cup\overline\PP_3\cup\PP_3\cup\overline\PP_2\cup\PP_2\cup\overline\PP_1\cup\PP_1,$$
where $\PP_{i+1}$ and $\overline\PP_i$ are identified along their southern boundaries, $\overline\PP_i$ and $\PP_i$ are identified along their northern boundaries, and the two copies of $\PP_1$ are identified point-wise.  Note that the cyclic ordering here is the opposite of what one might expect.
This important subtlety is explained in the proof of Proposition~\ref{prop:tri-plane_braid} below.
The corresponding union of the disks of the tri-plane
$$E_1\cup \overline E_3\cup E_3\cup \overline E_2\cup E_2\cup \overline E_1\cup E_1$$
is topologically a two-sphere $S^2$.  In particular, the 0--spheres $S^0_i$ have all been identified with a single 0--sphere $S^0$, which we think of as poles of the two-sphere.
We represent this two-sphere in the plane by cutting open along $\partial_-E_1$ and embedding the resulting bigon so that the $E_i$ and $\overline E_i$ lie in the $yz$--plane with $E_3\cap\overline E_2$ on the $y$--axis.  See Figure~\ref{fig:monodromy2}
In this way, the diagram $\circ\PP$ encodes a link in a three-sphere. The unknotted binding $B$ in $S^3$ can be thought of as the unit circle in the $xy$--plane.  (The positive $x$--axis points out of the page.) Each longitudinal arc on $S^2$, including the northern and southern arcs of each $E_i$, corresponds to a distinct page, given six in all.  However, the ambient three-sphere in which this link lives is not $S^3=\partial B^4$, as the proof of Proposition~\ref{prop:tri-plane_braid} will make clear.

Note that the diagram $\circ\widehat\PP$ will have only two types of connected components: (1) components that meet each disk $E_i$ and are homotopically essential in $S^2\setminus\nu(S^0)$, and components that are null-homotopic and are contained in some pair $E_{i+1}\cup\overline E_i$.  Components of type (1) will correspond to the boundary link $(S^3,\Ll)$, while components of the second kind will correspond to split unknots.
The components of type (1) are not braided in the sense of being everywhere transverse to the longitudinal arcs of $S^2$ but, as we shall justify below, they become braided after a sequence of Reidemeister moves and isotopies that are supported away from $S^0$. 
Define $\circ\PP$ to be the result of discarding all components of type (2) from $\circ\widehat\PP$, then straightening out the arcs of type (1) until they give a braid diagram in the sense that they are everywhere transverse to the longitudinal arcs connecting the poles $S^0\subset S^2$.

\begin{proposition}
\label{prop:tri-plane_braid}
	Suppose $\PP = (\PP_1,\PP_2,\PP_3)$ is a tri-plane diagram for $(B^4,\Ff)$.  Then the diagram $\circ\PP$ is a braid diagram for the boundary link $(S^3,\Ll) = \partial(B^4,\Ff)$.
\end{proposition}

\begin{proof}
	Consider the spine $H_1\cup H_2\cup H_3$ of the genus zero trisection $\T_0$ of $B^4$.  Let $N$ be a small lensed neighborhood of this spine inside $B^4$.  Here, the qualifier `lensed' has the effect that $N\cap\partial B^4$ is unchanged:
	$$N\cap\partial X = P_1\sqcup P_2\sqcup P_3.$$
	We can decompose $\partial N$ into six pieces:
	$$\partial N = H_1^+\cup H_3^-\cup H_3^+\cup H_2^-\cup H_2^+\cup H_1^-,$$
	where the pieces intersect cyclically in the following manner:  the $H_{i+1}^+\cap H_i^- = \Sigma_{i}^-$ are the three obvious push-offs of $\Sigma$ into $\partial N$, and $H_i^-\cap H_i^+ = P_i$.
	Because $B = \partial \Sigma =\partial \Sigma_i = \partial P_i$, it follows that $\partial N$ is a closed 3--manifold.  In fact, there is an obvious `radial' diffeomorphism $N\to B^4$ that pushes $H_{i+1}^+\cup H_i^-$ onto $Y_i$ into an \emph{orientation preserving} way.
	To unpack this last statement, recall that $Z_i$ induces an orientation on it's boundary such that
	$$\partial Z_i = H_i\cup_\Sigma \overline H_{i+1}\cup Y_i.$$
	In $\partial N$, we have corresponding pieces $H_{i+1}^+\cup_{\Sigma^-_i}H_i^+$, but the correspondences
	$$H_i\leftrightarrow H_i^-, \hspace{.5cm} \overline H_{i+1}\leftrightarrow H_{i+1}^+, \hspace{.5cm} \text{and} \hspace{.5cm} \Sigma\leftrightarrow\Sigma_i^-$$
	all reverse orientation.
	This is because the outward normal to $N$ points into $Z_i$.  Figure~\ref{fig:Morse1} provides a potentially helpful schematic.
	
	Bringing the surface $\Ff$ into the picture, we have the identification
	$$\partial N\cap\Ff = (H_1,\Tt_1)\cup\overline{(H_2,\Tt_2)}\cup(H_2,\Tt_2)\cup\overline{(H_3,\Tt_3)}\cup(H_3,\Tt_3)\cup\overline{(H_1,\Tt_1)}.$$
	If $\Ee = E_1\cup E_2\cup E_3$ was our original tri-plane, then there are obvious disks $E_i^\pm\subset H_i^\pm$ onto which $\partial N\cap\Ff$ can be projected.
	As discussed in the text preceding this proposition, the union of the $E_i^\pm$ is a two-sphere, which can be identified with the plane, as discussed.
	Adopting this identification, we find that the induced diagram $\circ\PP$ is a planar diagram for $\partial N\cap \Ff$.

	Recall that, by definition, $\PP_{i+1}\cup\overline\PP_i$ is a diagram for (the mirror of) a split union of a braid with an unlink.
	Thus, the total union $\circ\PP$ is (currently) a diagram for a closed braid split union three unlinks.
	Note that although the diagram describes a closed (geometric) braid, the diagram may not be braided. See Figure~\ref{fig:monodromy2}.
	
	It remains to observe how this diagram changes as the neighborhood $N$ is enlarged until it fills up all of $B^4$ and $\partial N$ coincides with $S^3 = \partial B^4$.
	Two things happen in the course of this.
	First, the unlinks will shrink to points and disappear as the neighborhood $N$ is enlarged to encompass the flat patches of the trivial disk-tangles that cap them off.
	Second, the portions of the diagram corresponding to the closed braid will `straighten out', meaning they will deperturb until the diagram is an honest braid diagram.
	Finally, the neighborhood $N$ will coincide with all of $B^4$, the union of the $E_i^\pm$ will live in $S^3$, and the diagram $\circ\PP$ will correspond to a braid diagram for $\Ll = \partial \Ff$, as desired.
\end{proof}

\begin{figure}[h!]
\begin{subfigure}{\textwidth}
  \centering
  \includegraphics[width=.8\linewidth]{f85}
  \caption{}
  \label{fig:monodromy1}
\end{subfigure}
\par\vspace{5mm}
\begin{subfigure}{.25\textwidth}
  \centering
  \includegraphics[width=.8\linewidth]{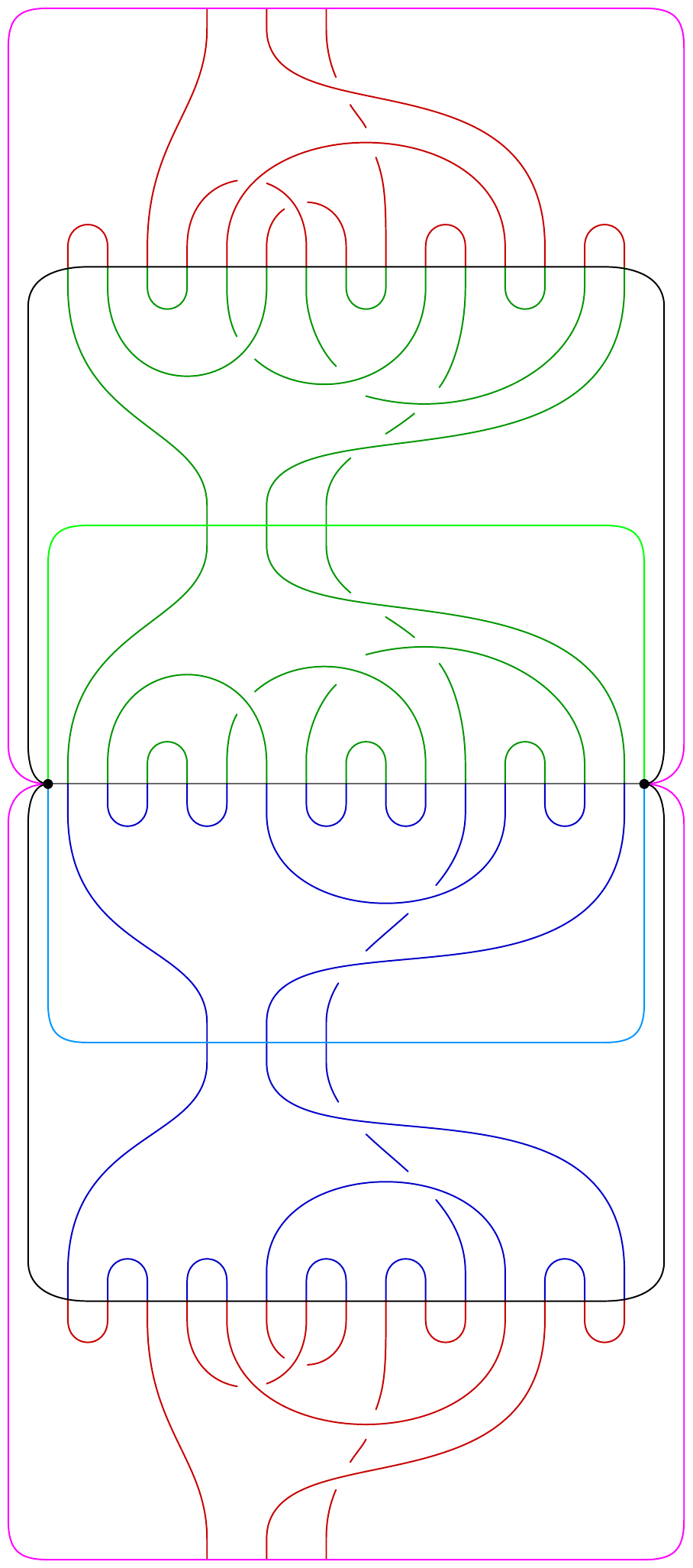}
  \caption{}
  \label{fig:monodromy2}
\end{subfigure}%
\begin{subfigure}{.25\textwidth}
  \centering
  \includegraphics[width=.8\linewidth]{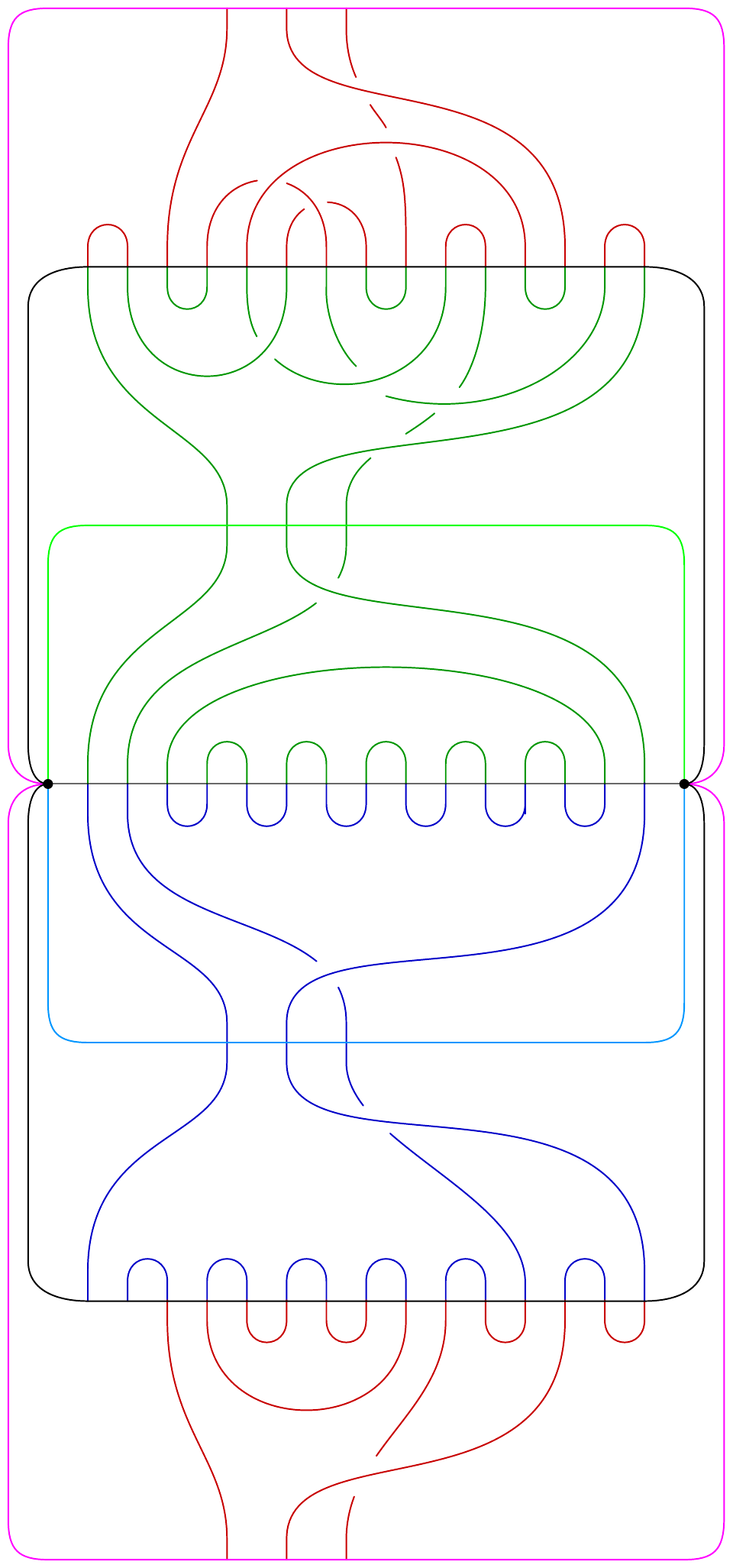}
  \caption{}
  \label{fig:monodromy3}
\end{subfigure}%
\begin{subfigure}{.25\textwidth}
  \centering
  \includegraphics[width=.8\linewidth]{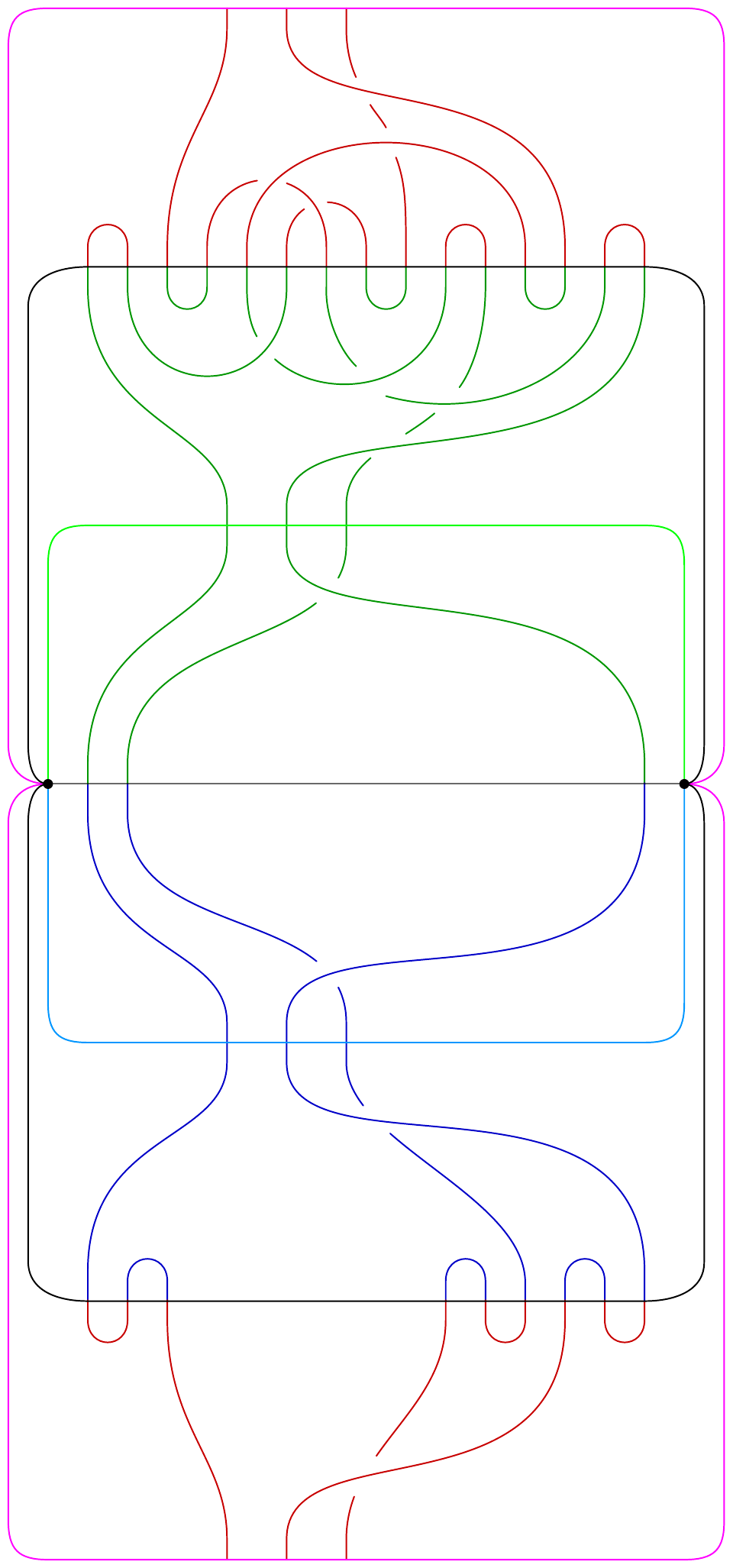}
  \caption{}
  \label{fig:monodromy4}
\end{subfigure}%
\begin{subfigure}{.25\textwidth}
  \centering
  \includegraphics[width=.8\linewidth]{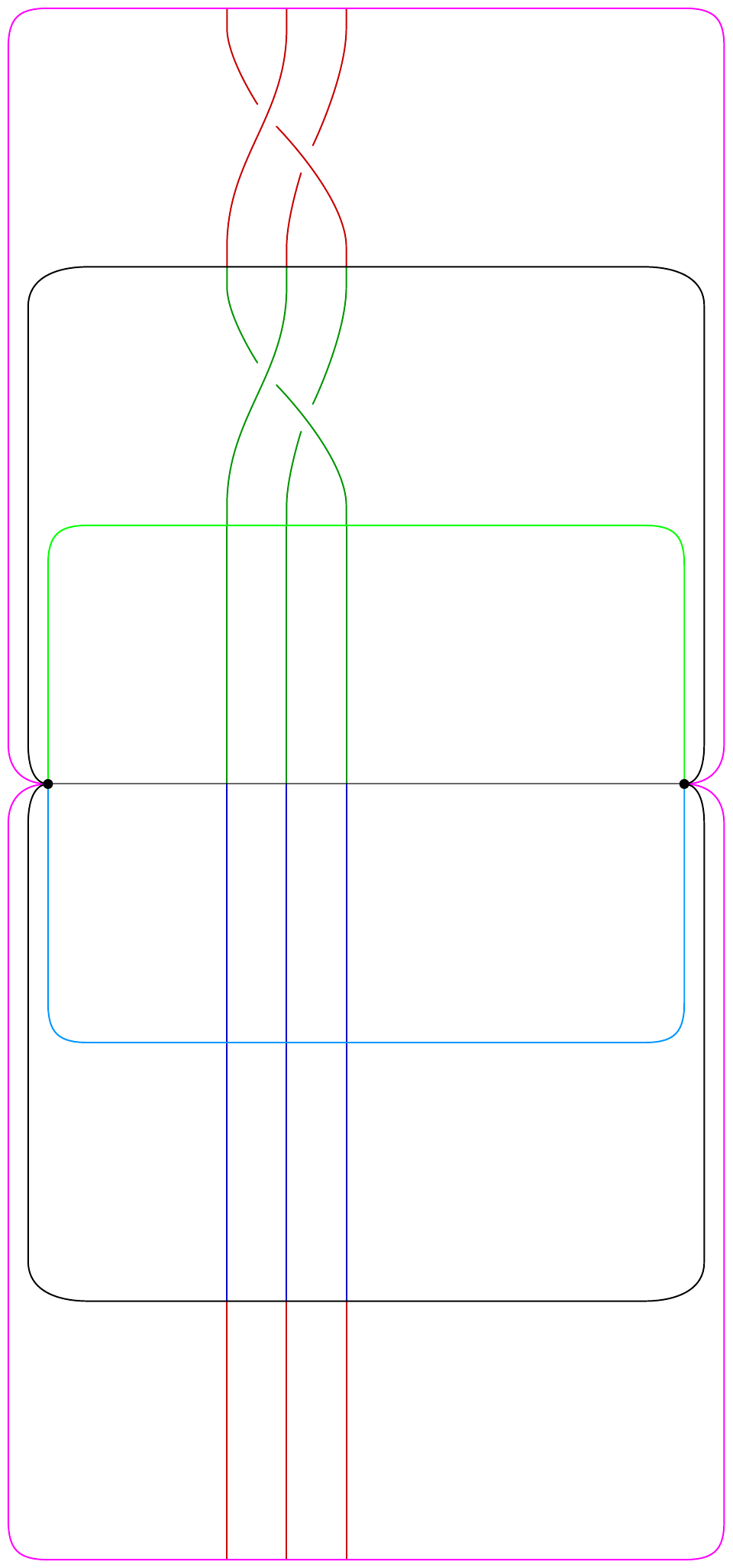}
  \caption{}
  \label{fig:monodromy5}
\end{subfigure}
\caption{Recovering the boundary braid (E) from a tri-plane diagram (A).  Compare with Figure~\ref{fig:f8}.}
\label{fig:monodromy}
\end{figure}

\begin{example}
\label{ex:monodromy}
	Figure~\ref{fig:monodromy2} shows the diagram $\circ\widehat\PP$ corresponding to the tri-plane diagram in Figure~\ref{fig:monodromy1}. (This tri-plane diagram corresponds to the Seifert surface for the figure-8 knot; see Figure~\ref{fig:f8} for more details.) The two black dots represent the braid axis, and each arc connecting the these dots corresponds to a disk page of the braid axis.
	
	As described in the proof of Proposition~\ref{prop:tri-plane_braid}, the sequence Figures~\ref{fig:monodromy2}--\ref{fig:monodromy5} can be thought of as describing the cross-section of the bridge trisected surface with concentric shells in $B^4$, starting with the boundary of a regular neighborhood of the spine of the trisection of $B^4$ and terminating in the boundary of $B^4$.  Moving from Figure~\ref{fig:monodromy2} to Figure~\ref{fig:monodromy3}, the cross-section changes only by isotopy, revealing clearly the presence of two unknotted, type (2) components.  In the transition to Figure~\ref{fig:monodromy4}, these components cap off and disappear.  In the transition to Figure~\ref{fig:monodromy5}, the flat structure is forgotten as we deperturb.  The end result is the boundary of the surface, described by a braid.
\end{example}

For more examples, see Figures~\ref{fig:square} and~\ref{fig:mono}, which were discussed in Examples~\ref{ex:square} and~\ref{ex:mono}, respectively.

\section{Shadow diagrams}
\label{sec:shadow}

Consider a $(g,b;\bold p,\bold f,\bold v)$--tangle $(H,\Tt)$. Let $\Delta$ be a bridge disk system for $\Tt$.  We now fix some necessary notation.
\begin{itemize}
	\item Let $\Sigma = \partial_+H$.
	\item Let $\alpha\subset\Sigma$ be a defining set of curves for $H$, disjoint from $\Delta$.
	\item Let $\frak a$ denote a collection of neatly embedded arcs, disjoint from $\Delta$ and $\alpha$ such that surgering $\Sigma$ along $\alpha$ and $\frak a$ results in a disjoint union of disks.  We assume $|\frak a|$ is minimized.
	\item Let $\Tt^*$ denote the shadows of the flat strands of $\Tt$ -- i.e. those coming from the bridge semi-disks.
	\item Let $\Aa^*$ denote the shadows for the vertical strands -- i.e. those coming from the bridge triangles.
	\item Let $\bold x = \Tt\cap\Sigma$.
\end{itemize}
The tuple $(\Sigma,\alpha,\Tt^*,\bold x)$ is called a \emph{tangle shadow} for the pair $(H,\Tt)$.
The tuple $(\Sigma,\alpha,\frak a,\Tt^*,\Aa^*,\bold x)$ is called an \emph{augmented tangle shadow} for the pair $(H,\Tt)$.  We will say that an augmented tangle shadow is an \emph{augmenting} of the underlying tangle shadow.
Figure~\ref{fig:std_diag} shows a pair of augmented tangle shadows: One is found by considering the red, pink, and orange arcs and curves, while the other is found by considering the dark blue, light blue, and orange arcs and curves.  Note that we consider tangle shadows up to isotopy rel-$\partial$.

\begin{figure}[h!]
\centering
\includegraphics[width=.8\linewidth]{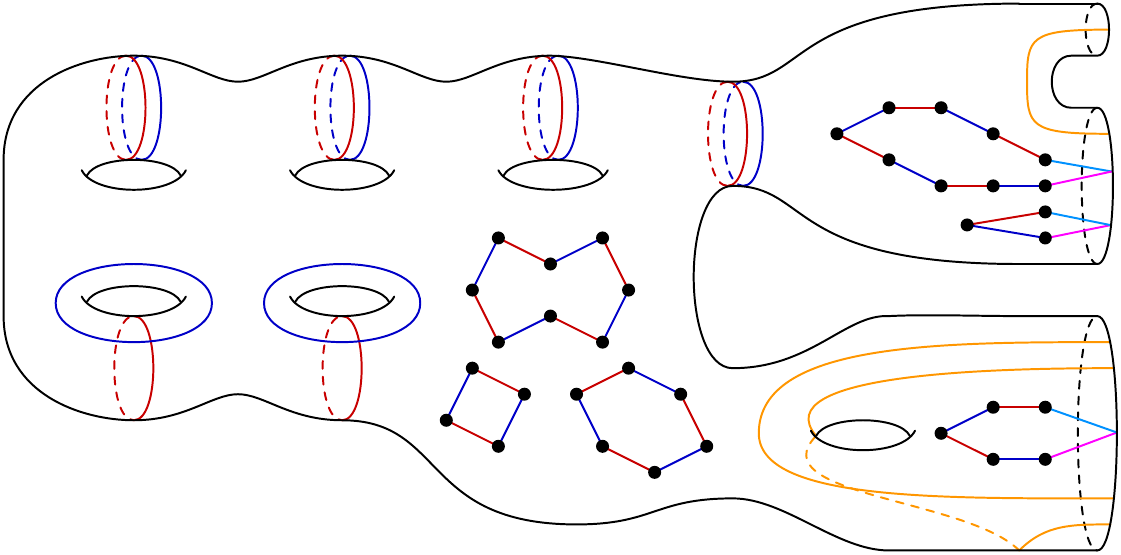}
\caption{A pair of (augmented) tangle shadows that, taken together, give a standard (augmented) splitting shadow.  The relevant parameters for each handlebody are $g=6$, $n=2$, $m=3$, $\bold p=(0,1)$, $\bold f = (2,1)$.  The relevant parameters for each tangle are $b=16$ and $\bold v=(0,2,1)$.  The arcs and curves of the $\alpha_i$ and $\Tt_i^*$ for $i=1,2$ are shown in red and blue, respectively, while the arcs of the $\Aa_i$ are shown in pink and light blue, respectively, and the arcs of $\frak a_1=\frak a_2$ are shown in orange.}
\label{fig:std_diag}
\end{figure}

\begin{lemma}
\label{lem:tangle_shadow}
	A tangle shadow determines a tangle $(H,\Tt)$.
\end{lemma}

\begin{proof}
	Given a shadow diagram $(\Sigma,\alpha,\Tt^*,\bold x)$, let $H$ be the lensed cobordism obtained from the spread $H\times[0,1]$ by attaching 3--dimensional 2--handles along the curves $\alpha\times\{1\}$.  Let $\Tt\subset H$ be obtained by perturbing the interiors of the shadows $\Tt^*\times\{0\}$ into the interior of $H$ to obtain the flat strands of $\Tt$ and extending the marked points $\bold x$ to vertical arcs $\bold x\times[0,1]$ using the product structure of the spread to obtain the vertical strands of $\Tt$.  This completes the proof the first claim.
\end{proof}

As a matter of convention, we have assumed without loss of generality that the curves and arcs of $\alpha\cup\frak a\cup\Tt^*\cup\Aa^*$ are all pairwise disjoint; it is not strictly necessary, for example, to assume $\alpha\cap\Tt^*=\emptyset$, but this can always be achieved.
Given a tangle shadow $(\Sigma,\alpha,\Tt^*,\bold x)$, we recall two standard moves:  Let $\alpha_1$ and $\alpha_2$ be two curves in $\alpha$, and let $\omega$ be a embedded arc in $\Sigma$ connecting $\alpha_1$ to $\alpha_2$ such that $\Int(\omega)\cap(\alpha\cup\Tt^*\cup\bold x)=\emptyset$.
Then $N=\nu(\alpha_1\cup\omega\cup\alpha_2)$ is a pair of pants.  Let $\alpha_1'$ be the boundary component of $N$ not parallel to $\alpha_1$ or $\alpha_2$.  Then $\alpha' = \alpha\setminus\{\alpha_1\}\cup\{\alpha_1'\}$ is a new defining set of curves for $H$.
We say that  $\alpha'$ is obtained from $\alpha$ by a \emph{curve slide} of $\alpha_1$ over $\alpha_2$ along $\omega$.
Now let $\tau_1^*$ be an arc of $\Tt^*$ and let $\alpha_2$ be a curve in $\alpha$ (respectively, the boundary of a regular neighborhood of another arc $\tau_2^*$ of $\Tt^*$). Let $\omega$ be a embedded arc in $\Sigma$ connecting $\tau_1^*$ to $\alpha_2$ such that $\Int(\omega)\cap(\alpha\cup\Tt^*\cup\bold x)=\emptyset$.
Let $(\tau_1^*)'$ denote the arc obtained by banding $\tau_1^*$ to $\alpha_2$ using the surface-framed neighborhood of $\omega$.
Then, $(\Tt^*)' = \Tt^*\setminus\tau_1^*\cup(\tau_1^*)'$ is a new collection of shadows for the flat strands of $\Tt$.  We say that $(\Tt^*)'$ is obtained from $\Tt^*$ by an \emph{arc slide} of $\tau_1^*$ over $\alpha_2$ (respectively, $\tau_2^*)$ along $\omega$.
Two shadow diagrams for $(H,\Tt)$ are called \emph{slide-equivalent} if they can be related by a sequence of curve slides and arc slides.

Given an augmented tangle shadow $(\Sigma,\alpha,\frak a,\Tt^*,\Aa^*,\bold x)$, we have further moves.  Similar to above, we have arc slide moves that allow one to slide arcs of $\frak a$ or $\Aa^*$ over arcs and curves of $\alpha$ and $\Tt^*$.  Note that we do not allow an arc or curve of any type to slide over an arc of $\frak a$ nor $\Aa^*$.
Two (augmented) shadow diagrams that are related by a sequence of these two types of moves are called \emph{slide-equivalent}.
The following is a modest generalization of a foundational result of Johansson~\cite{Joh_95_Topology-and-combinatorics-of-3-manifolds}, and follows from a standard argument.

\begin{proposition}
\label{prop:Joh}
	Two (augmented) tangle shadows for a given tangle are slide-equivalent.
\end{proposition}

A tuple $(\Sigma,\alpha_1,\alpha_2,\Tt_1^*,\Tt_2^*,\bold x)$ is called a \emph{splitting shadow} if each tuple $(\Sigma,\alpha_i,\Tt_i^*,\bold x)$ is a tangle shadow. A splitting shadow gives rise to a bridge splitting of pair $(M,K)$ in the same way that a tangle shadow gives rise to a tangle.  Recall the notion of a standard bridge splitting of $(M,K)$ from Subsection~\ref{subsec:Bridge}.  If a splitting shadow corresponds to a standard bridge splitting, then the tangle shadows $(\Sigma,\alpha_i,\Tt_i^*, \bold x)$ are (respectively, for $i=1,2$) slide-equivalent to tangle shadows $(\Sigma,\alpha_i',(\Tt_i^*)', \bold x)$ such that $(\Sigma,\alpha_1',\alpha_2')$ is a standard Heegaard diagram (cf. Subsection~\ref{subsec:Heegaard}) and $(\Tt_1^*)'\cup(\Tt_2^*)'$ is a neatly embedded collection of polygonal arcs and curves such that the polygonal curves bound disjointly embedded disks.  We call such a splitting shadow $(\Sigma,\alpha_1,\alpha_2,\Tt_1^*,\Tt_2^*,\bold x)$ \emph{standard}.  Figure~\ref{fig:std_diag} shows a standard splitting shadow (ignore the pink, light blue, and orange arcs for now).
Two splitting shadows are called \emph{slide-equivalent} if the two pairs of corresponding tangle shadows are slide-equivalent.

\begin{definition}
	A \emph{$(g,\bold k, f,\bold c;\bold p, \bold f, \bold v)$--shadow diagram} is a tuple $(\Sigma, \alpha_1, \alpha_2, \alpha_3,\Tt_1^*,\Tt_2^*,\Tt_3^*,\bold x)$, such that the tuple $(\Sigma,\alpha_i,\alpha_{i+1},\Tt_i^*,\Tt_{i+1}^*,\bold x)$ slide-equivalent to a standard splitting shadow for each $i\in\Z_3$.
	
	Two shadow diagrams are called \emph{slide-equivalent} if the three pairs of corresponding tangle shadows are slide-equivalent.
\end{definition}

\begin{figure}[h!]
\centering
\includegraphics[width=.5\linewidth]{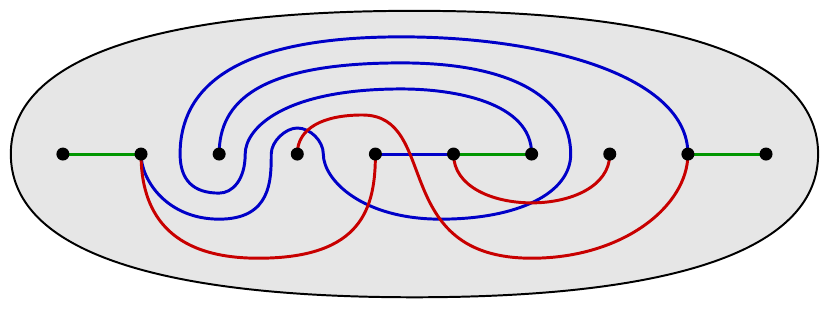}
\caption{A shadow diagram for the bridge trisection given in Figure~\ref{fig:steve}, which corresponds a ribbon disk for the stevedore knot.}
\label{fig:steve_shadow}
\end{figure}

\begin{proposition}
\label{prop:shadow}
	A $(g,\bold k, f,\bold c;\bold p, \bold f, \bold v)$--shadow diagram uniquely determines the spine of a $(g,\bold k, f,\bold c;\bold p, \bold f, \bold v)$--bridge trisection.  Any two shadow diagrams for a fixed bridge trisection are slide-equivalent.
\end{proposition}

\begin{proof}
	First, note that a shadow diagram determines the spine of a bridge trisection. This follows immediately from the definition of a shadow diagram, and Lemma~\ref{lem:tangle_shadow}.  The first claim follows from the fact that a bridge trisection is determined up to diffeomorphism by its spine.  The second claim follows from
	Proposition~\ref{prop:Joh}.
\end{proof}

Since bridge trisections are determined by their spines (Corollary~\ref{coro:spine}), we find that any surface $(X,\Ff)$ can be described by a shadow diagram.

\begin{corollary}
\label{coro:shadow_describe}
	Let $X$ be a smooth, orientable, compact, connected four-manifold, and let $\Ff$ be a neatly embedded surface in $X$.  Then, $(X,\Ff)$ can be described by a shadow diagram.
\end{corollary}

\subsection{Recovering the boundary braid from a shadow diagram}
\label{subsec:shadow_braid}
\ 

We now see how to recover the information about the boundary of a bridge trisected pair $(X,\Ff)$.  By augmenting a shadow diagram for the bridge trisection, we will recover this information in the form of an abstract open-book braiding, as defined in Subsection~\ref{subsec:OBD}.  What follows is based on the monodromy algorithm described by Castro, Gay, Pinz\'on-Caicedo in~\cite{CasGayPin_18_Diagrams-for-relative-trisections} and is closely related to the notion of an arced relative trisection diagram, as described in~\cite{GayMei_18_Doubly-pointed}.

To start, we return our attention to pairs of augmented tangle shadows.  A tuple $(\Sigma,\alpha_1,\alpha_2,\frak a, \Tt_1^*,\Tt_2^*,\Aa_1^*,\Aa_2^*,\bold x)$ is called a \emph{standard augmented splitting shadow} if
\begin{itemize}
	\item For each $i=1,2$, $(\Sigma,\alpha_i,\frak a, \Tt_i^*,\Aa_i^*,\bold x)$ is a augmented tangle shadow;
	\item $(\Sigma,\alpha_1,\alpha_2, \Tt_1^*,\Tt_2^*,\bold x)$ is a standard splitting shadow;
	\item The components of $\Tt_1^*\cup\Tt_2^*\cup\Aa_1^*\cup\Aa_2^*$ intersecting $\partial\Sigma$ bound disjointly embedded polygonal disks, each of which intersects $\partial \Sigma$ in a single point.
\end{itemize}
See Figure~\ref{fig:std_diag} for an example of a standard augmented splitting shadow.

\begin{definition}[\textbf{\emph{augmented shadow diagram}}]
	An \emph{augmented $(g,\bold k, f,\bold c;\bold p, \bold f, \bold v)$--shadow diagram} is a tuple $(\Sigma, \alpha_1, \alpha_2, \alpha_3,\frak a_1, \Tt_1^*,\Tt_2^*,\Tt_3^*,\Aa_1^*,\bold x)$, such that the tuple $(\Sigma, \alpha_1, \alpha_2, \alpha_3,\Tt_1^*,\Tt_2^*,\Tt_3^*,\bold x)$ is a shadow diagram, and $(\Sigma, \alpha_1, \frak a_1,\Tt_1^*,\Aa_1^*,\bold x)$ is an augmented tangle shadow.

	A \emph{fully augmented $(g,\bold k, f,\bold c;\bold p, \bold f, \bold v)$--shadow diagram} is a tuple
	$$\left(\Sigma, \alpha_1, \alpha_2, \alpha_3,\frak a_1,\frak a_2,\frak a_3,\frak a_4, \Tt_1^*,\Tt_2^*,\Tt_3^*,\Aa_1^*,\Aa_2^*,\Aa_3^*,\Aa_4^*,\bold x\right),$$
	such that the tuple $(\Sigma, \alpha_1, \alpha_2, \alpha_3,\Tt_1^*,\Tt_2^*,\Tt_3^*,\bold x)$ is a shadow diagram,  the tuples $(\Sigma, \alpha_1, \frak a_1,\Tt_1^*,\Aa_1^*,\bold x)$ and $(\Sigma, \alpha_1, \frak a_4,\Tt_1^*,\Aa_4^*,\bold x)$ are augmented tangle shadows for the same tangle, and the following hold:
	\begin{enumerate}
		\item For each $i\in\Z_3$, the diagram
			$$\left(\Sigma, \alpha_i,\alpha_{i+1}, \frak a_i, \frak a_{i+1}, \Tt_i^*, \Tt_{i+1}^*, \Aa_i^*, \Aa_{i+1}^*, \bold x\right)$$
			is slide-equivalent to a diagram
			$$\left(\Sigma, \alpha_i', \alpha_{i+1}', \frak a_i', \frak a_{i+1}', (\Tt_i^*)', (\Tt_{i+1}^*)', (\Aa_i^*)', (\Aa_{i+1}^*)', \bold x\right)$$
			that is a standard augmented splitting shadow, with $\frak a_i'=\frak a_{i+1}'$.
		\item The diagram
			$$\left(\Sigma, \alpha_3,\alpha_1, \frak a_3, \frak a_4, \Tt_3^*, \Tt_1^*, \Aa_3^*, \Aa_4^*, \bold x\right)$$
			is slide-equivalent to a diagram
			$$\left(\Sigma, \alpha_3'', \alpha_1'', \frak a_3'', \frak a_4'', (\Tt_3^*)'', (\Tt_1^*)'', (\Aa_3^*)'', (\Aa_4^*)'', \bold x\right)$$
			that is a standard augmented splitting shadow, with and satisfies $\frak a_3''=\frak a_4''$.
	\end{enumerate}
	We say that an augmented shadow diagram is an \emph{augmenting} of the underlying shadow diagram and that a fully  augmented shadow diagram is a \emph{full-augmenting} of the underlying (augmented) shadow diagram.
\end{definition}

We now describe how the data of an augmented shadow diagram allows us to recover the boundary open-book braiding  $(Y,\Ll)$ of the corresponding bridge trisected pair $\partial(X,\Ff)$.  First, we note the following crucial connection between augmented shadow diagrams and fully augmented shadow diagrams.

\begin{proposition}
\label{prop:aug}
	There is an algorithmic way to complete an augmented shadow diagram to a fully augmented shadow diagram, which is unique up to slide-equivalence.
\end{proposition}

\begin{proof}
	Start with an augmented shadow diagram $(\Sigma, \alpha_1, \alpha_2, \alpha_3,\frak a_1, \Tt_1^*,\Tt_2^*,\Tt_3^*,\Aa_1^*,\bold x)$.
	Restrict attention to the splitting shadow $(\Sigma,\alpha_1,\alpha_2,\Tt_1^*,\Tt_2^*,\bold x)$.
	By definition, this diagram is slide-equivalent to a standard splitting shadow $(\Sigma,\alpha_1',\alpha_2',(\Tt_1^*)',(\Tt_2^*)',\bold x)$.
	Choose a sequence of arc and curve slides realizing this equivalence. 
	Whenever a slide involving the arcs and curves of $\alpha_1\cup\Tt_1^*$ would be performed along an arc $\omega$ that intersects $\frak a_1\cup\Aa_1^*$, first slide the offending arcs of $\frak a_1\cup\Aa_1^*$ out of the way using the same slide-arc $\omega$.
	Now the splitting shadow has been standardized, but the arcs of $\frak a_1\cup\Aa_1^*$ may intersect the curves and arcs of $\alpha_2'\cup(\Tt_2^*)'$.
	Intersections of $\frak a_1\cup\Aa_1^*$ with the curves of $\alpha_2'$ can be removed via slides over the curves of $\alpha_1'$ dual to curves of~$\alpha_2'$.
	Recall that the closed components of $(\Tt_1^*)\cup(\Tt_2^*)'$ are embedded polygonal curves, while the non-closed components are embedded polygonal arcs.
	Moreover, the arcs of $\Aa_1^*$ connect one end of each polygonal arc to $\partial\Sigma$.
	Intersections of (the interior of) $\frak a_1\cup\Aa_1^*$ with the polygonal curves of $(\Tt_1^*)\cup(\Tt_2^*)'$ can be removed via slides over the arcs of $(\Tt_1^*)'$ included in these polygonal curves.
	Intersections of (the interior of) $\frak a_1\cup\Aa_1^*$ with the polygonal arcs of $(\Tt_1^*)\cup(\Tt_2^*)'$ can be removed via slides over the arcs of $(\Tt_1^*)'$ included in these polygonal arc, provided one is careful to slide towards the end of the polygonal arc that is not attached to~$\Aa_1^*$.
	
	Once the described slides have all been carried out, the collection $\frak a_1$ and $\Aa_1^*$ of arc will have been transformed into new collection, which we denote $\frak a_1'$ and $(\Aa_1^*)'$, respectively.
	The key fact is that $\frak a_1'$ and $(\Aa_1^*)'$ are disjoint (in their interiors) from the arcs and curves of $\alpha_2'\cup(\Tt_2^*)'$.
	Set $\frak a_2=\frak a_1'$, and note that $\frak a_2$ has the desired property of being (vacuously) slide-equivalent to $\frak a_2' = \frak a_1'$.
	To define $\Aa_2^*$, note that at this point the union of the polygonal arcs of $(\Tt_1^*)'\cup(\Tt_2^*)'$ with $(\Aa_1^*)'$ is a collection of embedded `augmented' polygonal arcs each of which intersects $\partial\Sigma$ in a single point.
	Let $\Aa_2^*$ be the collection of arcs obtained by pushing each augmented polygonal arc off itself slightly, while preserving its endpoint that lies in the interior of $\Sigma$.
	See Figure~\ref{fig:push-off}.
	This can be thought of as sliding the endpoint of $(\Aa_1^*)'$ that lies in the interior of $\Sigma$ along the polygonal arc of $(\Tt_1^*)'\cup(\Tt_2^*)'$ that it intersects until it reaches the end.
	Having carried out these steps, we have that $(\Sigma, \alpha_1', \alpha_2',\frak a_1', \frak a_2, (\Tt_1^*)',(\Tt_2^*)',(\Aa_1^*)',\Aa_2^*,\bold x)$ is a standard augmented splitting shadow, as desired.
	
\begin{figure}[h!]
\centering
\includegraphics[width=.25\linewidth]{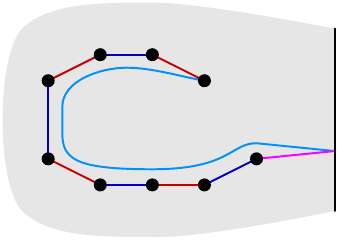}
\caption{Obtaining $\Aa_2^*$ from $(\Aa_1^*)'$.}
\label{fig:push-off}
\end{figure}

	Next, we repeat the process outlined in the first two paragraph, starting this time with the splitting shadow $(\Sigma,\alpha_2',\alpha_3,(\Tt_2^*)',\Tt_3^*,\bold x)$:
	Standardize the splitting shadow, and include the arcs of $\frak a_2\cup\Aa_2^*$ in the slides when necessary.
	Perform additional slides to obtain the new collection of arcs $\frak a_2'$, and $(\Aa_2^*)'$ whose interiors are disjoint from all other arcs and curves.
	Let $\frak a_3 = \frak a_2'$, and obtain $\Aa_3^*$ from $(\Aa_2^*)'$ in the same way as before, so that the new diagram $(\Sigma, \alpha_2'', \alpha_3',\frak a_2', \frak a_3, (\Tt_2^*)'',(\Tt_3^*)',(\Aa_2^*)',\Aa_3^*,\bold x)$ is a standard augmented splitting shadow, as desired.
	Note that $(\Sigma,\alpha_2'',(\Tt_2^*)'',\bold x)$ is slide-equivalent to the original diagram $(\Sigma,\alpha_2,\Tt_2^*,\bold x)$.

	Finally, repeat the process once more, starting with the splitting shadow $(\Sigma,\alpha_3',\alpha_1',(\Tt_3^*)',(\Tt_1^*)',\bold x)$ and performing slides until we can obtain new collections $\frak a_4$ and $\Aa_4^*$ from the modified collections $\frak a_3'$ and $(\Aa_3^*)'$, as before.
	At this point, there is a minor wrinkle.
	We are not finished once we set $\frak a_4 = \frak a_3'$ and obtain $\Aa_4^*$ from $(\Aa_3^*)'$ as before.
	The reason is that these choices for $\frak a_4$ and $\Aa_4$ might not be compatible with the original tangle shadow $(\Sigma, \alpha_1,\Tt_1^*,\bold x)$, rather these choices are compatible with the slide-equivalent tangle shadow $(\Sigma, \alpha_1'', (\Tt_1^*)'',\bold x)$.
	To remedy this issue, we perform the slides to change this latter tangle shadow to the former one, and we carry $\frak a_4$ and $\Aa_4^*$ with us along the way, sliding them over arcs and curves when necessary.
	In abuse of notation, we denote the results of this transformation $\frak a_4$ and $\Aa_4^*$.
	
	In summary, we have produce the collections of arcs $\frak a_2$, $\frak a_3$, $\frak a_4$, $\Aa_2^*$, $\Aa_3^*$, and $\Aa_4^*$ required to fully augment the original augmented shadow diagram.
\end{proof}

Following Castro, Gay, and Pinz\'on-Caicedo, we refer to the above algorithm as the \emph{monodromy algorithm}.  What follows a is generalization of the discussion of~\cite[Section~3]{GayMei_18_Doubly-pointed}; see also~\cite[Section~4]{CasGayPin_18_Diagrams-for-relative-trisections} and~\cite[Section~2]{CasOzb_19_Trisections-of-4-manifolds-via-Lefschetz}.

Given an augmented shadow diagram $\DD=(\Sigma, \alpha_1, \alpha_2, \alpha_3,\frak a_1, \Tt_1^*,\Tt_2^*,\Tt_3^*,\Aa_1^*,\bold x)$, let $(H,\Tt)$ denote the tangle determined by the tangle shadow $(\Sigma,\alpha_1,\Tt_1^*,\bold x)$.  Let $(P,\bold y)_\DD = \partial_-(H,\Tt)$.  We call $(P,\bold y)_\DD$ the \emph{page} of the shadow diagram.  Fix an identification  $\Id\colon(P,\bold y)_\DD\to(\Sigma_{\bold p,\bold f},\bold x_{\bold p,\bold f})$. We use the standard Morse structure on $H$ to consider $\frak a_1$ and $\Aa_1^*$ as lying in $P$. Consider the arcs $\frak a=\Id(\frak a_1)$, which cut the standard surface into a collection of disks, and the arcs $\Aa^*=\Id(\Aa^*_1)$, which connect the marked points to the boundary in the standard pair.

Let $\DD^+=\left(\Sigma, \alpha_1, \alpha_2, \alpha_3,\frak a_1,\frak a_2,\frak a_3,\frak a_4, \Tt_1^*,\Tt_2^*,\Tt_3^*,\Aa_1^*,\Aa_2^*,\Aa_3^*,\Aa_4^*,\bold x\right)$ be a full-augmenting of $\DD$.  We consider the arcs $\frak a_4$ and $\Aa^*_4$ as lying in $P$, as well.  Consider the arcs $\frak a'=\Id(\frak a_4)$ and the arcs $(\Aa^*)'=\Id(\Aa^*_4)$.  Let $\phi_\DD$ be the automorphism of $(\Sigma_{\bold p,\bold f},\bold x_{\bold p,\bold f})$ satisfying $\phi_\DD(\frak a\cup\Aa^*) = \frak a'\cup(\Aa^*)'$, noting that $\phi$ is unique up to isotopy.  We call $\phi_\DD$ the \emph{monodromy} of the shadow diagram. It is straight-forward to check that monodromy $\phi_\DD$ of an augmented shadow diagram $\DD$ depends only on the underlying shadow diagram (not on the choice of augmentation).  The relevance of $\phi_\DD$ is given in the following proposition; we refer the reader to Subsection~\ref{subsec:OBD} for relevant notation and terminology regarding open-book decompositions and braidings.  The following is a generalization of~\cite[Theorem~5]{CasGayPin_18_Diagrams-for-relative-trisections} and~\cite[Lemma~3.1]{GayMei_18_Doubly-pointed}.

\begin{proposition}
\label{prop:monodromy}
	Suppose that $\DD$ is a shadow diagram for a bridge trisection $\T$ of a pair $(X,\Ff)$. Let $\phi_\DD$ denote the monodromy of the shadow diagram, and let $(Y_{\phi_\DD},\Ll_{\phi_\DD})$ denote the model open-book braiding corresponding to the abstract open-book braiding $(\Sigma_{\bold p,\bold f},\bold x_{\bold p,\bold f},\phi_\DD)$. Then, there is a canonical (up to isotopy) diffeomorphism
	$$\psi_\DD\colon \partial(X,\Ff)\to(Y_{\phi_\DD},\Ll_{\phi_\DD}).$$
\end{proposition}

\begin{proof}
	Let $(H_1,\Tt_1)\cup(H_2,\Tt_2)\cup(H_3,\Tt_3)$ denote the spine of the bridge trisection determined by the diagram~$\DD$; cf. Proposition~\ref{prop:spine} and Proposition~\ref{prop:shadow}.  Fix an identifcation $\psi\colon(P_1\bold y_1)\to(\Sigma_{\bold p,\bold f},\bold x_{\bold p,\bold f})$ and regard this latter pair as a page $(P,\bold y)\times\{0\}$ in the model open-book braiding $(Y_{\phi_\DD},\Ll_{\phi_\DD})$, which we think of as $(P,\bold y)\times_{\phi_\DD}S^1$.
	
	Choose an augmenting of $\DD$ by picking arcs $\frak a_1$ and $\Aa_1^*$, which we consider as having been isotoped vertically to lie in $(P_1,\bold y_1)$.  Let $\frak a\times\{0\}$ and $\Aa^*\times\{0\}$ denote the arcs on $(P,\bold y)\times\{0\}$ that are the images of $\frak a_1$ and $\Aa_1^*$ under $\psi$.  Apply the monodromy algorithm of Proposition~\ref{prop:monodromy} to obtain a full-augmenting of $\DD$.  Consider the arcs $\frak a_1'$, $(\Aa_1^*)'$, and $(\Aa_2^*)'$ coming from the standard augmented splitting diagram for
	$$(M_1,K_1)=(H_1,\Tt_1)\cup_{(\Sigma,\bold x)}\overline{(H_2,\Tt_2)},$$
	noting that, regarded as arcs in $(P_1,\bold y_1)$, $\frak a_1$ and $\frak a_1'$ are isotopic rel-$\partial$, as are $\Aa_1^*$ and $(\Aa_1^*)'$.
	These arcs determine the identity map $\Id_{(M_1,K_1,\Sigma)}$ described in Lemma~\ref{lem:BridgeDouble}.  In particular, this gives a unique extension of $\psi$ to a diffeomorphism from the spread $(Y_1,\beta_1)$ in $\partial(X,\Ff)$ to the spread $(P,\bold y)\times[0,1/3]$ in $(Y_{\phi_\DD},\Ll_{\phi_\DD})$.
	
	Repeating the step described above ($i=1$) for $i=2$ and $i=3$ allows us to extend $\psi_1$ to a map $\psi_\DD$ whose domain is the entire boundary
	$$\partial(X,\Ff) = (Y_1,\beta_1)\cup(Y_2,\beta_2)\cup(Y_3,\beta_3)$$
	and whose codomain is $(P,\bold y)\times[0,1]$, equipped with the identification $x\sim\phi'(x)$, where $\phi'$ must take the arcs $\frak a_4\cup\Aa_4^*$ to the arcs $\frak a_1\cup\Aa_1^*$, in order for $\psi_\DD$ to be continuous.  However, this implies that $\phi'$ is isotopic rel-$\partial$ to $\phi_\DD$, by definition, and we have that $\psi_\DD$ respects the original identification space structure on $(Y_{\phi_\DD},\Ll_{\phi_\DD})$, hence is a diffeomorphism, as desired.
\end{proof}

\begin{figure}[h!]
\begin{subfigure}{.33\textwidth}
	\centering
	\includegraphics[width=.9\linewidth]{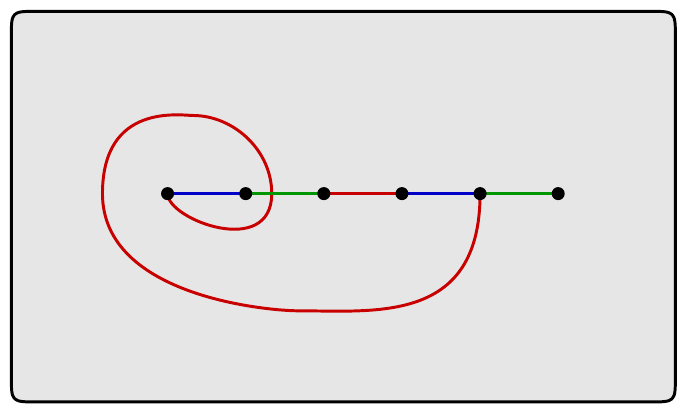}
	\caption{}
	\label{fig:Mob_sh_1}
\end{subfigure}%
\begin{subfigure}{.33\textwidth}
	\centering
	\includegraphics[width=.9\linewidth]{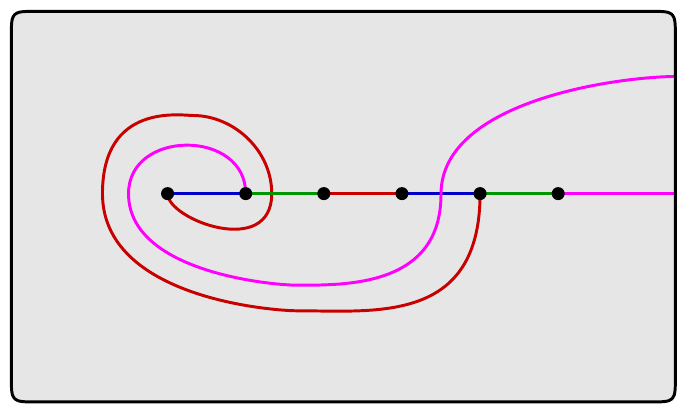}
	\caption{}
	\label{fig:Mob_sh_2}
\end{subfigure}%
\begin{subfigure}{.33\textwidth}
	\centering
	\includegraphics[width=.9\linewidth]{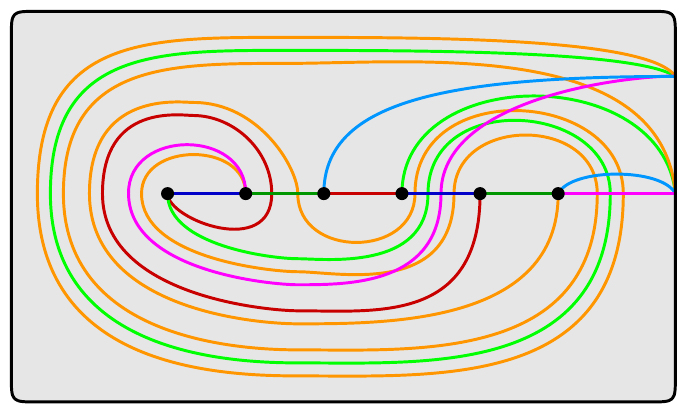}
	\caption{}
	\label{fig:Mob_sh_3}
\end{subfigure}
\begin{subfigure}{.33\textwidth}
	\centering
	\includegraphics[width=.9\linewidth]{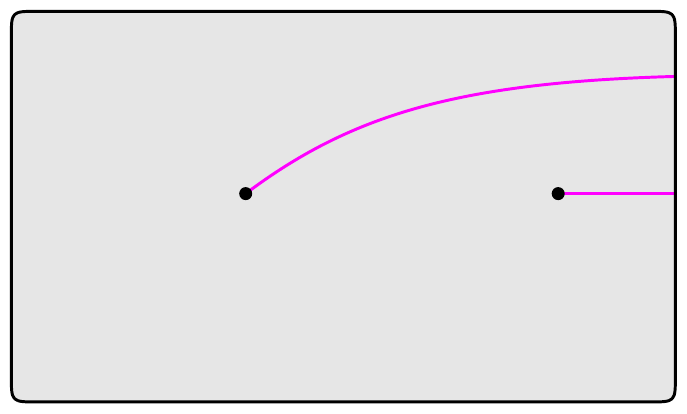}
	\caption{}
	\label{fig:Mob_sh_4}
\end{subfigure}%
\begin{subfigure}{.33\textwidth}
	\centering
	\includegraphics[width=.9\linewidth]{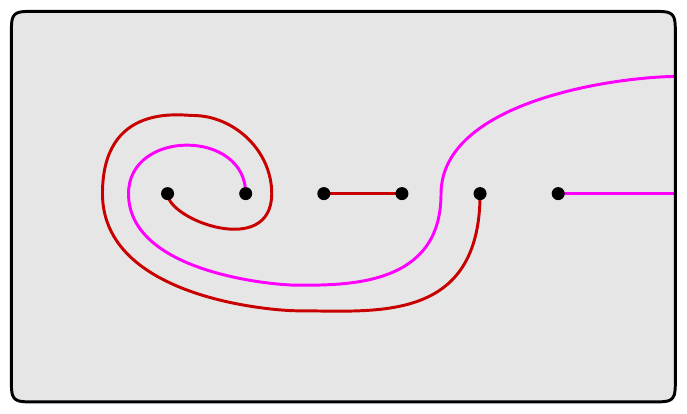}
	\caption{}
	\label{fig:Mob_sh_5}
\end{subfigure}%
\begin{subfigure}{.33\textwidth}
	\centering
	\includegraphics[width=.9\linewidth]{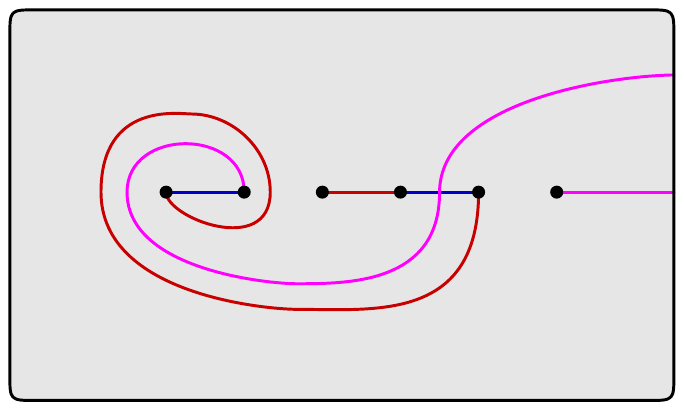}
	\caption{}
	\label{fig:Mob_sh_6}
\end{subfigure}
\begin{subfigure}{.33\textwidth}
	\centering
	\includegraphics[width=.9\linewidth]{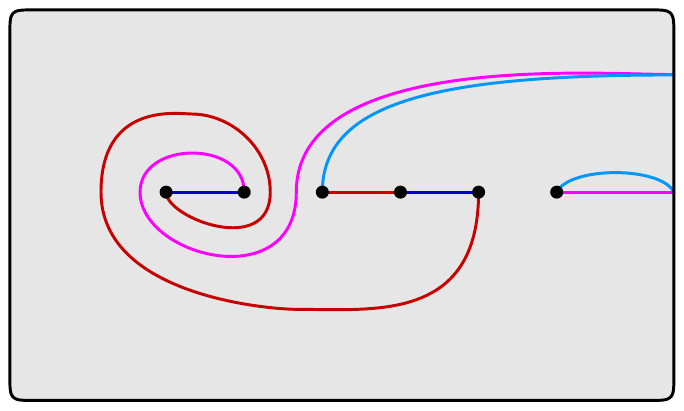}
	\caption{}
	\label{fig:Mob_sh_7}
\end{subfigure}%
\begin{subfigure}{.33\textwidth}
	\centering
	\includegraphics[width=.9\linewidth]{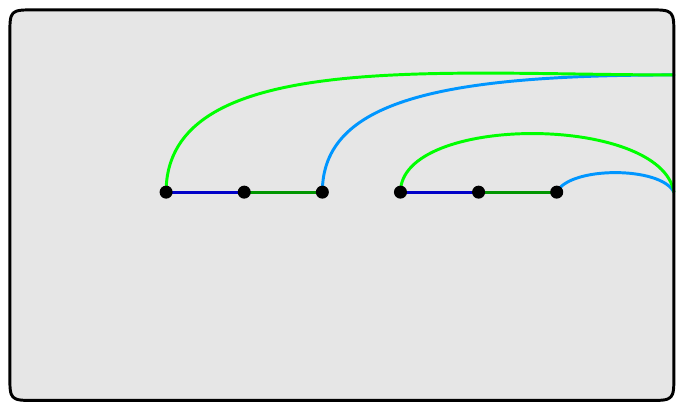}
	\caption{}
	\label{fig:Mob_sh_8}
\end{subfigure}%
\begin{subfigure}{.33\textwidth}
	\centering
	\includegraphics[width=.9\linewidth]{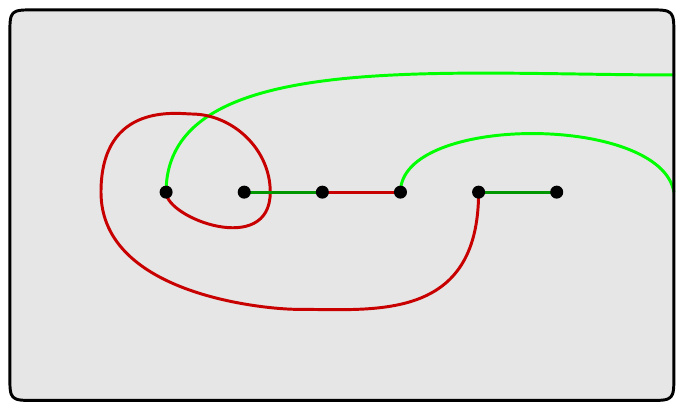}
	\caption{}
	\label{fig:Mob_sh_9}
\end{subfigure}
\begin{subfigure}{.33\textwidth}
	\centering
	\includegraphics[width=.9\linewidth]{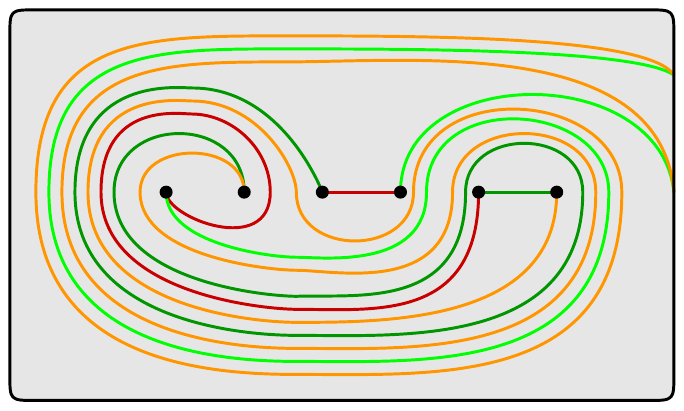}
	\caption{}
	\label{fig:Mob_sh_10}
\end{subfigure}%
\begin{subfigure}{.33\textwidth}
	\centering
	\includegraphics[width=.9\linewidth]{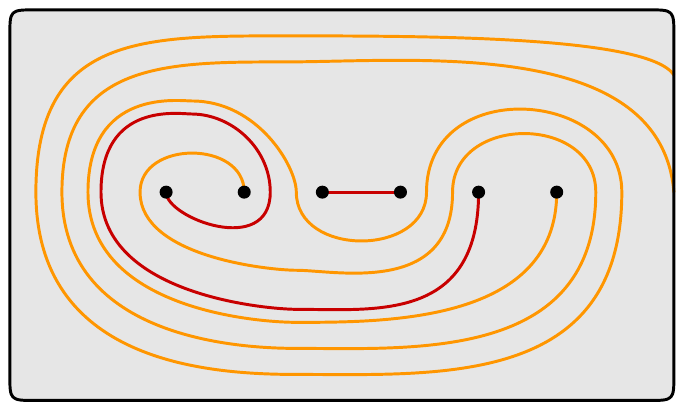}
	\caption{}
	\label{fig:Mob_sh_11}
\end{subfigure}%
\begin{subfigure}{.33\textwidth}
	\centering
	\includegraphics[width=.9\linewidth]{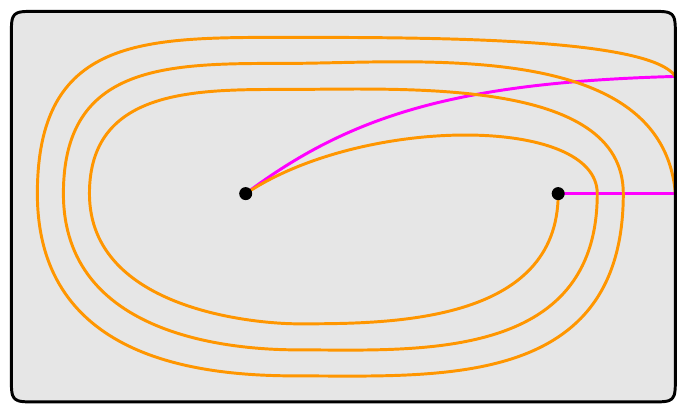}
	\caption{}
	\label{fig:Mob_sh_12}
\end{subfigure}
\caption{A shadow diagram (A), a augmented shadow diagram (B), and a fully augmented shadow diagram (C) for a bridge trisection for the M\"obius band bounded by the right-handed trefoil in $S^3$.  (E)--(K) illustrate the process described by the monodromy algorithm of Proposition~\ref{prop:monodromy}, used to find the full-augmenting (C) of the augmented shadow diagram (B). We recover the braiding induced on the boundary of the bridge trisection by studying  (L), which show the arcs $\frak a$ and $\frak a'$ in the page $(P,\bold Y)$.}
\label{fig:Mob_sh}
\end{figure}

\begin{example}
\label{ex:Mob_sh}
	\textbf{(M\"obius band for the trefoil)}
	Figure~\ref{fig:Mob_sh_1} shows a shadow diagram corresponding to the bridge trisection of the M\"obius band bounded by the right-handed trefoil in $S^3$ that was discussed in Example~\ref{ex:mono}.  Since this is a $(2;0,2)$--bridge trisection, we have that $(P,\bold y) = \partial_-(H_1,\Tt_1)$ is a disk with two distinguished points in its interior.  This pair is shown in Figure~\ref{fig:Mob_sh_4}, together with a pair of arcs that connect the points $\bold y$ to $\partial P$.  Using the Morse function on $(H_1,\Tt_2)$, these arcs can be flowed rel-$\partial$ to lie in $\Sigma$, as shown in Figure~\ref{fig:Mob_sh_5}.  In Figure~\ref{fig:Mob_sh_6}, the shadows for $(H_2,\Tt_2)$ have been added, making an splitting shadow for $(M_1,K_1)$, which is a geometric 2--braid in $D^2\times I$, one component of which is twice-perturbed, while the other is not perturbed.  In Figure~\ref{fig:Mob_sh_7}, a slide of an arc of $\Aa^*_1$ has been performed to arrange that all arcs are disjoint in their interiors, and the arcs of $\Aa_2^*$ have been obtained, as described in the proof of Proposition~\ref{prop:monodromy}; this is an augmented splitting shadow for $(M_1,K_1)$.  Figure~\ref{fig:Mob_sh_8} shows a splitting shadow for $(M_2,K_2)$, with $\Aa_2^*$ remembered, and since all arcs are disjoint in their interiors, the arcs of $\Aa_3^*$ have been derived.  Figure~\ref{fig:Mob_sh_9} shows a splitting shadow for $(M_3,K_3)$, with the arcs of $\Aa_3^*$ remembered, and Figure~\ref{fig:Mob_sh_10} is obtained from this diagram by arc slides of arcs from $\Tt_3^*\cup\Aa_3^*$, before $\Aa_4^*$ is obtained.  In Figure~\ref{fig:Mob_sh_11}, the arcs of $\Aa_1^*$ and $\Aa_4^*$ are shown with the arcs of $\Tt_1^*$ in $\Sigma$.  Figure~\ref{fig:Mob_sh_12} shows the result of flowing $\Aa_1^*\cup\Aa_4^*$ up to the page $(P,\bold y)$.

\begin{figure}[h!]
\begin{subfigure}{.33\textwidth}
	\centering
	\includegraphics[width=.9\linewidth]{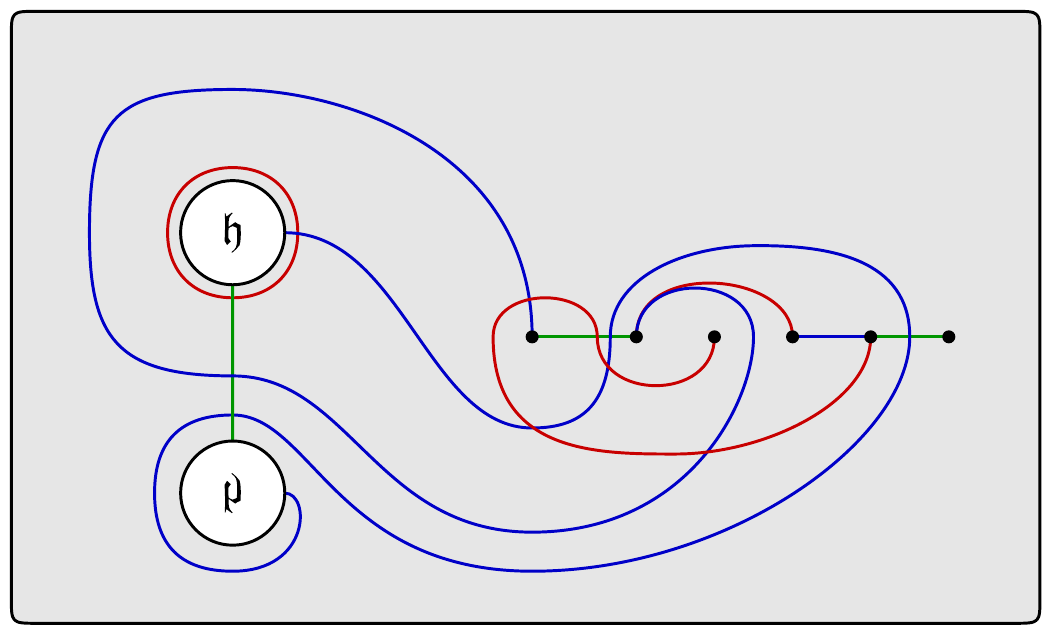}
	\caption{}
	\label{fig:tref_disk_1}
\end{subfigure}%
\begin{subfigure}{.33\textwidth}
	\centering
	\includegraphics[width=.9\linewidth]{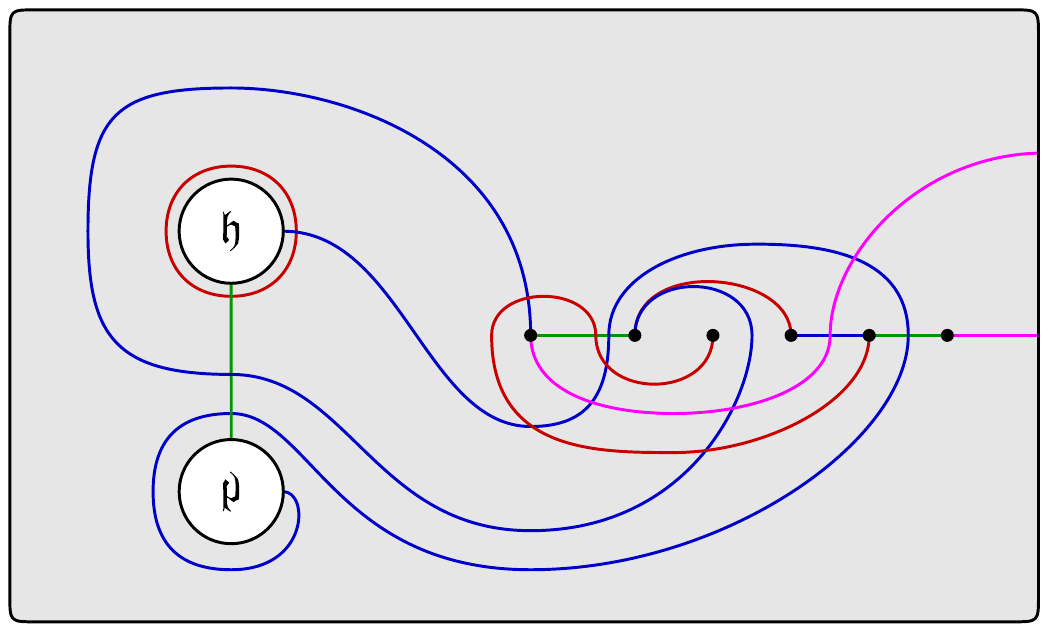}
	\caption{}
	\label{fig:tref_disk_2}
\end{subfigure}%
\begin{subfigure}{.33\textwidth}
	\centering
	\includegraphics[width=.9\linewidth]{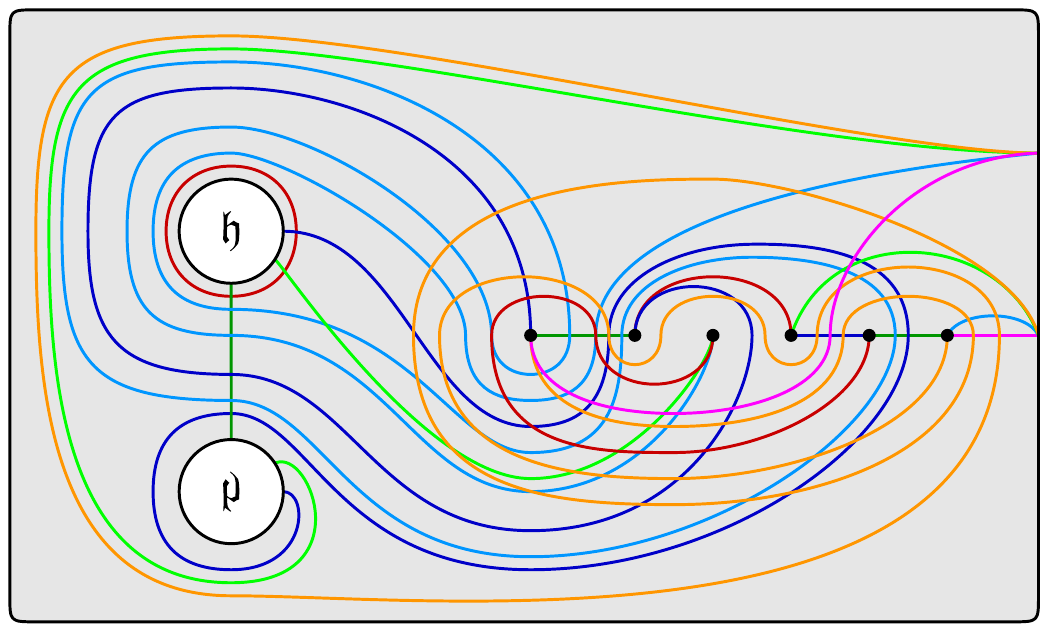}
	\caption{}
	\label{fig:tref_disk_3}
\end{subfigure}
\begin{subfigure}{.33\textwidth}
	\centering
	\includegraphics[width=.9\linewidth]{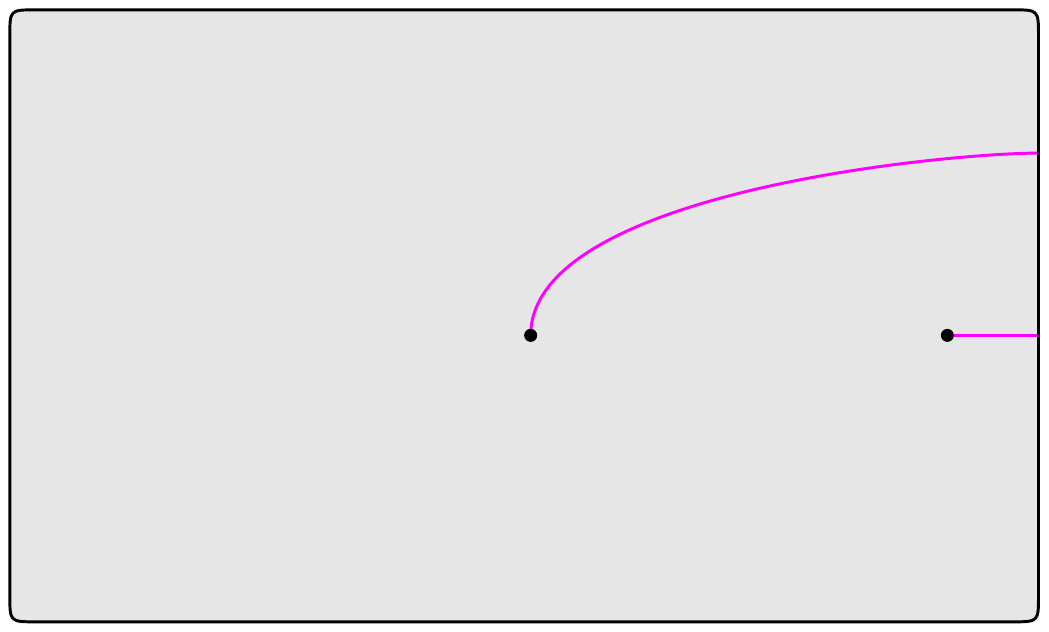}
	\caption{}
	\label{fig:tref_disk_4}
\end{subfigure}%
\begin{subfigure}{.33\textwidth}
	\centering
	\includegraphics[width=.9\linewidth]{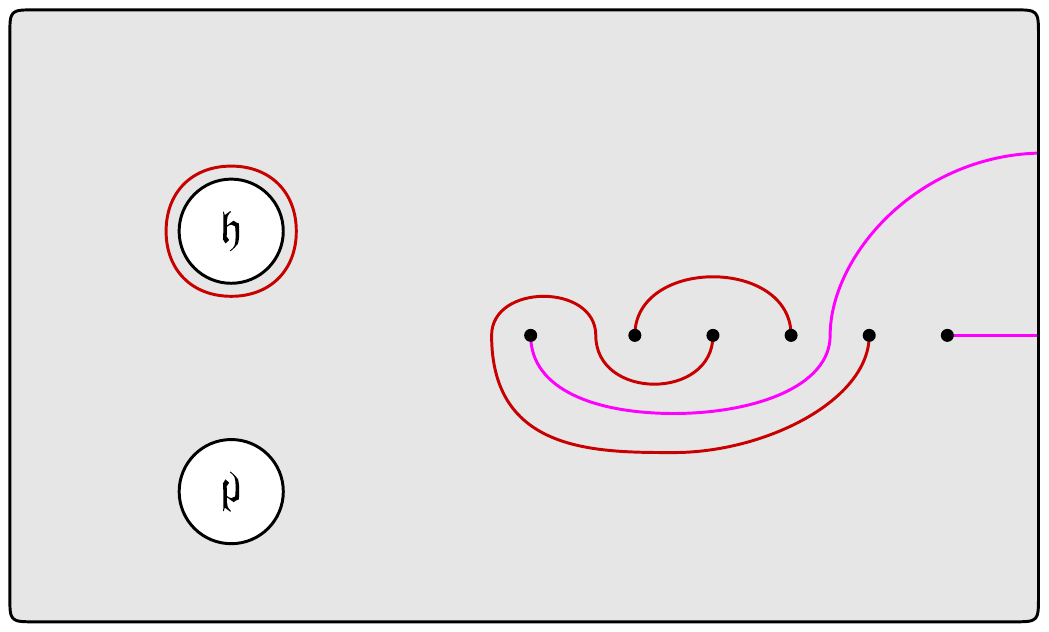}
	\caption{}
	\label{fig:tref_disk_5}
\end{subfigure}%
\begin{subfigure}{.33\textwidth}
	\centering
	\includegraphics[width=.9\linewidth]{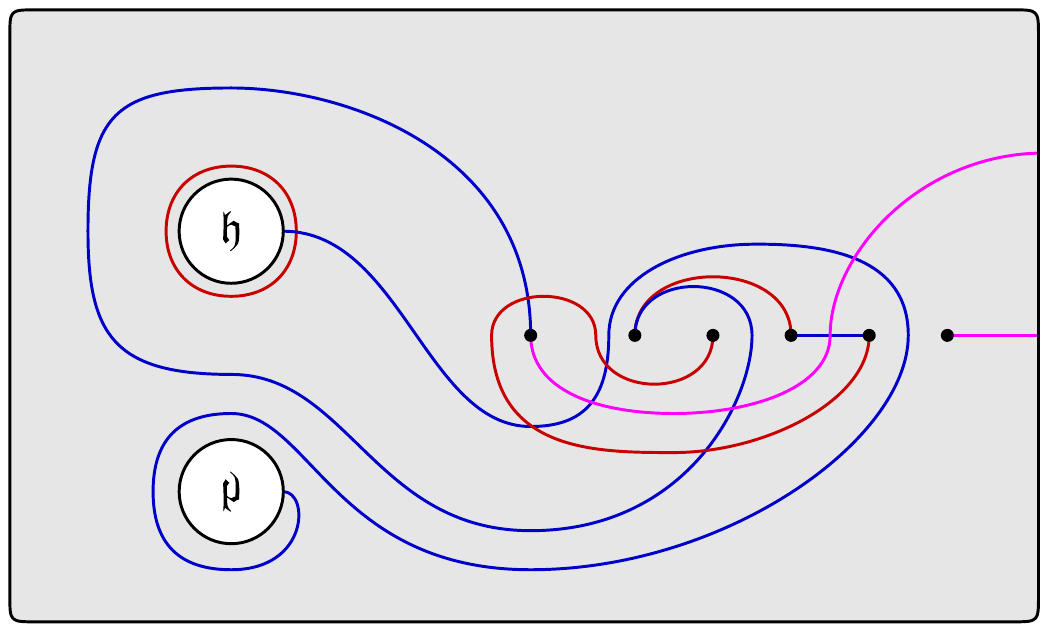}
	\caption{}
	\label{fig:tref_disk_6}
\end{subfigure}
\begin{subfigure}{.33\textwidth}
	\centering
	\includegraphics[width=.9\linewidth]{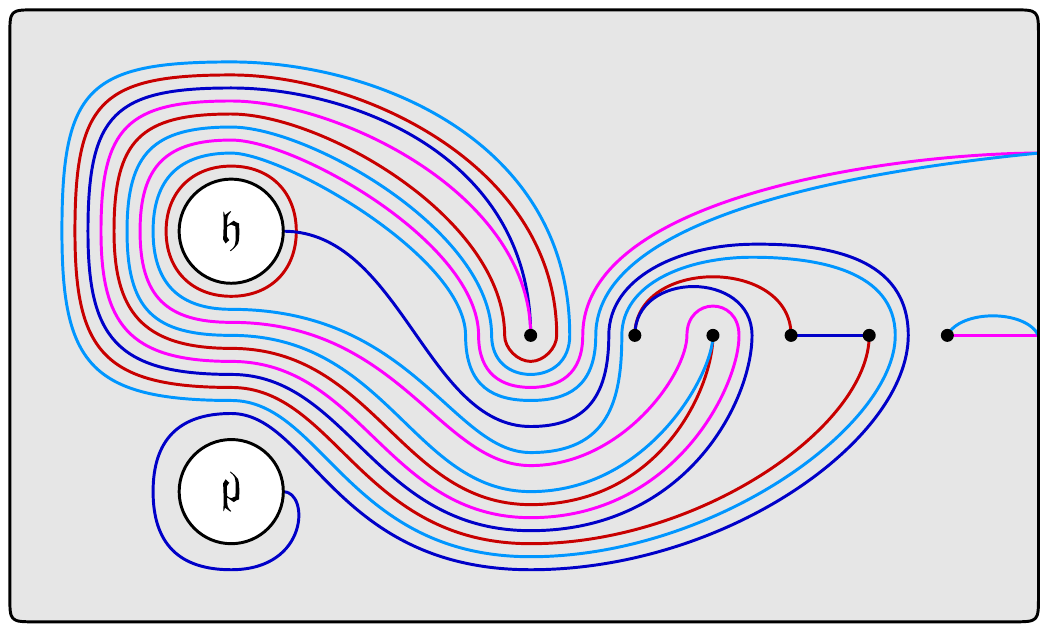}
	\caption{}
	\label{fig:tref_disk_7}
\end{subfigure}%
\begin{subfigure}{.33\textwidth}
	\centering
	\includegraphics[width=.9\linewidth]{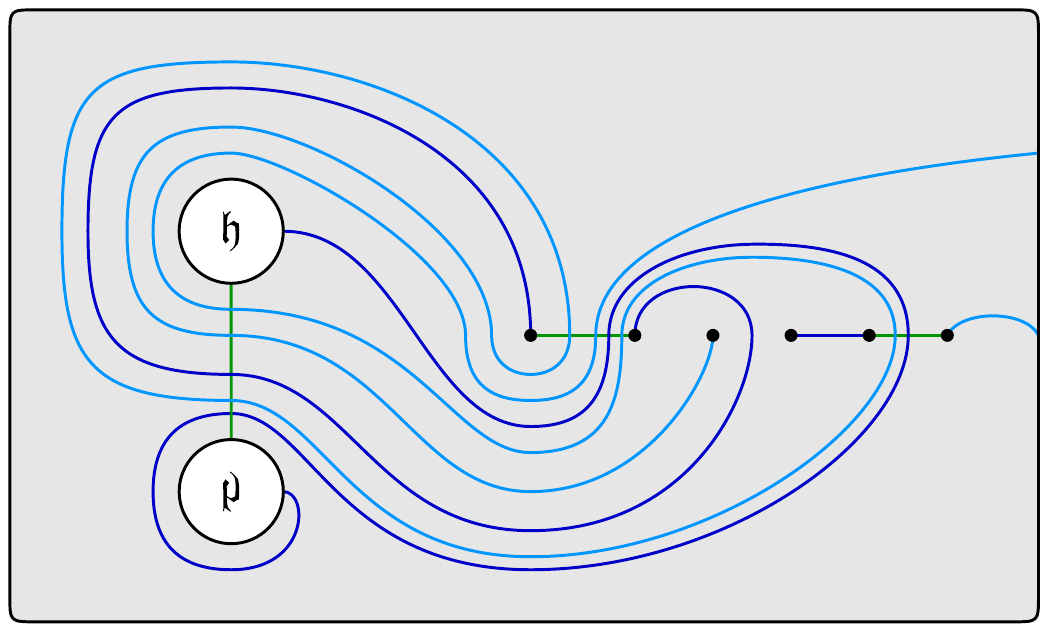}
	\caption{}
	\label{fig:tref_disk_8}
\end{subfigure}%
\begin{subfigure}{.33\textwidth}
	\centering
	\includegraphics[width=.9\linewidth]{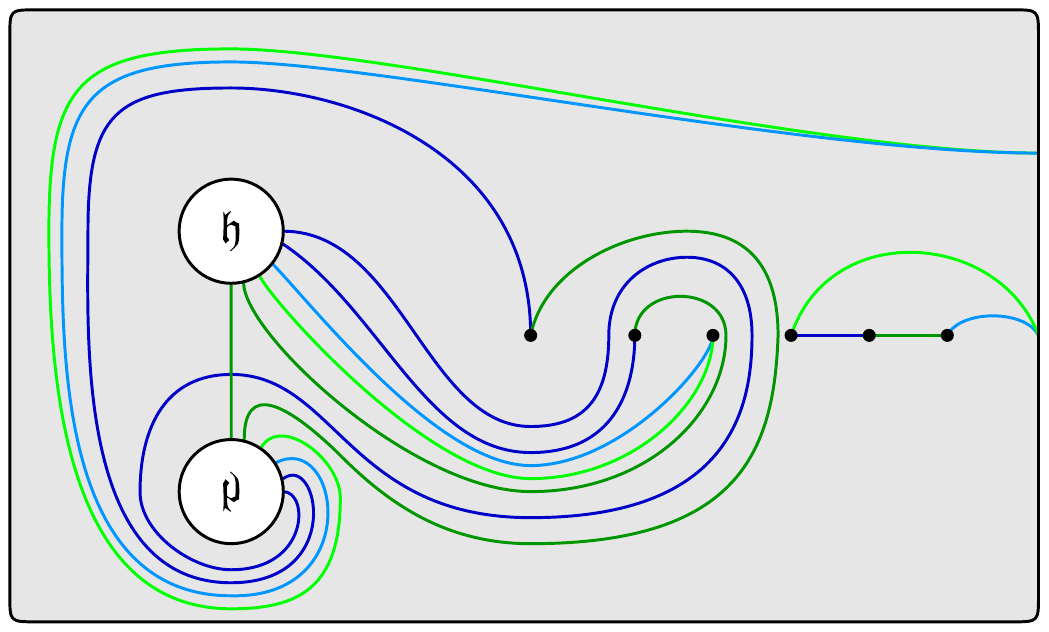}
	\caption{}
	\label{fig:tref_disk_9}
\end{subfigure}
\begin{subfigure}{.33\textwidth}
	\centering
	\includegraphics[width=.9\linewidth]{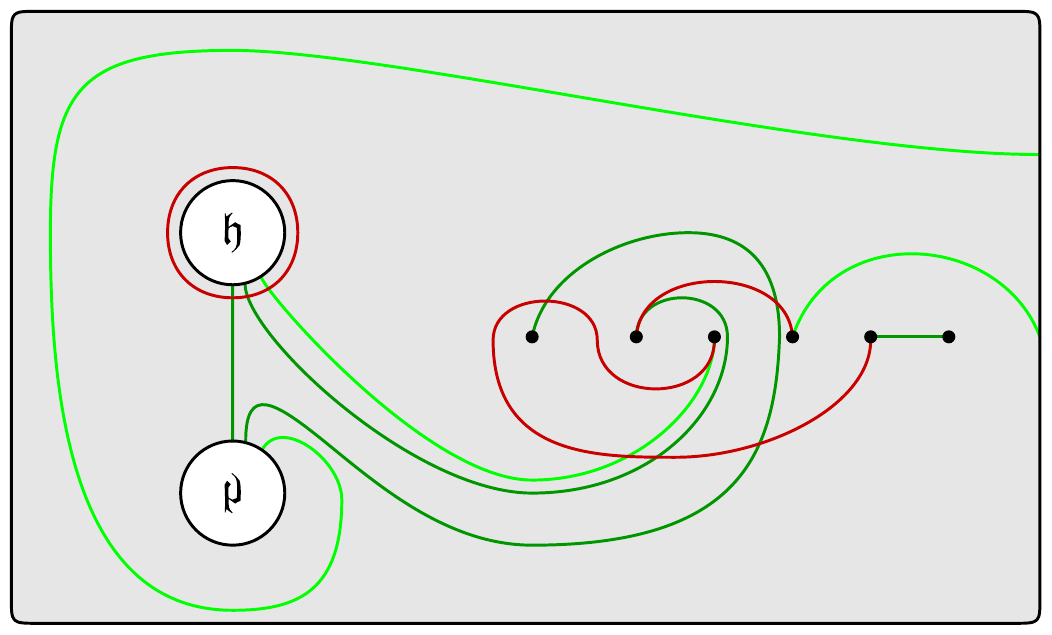}
	\caption{}
	\label{fig:tref_disk_10}
\end{subfigure}%
\begin{subfigure}{.33\textwidth}
	\centering
	\includegraphics[width=.9\linewidth]{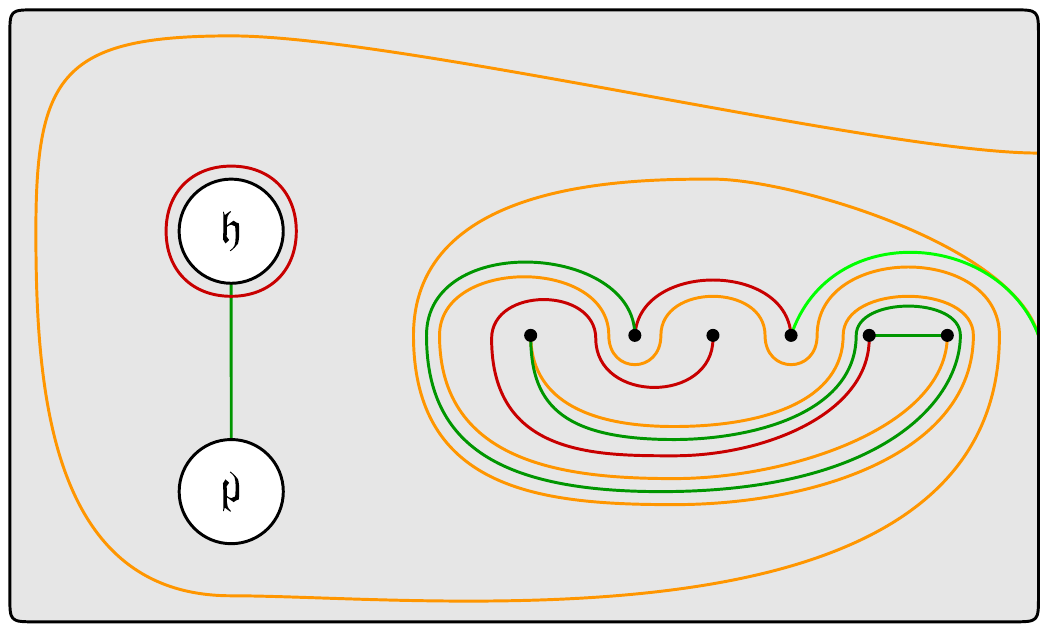}
	\caption{}
	\label{fig:tref_disk_11}
\end{subfigure}%
\begin{subfigure}{.33\textwidth}
	\centering
	\includegraphics[width=.9\linewidth]{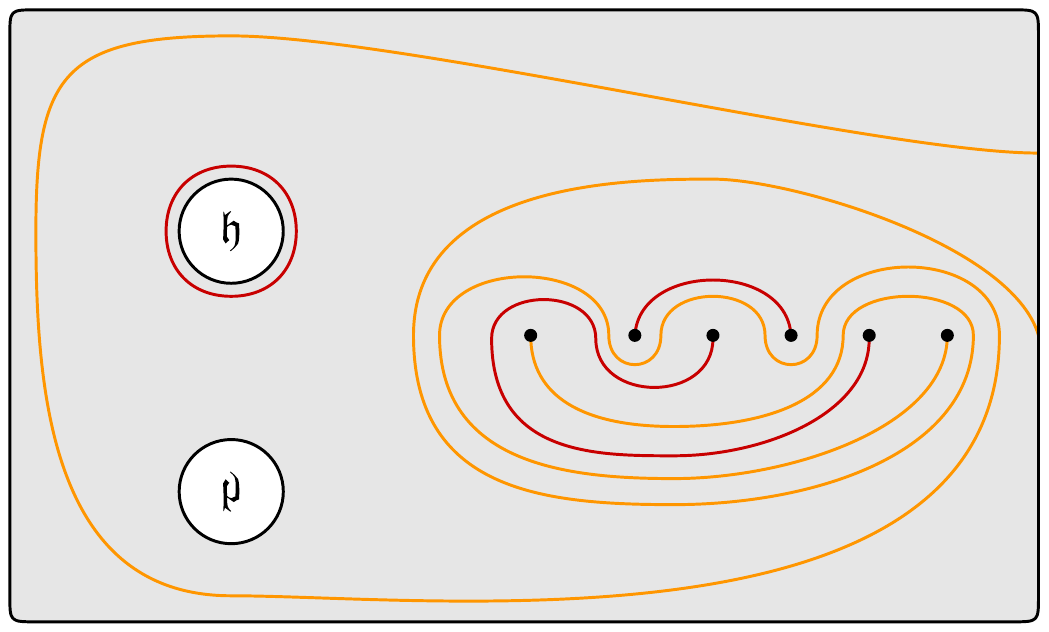}
	\caption{}
	\label{fig:tref_disk_12}
\end{subfigure}
\begin{subfigure}{.33\textwidth}
	\centering
	\includegraphics[width=.9\linewidth]{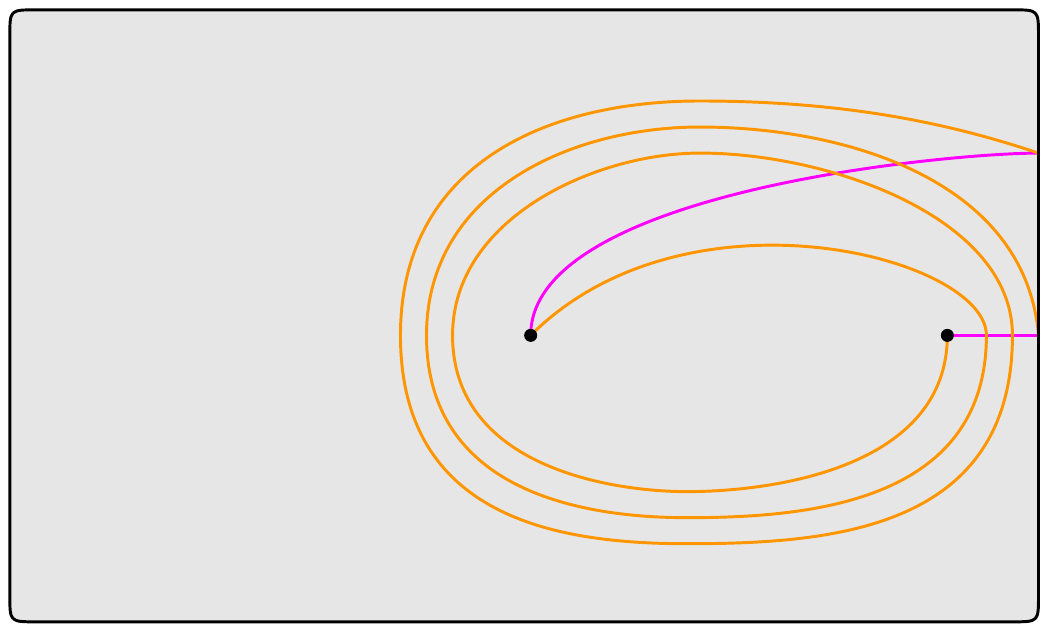}
	\caption{}
	\label{fig:tref_disk_13}%
\end{subfigure}
\caption{A shadow diagram (A), a augmented shadow diagram (B), and a fully augmented shadow diagram for a bridge trisection for the disk bounded by the right-handed trefoil in $(\CP^2)^\circ$.  (D)--(M) illustrate the process described by the monodromy algorithm of Proposition~\ref{prop:monodromy}, used to find a full-augmenting of a shadow diagram and recover the braiding induced on the boundary of the bridge trisection.}
\label{fig:tref_disk}
\end{figure}
	
	Figure~\ref{fig:Mob_sh_12} allows us to see that the braiding induced on the boundary of the bridge trisection is diffeomorphic to the abstract open-book $(P,\bold y, \tau^3)$, where $P$ is a disk, $\bold y$ is two points, and $\tau$ is a right-handed Dehn twist about a curve parallel to $\partial P$ in $P\setminus\nu(\bold y)$.  This derivation is a shadow diagram version of the calculation of this braiding given in Example~\ref{ex:mono} and Figure~\ref{fig:mono}.
\end{example}

\begin{example}
\label{ex:tref_disk}
	\textbf{(Disk for the trefoil in $(\CP^2)^\circ$)}
	Figure~\ref{fig:tref_disk_1} shows a shadow diagram corresponding to a bridge trisection of a disk bounded by the right-handed trefoil in $(\CP^2)^\circ$, the result of removing a neighborhood of a point from $\CP^2$. The two circles represent the foot of a handle for the surface $\Sigma$ and are identified via vertical reflection.
	If one forgets the bridge points $\bold x$ and all shadow arcs, one obtains a $(1,0;0,1)$--trisection diagram for this four-manifold.  The bridge trisection itself is type $(2,(0,1,0);2)$; the union of the blue and green shadows includes a bigon.  As in the previous example, we have that $(P,\bold y) = \partial_-(H_1,\Tt_1)$ is a disk with two distinguished points in its interior.  This pair is shown in Figure~\ref{fig:tref_disk_4}, together with a pair of arcs that connect the points $\bold y$ to $\partial P$.  Using the Morse function on $(H_1,\Tt_2)$, these arcs can be flowed rel-$\partial$ to lie in $\Sigma$, as shown in Figure~\ref{fig:tref_disk_5}.
	In Figure~\ref{fig:tref_disk_6}, the shadows for $(H_2,\Tt_2)$ have been added, giving a splitting shadow for $(M_1,K_1)$, which is a geometric 2--braid in $D^2\times I$, one component of which is twice-perturbed with respect to the once-stabilized Heegaard splitting of this spread.
	In Figure~\ref{fig:tref_disk_7}, a number of arc slides of $\Tt_1^*\cup\Aa^*_1$ have been performed to arrange that all arcs and curves are disjoint in their interiors, save the standard curve pair $\alpha_1\cup\alpha_2$.  From this standard splitting shadow, the arcs of $\Aa_2^*$ have been obtained, as described in the proof of Proposition~\ref{prop:monodromy}.  Figure~\ref{fig:tref_disk_8} shows a splitting shadow for $(M_2,K_2)$, with $\Aa_2^*$ remembered.  Figure~\ref{fig:tref_disk_9} shows the standard augmented splitting shadow resulting from  a number or arc slides, together with the arcs of $\Aa_3^*$.
	Figure~\ref{fig:tref_disk_10} shows a splitting shadow for $(M_3,K_3)$, with the arcs of $\Aa_2^*$ remembered, and Figure~\ref{fig:tref_disk_11} shows a slide-equivalent standard splitting shadow, with $\Aa_\circ^*$ derived.
	In Figure~\ref{fig:tref_disk_12}, the arcs of $\Aa_1^*$ and $\Aa_\circ^*$ are shown with the arcs and curves of the original tangle shadow for $(H_1,\Tt_1)$ in $\Sigma$.
	Figure~\ref{fig:tref_disk_13} shows the result of flowing $\Aa_1^*\cup\Aa_\circ^*$ up to the page $(P,\bold y)$.
	
	Figure~\ref{fig:tref_disk_13} allows us to see that the braiding induced on the boundary of the bridge trisection is diffeomorphic to the abstract open-book $(P,\bold y, \tau^3)$, where $P$ is a disk, $\bold y$ is two points, and $\tau$ is a right-handed Dehn twist about a curve parallel to $\partial P$ in $P\setminus\nu(\bold y)$.  This proves that this bridge trisection corresponds to a surface bounded by the right-handed trefoil in $(\CP^2)^\circ$.  From the bridge trisection parameters, we conclude that the surface is a disk, since it has Euler characteristic one and is connected.

\begin{figure}[h!]
\centering
\includegraphics[width=.5\linewidth]{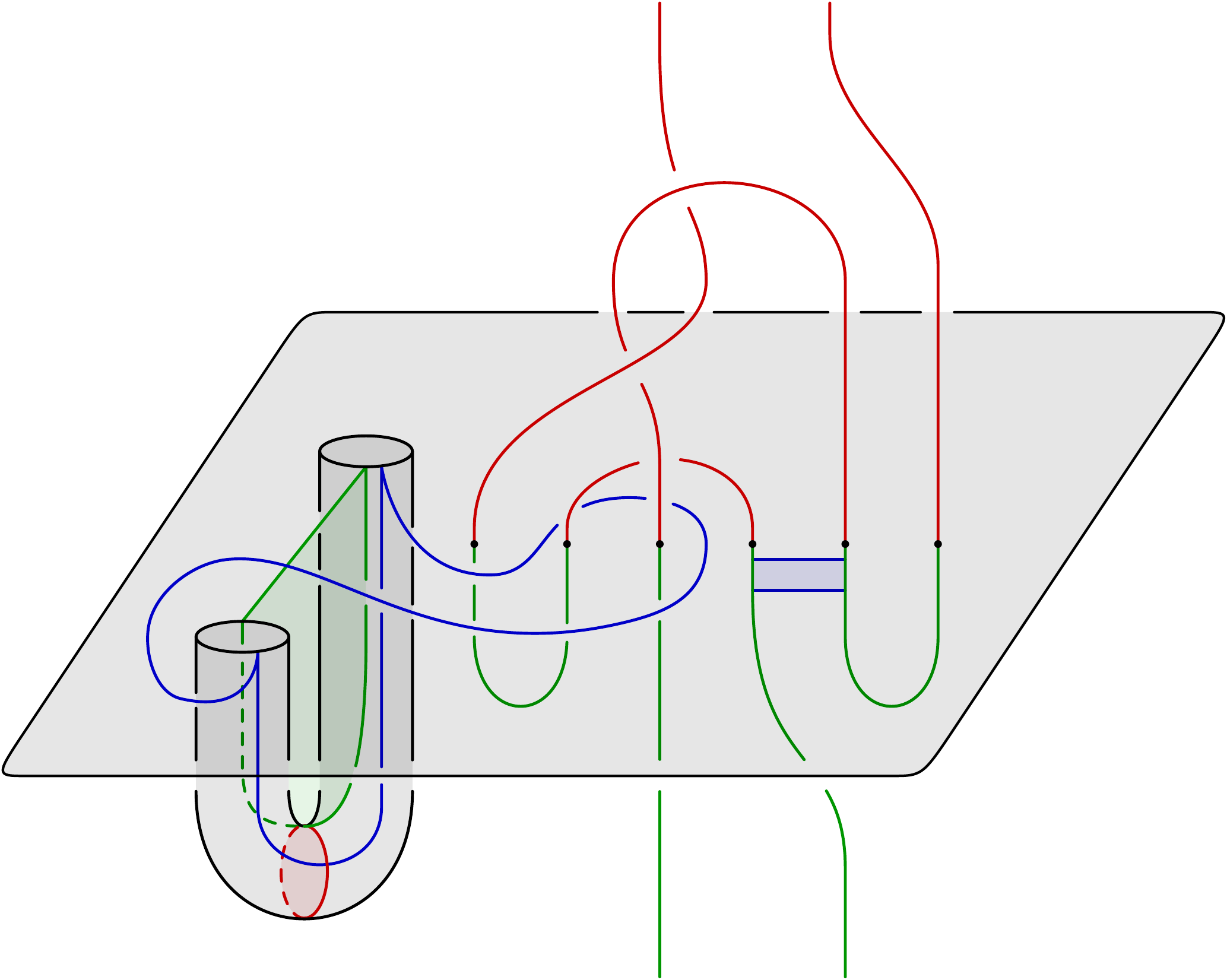}
\caption{A three-dimensional rendering of the shadow diagram in Figure~\ref{fig:tref_disk_1} corresponding to the disk bounded by the right-handed trefoil in $(\CP^2)^\circ$.}
\label{fig:genus_one_ex}
\end{figure}
	
	A three-dimensional rendering for this example is given in Figure~\ref{fig:genus_one_ex}.  The ambient 3--manifold is $S^3 = \partial(CP^2)^\circ$, equipped with the Heegaard-page structure coming from the compressionbody $H_{1,0,1}$.  The right-handed trefoil is in 2--braid position, and perturbed twice with respect to the genus one Heegaard surface $\Sigma$.  The closed curve shown in blue is the belt-sphere for the 2--handle that is attached to a 0--cell $B^4$ to build $(\CP^2)^\circ$.  The curve lies on $\Sigma$ with surface-framing $-1$.  This reflects the fact that $(\CP^2)^\circ$ can be thought of as being built from $\overline{S^3}\times[-1,0]$ by attaching a $(+1)$--framed 2--handle along the corresponding curve in the mirror manifold $\overline{S^3}\times\{-1\}$, before capping off with a 0--handle below.
	A single band is shown for the boundary knot, but this band is a helper-band in the sense of Remarks~\ref{rmk:helpers} and~\ref{rmk:helpers2} and Subsection~\ref{subsec:braiding_presentations} more generally.  In fact, relative to the Morse function on $(\CP^2)^\circ$, the disk bounded by the trefoil can be (and has been) assumed to have no saddle points, just a single minimum.  However, the Morse function on $(\CP^2)^\circ$ coming from the bridge trisection will require the disk to be built from a pair of vertical disks (since we require a 2--braid on the boundary), and the helper-band joins these disks together.  Compare with the Morse-theoretic proof of Theorem~\ref{thm:general}.
\end{example}

\section{Gluing bridge trisected surfaces and shadow diagrams}
\label{sec:gluing}

In this section, we describe how to glue bridge trisected surfaces along portions of their boundary in a way that respects the  bridge trisection structure.  The gluing of trisections was first discussed by Castro~\cite{Cas_17_Trisecting-smooth-4--dimensional}, with further development given by Castro and Ozbagci~\cite{CasOzb_19_Trisections-of-4-manifolds-via-Lefschetz} and by the author and Gay~\cite{GayMei_18_Doubly-pointed}.  We conclude this section with some examples of simple gluings of bridge trisected pairs with disconnected boundary, as well as a more complicated example involving the surfaces bounded by the right-handed trefoil discussed above.  We refer the reader to Section~\ref{sec:shadow} for necessary concepts related to shadow diagrams.

The development below is a generalization of previous developments to the setting of bridge trisections for four-manifold pairs and is complicated by the fact that we allow the four-manifolds being glued to have multiple boundary components and for the gluings to involve proper submanifolds of these boundaries.  To account for this, we will allow our gluing maps to be \emph{partial diffeomorphisms}, which means that they may be defined on proper subsets of their domain.  This subset is called the \emph{domain of definition} of the map; the image of the domain of definition is called the \emph{range}, and may be a proper subset of the codomain.  The domain of definition and range of our partial diffeomorphisms will always be closed submanifolds of the domain and codomain, respectively.

Let $\T$ be a bridge trisection of a pair $(X,\Ff)$, and let $\DD$ be a shadow diagram for $\T$.  Let $(P,\bold y) = \partial_-(H_1,\Tt_1)$, and let $\phi_\DD\colon(P,\bold y)\to(P,\bold y)$ be the monodromy automorphism determined by $\DD$ according to Proposition~\ref{prop:aug}.  Let $\psi_\DD\colon\partial(X,\Ff)\to(Y_{\phi_\DD},\Ll_{\phi_\DD})$ be the diffeomorphism given by Proposition~\ref{prop:monodromy}, where $(Y_{\phi_\DD},\Ll{\phi_\DD})$ is the model pair of the abstract open-book $(P,\bold y,\phi_\DD)$.
We note that both $\phi_\DD$ and $\psi_\DD$ depend on the underlying bridge trisection $\T$, and are determined up to post-composing with a automorphism of $(P,\bold y)$.  Thus, we might as well denote these maps by $\phi_\T$ and $\psi_\T$; we will adopt either decoration, depending on whether we wish to emphasize the shadow diagram or the underlying bridge trisection.

We work in the generality of bridge trisected pairs with disconnected boundary, so we emphasize the decomposition
$$(Y,\Ll) = (Y^1,\Ll^1)\sqcup\cdots\sqcup(Y^n,\Ll^n)$$
of $(Y,\Ll) = \partial(X,\Ff)$ into connected components of $Y$; for any connected component $Y^j$ of $Y$, we may have $\Ll^j$ disconnected -- i.e. a link.  Thus, we have corresponding decomposition of the pairs $(P,\bold y)$, $(P_{\phi_\T},\bold y_{\phi_\T})$, and $(Y_{\phi_\T},\Ll_{\phi_\T})$, and of the maps $\phi_\T$ and $\psi_\T$.

Our first result is that bridge trisections that induce diffeomorphic braidings on some portion of their boundaries can be glued along those boundaries to obtain a new bridge trisection.  By a \emph{diffeomorphism of open-book braidings} we mean a diffeomorphism of three-manifold pairs that restricts to a diffeomorphism of pages (hence, commutes with the monodromies).

\begin{proposition}
\label{prop:glue_tri}
	Let $\T'$ and $\T''$ be bridge trisections for pairs $(X',\Ff')$ and $(X'',\Ff'')$. Suppose we have an orientation reversing partial diffeomorphism of open-book braidings $\Psi\colon\partial(X',\Ff')\to\partial(X'',\Ff'')$.  Then the pair $(X,\Ff) = (X',\Ff')\cup_\Psi(X'',\Ff'')$ inherits a canonical bridge trisection $\T = \T'\cup_\Psi\T''$.
\end{proposition}

\begin{proof}
	Let $(Y',\Ll')$ and $(Y'',\Ll'')$ denote the domain of definition and range of $\Psi$, respectively, noting that these are closed (possibly proper) submanifolds of $\partial(X',\Ff')$ and $\partial(X'',\Ff'')$, respectively.

	After potentially changing $\Psi$ by an isotopy through diffeomorphisms of open-book braidings, we can assume that $\Psi(P_i',\bold y_i') = (P_i'',\bold y_i'')$ for each $i\in\Z_3$.  We will verify that gluing the various corresponding pieces of $\T'$ and $\T''$ together according to $\Psi$ results in a collection of pieces giving a bridge trisection of $(X,\Ff)$.
	
	Consider the restriction of $\Psi$ to the binding $B'$ of the open-book decomposition of $(Y',\Ll')$, recalling that $B' = \partial(\Sigma',\bold x')$ and $B'' =\Psi(B') = \partial(\Sigma'',\bold x'')$.  Let $(\Sigma,\bold x) = (\Sigma',\bold x')\cup_\Psi(\Sigma'',\bold x'')$, which is simply the union of two surfaces with marked points and boundary along closed subsets of their respective boundaries, hence a new surface with marked points and (possibly empty) boundary.
	
	Consider the restriction of $\Psi$ to the pages $P_i'$ for each $i\in\Z_3$, recalling that $(P_i',\bold y_i') = \partial(H_i',\Tt_i')$ and $(P_i'',\bold y_i'') =\Psi(P_i',\bold y_i') = \partial(H_i'',\Tt_i'')$.  Let $(H_i,\Tt_i) = (H_i',\Tt_i')\cup_{\Psi_{(P_i',\bold y_i')}}(H_i'',\Tt_i'')$, noting that
	$$\partial(H_i,\Tt_i) = (\Sigma,\bold x)\cup_B\left((\partial_-(H_i',\Tt_i')\setminus(P_i',\bold y'))\sqcup(\partial_-(H_i'',\Tt_i'')\setminus(P_i'',\bold y''))\right).$$
	(A word of caution regarding notation: The fact that we are considering gluings along potentially strict subsets of the boundaries complicates the exposition notationally.  For example, earlier in the paper, we would have written $(P_i',\bold y_i') = \partial_-(H_i',\Tt_i')$, but here we regard $(P_i',\bold y_i')\subset\partial_-(H_i',\Tt_i')$ as the portion of $\partial_-(H_i',\Tt_i')$ lying in the domain of definition.)
	
	For each $i\in\Z_3$, let $\frak a_i'$ be a neatly embedded collection of arcs in $P_i'\setminus\bold y_i'$ such that surgery along the arcs reduces $P_i'$ to a collection of disks with the number of connected components as $P_i'$.  Moreover, we require that $\frak a_i'$ and $\frak a_{i+1}'$ be isotopic rel-$\partial$ in $Y'\setminus\Ll'$ via an isotopy that is monotonic with respect to the open-book structure. Let $\frak a_i'' = \Psi(\frak a_i')$. For each $i\in\Z_3$, let $\Aa_i$ be an embedded collection of arcs connecting the points of $\bold y_i'$ to $\partial P_i'$, and assume, as before, that $\Aa_i'$ and $\Aa_{i+1}'$ are isotopic via an isotopy that fixes $\Aa_i'\cap\partial P_i'$ and is monotonic with respect to the open-book-braiding structure; the free endpoints of $\Aa_i'$ will move along $\Ll'$.  Let $\Aa_i'' = \Psi(\Aa_i')$.
	
	Using the Morse structure on $(H_i',\Tt_i')$, flow the arcs of $\frak a_i'$ and $\Aa_i'$  down to $\Sigma'$, and denote the results $(\frak a_i^*)'$ and $(\Aa_i^*)'$, respectively. Let $E_i'$ and $T_i'$ denote the traces of the respective isotopies, noting that the $E_i'$ are compression disks for the $H_i'$, and that the $T_i'$ are bridge triangles for the vertical strands $\bold y_i'\times[0,1]$. Do the same for $\frak a_i''$ and $\Aa_i''$ to obtain $(\frak a_i^*)''$ and $(\Aa_i^*)''$ on $\Sigma''$, with corresponding traces $E_i''$ and $T_i''$.
	
	Let $D_i'$ and $D_i''$ be collections of neatly embedded disks in $H_i'$ and $H_i''$, respectively, such that surgery along $D_i'$ and $D_i''$ reduces $H_i'$ and $H_i''$, respectively, to spreads $\partial_-(H_i',\Tt_i')\times[0,1]$ and $\partial_-(H_i'',\Tt_i'')\times[0,1]$.  For each connected component of $(P_i',\bold y')$, pick a disk of $D_i'$ adjacent to that component in the sense that one of the two scars resulting from surgery along the chosen disk lies in the corresponding component of $(P_i',\bold y')\times[0,1]$.  (Equivalently, the chosen disk is the cocore of a 1--handle connecting the component of $(P_i',\bold y')\times[0,1]$ to another component of the spread obtained by surgery.) Let $F_i'\subset D_i'$ denote the chosen disks.  Then, we claim that
	$$D_i = (D_i'\setminus F_i')\sqcup (E_i'\cup_\Psi E_i'')\sqcup D_i''$$
	is a collection of compression disks in $H_i$ such that surgery along $D_i$ reduces $H_i$ to
	$$(\partial_-(H_i')\setminus P_i')\sqcup(\partial_-(H_i'')\setminus P_i'').$$
	
	To see that this is the case, note that the result of surgering $H_i$ along $D_i\sqcup F_i'$ is precisely 
	$$((\partial_-(H_i')\setminus P_i')\times[0,1])
	\sqcup(\sqcup_{m'}D^2\times[0,1])
	\sqcup((\partial_-(H_i'')\setminus P_i'')\times[0,1]),$$
	where $m'$ is the number of connected components of $Y_i'$, $P_i'$, and $F_i'$.  The effect of removing the disks of $F_i'$ from this collection of compression disk is to attach 1--handles, one for each $D^2\times[0,1]$ in the above decomposition, connecting the $m'$ copies of $D^2\times[0,1]$ to the rest of the spread. 
	It follows that $H_i$ is a compressionbody with $\partial_+H_i = \Sigma$ and $\partial_-(H_i) = (\partial_-(H_i')\setminus P_i')\sqcup(\partial_-(H_i'')\setminus P_i'')$, as desired.
	
	Moreover, let $\Delta_i'$ and $\Delta_i''$ be bridge disks for the flat strands of $\Tt_i'$ and $\Tt_i''$, respectively.  Then,
	$$\Delta_i = \Delta_i'\sqcup(T_i'\cup_\Psi T_i'')\sqcup\Delta_i''$$
	is a collection of bridge semi-disks and triangles for the strands of $\Tt_i'\cup_\Psi\Tt_i''$ in $H_i$.  The key thing to note here is that the bridge triangles $T_i'$ for the vertical strands $\bold y_i'\times[0,1]$ glue to the corresponding bridge triangles $T_i''$ for the vertical strands of $\bold y_i''\times[0,1]$ along the identified arcs $\Aa_i'\cup_\Psi\Aa_i''$ to give bridge disks for the new flat strands $(\bold y_i'\times[0,1])\cup_\Psi(\bold y_i''\times[0,1])$.
	
	Finally, consider the restriction of $\Psi$ to the spreads $(Y_i',\beta_i')$ cobounded by $(P_i',\bold y_i')$ and $(P_{i+1}',\bold y_{i+1}')$ in $(Y',\Ll)$, recalling that $(Y_i',\beta_i') = (Z_i',\Dd_i')\cap\partial(X',\Ff')$, and noting that $\Psi(Y_i',\beta_i') = (Y_i'',\beta_i'')$.
	Let $(Z_i,\Dd_i) = (Z_i',\Dd_i')\cup_\Psi(Z_i'',\Dd_i'')$ for each $i\in\Z_3$.  We claim that the fact that the $(Z_i,\Dd_i)$ are trivial disk-tangles follows easily from the detailed argument just given that the $(H_i,\Tt_i)$ are trivial tangles.
	The reason is that a trivial disk-tangle $(Z,\Dd)$ can be naturally viewed as the lensed product $(H,\Tt)\times[0,1]$ such that the decomposition of $\partial(H,\Tt) = (S,\bold x)\cup_{\partial S}(P,\bold y)$ gives rise to a bridge-braid structure on $\partial(Z,\Dd)$.  Precisely, the lensed product $(H_{g,\bold p,\bold f},\Tt_{b,\bold v})\times[0,1]$ is $(Z_{g,k;\bold p,\bold f},\Dd_{c;\bold v})$, where $k = g+p+f-n$ and $n$ is the length of the partition $\bold p$.  The structure on the boundary is that of a symmetric Heegaard double.  Moreover, we have that $\partial_-(Z,\bold D) = \partial_-(H,\Tt)\times[0,1]$, so gluing two trivial disk-tangles along a portion of their negative boundaries is the same as gluing the corresponding trivial tangles (of which the trivial disk-tangles are lensed products) along the corresponding portions of their negative boundaries, then taking the product with the interval.  Succinctly, the gluings along portions of the negative boundaries commute with the taking of the products with the interval.  Therefore, the $(Z_i,\Dd_i)$ are trivial disk-tangles, as desired.
	
	It remains only to verify that $(Z_i,\Dd_i)\cap(Z_{i-1},\Dd_{i-1}) = (H_i,\Tt_i)$ and $(H_i,\Tt_i)\cap(H_{i+1},\Tt_{i+1}) = \Sigma$, but this is immediate.
\end{proof}

\begin{remark}
\label{rmk:self_glue}
	It is interesting to note that Proposition~\ref{prop:glue_tri} holds in the case that $\T'=\T''$ and $\Psi$ is a (partial) \emph{self}-gluing!  See Example~\ref{ex:torus1} below.
\end{remark}

Having established how to glue bridge trisections from the vantage point of bridge trisected pairs, we now turn our attention to understanding gluings diagrammatically.  Suppose that $\T'$ and $\T''$ are bridge trisections of pairs $(X',\Ff')$ and $(X'',\Ff'')$ with augmented shadow diagrams $\DD'$ and $\DD''$, respectively.  Let $f\colon\partial(\Sigma,\frak a_1',(\Aa_1^*)')\to\partial(\Sigma',\frak a_1',(\Aa_1^*)')$ be an orientation-reversing \emph{partial} diffeomorphism.  We call $\DD'$ and $\DD''$ \emph{gluing compatible} if there is an orientation-reversing \emph{partial} diffeomorphism
$$\psi_f(\DD',\DD'')\colon(P_1',\bold y_1')\to(P_1'',\bold y_1'')$$
that extends $f$ and commutes with the monodromies of the diagrams -- i.e., $\psi_f(\DD',\DD'')\circ\phi_{\DD'} = \phi_{\DD''}$ -- where this composition is defined.  In this case, we call $f$ a \emph{compatible (partial) gluing}.

The map $\psi_f(\DD',\DD'')$ determines an orientation-reversing (partial) diffeomorphism
$$\Upsilon_f(\DD',\DD'')\colon (Y_{\phi_{\DD'}},\Ll_{\phi_{\DD'}})\to(Y_{\phi_{\DD''}},\Ll_{\phi_{\DD''}})$$ of abstract open-book braidings.  So, we can define a (partial) gluing map $\Psi_f(\DD',\DD'')\colon\partial(X',\Ff')\to\partial(X'',\Ff'')$ of the bridge trisected pairs by
$$\Psi_f(\DD',\DD'') = \psi_{\DD''}^{-1}\circ\Upsilon_f(\DD',\DD'')\circ\psi_{\DD'}.$$
Again, we are interested in partial boundary-gluings, so we reiterate that the above caveats regarding the domain and codomain apply to $\Psi_f(\DD',\DD'')$.  Given this set-up, we can now describe how gluing shadow diagrams corresponds to gluing bridge trisected four-manifold pairs.

\begin{proposition}
\label{prop:glue_diag}
	Suppose that $\T'$ and $\T''$ are bridge trisections of four-manifold pairs $(X',\Ff')$ and $(X'',\Ff'')$, respectively, and that the corresponding fully-augmented shadow diagrams $\DD'$ and $\DD''$ admit a compatible gluing $f$.  Let $\DD = \DD'\cup_f\DD''$, and let $(X,\Ff) = (X',\Ff')\cup_{\Psi_f(\DD',\DD'')}(X'',\Ff'')$.  Then, $\DD$ is a fully-augmented shadow diagram for the bridge trisection on $(X,\Ff)$ given in Proposition~\ref{prop:glue_tri}, once it is modified in the following two ways:
	\begin{enumerate}
		\item The arcs of $(\frak a_4)'\sqcup(\Aa_4^*)'$ and $(\frak a_4)''\sqcup(\Aa_4^*)''$ whose endpoints lie in the domain of definition and range of $f$ should be deleted.
		\item If $\partial X''$  is disconnected, then, for each component $Y''$ of the range of $\Psi_f(\DD',\DD'')$ there is a subcollection of curves of $\alpha_i''$, for each $i\in\Z_3$, that separate the components of $\partial \Sigma''$ corresponding to $Y''$ from the other components of $\partial\Sigma''$.  Throw out one curve from the subcollection of curves corresponding to each connected component of the range of $\Psi_f(\DD',\DD'')$.
		\item If $\partial X''$  is connected but $\partial X'$ is disconnected, then, for each component $Y'$ of the domain of definition  of $\Psi_f(\DD',\DD'')$ there is a subcollection of curves of $\alpha_i'$, for each $i\in\Z_3$, that separate the components of $\partial \Sigma'$ corresponding to $Y'$ from the other components of $\partial\Sigma'$. Throw out one curve from the subcollection of curves corresponding to each connected component of the domain of definition of $\Psi_f(\DD',\DD'')$.
	\end{enumerate}
\end{proposition}

The first modification required above is a minor issue.  If this is not done, then the would-be-deleted arcs give rise to extra shadows and curves that are redundant in the encoding of the trivial tangle $(H_1,\Tt_1)$.  The next two modifications are more serious, and are required to ensure that the resulting diagram is a shadow diagram. The rationale was made clear in the proof of Proposition~\ref{prop:glue_tri}, where this precise discarding was carried out at the level of compression disks.  Note that only one of the final two modification will need to be made in practice.

\begin{proof}
	The proof of the proposition follows from the proof of Proposition~\ref{prop:glue_tri}, as applied to the gluing $\Psi_f(\DD',\DD'')$.
\end{proof}

We conclude this section with some examples illustrating gluings of bridge trisected four-manifold pairs.

\begin{example}
\label{ex:sphere}
	First, we recall the bridge trisected surfaces bounded by the right-handed trefoil discussed in Examples~\ref{ex:Mob_sh} and~\ref{ex:tref_disk}.  Let $\DD'$ denote the fully-augmented shadow diagram in Figure~\ref{fig:sphere1}, which corresponds to a bridge trisection of the pair $(X',\Ff')$, where $\Ff'$ is a disk bounded by the right-handed trefoil in $X'=(\CP^2)^\circ$. Let $\DD''$ denote the fully-augmented shadow diagram in Figure~\ref{fig:sphere2}, which corresponds to the pair $(X'',\Ff'')$, where  $\Ff''$ is the M\"obius band bounded by the left-handed trefoil in $S^3$, which we imagine as being perturbed so that its interior lies in $X''=B^4$. Note that $\DD''$ is the mirror of the diagram show in Figure~\ref{fig:Mob_sh_3}. Orientations for the boundaries of the diagrams are shown.

\begin{figure}[h!]
\begin{subfigure}{.5\textwidth}
	\centering
	\includegraphics[width=.8\linewidth]{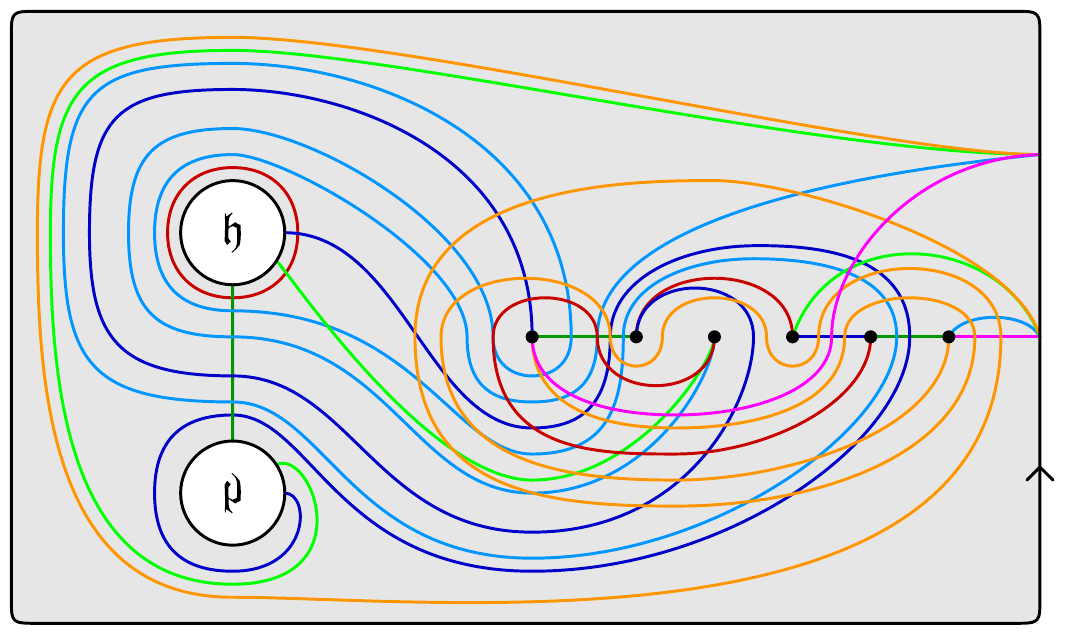}
	\caption{}
	\label{fig:sphere1}
\end{subfigure}%
\begin{subfigure}{.4\textwidth}
	\centering
	\includegraphics[width=.7\linewidth]{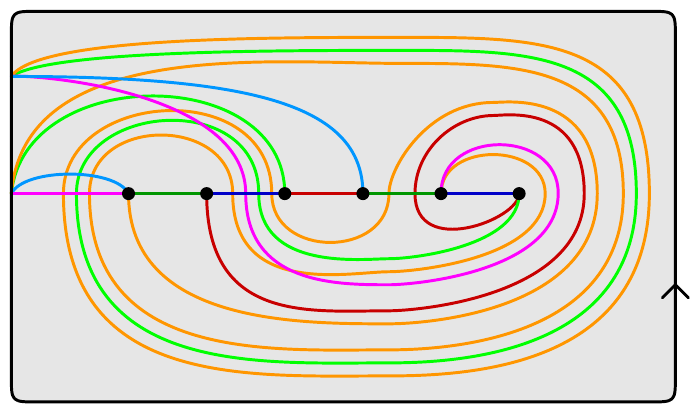}
	\caption{}
	\label{fig:sphere2}
\end{subfigure}
\par\vspace{5mm}
\begin{subfigure}{.7\textwidth}
	\centering
	\includegraphics[width=.9\linewidth]{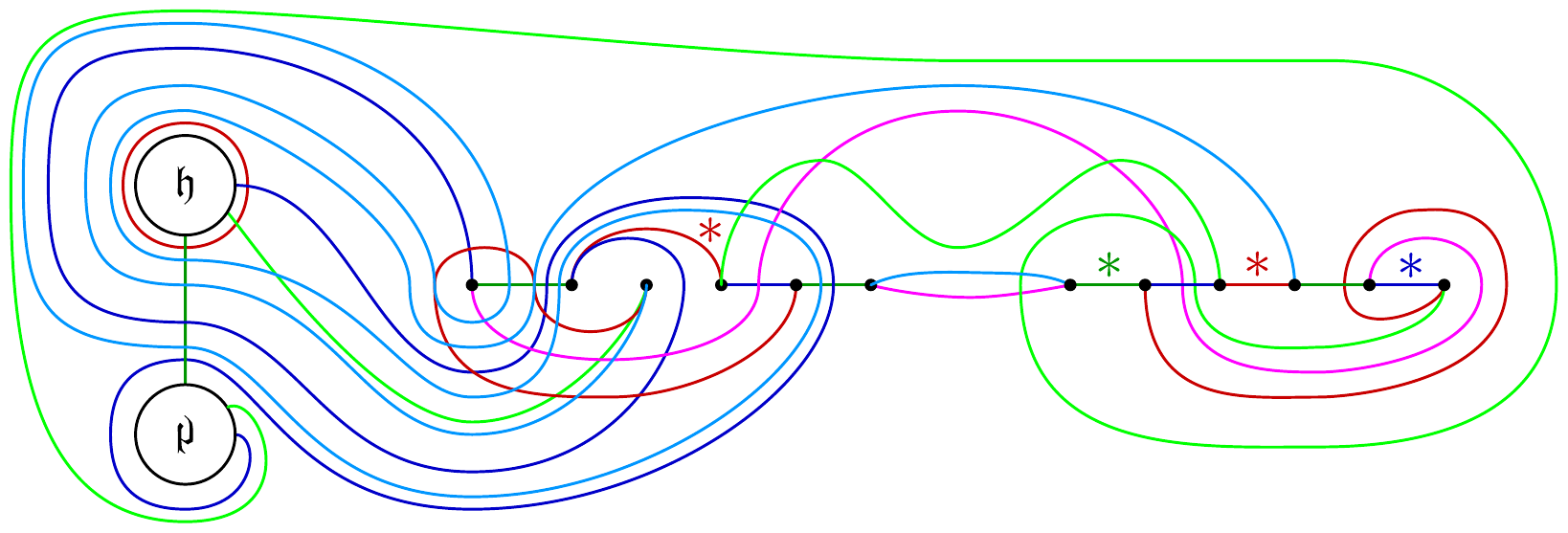}
	\caption{}
	\label{fig:sphere3}
\end{subfigure}%
\begin{subfigure}{.3\textwidth}
	\centering
	\includegraphics[width=.9\linewidth]{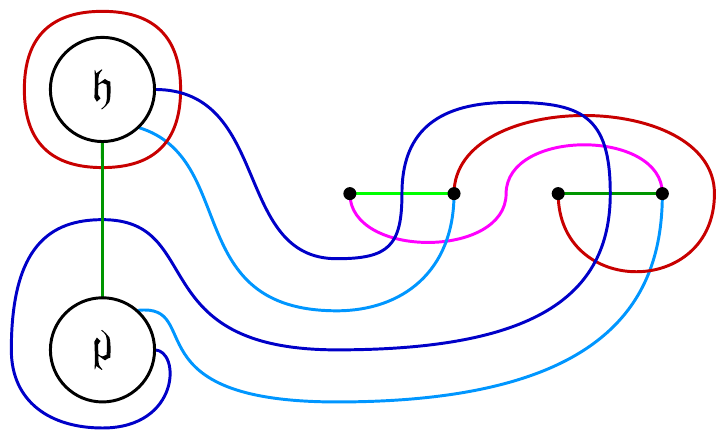}
	\caption{}
	\label{fig:sphere4}
\end{subfigure}
\caption{(A) A shadow diagram for the disk bounded by the right-handed trefoil in $(\CP^2)^\circ$. (B) A shadow diagram for the M\"obius band bounded by the right-handed trefoil in $B^4$. (C) The result of gluing these diagrams via the unique compatible gluing: a shadow diagram for a projective plane in $\CP^2$. (D) is obtained from (C) by deperturbing along the indicated shadows.}
\label{fig:sphere}
\end{figure}

	These bridge trisections induce open-book braidings on the boundaries of their corresponding manifold pairs that are orientation-reversing diffeomorphic. Both open-book braidings have disk page and boundary link in 2--braid position: For $\DD'$, the monodromy is three \emph{positive} half-twists about the two braid points.  This was described in Example~\ref{ex:tref_disk} and Figure~\ref{fig:tref_disk}.  However, for $\DD''$, the half-twists are \emph{negative}, since $\DD''$ is the mirror of the diagram discussed in Example~\ref{ex:Mob_sh} and Figure~\ref{fig:Mob_sh}. 
	
	Let $f\colon\partial\DD'\to\partial\DD''$ be the orientation-reversing diffeomorphism that matches the endpoints of the arcs $(\Aa^*_1)'$ with those of $(\Aa^*_1)''$.  There is an orientation-reversing diffeomorphism $\psi_f(\DD',\DD'')\colon (P_1',\bold y_1')\to(P_2'',\bold y_2'')$ that extends $f$; simply pick the obvious diffeomorphism relating pair in Figure~\ref{fig:tref_disk_4} to the mirror of the pair in Figure~\ref{fig:Mob_sh_4}.  It follows that is a compatible gluing corresponding to an orientation-reversing diffeomorphism $\Psi_f(\DD',\DD'')$.
	
	Let $(X,\Ff) = (X',\Ff')\cup_{\Psi_f(\DD',\DD'')}(X'',\Ff'')$.  By Proposition~\ref{prop:glue_diag}, the $\DD=\DD'\cup_f\DD''$ shown in Figure~\ref{fig:sphere3} is a shadow diagram for $(X,\Ff)$.
	Observe how the arcs $(\Aa_4^*)'$ and $(\Aa_4^*)''$ have been discarded according with the first modification required by Proposition~\ref{prop:glue_diag}.  (The second and third modification are not necessary in this example, since $\partial X'$ and $\partial X''$ are connected.)
	A brief examination reveals that this diagram can be deperturbed three times, using the indicated shadows. (See Section~\ref{sec:stabilize} for details about perturbation.)  Doing so produces the diagram of Figure~\ref{fig:sphere4}.
	
	We have that $X\cong\CP^2$ and $\Ff\cong\RP^2$, but it is not true that $(X,\Ff)\cong(\CP^2,\RP^2)$, where the latter pair is the projectivization of the standard pair $(\C^3,\R^3)$. The standard projective pair $(\CP^2,\RP^2)$ is depicted in Figure~2 of~\cite{MeiZup_18_Bridge-trisections-of-knotted}. One way to distinguish these two pairs is to note that  $\Ff$ has normal Euler number +6, while $\RP^2$ has normal Euler number +2.  Moreover,  $\pi_1(X\setminus\nu(\Ff))\cong\Z/2\Z$, while $\pi_1(\CP^2\setminus\nu(\RP^2))\cong 1$.  These facts are left as exercises to the reader.
\end{example}

\begin{example}
\label{ex:annulus}
	Consider the shadow diagram  $\DD'$ shown in Figure~\ref{fig:annulus1}, which corresponds to a bridge trisection of the cylinder pair $(X',\Ff')=(S^3\times I, S^1\times I)$.  The underlying trisection of $S^3\times I$ can be thought of as follows. If one ``trisects" $S^3$ into three three-balls, which meet pairwise along disk pages of the open-book decomposition with unknotted boundary -- so the triple intersection of the three-balls is this binding, then the trisection of $S^3\times I$ can be thought of as the product of this ``trisection" of $S^3$ with the interval, and the core $\Sigma$ is simply the product of the binding with the interval. So, the diagram $\DD'$ can be thought of as a bridge trisection for a copy $\Ff$ of $\Sigma$.  To carry this out, the copy $\Ff$ of the annular core must be perturbed relative the original copy $\Sigma$ of the core.  We leave it as an exercise to the reader to verify that $\DD'$ describes the cylinder pair, as claimed.

\begin{figure}[h!]
\begin{subfigure}{.3\textwidth}
	\centering
	\includegraphics[width=.8\linewidth]{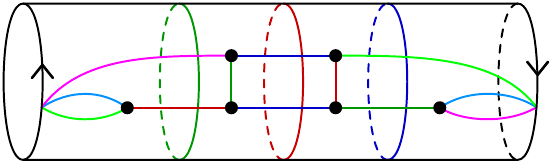}
	\caption{}
	\label{fig:annulus1}
\end{subfigure}%
\begin{subfigure}{.6\textwidth}
	\centering
	\includegraphics[width=.9\linewidth]{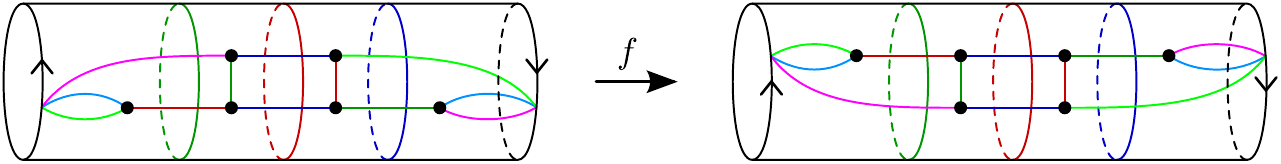}
	\caption{}
	\label{fig:annulus2}
\end{subfigure}
\par\vspace{5mm}
\begin{subfigure}{.6\textwidth}
	\centering
	\includegraphics[width=.8\linewidth]{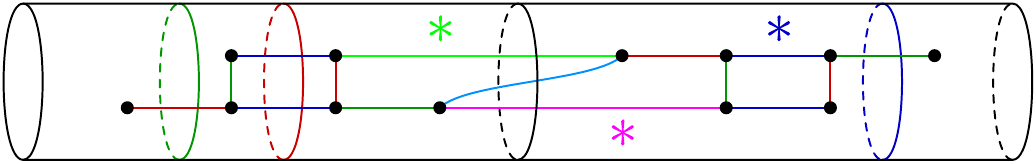}
	\caption{}
	\label{fig:annulus3}
\end{subfigure}%
\begin{subfigure}{.3\textwidth}
	\centering
	\includegraphics[width=.8\linewidth]{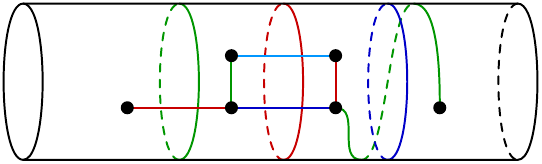}
	\caption{}
	\label{fig:annulus4}
\end{subfigure}
\caption{(A) A shadow diagram for $S^3\times I$. (B) A copy of this diagram and a copy of its mirror, with compatible gluing $f$ indicated. (C) The result of the gluing, $S^3\times I$. (D) is obtained from (C) by deperturbing along the indicated shadows.}
\label{fig:annulus}
\end{figure}
	
	Now, let $\DD''$ denote a mirror copy of $\DD'$ that corresponds to a second copy of cylinder pair: $(X'',\Ff'')=(S^3\times I,S^1\times I)$.  Each of the two boundary components of both $(X',\Ff')$ and $(X'',\Ff'')$ have induced open-book braidings with page a disk with one braid point.  Let $f\colon\partial\DD'\to\partial\DD''$ be the orientation-reversing partial diffeomorphism shown in Figure~\ref{fig:annulus2} -- i.e., $f$ maps the boundary component $S^1\times\{1\}$ fo $\DD'$ to the boundary component $S^1\times\{0\}$ of $\DD''$.	Trivially, $f$ extends to an orientation-reversing partial diffeomorphism $\psi_f(\DD',\DD'')\colon (P_1',\bold y_1')\to(P_2'',\bold y_2'')$ between the page pairs corresponding to the boundary components of the chosen boundary components of $\DD'$ and $\DD''$.  Thus, we have an orientation-reversing partial diffeomorphism $\Psi_f(\DD',\DD'')\colon\partial(X',\Ff')\to\partial(X'',\Ff'')$.
	
	Let $(X,\Ff) = (X',\Ff')\cup_{\Psi_f(\DD',\DD'')}(X'',\Ff'')$.  By Proposition~\ref{prop:glue_diag}, the diagram $\DD = \DD'\cup_f\DD''$ shown in Figure~\ref{fig:annulus3} is a shadow diagram for $(X,\Ff)$. Note that one curve of each color has been discarded in accordance with modification (2). As before, the diagram obtained from gluing can be deperturbed. (This is a common phenomenon when gluing shadow diagrams.)  The diagram obtained after deperturbing, shown in Figure~\ref{fig:annulus4}, is diffeomorphic to the original diagram $\DD'$.  Of course, $(X,\Ff)\cong(S^3\times I,S^1\times I)$.
	
	In this example, modification (1) of Proposition~\ref{prop:glue_diag} is implicit; the arcs $\frak a_4'$, $(\Aa_4^*)'$, $\frak a_4''$, and $(\Aa_4^*)''$ were never drawn and were never needed.  More interestingly, we see how modification (2) is required.  The curves of $\DD''$ have been discarded upon gluing.  Had this not been done, there would have been parallel curves in $\alpha_i$ for each $i\in\Z_3$.  This would imply that $P_i = \partial_-H_i$ would have a two-sphere component, which is not allowed.
\end{example}

\begin{example}
\label{ex:torus1}
	Finally, we consider two more compatible gluings involving $\DD'$.  First, let $\DD''$ denote a mirror copy of $\DD'$, and let $f\colon\partial\DD'\to\partial\DD''$ be the compatible gluing shown in Figure~\ref{fig:torus1}.  This compatible gluing is similar to the one explored in Example~\ref{ex:annulus}, but this time $f$ is not a partial diffeomorphism.  The induced gluing $\Psi_f(\DD',\DD'')$ matches the two boundary components of $(X,\Ff)$ with the corresponding components of $(X'',\Ff'')$. As a result,  $(X,\Ff) = (X',\Ff')\cup_{\Psi_f(\DD',\DD'')}(X'',\Ff'')$ is the closed four-manifold pair $(S^3\times S^1,S^1\times S^1)$, and the diagram $\DD = \DD'\cup_f\DD''$ for this pair is shown in Figure~\ref{fig:torus2}.  As in Example~\ref{ex:annulus}, the redundant arcs have been suppressed, and the curves $\alpha_i''$ have been discarded upon gluing.  Also, we can again deperturb, arriving at the diagram of Figure~\ref{fig:torus4}.

\begin{figure}[h!]
\begin{subfigure}{.3\textwidth}
	\centering
	\includegraphics[width=.8\linewidth]{annulus1}
	\caption{}
	\label{fig:torus0}
\end{subfigure}%
\begin{subfigure}{.3\textwidth}
	\centering
	\includegraphics[width=.7\linewidth]{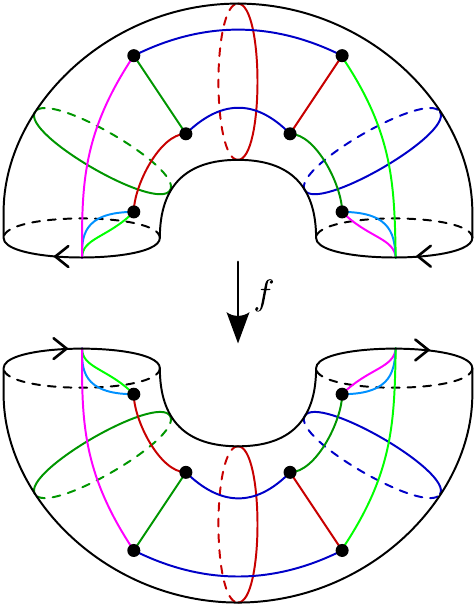}
	\caption{}
	\label{fig:torus1}
\end{subfigure}%
\begin{subfigure}{.3\textwidth}
	\centering
	\includegraphics[width=.7\linewidth]{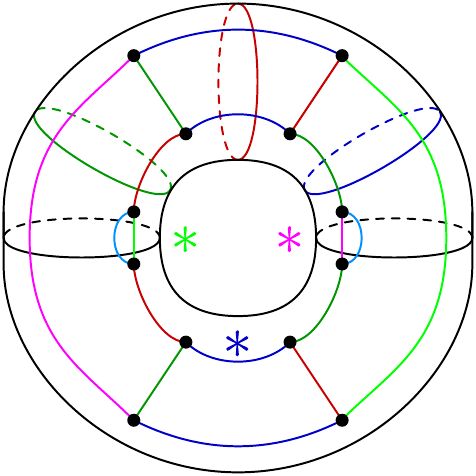}
	\caption{}
	\label{fig:torus2}
\end{subfigure}
\par\vspace{5mm}
\begin{subfigure}{.3\textwidth}
	\centering
	\includegraphics[width=.7\linewidth]{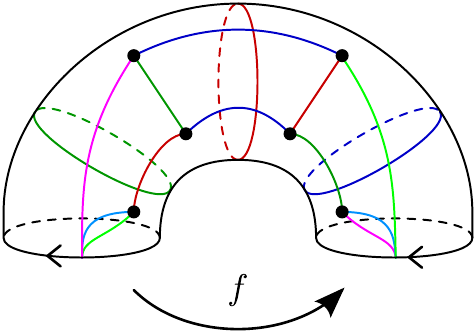}
	\caption{}
	\label{fig:torus3}
\end{subfigure}
\begin{subfigure}{.3\textwidth}
	\centering
	\includegraphics[width=.7\linewidth]{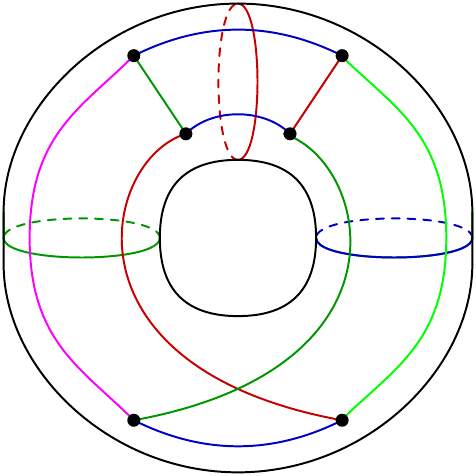}
	\caption{}
	\label{fig:torus4}
\end{subfigure}
\caption{(A) A shadow diagram for $S^3\times I$. (B) A copy of this diagram and a copy of its mirror, with compatible gluing $f$ indicated. (C) The result of the gluing, $S^3\times S^1$. (D) A compatible self-gluing of the diagram. (E) The result of the self gluing, $S^3\times S^1$.  (E) is obtained from (C) by deperturbing along the indicated shadows.}
\label{fig:torus}
\end{figure}

	Now, let $f$ denote the compatible self-gluing shown in Figure~\ref{fig:torus3}.
	The induced self-map of $(S^3\times I, S^1\times I)$ is   $\Psi_f(\DD')\colon(S^3\times\{0\},S^1\times\{0\})\to(S^3\times\{1\},S^1\times\{1\})$.
	The diagram resulting from the compatible self-gluing $f$ is the diagram of Figure~\ref{fig:torus4}, which describes $(S^3\times S^1,S^1\times S^1)$, as noted before.
\end{example}

\section{Classification and Examples}
\label{sec:class}

In this section, we classify $(b,\bold c;v)$--bridge trisections in the trivial cases where one or more of the parameters is sufficiently small.  Then, we present families of examples representing more interesting choices of parameters and pose questions about further possible classification results.  To get started, we discuss the connected sum and boundary connected sum operations, then we introduce some notions of reducibility for bridge trisections.

\subsection{Connected sum of bridge trisections}
\label{subsec:sum}
\ 

Given trisections $\T'$ and $\T''$ for four-manifolds $X'$ and $X''$, it is straight-forward to see that there is a trisection $\T=\T'\#\T''$ describing  $X'\#X''$.  Let $\varepsilon\in\{',''\}$.  All that needs to be done is to choose the points $x^\varepsilon\in X^\varepsilon$ that determine the connected sum to lie on the respective cores.  Having done so, the pieces of the trisection $\T$ can be described as by: $\Sigma = \Sigma'\#\Sigma''$, $H_i = H_i'\natural H_i''$, and $Z_i = Z_i'\natural Z_i''$.  Note that $\T$ is independent of the choice of points made above.

\begin{remark}
\label{rmk:sum_glue}
	Note that the connected sum operation, as described, is a very simple example of a gluing of trisections, as described in detail in Section~\ref{sec:gluing}. Each of $\T^\varepsilon\setminus\nu(x^\varepsilon)$ is automatically a trisection with one new boundary component diffeomorphic to $S^3$.
	If $\DD^\varepsilon$ is a shadow diagram for $\T^\varepsilon$, then $\DD^\varepsilon\setminus\nu(x^\varepsilon)$ is a diagram for $\T^\varepsilon\setminus\nu(x^\varepsilon)$ after a simple modification is made in the case that $\partial X\not=\emptyset$:
	In this case, a curve $\delta$ must be added to each of the $\alpha_i$ that is parallel to the curve $\partial\nu(x^\varepsilon)$ where $\Sigma'$ and $\Sigma''$ were glued together.  (This is a separating reducing curve in the sense of Definition~\ref{def:reduce}, below.)
\end{remark}

There is a complication in extending this interpretation to connected sum of bridge trisections with boundary that was not present in discussions of the connected sum of \emph{closed} bridge trisections elsewhere in the literature.
The na\"ive idea is to simple choose the connected sum points $x^\varepsilon$ to be bridge points.  This works for closed bridge trisections, because every bridge point is incident to a flat strand in each of the three trivial tangles.
This is not the case for bridge trisections with boundary.  To convince oneself of the problem, try to form the connect sum of two bridge trisections, each of which is a copy of the bridge trisection described in Figure~\ref{fig:b=11}, which corresponds to the standard positive M\"obius band.  It is simply not possible: The removal of an open neighborhood around any bridge point has the effect that one of the trivial tangles will no longer be trivial, since it will have a strand with no endpoints on $\Sigma$.

One might think that perturbing the bridge trisection (see Subsection~\ref{subsec:interior_perturbation} below) would fix the problem by creating a bridge point that is incident to flat strands in each arm, however, the problem persists due to consideration of the vertical patches.  Since vertical patches are only allowed to be incident to one component of $\partial X$, we cannot puncture our bridge trisection at a bridge point that is incident to a vertical patch.

The next lemma makes precise when puncturing a bridge trisection at a bridge point produces a new bridge trisection and indicates how to form the connected sum of bridge trisections.

\begin{lemma}
\label{lem:punc}
	Let $\T$ be a bridge trisection for a pair $(X,\Ff)$, and let $x$ be a bridge point.  Then, $\T\setminus\nu(x)$ is a bridge trisection for the pair $(X\setminus\nu(x),\Ff\setminus\nu(x))$ if and only if $x$ is incident to a flat patch of $\Dd_i$ for each~$i\in\Z_3$.
	
	If $\DD = (\Sigma,\alpha_1,\alpha_2,\alpha_3,\Tt_1^*,\Tt_2^*,\Tt_3^*,\bold x)$ is a shadow diagram for $\T$, then a shadow diagram for $\T\setminus\nu(x)$ can be obtained as follows: Let $\delta = \partial\nu(x)$ in $\DD$.  Let $\delta_i$ denote the result of sliding $\delta$ off the arc $\tau_i^*$ of $\Tt_i^*$ that is incident to $x$.  Let $\Sigma' = \Sigma\setminus\nu(x)$, $\alpha_i' = \alpha_i\cup\delta_i$, $(\Tt_i^*)' = \Tt_i^*\setminus\tau_i^*$, and $\bold x' = \bold x\setminus\{x\}$. Then, there are two cases:
	If $\partial X=\emptyset$, then
	$$\DD = (\Sigma',\alpha_1,\alpha_2,\alpha_3,(\Tt_1^*)',(\Tt_2^*)',(\Tt_3^*)',\bold x'),$$
	is a shadow diagram for $T\setminus\nu(x)$.
	If $\partial X\not=\emptyset$, then
	$$\DD = (\Sigma',\alpha_1',\alpha_2',\alpha_3',(\Tt_1^*)',(\Tt_2^*)',(\Tt_3^*)',\bold x'),$$
	is a shadow diagram for $T\setminus\nu(x)$.
\end{lemma}

\begin{proof}
	If $x$ is incident to a flat patch of $\Dd_i$ for each $i\in\Z_3$, then it is straight-forward to verify that the pieces of $\T\setminus\nu(x)$ form a bridge trisection.
	The main substantive changes are that (1) the number of components of $\partial X$, $\partial\Sigma$, and $\partial_-H_i$ all increase by one; and (2) for each $i\in\Z_3$, the flat strand of $\Tt_i$ becomes a vertical strand and the flat patch of $\Dd_i$ incident to $x$ becomes a vertical patch.
	Conversely, if $x$ is incident to a vertical patch $D\subset\Dd_i$ for some $i\in\Z_3$, then $\Dd_i\setminus\nu(x)$ is no longer a trivial disk-tangle, since $D\setminus\nu(x)$ is neither vertical nor flat, as it intersects multiple components of $\partial X$.
	
	If $\partial X=\emptyset$, then the $H_i$ are handlebodies and the $H_i'$ are compressionbodies with $\partial_-H_i'\cong D^2$.  In this case, the curves $\alpha_i$ still encode $H_i'$ without modification.
	If $\partial X\not=\emptyset$, then $\partial_-H_i' \cong \partial_-H_i\sqcup D^2$.  In this case, $\delta$ must be added to $\alpha_i$ in order to encode the fact that the new component of $\partial_-H_i'$ is disjoint from the original ones. As curves in a defining set, $\delta$ and $\delta_i$ serve the same role, since they are isotopic.  The only reason for pushing $\delta$ off $\tau_i^*$ is to satisfy our convention that the shadow arcs be disjoint from the defining set of curves for the handlebody.
	The shadow arcs $\tau_i^*$ are deleted regardless of whether $\partial X$ is empty, since these shadows correspond to flat strands that become vertical strands upon removal of $\nu(x)$. 
\end{proof}

\begin{example}
\label{ex:conn_sum}
	Consider the shadow diagram $\DD$ shown in Figure~\ref{fig:conn_sum1	}, which corresponds to a bridge trisection of the trivial disk in the four-ball.  Figure~\ref{fig:conn_sum2} shows the diagram corresponding to the bridge trisection $\T' = \T\setminus\nu(x)$ for $(X',\Ff') = (B^4\setminus\nu(x),D^2\setminus\nu(x))$.  Note that this diagram is equivalent to that of Figures~\ref{fig:annulus1} and~\ref{fig:torus0}.
\end{example}

\begin{figure}[h!]
\begin{subfigure}{.5\textwidth}
	\centering
	\includegraphics[width=.8\linewidth]{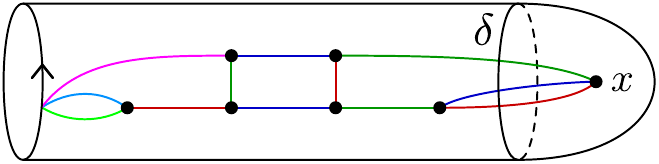}
	\caption{}
	\label{fig:conn_sum1	}
\end{subfigure}%
\begin{subfigure}{.5\textwidth}
	\centering
	\includegraphics[width=.67\linewidth]{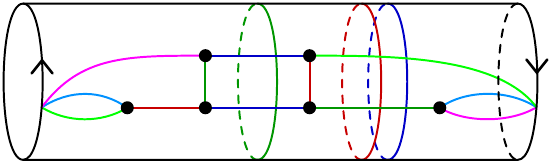}
	\caption{}
	\label{fig:conn_sum2}
\end{subfigure}
\caption{(A) A shadow diagram for a bridge trisection of $(B^4,D^2)$. (B) The diagram obtained by puncturing at the bridge point $x$.}
\label{fig:conn_sum}
\end{figure}

In light of this lemma, it is clear that we can obtain a bridge trisection for the connected sum of surfaces by choosing the connected sum points to be bridge points incident only to flat patches.  Though such bridge points need not always exist (see the M\"obius band example reference above), they can be created via interior perturbation -- at most one in each direction.  The punctured trisections $\T^\varepsilon\setminus\nu(x^\varepsilon)$ can be canonically glued along the novel three-sphere-unknot boundary-component, according to the techniques of Section~\ref{sec:gluing}.  Note that in the case that at least one of $X'$ and $X''$ have boundary, then at least one of the curves $\delta_i'$ or the curves $\delta_i''$ should be discarded upon gluing, as dictated by Propositions~\ref{prop:glue_tri} and~\ref{prop:glue_diag}.  Compare Example~\ref{ex:conn_sum} to Example~\ref{ex:annulus}.

So far, we have viewed the connected sum of bridge trisections as a special case of gluing bridge trisections, and it has been noted that, for this approach to work, we must form the connected sum at bridge points that are incident to flat patches in each disk-tangle.  However, it is possible to work in a slightly more general way so that the punctured objects need not be bridge trisections themselves, but their union will be a bridge trisection of the connected sum.

\begin{lemma}
\label{lem:conn_sum}
	Let $\T'$ and $\T''$ be bridge trisections for pairs $(X',\Ff')$ and $(X'',\Ff'')$, respectively, and let $x'$ and $x''$ be bridge points such that, for each $i\in\Z_3$, one of $x'$ or $x''$ is incident to a flat patch in $\T^\varepsilon$.  Then, the result
	$$\T = (\T'\setminus\nu(x'))\cup(\T''\setminus\nu(x''))$$
	obtained by removing open neighborhoods of the $x^\varepsilon$ from the $\T^\varepsilon$ and gluing along resulting boundaries so that the corresponding trisection pieces are matched is a bridge trisection for $(X,\Ff) = (X',\Ff')\#(X'',\Ff'')$.
\end{lemma}

\begin{proof}
	Let $D_i^\varepsilon$ be the patch of $\Dd_i^\varepsilon$ containing $x^\varepsilon$ for each $i\in\Z_3$ and each $\varepsilon\in\{',''\}$.  Let $D_i = D_i'\cup_{\partial\nu(x^\varepsilon)}D_i''$.
	Then
	$$\Dd_i = \Dd_i'\cup_{\partial\nu(y^\varepsilon)}\Dd_i'' = (\Dd_i'\setminus D_i')\sqcup(\Dd_i''\setminus D_i'')\sqcup D_i.$$
	 For each $i\in\Z_3$,  one of the $D_i^\varepsilon$ will be flat, so $D_i$ will be flat or vertical, according to whether the other of the $D_i^\varepsilon$ is flat or vertical. In any event, each disk of  $\Dd_i$ has at most one critical point, we have a trivial disk-tangle, since the boundary sum of trivial disk-tangles is a trivial disk-tangle.
	 
	 A similar argument shows that the arms of $\T$ are just the boundary sum of the arms of the $\T^\varepsilon$ and that each strand is vertical or flat, as desired.  The details are straight-forward to check.
\end{proof}

Note that while the parameters $g$ and $\bold k$ are additive under connected sum, the parameters $b$ and $\bold c$ are $(-1)$--subadditive (eg. $b = b'+b''-1$).  In the case that the $(X^\varepsilon,\Ff^\varepsilon)$ have non-empty boundary, the boundary parameters $\bold p$, $\bold f$, $\bold v$, and $n$ are all additive, since we are discussing connected sum at an \emph{interior point} of the pairs.  Unlike the case of the connected sum of two four-manifold trisections, here, the resulting bridge trisection is highly dependent on the choice of bridge points made above.

\subsection{Boundary connected sum of bridge trisections}
\label{subsec:bound_sum}
\ 

Now consider the operation of boundary connected sum of four-manifolds.  We start with the set-up as above, but now we choose the summation points to be points $y^\varepsilon$ lying in components $K^\varepsilon$ of the bindings $\partial \Sigma^\varepsilon$ for each $\varepsilon\in\{',''\}$.  In this case, the pieces of the trisection $\T = \T'\natural\T''$ can be described as follows:  $\Sigma = \Sigma'\natural\Sigma''$, $H_i = H_i'\natural H_i''$, $Z_i = Z_i'\natural Z_i''$, $B = B'\#B''$, $P_i = P_i'\natural P_i''$, and $Y_i = Y_i'\natural Y_i''$.  And in this case, $g$, $\bold k$, and $\bold p$ are additive, while $\bold f$ and $n$ are $(-1)$--subadditive, and $\T$ is highly dependent on the choice of binding component $K^\varepsilon$ made above.

The situation becomes more complicated when we consider boundary connected sum of bridge trisected pairs.  The issue here is that $\Ff^\varepsilon\cap\partial\Sigma^\varepsilon =\emptyset$, so we cannot choose the $y^\varepsilon$ to lie simultaneously on $\Sigma^\varepsilon$ and on $\Ff^\varepsilon$.  Our approach is to first perform the boundary connected sum of the ambient four-manifolds, as just described, then consider the induced bridge trisection of the split union $(X,\Ff'\sqcup\Ff'')$ of surface links.  We now describe a modification of this bridge trisection that will produce a bridge trisection of $(X,\Ff'\natural\Ff'')$.  

Suppose that we would like to form the boundary connected sum of $(X',\Ff')$ with $(X'',\Ff'')$ at points $y^\varepsilon\in\partial\Ff^\varepsilon$.  Without loss of generality, we can assume that $y^\varepsilon\in \Ff^\varepsilon\cap P_i^\varepsilon$; in relation to the open-book structure on (the chosen component of) $\partial X^\varepsilon$, we assume that $y^\varepsilon$ lies on the page $P_i^\varepsilon$.  Henceforth, our model is dependent on the choice of $i\in\Z_3$.

Choose arcs $\omega^\varepsilon$ connecting the points $y^\varepsilon$ to the chosen binding components $K^\varepsilon\subset B^\varepsilon$.  Let $z^\varepsilon$ denote the points of $\omega^\varepsilon\cap K^\varepsilon$.  Form the boundary connected sum of the ambient four-manifolds at the points $z^\varepsilon$, as described above, so that $\Ff'\sqcup\Ff''$ is in bridge position with respect to $\T$.  Note that the arcs $\omega^\varepsilon$ give rise to an arc $\omega$ in the page of $P_i$ connecting the points $y^\varepsilon$.

Use the height function on $H_i$ to flow $\omega$ down to the core $\Sigma$.  Let $Q$ represent the square traced out by this isotopy, and let $\omega_* = Q\cap\Sigma$.  Let $N$ be a regular neighborhood of $Q$ in $X$.  We will change $\Ff'\sqcup\Ff''$ to $\Ff'\natural\Ff''$ in a way that will produce a bridge trisection for the latter from the bridge trisection of the former, and this change will be supported inside $N$.  See Figure~\ref{fig:bcs_schem1} for a (faithful) schematic of this set-up.  The figures depict the case of $i=1$.

\begin{proposition}
\label{prop:bcs}
	A bridge trisection for $(X,\Ff) = (X'\natural X'',\Ff'\natural\Ff'')$ can be obtained from the bridge trisection of $(X,\Ff'\sqcup\Ff'')$ described above by replacing the local neighborhood $N$ of $Q$ shown in Figure~\ref{fig:bcs_schem1} with the local neighborhood $N'$ shown in Figure~\ref{fig:bcs_schem2}.  The replacement can be seen in a shadow diagram as the local replacement of the portion of the diagram supported near $\omega_*$ shown in Figure~\ref{fig:bcs_schem3} with the portion shown in Figure~\ref{fig:bcs_schem4}.
\end{proposition}

\begin{proof}
	Near $\omega_*$, the neighborhood $N$ is precisely the $(0;0,2)$--bridge trisection of two copies of the trivial disk in $B^4$.  To recover all of $N$, we extend upward along $Q$.  Because $\omega$ was lowered to $\Sigma$ along a pair of vertical strands of $(H_i,\Tt_i'\sqcup\Tt_i'')$, we see that the entirety of $N$ is still just the 2--bridge trisection of two copies of the trivial disk.  In other words, $N$ is isolating, in a bridge-trisected way, a small disk from each of the $\Ff^\varepsilon$. 

\begin{figure}[h!]
\begin{subfigure}{.5\textwidth}
  \centering
  \includegraphics[width=.8\linewidth]{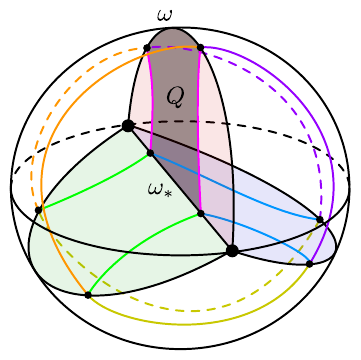}
  \caption{The neighborhood $N$.}
  \label{fig:bcs_schem1}
\end{subfigure}%
\begin{subfigure}{.5\textwidth}
  \centering
  \includegraphics[width=.8\linewidth]{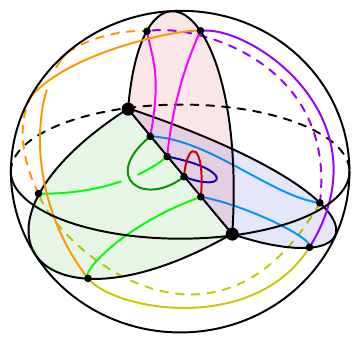}
  \caption{The neighborhood $N'$.}
  \label{fig:bcs_schem2}
\end{subfigure}
\par\vspace{5mm}
\begin{subfigure}{.5\textwidth}
  \centering
  \includegraphics[height=.7in]{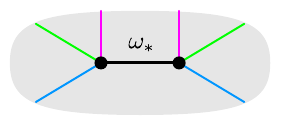}
  \caption{Local shadow diagram before...}
  \label{fig:bcs_schem3}
\end{subfigure}%
\begin{subfigure}{.5\textwidth}
  \centering
  \includegraphics[height=.7in]{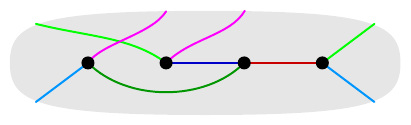}
  \caption{...and after boundary connected sum.}
  \label{fig:bcs_schem4}
\end{subfigure}
\caption{The trisected local neighborhood (A) is exchanged for the trisected local neighborhood (B) to carry out an ambient boundary connected sum of surface-links. The local change is depicted with shadow diagrams in the change from (C) to (D).  Note that, globally, the pink shadow arcs necessarily correspond to vertical strands of $\Tt_1$, while the light blue and light green shadow arcs may correspond (globally) to either flat or vertical strands.}
\label{fig:bcs}
\end{figure}

	Now, to perform the (ambient) boundary connected sum of the $\Ff^\varepsilon$ at the points $y^\varepsilon$, we must attach a half-twisted band $\frak b$ connecting these points.  (It should be half-twisted because $\partial \Ff'$ and $\partial\Ff''$ are braided about $B$; the half-twist will ensure that the result $\partial\Ff'\#\partial\Ff''$ is still braided about $B$.) We also assume that the core of $\frak b$ lies in $P_i$.  The change affected by attaching the half-twisted band is localized to the neighborhood $N$.  Therefore, it suffices to understand how $N$ is changed.
	
	Although we are describing an ambient boundary connected sum of surface in a four-manifold $X$ that may be highly nontrivial, the neighborhood $N$ is a four-ball, so it makes sense to import the bridge-braided band presentation technology from Section~\ref{sec:four-ball}.
	Figure~\ref{fig:bcs_band1} shows a bridge-braided ribbon presentation for~$N$, together with the half-twisted band~$\frak b$. 
	Figure~\ref{fig:bcs_band2} shows the effect of attaching the band, together with the dual band; this is a ribbon presentation for the boundary connected sum of the two disks in $N$.  Figure~\ref{fig:bcs_band3} shows a bridge-braided ribbon presentation for this object, which we denote by $N'$.  Note that the boundaries of $N$ and $N'$ are both 2--braids and are identical, except where they differ by a half-twist.  As stated before, we assume this difference is supported near $P_i$.  (Note that in the schematic of Figure~\ref{fig:bcs_schem2}, the half-twist is shown in the spread $Y_{i-1}$, rather than in $P^i$, due the reduction in dimension.  Similarly, in the frames of Figure~\ref{fig:bcs_band1} and~\ref{fig:bcs_band2}, the band $\frak b$ and the crossing are similarly illustrated away from $P_i$.)

\begin{figure}[h!]
\begin{subfigure}{.25\textwidth}
  \centering
  \includegraphics[width=.8\linewidth]{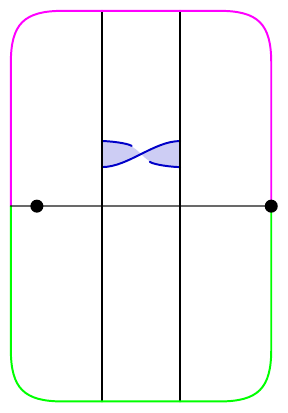}
  \caption{}
  \label{fig:bcs_band1}
\end{subfigure}%
\begin{subfigure}{.25\textwidth}
  \centering
  \includegraphics[width=.8\linewidth]{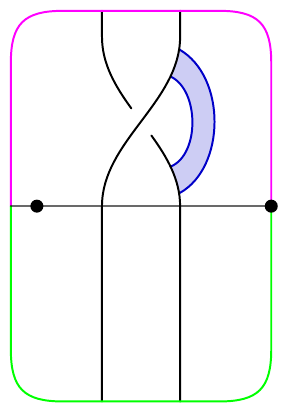}
  \caption{}
  \label{fig:bcs_band2}
\end{subfigure}%
\begin{subfigure}{.25\textwidth}
  \centering
  \includegraphics[width=.8\linewidth]{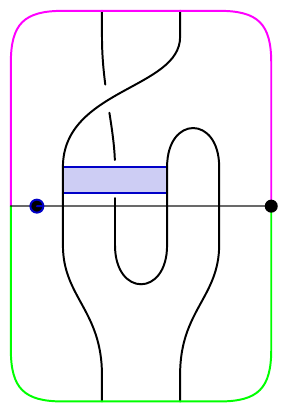}
  \caption{}
  \label{fig:bcs_band3}
\end{subfigure}%
\caption{(A) A ribbon presentation for $N$, together with the band $\frak b$ realizing the boundary connected sum; (B) A ribbon presentation for $N'$, the result of the boundary connected sum; (C) a bridge-braided ribbon presentation for $N'$.}
\label{fig:bcs_band}
\end{figure}

	The neighborhood $N'$ is the $(1;0,2)$--bridge trisection of the spanning disk for the unknot that induces the braiding of the unknot as a $(2,1)$--curve in the complement of the (unknotted) binding.  The corresponding bridge-braided ribbon presentation has one band, which is a helper band in the sense of Remarks~\ref{rmk:helpers} and~\ref{rmk:helpers2}.  This helper band is the dual band to $\frak b$.
	
	Because $\partial N$ and $\partial N'$ are identical away from a neighborhood of $\omega$, we can cut $N$ out and glue in $N'$ to realize the attaching of $\frak b$; i.e., to realize the ambient boundary connected sum.
\end{proof}

\subsection{Notions of reducibility}
\label{subsec:reduce}
\ 

We now discuss three notions of reducibility for trisections of pairs that we will show correspond with the connected sum and boundary connected sum operations discussed above.  These properties are distinct from, but related to, the properties of being stabilized or perturbed, which are discussed in Section~\ref{sec:stabilize}.

\begin{definition}
\label{def:reduce}
	Let $\T$ be a bridge trisection for a pair $(X,\Ff)$. Let $\delta\subset\Sigma\setminus\nu(\bold x)$ be an essential simple closed curve.
	\begin{enumerate}
		\item The curve $\delta$ is called a \emph{reducing curve} if, for each $i\in\Z_3$, there exists a disk $E_i\subset H_i\setminus\nu(\Tt_i)$ with~$\partial E_i = \delta$.
		\item The curve $\delta$ is called a \emph{decomposing curve} if, for each $i\in\Z_3$, there exists a disk $E_i\subset H_i$ with~$\partial E_i = \delta$ and with $|E_i\cap\Tt_i|=1$.  A decomposing curve is called \emph{trivial} if it bounds a disk in $\Sigma$ containing a single bridge point.
		\item An embedded three-sphere $S\subset X$ is a \emph{trisected reducing sphere} if $Z_i\cap S$ is a three-ball and $H_i\cap S$ is a disk for each~$i\in\Z_3$, and $\Sigma\cap S$ is a reducing curve.
		\item An embedded three-sphere-unknot pair $(S,K)\subset(X,\Ff)$ is a \emph{(nontrivial) trisected decomposing sphere pair} if
		$$(Z_i\cap S,\Dd_i\cap S)\cong(B^3,I)$$
		is a trivial 1--strand tangle in a three-ball for each~$i\in\Z_3$, and $\Sigma\cap S$ is a (nontrivial) decomposing curve.
		\item A trisection is \emph{reducible} (resp., \emph{decomposable}) if it admits a reducing curve (resp., a nontrivial decomposing curve).
	\setcounter{saveenum}{\value{enumi}}
	\end{enumerate}
		Let $\eta\subset\Sigma\setminus\nu(\bold x)$ be an essential, neatly embedded arc.
	\begin{enumerate}
	\setcounter{enumi}{\value{saveenum}}
		\item The arc $\eta$ is called a \emph{reducing arc} if, for each $i\in\Z_3$, there exists a neatly embedded arc $\eta_i\subset P_i$ and a disk $E_i\subset H_i\setminus\nu(\Tt_i)$ with $\partial E_i = \eta\cup\eta_i$.
		\item A neatly embedded three-ball $B\subset X\setminus\Ff$ is a \emph{trisected boundary-reducing ball} if, for all $i\in\Z_3$, we have $Z_i\cap B$ is a three-ball and $H_i\cap B$ is a disk, and $\Sigma\cap B$ is a reducing arc.
		\item A trisection is \emph{boundary-reducible}  if it admits a reducing arc.
	\end{enumerate}
\end{definition}

\begin{lemma}
\label{lem:reduce}
	If a trisection $\T$ is reducible, decomposable, or boundary-reducible, then $\T$ admits, respectively, a trisected reducing sphere, a nontrivial trisected decomposing sphere pair, or a trisected boundary-reducing ball.
\end{lemma}

\begin{proof}
	What follows is closely based on the proof of Proposition~3.5 from~\cite{MeiSchZup_16_Classification-of-trisections-and-the-Generalized}, where reducing curves are assumed (implicitly) to be separating, and some clarification is lacking.  Here, we give added detail and address the latter two conditions, which are novel.
	
	Suppose $\T$ is either reducible or decomposable, with reducing or decomposing curve $\delta$ bounding disks $E_i$ in the $H_i$.  Let $R_i = E_i\cup_\delta\overline E_{i+1}$ be the given two-sphere in $H_i\cup_\Sigma\overline H_{i+1}\subset \partial Z_i$.
	Recall (Subsection~\ref{subsec:DiskTangles}) that $Z_i$ is built by attaching 4--dimensional 1--handles the lensed product $Y_i\times[0,1]$ along $Y_i\times\{1\}$. A priori, the $R_i$ may not be disjoint from the belt spheres of the 1--handles in $Z_i$; however, by~\cite{Lau_73_Sur-les-2-spheres-dune-variete}, it can be arranged via handleslides and isotopies of the 1--handles that $R_i$ is disjoint from the belt spheres.
	Thus, we can assume that either (1) $R_i$ is parallel to a belt sphere, or (2) $R_i$ is contained in $Y_i\times\{1\}$.  These cases correspond to whether $\delta$ is non-separating or separating, respectively.
	In case (1), $R_i$ bounds the cocore of the 1--handle, which is a three-ball in $Z_i$.  In case (2), since $Y_i$ is irreducible, $R_i$ bounds a three-ball in $Y_i$ whose interior can be perturbed into $Z_i$.  In either case, we get a three-ball $B_i$ in $Z_i$ whose boundary is $E_i\cup_\delta\overline E_{i+1}$, and the union $S_\delta = B_1\cup B_2\cup B_3$ gives a trisected three-sphere.
	
	In the case that $\delta$ is reducing, we are done: $S_\delta$ is a trisected reducing sphere.  In the case that $\delta$ is a decomposing curve, it remains to show that $S_\delta\cap\Ff$ is unknotted and $B_i\cap\Ff$ is a trivial arc; the former is implied by the latter, which we now show.  Note that $B_i$ and $\Dd_i$ are both neatly embedded in $Z_i$ and that $\Dd_i$ is boundary parallel.
	Using the boundary parallelism of $\Dd_i$, we can arrange that a component $D$ of $\Dd_i$ intersects $B_i$ if and only if $D$ intersects $R_i = \partial B_i$.  It follows that there is a unique component $D\subset \Dd_i$ that intersects $B_i$.  If we isotope $D$ to a disk $D_*\subset\partial Z_i$, then we find that $D_*\cap R_i$ consists of an arc and some number of simple close curves.
	By an innermost curve argument, we may surgery $D_*$ to obtain a new disk $D_*'$ such that $D_*'\cap R_i$ consists solely of an embedded arc.  Since $D_*'$ and $D_*$ have the same boundary, they are isotopic rel-$\partial$ in $Z_i$ by Proposition~\ref{prop:triv_disks}.
	Reversing this ambient isotopy, we can arrange that $B_i\cap\Dd = B_i\cap D$ consists of a single arc.  Moreover, this arc is trivial, since it is isotopic to the arc $R_i\cap D_*$ in $\partial Z_i$, and $R_i$ is a decomposing sphere for either (a) the unknot $\partial D$ or (b) the unknotted, vertical strand $\Dd\cap H_i\cup_\Sigma\overline H_{i+1}$. Either way, $R_i$ cuts off an unknotted arc. Thus, $(S_\delta,K)$ can be constructed to be a decomposing sphere for the trisection, as desired, where $K$ is the three-fold union of the trivial arcs $B_i\cap\Ff$.
	
	Now suppose that $\T$ is boundary-reducible, with reducing arc $\eta$ and arcs $\eta_i$ such that $\eta\cup\eta_i$ bounds a disk $E_i\subset H_i$.  Consider the neatly embedded 2--disk $R_i = E_i\cup_\eta \overline E_{i+1}$ in $H_i\cup_\Sigma\overline H_{i+1}\subset\partial Z_i$.  Let $B_i$ be the trace of a small isotopy that perturbs the interior of $R_i$ into $Z_i$.  Then the union $B_\eta = B_1\cup B_2\cup B_3$ is a trisected three-ball.  If $\eta$ is a reducing arc, we are done.
\end{proof}

\begin{remark}[\textbf{Regarding non-separating curves}]
\label{rmk:nonsep}
	Reducing curves are almost always separating in the following sense.  Suppose that $\delta$ is a non-separating reducing curve. Then there is a curve $\eta\subset\Sigma$ that is dual to $\delta$.  Let $\delta' = \partial\nu(\delta\cup\eta)$.  Then $\delta'$ is a separating reducing curve, unless it is inessential (i.e., parallel to a boundary component of $\Sigma$ or null-homotopic in $\Sigma$).  This only occurs if $\Sigma$ is the core of the genus one trisection for $S^1\times S^3$ or for its puncture, $(S^1\times S^3)^\circ$.  In any event, the neighborhood $\nu(S_\delta\cup\eta)$, where $S_\delta$ is the reducing sphere corresponding to $\delta$ as in Lemma~\ref{lem:reduce} below, is diffeomorphic to $(S^1\times S^3)^\circ$.

	If $\delta$ is a non-separating decomposing curve with corresponding decomposing pair $(S_\delta,K_\delta)$, then $K_\delta$ can be separating or non-separating as a curve in $\Ff$.  If $K_\delta$ is non-separating, then we can surgery $(X,\Ff)$ along the pair $(S,K)$ to obtain a new pair $(X',\Ff')$.  That the surgery of $\Ff$ along $K$ can be performed ambiently uses the fact that $K$ is an unknot in $S$, hence bounds a disk in $X\setminus\Ff$.  Working backwards, there is a $S^0\subset\Ff'\subset X$ along which we can surger $(X',\Ff')$ to obtain $(X,\Ff)$.  It follows that $X = X'\#(S^1\times S^3)$ and $\Ff$ is obtained from $\Ff'$ by tubing.  Diagrammatically, the surgery from $(X,\Ff)$ to $(X',\Ff')$ is realized by surgering $\Sigma$ along $\delta$.  Note that this tubing is not necessarily trivial in the sense that it may or may not be true that $(X,\Ff) = (X',\Ff')\#(S^1\times S^3,S^1\times S^1)$.
\end{remark}

A bridge trisection satisfying one of the three notions of reducibility decomposes in a natural way.  See Subsection~\ref{subsec:sum} for a detailed discussion of connected sum and boundary connected sum operations.  For example, presently, we let $\T'\#\T''$ denote the connected sum of trisections, regardless of whether the connected summing point is a bridge point or not.

\begin{proposition}
\label{prop:reduce}
	Let $\T$ be a bridge trisection for a pair $(X,\Ff)$.
	\begin{enumerate}
		\item If $\T$ admits a separating reducing curve, then there exist pairs $(X',\Ff')$ and $(X'',\Ff'')$ with trisections $\T'$ and $\T''$ such that $\T = \T'\#\T''$ and
		$$(X,\Ff) = (X'\#X'',\Ff'\sqcup\Ff'').$$
		\item If $\T$ admits a nontrivial, separating decomposing curve, then there exist pairs $(X',\Ff')$ and $(X'',\Ff'')$ with trisections $\T'$ and $\T''$ such that $\T = \T'\#\T''$ and
		$$(X,\Ff) = (X'\#X'',\Ff'\#\Ff'').$$
		\item If $\T$ admits a separating reducing arc, then there exist pairs $(X',\Ff')$ and $(X'',\Ff'')$ with trisections $\T'$ and $\T''$ such that $\T = \T'\natural\T''$ and
		$$(X,\Ff) = (X'\natural X'',\Ff'\sqcup\Ff'').$$
	\end{enumerate}
\end{proposition}

\begin{proof}
	If $\T$ admits a separating reducing curve $\delta$, then it admits a separating trisected reducing sphere $S_\delta$, by Lemma~\ref{lem:reduce}.  Cutting open along $S_\delta$ and capping off the two resulting three-sphere boundary components with genus zero trisections of $B^4$ results in two new trisections $\T'$ and $\T''$ for pairs $(X',\Ff')$ and $(X'',\Ff'')$, as desired in part (1). For part (2), we proceed as above, except we cap off with two genus zero 0--bridge trisections of $(B^4,D^2)$ to achieve the desired result. (If any of the disks $E_i$ bounded by $\delta$ in the $H_i$ intersect vertical strands $\tau_i$, then we can perturb to make these intersecting strands flat. If such perturbations are performed before cutting, they can be undone with deperturbation after gluing.  This is related to the discussion immediately preceding Lemma~\ref{lem:conn_sum}.)
	
	If $\T$ admits a separating reducing arc $\eta$, then it admits a separating trisected reducing ball $B_\eta$, by Lemma~\ref{lem:reduce}.  Cutting open along $B_\eta$ results in two new trisections $\T'$ and $\T''$ for pairs $(X',\Ff')$ and $(X'',\Ff'')$, as desired in part (3).  
\end{proof}

\begin{remark}[\textbf{Boundary-decomposing arcs}]
\label{rmk:bda}
	Conspicuously absent from the above notions of reducibility is a characterization of what might be referred to as boundary-decomposability -- in other words, a characterization of when we have
	$$(X,\Ff) = (X'\natural X'',\Ff'\natural\Ff'').$$
	The obvious candidate for such a notion would be the existence of a neatly embedded, essential arc $\eta\subset\Sigma$, similar to the one involved in the notion of boundary-reducibility, but where the disks $E_i$ each intersect the respective $\Tt_i$ in precisely one point.  However, a lengthy examination of such arcs reveals that they rarely correspond to surfaces that are boundary connected sums in the desired way. To the point, many of the examples given later in this section admit such arc, but are not boundary-connected sums of bridge trisected surfaces.  We have been unable to find a satisfying characterization of when this occurs.
\end{remark}

\subsection{Classification for small parameters}
\label{subsec:small}
\ 

As a first example, consider the $(4,(2,4,2);3)$--bridge trisection shown in Figure~\ref{fig:c=b}, which is the boundary sum of a 1--bridge trisection, a 3--bridge trisection that is perturbed, and three 0--bridge trisections  and corresponds to $(B^4,S^2\sqcup S^2\sqcup D^2\sqcup D^2\sqcup D^2)$.  It turns out that such a bridge trisection is obtained whenever $c_i=b$ for some $i\in\Z_3$.  (Recall that $\partial D_i$ contains a flat $b$--bridge $c_i$ component unlink, so $b\geq c_i$ for all $i\in\Z_3$.)

\begin{figure}[h!]
\begin{subfigure}{.5\textwidth}
  \centering
  \includegraphics[width=.9\linewidth]{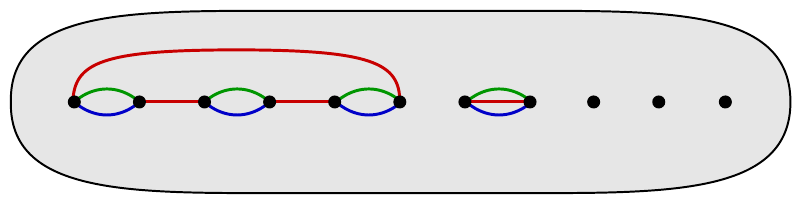}
  \caption{}
  \label{fig:c=b1}
\end{subfigure}%
\begin{subfigure}{.5\textwidth}
  \centering
  \includegraphics[width=.9\linewidth]{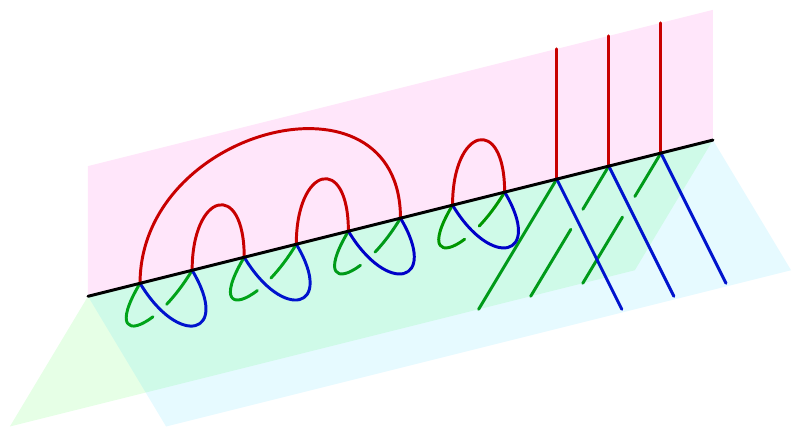}
  \caption{}
  \label{fig:c=b2}
\end{subfigure}
\caption{A shadow diagram (A) and schematic tri-plane diagram (B) for the unique $(4,(2,4,2);3)$--bridge trisection, which is totally reducible.}
\label{fig:c=b}
\end{figure}

\begin{proposition}
\label{prop:c=b}
	Let $\T$ be a $(b,\bold c;v)$--bridge trisection for a surface $(B^4,\Ff)$.  If $b=c_i$ for some $i\in\Z_3$, then $c_{i+1}=c_{i+2}=c$ and
		$$(B^4,\Ff) = \left(B^4,(\sqcup_cS^2)\sqcup(\sqcup_vD^2)\right),$$ and $\T$ is the boundary sum of $c$ genus zero bridge trisections of $(B^4,S^2)$, each of which is a finger perturbation of the 1--bridge trisection, and $v$ genus zero 0--bridge trisections of $(B^4,D^2)$.
\end{proposition}

\begin{proof}
	Suppose without loss of generality that $c_2=b$. By Proposition~\ref{prop:bridge_to_band}, $\Ff$ admits a $(b,\bold c;v)$--bridge-braided band presentation. In particular, $\Ff$ can be built with $n = b-c_2=0$ bands. It follows that $c_1=c_3$. It also follows that the flat disks of $(Z_2,\Dd_2)$ are given as products on the $b$ flat strands of $(H_2,\Tt_2)$.
	
	We can assume that the union of the red and blue shadow arcs is a collection of $c_1$ embedded polygons in $\Sigma$, since they determine a $b$--bridge $c_1$--component unlink in $H_1\cup_\Sigma\overline H_2$.  We can also assume that the green shadow arcs coincide with the blue shadow arcs, due to the product structure on the flat disks of $\Dd_2$. See Figure~\ref{fig:c=b1}.
	
	Let $\delta$ be a simple closed curve in $\Sigma\setminus\nu(\bold x)$ that separates the red/blue polygons from the bridge points that are adjacent to no shadow arc.  (Note that, here, every bridge point is adjacent to either 0 or 3 shadow arcs by the above considerations.)  Then $\delta$ is a reducing curve for $\T$ so that $\T=\T^1\#\T^2$, where $\T^1$ is a $(b,\bold c)$--bridge trisection for a pair $(S^4,\Ff^1)$ and $\T^2$ is a $(0,0;v)$--bridge trisection for a pair $(B^4,\Ff^2)$.
	
	Because the blue and green shadow arcs coincide, each polygon is a finger perturbation of the 1--bridge splitting of $(S^4,S^2)$, and $\Ff^1 = \sqcup_cS^2$.  Moreover, $\T^1$ admits $c-1$ reducing curves that completely separate the polygons.  It follows that $\T^1$ is connected sum of perturbations of the 1--bridge trisection of $(S^4,S^2)$, as desired.  Finally, the bridge trisection $\T^2$ admits $v-1$ reducing arcs that cut it up into $v$ copies of the genus zero 0--bridge trisection of $(B^4,D^2)$, as desired.
\end{proof}

Having dispensed of the case when $c_i=b$ for some $i\in\Z_3$, we consider the case when $b=1$ and, in light of the above, $c_i=0$ for all $i\in\Z_3$.  Two simple examples of such bridge trisections are given in Figure~\ref{fig:b=1}.

\begin{figure}[h!]
\begin{subfigure}{.5\textwidth}
  \centering
  \includegraphics[height=20mm]{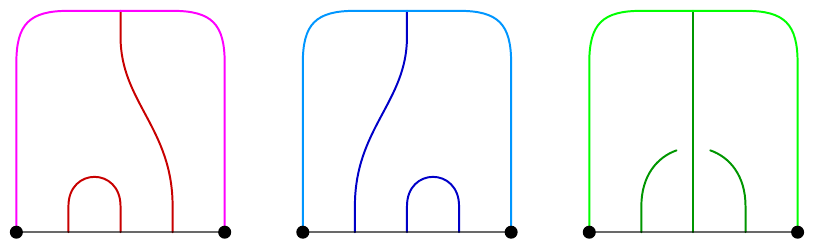}
  \caption{}
  \label{fig:b=11}
\setcounter{subfigure}{2}
\end{subfigure}%
\begin{subfigure}{.5\textwidth}
  \centering
  \includegraphics[height=20mm]{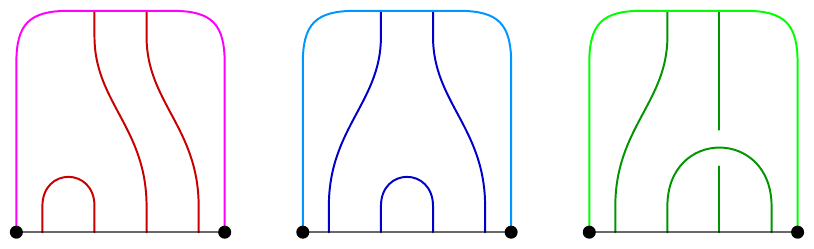}
  \caption{}
  \label{fig:b=13}
\setcounter{subfigure}{1}
\end{subfigure}
\par\vspace{5mm}
\begin{subfigure}{.5\textwidth}
  \centering
  \includegraphics[height=15mm]{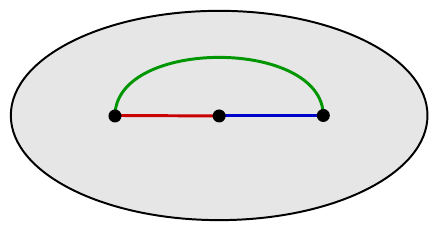}
  \caption{}
  \label{fig:b=12}
\setcounter{subfigure}{3}
\end{subfigure}%
\begin{subfigure}{.5\textwidth}
  \centering
  \includegraphics[height=15mm]{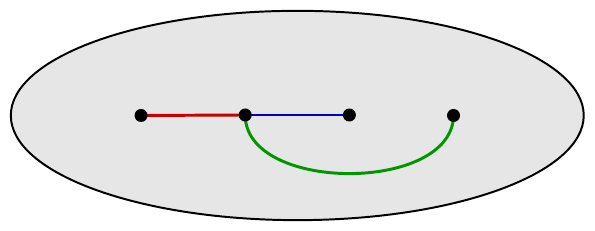}
  \caption{}
  \label{fig:b=14}
\end{subfigure}
	\caption{(A) and (B) The $(1,0;1)$--bridge trisection corresponding to the standard (positive) M\"obius band $(B^4,M^2)$.  (C) and (D) The $(1,0;2)$--bridge trisection corresponding to the unknotted disk $(B^4,D^2)$ with (positive) Markov stabilized, unknotted boundary.}
\label{fig:b=1}
\end{figure}

For a more interesting family of examples, consider the $(2,4)$--torus link $T_{2,4}$, which bounds the union of the trivial M\"obius band $M^2$ and the trivial disk $D^2$. (Imagine Figure~\ref{fig:T243} with the three parallel circles replaced with a single circle.)  Now, consider the surface $\Ff_v$ obtained by replacing the $D^2$ with $v-1$ parallel, trivial disks; Figure~\ref{fig:T243} shows the case of $v=4$.  A $(1,0;v)$--bridge trisection $\T_v$ for $(B^4,\Ff_v)$ is shown in Figures~\ref{fig:T241} and~\ref	{fig:T242}.  Note that when $v=1$, $\T_v$ corresponds the trivial (positive) M\"obius band with unknotted boundary and was given diagrammatically in Figures~\ref{fig:b=11} and~\ref{fig:b=12}.

\begin{figure}[h!]
\centering
\begin{tabular}{cc}
\multirow{2}{*}{
\begin{subfigure}{.4\textwidth}
\setcounter{subfigure}{0}
\centering
\includegraphics[width=.5\linewidth]{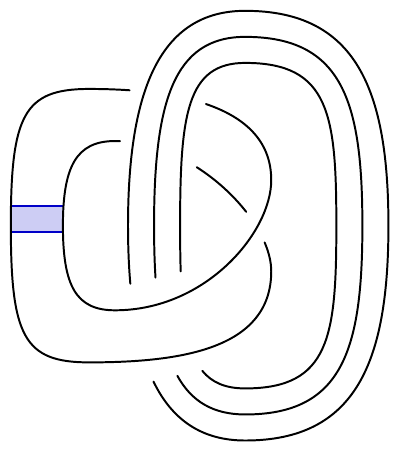}
\caption{}
\label{fig:T243}
\end{subfigure}%
}
&
\begin{subfigure}{.5\textwidth}
\centering
\includegraphics[width=\linewidth]{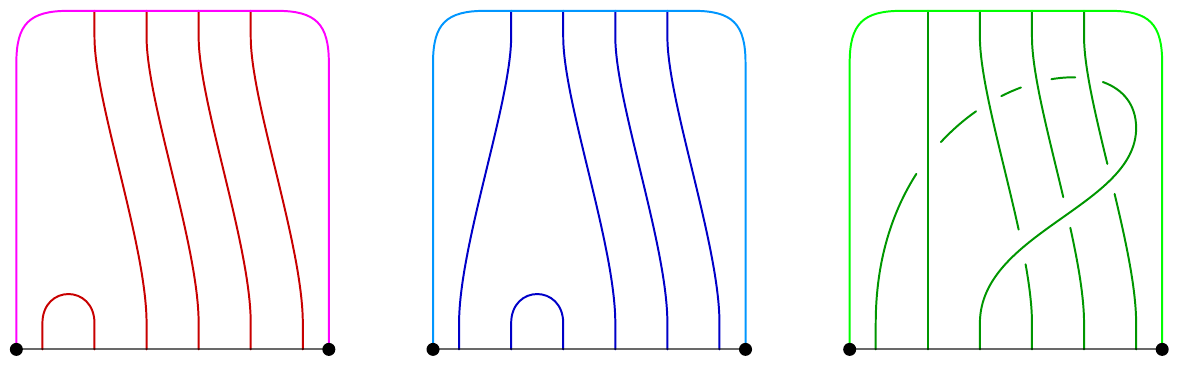}
\caption{}
\label{fig:T241}
\end{subfigure}%
\\[20mm]
&
\begin{subfigure}{.5\textwidth}
\centering
\includegraphics[width=.6\linewidth]{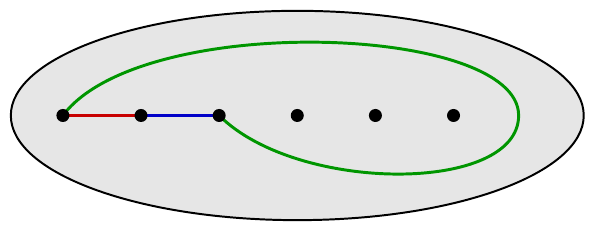}
\caption{}
\label{fig:T242}
\end{subfigure}%
\end{tabular}   
\caption{.}
\label{fig:T24}
\end{figure}

One can check using the techniques of Subsection~\ref{subsec:tri-plane_braid} that the bridge trisection $\T_v$ induces the $v$--braiding of $\partial \Ff_v$ given in Artin generators by $$(\sigma_1\sigma_2\cdots\sigma_{v-2}\sigma_{v-1}^2\sigma_{v-2}\cdots\sigma_2\sigma_1)^2.$$
In other words, one strand wraps twice around the other $v-1$ strands.
The link $\partial\Ff_v$ can be thought of as taking the $(v-1,0)$--cable of one component of $T_{2,4}$.

\begin{proposition}
\label{prop:b=1}
	The bridge trisection $\T_v$ is the unique (up to mirroring) totally irreducible $(1,0;v)$--bridge trisection.
\end{proposition}

\begin{proof}
	Suppose that $\T$ is a totally irreducible $(1,0;v)$--bridge trisection, and consider a shadow diagram for~$\T$.  Since $b=1$, there is a unique shadow arc of each of color.  Since $c=0$, the union of any two of these shadow arcs is a connected, embedded, polygonal arc in $\Sigma$.  There are two cases:  Either the union of the three shadow arcs is a circle, or the union of the three shadow arcs is a Y-shaped tree.
	
	Suppose the union is a Y-shaped tree.  Let $\eta$ be an arc connecting the tree to $\partial\Sigma$, and let $\omega$ be the arc boundary of the union of $\eta$ and the tree.  In other words, $\omega$ is a neatly embedded arc in $\Sigma\setminus\nu(\bold x)$ that separates the tree from the rest of the diagram.  If the rest of the diagram is nonempty, then $\delta$ is a reducing arc for the bridge trisection, and we have $\T = \T^1\natural\T^2$, where $\T^1$ is a $(1,0;2)$--bridge trisection (with Y-shaped shadow diagram) and $\T^2$ is a $(0,0;v)$--bridge trisection, with $v>0$. This contradicts the assumption that $\T$ was totally irreducible.  If $v=0$ (i.e. the rest of the diagram is empty), then $\T=\T^1$ is the Markov perturbation of the genus zero 0--bridge trisection and is shown in Figure~\ref{fig:b=14}, so $\T$ is not totally irreducible, another contradiction.
	
	Now suppose that the union of the three shadow arcs is a circle, and let $D\subset\Sigma$ denote the disk the union bounds.  Suppose there is a bridge point in $\Sigma\setminus D$.  Then there is a reducing arc separating the bridge point from $D$, so $\T$ is boundary reducible, a contradiction.  So, the $v-1$ bridge points that are not adjacent to a shadow arc are contained in $D$.  Therefore, the shadow diagram is the one given in Figure~\ref{fig:T242} or, in the case that $v=1$, in Figure~\ref{fig:b=12}.  This completes the proof.
\end{proof}

Having walked through these modest classification results, we now present some families of examples, as well as some questions and conjectures about further classification results.

\begin{example}
\label{ex:punc}
	Consider the three $(2,0;1)$--bridge trisections shown in Figure~\ref{fig:punc}, which correspond the punctured torus and two different Klein bottles.  All three surfaces are isotopic into $S^3$ and are bounded by the unknot.  The two Klein bottles decompose as boundary connected sums of M\"obius bands bounded by the unknot in $S^3$.  The Klein bottle depicted in Figures~\ref{fig:punc3} and~\ref{fig:punc4} is the boundary connected sum of two positive M\"obius bands; the Klein bottle depicted in Figures~\ref{fig:punc5} and~\ref{fig:punc6} is the boundary connected sum of a positive and a negative M\"obius bands

\begin{figure}[h!]
\begin{subfigure}{.5\textwidth}
\centering
\includegraphics[width=\linewidth]{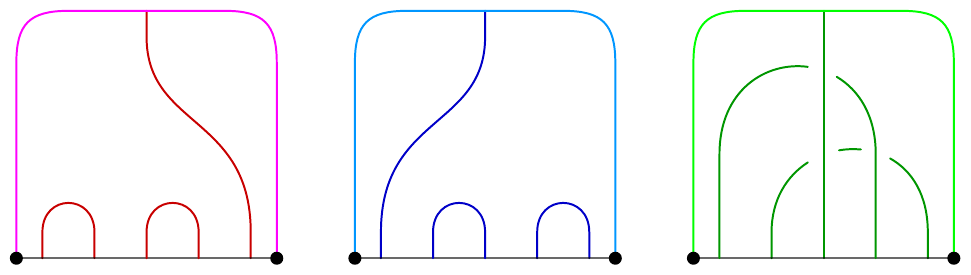}
\caption{}
\label{fig:punc1}
\end{subfigure}%
\begin{subfigure}{.4\textwidth}
\centering
\includegraphics[width=.6\linewidth]{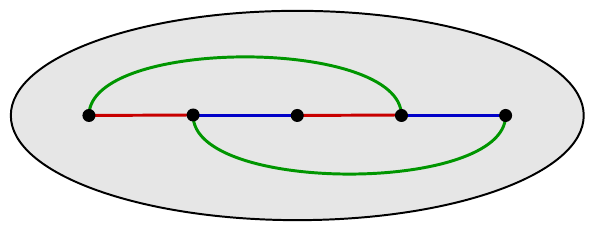}
\caption{}
\label{fig:punc2}
\end{subfigure}
\par\vspace{5mm}
\begin{subfigure}{.5\textwidth}
\centering
\includegraphics[width=\linewidth]{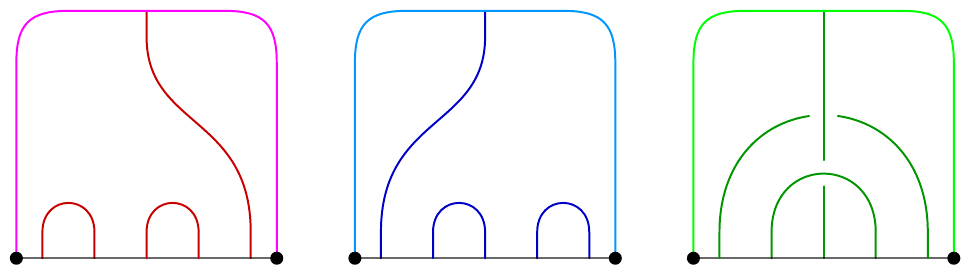}
\caption{}
\label{fig:punc3}
\end{subfigure}%
\begin{subfigure}{.4\textwidth}
\centering
\includegraphics[width=.6\linewidth]{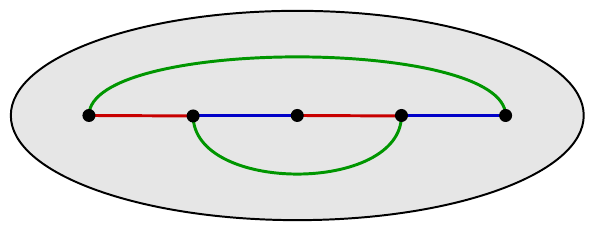}
\caption{}
\label{fig:punc4}
\end{subfigure}
\par\vspace{5mm}
\begin{subfigure}{.5\textwidth}
\centering
\includegraphics[width=\linewidth]{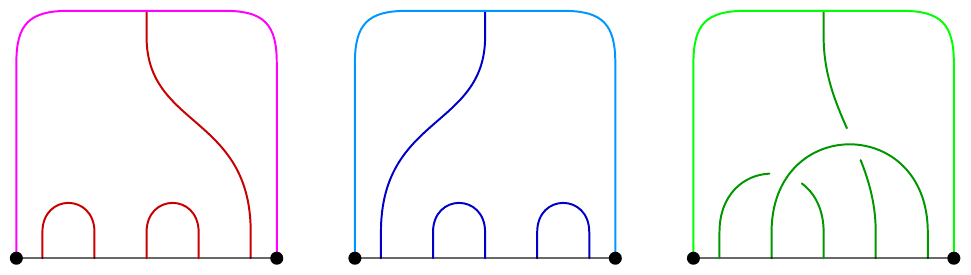}
\caption{}
\label{fig:punc5}
\end{subfigure}%
\begin{subfigure}{.4\textwidth}
\centering
\includegraphics[width=.6\linewidth]{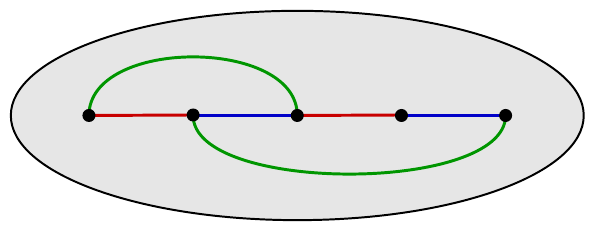}
\caption{}
\label{fig:punc6}
\end{subfigure}
\caption{Three $(2,0;1)$--bridge trisections for surfaces bounded by the unknot and isotopic into $S^3$. (A) and (B) describe a punctured torus; (C) and (D) describe the boundary connected sum of two positive M\"obius bands; (E) and (F) describe the boundary connected sum of a positive and a negative M\"obius band.}
\label{fig:punc}
\end{figure}

	These three bridge trisections can be obtained by taking the three unique $(3,1)$--bridge trisections of closed surfaces in $S^4$ and puncturing at a bridge point.

\end{example}

\begin{conjecture}
\label{conj:punc}
	There are exactly three (up to mirroring) totally irreducible $(2,0;1)$--bridge trisections.
\end{conjecture}

\begin{example}
\label{ex:2-braid}
	Consider the $(2,0;2)$--bridge trisection shown in Figure~\ref{fig:2-braid41}, which corresponds the annulus $S^3$ bounded by the $(2,4)$--torus link.  Compare with Example~\ref{ex:mono} and Figure~\ref{fig:mono1}.  By replacing the three positive half-twists with $n$ half-twists for some $n\in\Z$, gives a surface in $S^3$ bounded by the $(2,n)$--torus link that is a M\"obius band if $n$ is odd and an annulus if $n$ is even.
	
	One interesting aspect of the case when $n$ is even relates to the orientation of the boundary link.  The boundary link, which is the $(2,n)$--torus link, inherits an orientation as a 2--braid.  It also inherits an orientation from the spanning annulus that the bridge trisection describes.  These orientations don't agree!  In other words, the bridge trisections of the spanning annuli for these links induce a braiding of the links, but this braiding is not coherent with respect to the orientation of the links induced by the annuli.  Compare with Example~\ref{ex:T24coherent} below.
\end{example}

\begin{figure}[h!]
\begin{subfigure}{.5\textwidth}
\centering
\includegraphics[width=.9\linewidth]{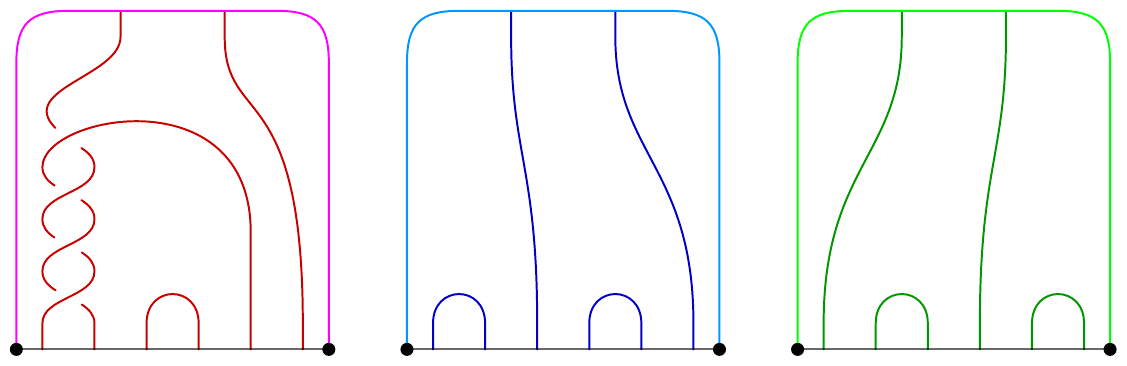}
\caption{}
\label{fig:2-braid41}
\end{subfigure}%
\begin{subfigure}{.5\textwidth}
\centering
\includegraphics[width=.5\linewidth]{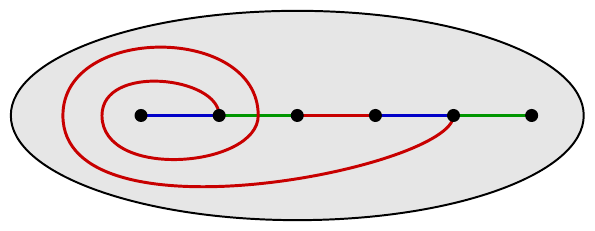}
\caption{}
\label{fig:2-braid42}
\end{subfigure}
\caption{Diagrams for a $(2,0;2)$--bridge trisection of the planar surface bounded by the $(2,n)$--torus link in $S^3$; shown is $n=4$.}
\label{fig:2-braid4}
\end{figure}

\begin{conjecture}
\label{conj:punc}
	Every $(2,0;2)$--bridge trisection is diffeomorphic to one of those described in Example~\ref{ex:2-braid} and in Figure~\ref{fig:2-braid4}.
\end{conjecture}

\begin{example}
\label{ex:T24coherent}
	Figure~\ref{fig:3-braid41} gives a $(3,0;3)$--bridge trisection for the annulus in $S^3$ bounded by the $(2,4)$--torus link. In contrast to the bridge trisection for this surface discussed in Example~\ref{ex:2-braid} and illustrated in Figure~\ref{fig:2-braid4}, this bridge trisection induces a coherent 3--braiding of the boundary link.  This example could be generalized to give a $(n+1,0;n+1)$--bridge trisection for the annulus bounded by the $(2,n)$--torus link for any even $n\in\Z$.
\end{example}

\begin{figure}[h!]
\begin{subfigure}{.6\textwidth}
\centering
\includegraphics[width=.9\linewidth]{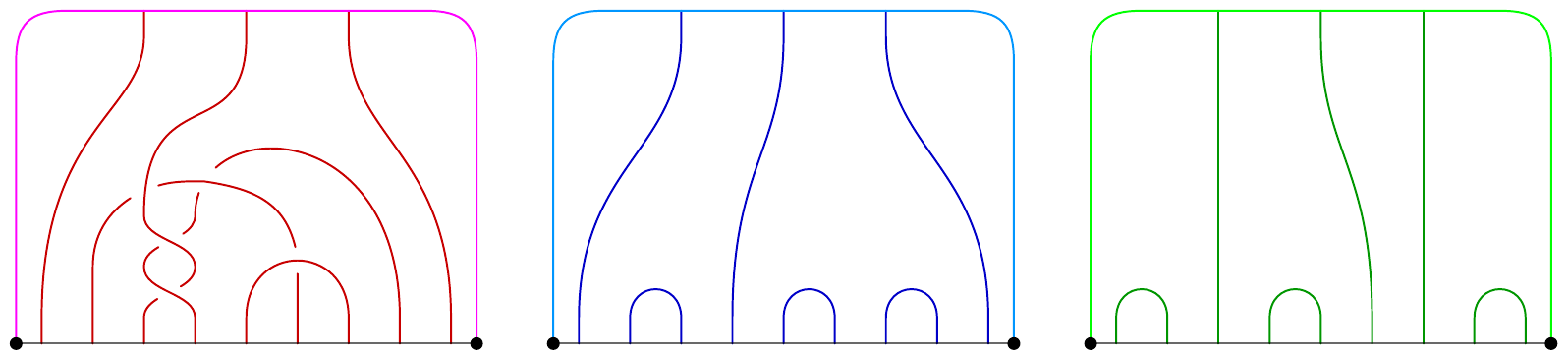}
\caption{}
\label{fig:3-braid41}
\end{subfigure}%
\begin{subfigure}{.4\textwidth}
\centering
\includegraphics[width=.7\linewidth]{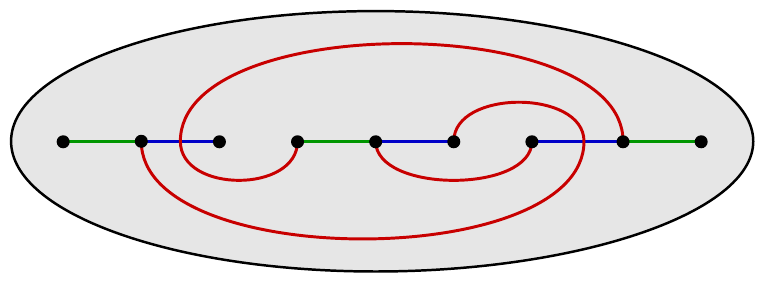}
\caption{}
\label{fig:3-braid42}
\end{subfigure}
\caption{Diagrams for a $(3,0;3)$--bridge trisection of the planar surface bounded by the $(2,n)$--torus link in $S^3$; shown is $n=4$.}
\label{fig:3-braid4}
\end{figure}

\section{Proof of Theorem~\ref{thm:general}}
\label{sec:gen_proof}

We now make use of the general framework outlined in Section~\ref{sec:general} to give a proof of Theorem~\ref{thm:general}, which we restate for convenience. We adopt the notation and conventions of Definition~\ref{def:Trisection}.

\begin{theorem}
\label{thm:general}
	Let $\T$ be a trisection of a four-manifold $X$ with $\partial X = Y$, and let $(B,\pi)$ denote the open-book decomposition of $Y$ induced by $\T$.  Let $\Ff$ be a neatly embedded surface in $X$; let $\Ll = \partial \Ff$; and fix a braiding $\wh\beta$ of $\Ll$ about $(B,\pi)$.  Then, $\Ff$ can be isotoped to be in bridge trisected position with respect to $\T$ such that $\partial \Ff = \wh\beta$.  If $\Ll$ already coincides with the braiding $\beta$, then this isotopy can be assumed to restrict to the identity on $Y$.
\end{theorem}

Note that if $X$ is closed, then Theorem~\ref{thm:general} is equivalent to Theorem~1 of~\cite{MeiZup_18_Bridge-trisections-of-knotted}.  For this reason, we assume henceforth that $Y=\partial X\not=\emptyset$.  We will prove Theorem~\ref{thm:general} using a sequence of lemmata.    Throughout, we will disregard orientations. All isotopies are assumed to be smooth and ambient. First, we describe the existence of a Morse function $\Phi_\T$ on (most of) $X$ that is well-adapted to the trisection $\T$.  We will want to think of $X$ as a lensed cobordism from $Y_1$ to $Y_2\cup_{P_3}Y_3$.

\begin{lemma}\label{lem:trisection_to_Morse}
	There is a self-indexing Morse function
	$$\Phi_\T\colon X\setminus\nu(P_1\cup_B P_2\cup_BP_3)\to [0,4]$$
	such that
	\begin{enumerate}
		\item $\Phi_\T$ has no critical points of index zero or four;
		\item $Y_1\setminus\nu(P_1\cup_B \overline P_2) = \Phi_\T^{-1}(0)$;
		\item $(H_1\cup_\Sigma \overline H_2)\setminus\nu(P_1\cup_B \overline P_2) = \Phi_\T^{-1}(1.5)$;
		\item $\Phi_\T(H_3\setminus\nu(P_3))\subset[1.5,2.5)$;
		\item $Y_3\setminus\nu(P_3\cup_B \overline P_1) = \Phi_\T^{-1}(4)$; and
		\item The index $j$ critical points of $\Phi_\T$ are contained in $\Int(Z_j)$.
	\end{enumerate}
\end{lemma}

Note that if $\Phi_\T(x)\geq 2.5$, then $x\in Z_3$.

\begin{proof}
	The existence of the Morse function and property (1) are standard consequences of the cobordism structure.  The other properties are easy and commonly discussed within the theory of trisections; see~\cite{GayKir_16_Trisecting-4-manifolds}, for example.  The set-up is made evident by the schematics of Figure~\ref{fig:Morse}.
\end{proof}

\begin{figure}[h!]
\begin{subfigure}{.5\textwidth}
  \centering
  \includegraphics[width=.6\linewidth]{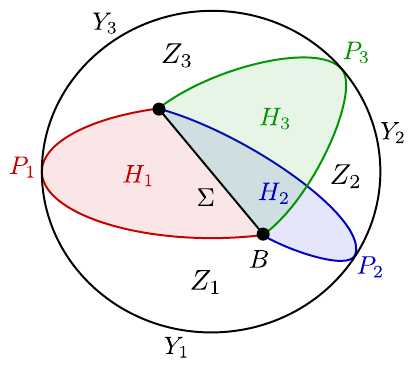}
  \caption{}
  \label{fig:Morse1}
\end{subfigure}%
\begin{subfigure}{.5\textwidth}
  \centering
  \includegraphics[width=\linewidth]{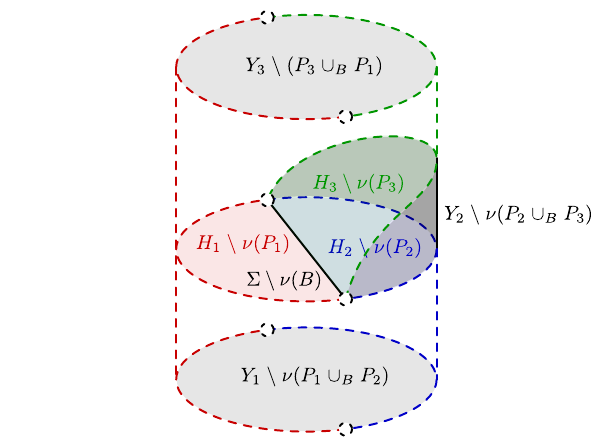}
  \caption{}
  \label{fig:Morse2}
\end{subfigure}
\caption{Passing from a trisection to a natural Morse function on $X\setminus\nu(P_1\cup P_2\cup P_3)$.}
\label{fig:Morse}
\end{figure}

Now, $Z_1$ is the result of attaching four-dimensional 1--handles to the lensed product $Y_1\times I$. The core $\Sigma$ can be assumed to satisfy $\Phi_\T(\Sigma\setminus\nu(B)) = 1.5$, and, together with $P_1$ and $P_2$, it gives a standard Heegaard double decomposition of $\partial Z_1$.  The attaching circles of the four-dimensional 2--handles are assumed to be contained in a (1--complex) spine of the compressionbody $H_2$, with the result of Dehn surgery thereupon being $H_3$.  The trace of this 2--handle attachment is $Z_2$, and $Z_3$ is the  (lensed cobordism) trace of attaching four-dimensional 3--handles to $H_3\cup_\Sigma \overline H_1$, the result of which is $Y_3$.  (Note that $Z_2$ is not quite a lensed cobordism from this perspective, since $Y_2$ is a vertical portion of its boundary $\partial Z_2 = H_2\cup Y_2\cup \overline H_3$.) 

For the remainder of the section, we let $\Phi = \Phi_\T$. Let $\Phi_i = \Phi\vert_{Z_i}$ for $i=1,2,3$.  Recall the standard Morse function on $Z_i\cong Z_{g,k_i,\bold p, \bold f}$ that was discussed in Subsection~\ref{subsec:DiskTangles}. By the above discussion, we have the following consequence of Lemma~\ref{lem:trisection_to_Morse}.

\begin{corollary}\label{coro:std_Morse}
	If $i=1$ or $i=3$, then $\Phi_i$ is a standard Morse function on $Z_i\cong Z_{g,k_i,\bold p, \bold f}$.
\end{corollary}

Presently, we will begin to isotope $\Ff$ to lie in bridge trisected position with respect to $\T$.

\begin{lemma}\label{lem:transversifying}
	After an isotopy of $\Ff$ that is supported near $\partial X$, we can assume that $\Ll=\wh\beta$.
\end{lemma}

\begin{proof}
	By the Alexander Theorem~\cite{Ale_20_Note-on-Riemann-spaces} or the generalization due to Rudolph~\cite{Rud_83_Constructions-of-quasipositive-knots}, $\Ll$ can be braided with respect to the open-book decomposition $(B,\pi)$.  By the Markov Theorem~\cite{Mar_35_Uber-die-freie-Aquivalenz} or its generalization to closed 3--manifolds~\cite{Sko_92_Closed-braids-in-3-manifolds,Sun_93_The-Alexander-and-Markov-theorems}, any two braidings of $\Ll$ with respect to $(B,\pi)$ are isotopic.  Thus, by an isotopy of $\Ff$ that is supported near $Y$, we can assume that $\Ll$ is given by the braiding to $\wh\beta$. 
\end{proof}

Any modifications made to $\Ff$ henceforth will be isotopies that restrict to the identity on $Y$.  Let $\Phi_\Ff$ denote the restriction of $\Phi$ to $\Ff$. (Note that by choosing a small enough collar $\nu(Y)$ in $X$, we can assume that $\Ff\cap\nu(Y) = \Ll\times I$.  By a small isotopy of $\Ff$ rel-$\partial$, we can assume that $\Phi_\Ff$ is Morse.)

\begin{lemma}\label{lem:Morse_Ff}
	After an isotopy of $\Ff$ rel-$\partial$, we can assume that $\Phi_\Ff\colon \Ff\to \R$ is Morse and that
\begin{enumerate}
	\item the minima of $\Phi_\Ff$ occur in $Z_1$,
	\item the saddles of $\Phi_\Ff$ occur in $ \Phi^{-1}(1.5)$, and
	\item the maxima of $\Phi_\Ff$ occur in $Z_3$.
\end{enumerate}
\end{lemma}

\begin{proof}
	That the critical points can be rearranged as desired follows from an analysis of their various ascending and descending manifolds.  A detailed analysis of this facet of (embedded) Morse theory can be found in~\cite{Bor}.  Here, we simply make note of the key points.
	
	The ascending (unstable) membrane of a maximum of $\Phi_\Ff$ is one-dimensional; think of a vertical arc emanating from the maximum and terminating in $Y_3$. (Vertical means the intersection with each level set is either a point or empty.)  Generically, such an arc will be disjoint from $\Ff$ and will be disjoint from the descending spheres of the critical points of $\Phi$ (which have index one, two, or three) in each level set.
	Thus, the gradient flow of $\Phi$ can be used to push the maxima up (and the minima down), and we obtain that the minima lie below $\Phi^{-1}(1.5)$ (i.e., in $Z_1$) and that the maxima lie above $\Phi^{-1}(2.5)$ (i.e., in $Z_3$).  Having arranged the extrema in this way, we move on to consider the saddles.
	
	The ascending membranes of the saddles of $\Phi_\Ff$ are two-dimensional, while the the descending spheres of the index one critical points of $\Phi$ are zero-dimensional.  Thus, we can flow the saddles up past the index one critical points of $\Phi$, until they lie in $\Phi^{-1}(1.5)$.  Symmetrically, we can flow saddles down past the index three critical points of $\Phi$ to the same result.
\end{proof}

Let $\Dd_i = \Ff\cap Z_i$ for $i=1,2,3$.  Assume that $\wh\beta$ is a braiding of $\Ll$ of multi-index $\bold v$.

\begin{lemma}\label{lem:Dd_1/3}
	If $\Phi_\Ff$ has $c_1$ minima and $c_3$ maxima, then $\Dd_1$ is a $(c_1,\bold v)$--disk-tangle, and $\Dd_3$ is a $(c_3,\bold v)$--disk-tangle.
\end{lemma}

\begin{proof}
	By Corollary~\ref{coro:std_Morse}, $\Phi_1$ is a standard Morse function on $Z_i$.  By Lemma~\ref{lem:one_min}, since $(\Phi_1)\vert_{\Dd_1}$ has $c_1$ minimal and no other critical points, and since $\Ff\cap Y_1 = \wh\beta\cap Y_1$ is a $\bold v$--thread, this implies that $\Dd_1$ is a $(c,\bold v)$--disk-tangle.  The corresponding result holds for $\Dd_3$, after turning $\Phi_3$ and $(Z_3,\Dd_3)$ upside down.
\end{proof}

Next, we see that the trisection $\T$ can be isotoped to ensure the intersections $\Tt_i = \Ff\cap H_i$ are trivial tangles for $i=1,2,3$.

\begin{lemma}\label{lem:Tt_i}
	After an isotopy of $\T$, we can assume that each $\Tt_i$ is a $(b,\bold v)$--tangle, for some $b\geq 0$.
\end{lemma}

\begin{proof}
	The level set $\Phi^{-1}(1.5)$ is simply $M = (H_1\cup_\Sigma \overline H_2)\setminus\nu(\overline P_1\cup_B P_2)$.
	The intersection $\Ff\cap\Phi^{-1}(1.5)$ is a 2--complex $L\cup\frak b$, where $L$ is a neatly embedded one-manifold $L$, and $\frak b$ is a collection of bands.  Here, we are employing the standard trick of flattening $\Ff$ near each of the saddle points of $\Phi_\Ff$. (See Subsection~\ref{subsec:band_pres} for a precise definition of a band.)
	
	We have a Heegaard splitting $(\Sigma;H_1,H_2)$ that induces a Morse function $\Psi\colon\Phi^{-1}(1.5)\to\R$.  In what follows, we will perturb this splitting (i.e., homotope this Morse function) to improve the arrangement of the 2--complex $L\cup\frak b$.  First, we perturb $\Sigma$ so that it becomes a bridge surface for $L$. At this point, we have arranged that $\Tt_1$ and $\Tt_2$ are $(b',\bold v)$--tangles, for some value $b'$ that will likely be increased by what follows.
	
	Next, we can perturb $\Sigma$ until the bands $\frak b$ can be isotoped along the gradient flow of $\Psi$ so that their cores lie in $\Sigma$. We can further perturb $\Sigma$ until $\frak b\cap\Sigma$ consists solely of the cores of $\frak b$, which are embedded in $\Sigma$; said differently, the bands of $\frak b$ are determined by their cores in $\Sigma$, together with the surface-framing give by the normal direction to $\Sigma$ in $\Psi^{-1}(1.5)$.  Finally, we can further perturb $\Sigma$ until each band is dualized by a bridge semi-disk for $\Tt_2$.  The details behind this approach were given in the proof of Theorem~1.3 (using Figures~10-12) of~\cite{MeiZup_17_Bridge-trisections-of-knotted} and discussed in~\cite{MeiZup_18_Bridge-trisections-of-knotted}.
	
	Finally, we isotope $\Sigma$ so that $\frak b$ is contained in $H_2$; in other words, we push the bands slightly into $H_2$ so as to be disjoint from $\Sigma$. Since each band of $\frak b$ is dualized by a bridge semi-disk for $\Tt_2$, the result $\Tt_3 = (\Tt_2)_\frak b$ of resolving $\Tt_2$ using the bands of $\frak b$ is a new trivial tangle.  The proof of this claim is explained in detail in Lemma~3.1 and Figure~8 of~\cite{MeiZup_17_Bridge-trisections-of-knotted}.  (Though it is not necessary, we can even perturb $\Sigma$ so that $\frak b$ is dualized by a bridge disk at \emph{both} of its endpoints, as in the aforementioned Figure~8.)
	
	Note that all of the perturbations of $\Sigma$ were supported away from $\nu(P_1\cup_B P_2)$, so each of the $\Tt_i$ contained precisely $\bold v$ vertical strands throughout.  In the end, each is a $(b,\bold v)$--tangle for some $b\geq0$.
\end{proof}

Finally, we verify that $\Dd_2$ is a trivial disk-tangle in $Z_2$.

\begin{lemma}\label{lem:Dd_2}
	If $c_2 = b - |\frak b|$, then $\Dd_2$ is a $(c_2,\bold v)$--disk-tangle.
\end{lemma}

\begin{proof}
	As in the preceding lemma, this follows exactly along the lines of Lemma~3.1 of~\cite{MeiZup_17_Bridge-trisections-of-knotted}, with only slight modification to account for the vertical strands.  This is particularly easy to see if one assumes that $\frak b$ meets dualizing disks at each of its endpoints, as in the aforementioned Figure~8.
\end{proof}

Thus, we arrive at a proof of Theorem~\ref{thm:general}.

\begin{proof}[Proof of Theorem~\ref{thm:general}]
	After performing the isotopies of $\Ff$ and $\T$ outlined in the lemmata above, we have arranged that, for $i=1,2,3$,  the intersection $\Dd_i = \Ff\cap Z_i$ is a $(c_i,\bold v)$--disk-tangle in $Z_i$ (Lemmata~\ref{lem:Dd_1/3} and,~\ref{lem:Dd_2}) and the intersection $\Tt_i = \Ff\cap H_i$ is a $(b,\bold v)$--tangle (Lemma~\ref{lem:Tt_i}).  Thus, $\Ff$ is in $(b,\bold c;\bold v)$--bridge trisected position with respect to $\T$, where $\bold c = (c_1, c_2, c_3)$, and the ordered partition $\bold v$ comes from the multi-index $\bold v$ of the braiding $\wh\beta$ of $\Ll = \partial \Ff$.
\end{proof}

\section{Stabilization operations}
\label{sec:stabilize}

In this section we describe various stabilization and perturbation operations that can be used to relate two bridge trisections of a fixed four-manifold pair.  We encourage the reader to refer back to the discussion of connected sums and boundary connected sums of bridge trisections presented in Section~\ref{sec:class}.

\subsection{Stabilization of four-manifold trisections}
\label{subsec:stabilization}
\ 

First, we'll recall the original stabilization operation of Gay and Kirby~\cite{GayKir_16_Trisecting-4-manifolds}, as developed in~\cite{MeiSchZup_16_Classification-of-trisections-and-the-Generalized}.

\begin{definition}[\textbf{\emph{core stabilization}}]
	Let $\T$ be a $(g,\bold k;\bold p, \bold f)$--trisection for a four-manifold $X$, and let $\omega$ be an arc in $\Int(\Sigma)$.  Fix an $i\in\Z_3$.  Perturb the interior of $\omega$ into $H_{i+1} = Z_i\cap Z_{i+1}$, and let $\Sigma'$ denote the surface obtained by surgering $\Sigma$ along $\omega$.  Then, $\Sigma'$ is the core of a $(g+1,\bold k';\bold p, \bold f)$--trisection $\T'$ for $X$, where $\bold k' = \bold k$, except that $k'_i = k_i+1$, which is called the \emph{core $i$--stabilization} of $\T$.
\end{definition}

The importance of this operation rests in the following result of Gay and Kirby.

\begin{theorem}[\cite{GayKir_16_Trisecting-4-manifolds}]
\label{thm:GK_unique}
	Suppose that $\T$ and $\T'$ are two trisections of a fixed four-manifold $X$, and assume that either $\partial X = \emptyset$ or $\T$ and $\T'$ induce isotopic open-book decomposition on each connected component of $\partial X$.  Then, $\T$ and $\T'$ become isotopic after they are each core stabilized some number of times.
\end{theorem}

Performing a core $i$--stabilization is equivalent to forming the (interior) connected sum with a simple trisection of $S^4$.  Let $\T_i$ denote the genus one trisection of $S^4$ with $k_i=1$. See~\cite{MeiSchZup_16_Classification-of-trisections-and-the-Generalized} for details.

\begin{proposition}\label{prop:core_stab}
	$\T'$ is a core $i$--stabilization of $\T$ if and only if $\T' = \T\#\T_i$.
\end{proposition}

Next, we recall the stabilization operation for trisections that corresponds to altering the induced open-book decomposition on the boundary by the plumbing of a Hopf band.  Let $\T_{\text{Hopf}}^+$ (respectively, $\T_{\text{Hopf}}^-$) denote the genus one trisection of $B^4$ that induces the open-book decomposition on $S^3$ with binding the positive (respectively, negative) Hopf link.

\begin{definition}[\textbf{\emph{Hopf stabilization}}]\label{def:Hopf}
	Let $\T$ be a $(g,\bold k;\bold p,\bold f)$--trisection for a four-manifold $X$.  Let $\omega\subset (\Sigma\setminus\alpha_i)$ be a neatly embedded arc, which we consider in $P_i$.  Let $\T^\pm$ denote the trisection obtained by plumbing $\T$ to $\T_\text{Hopf}^\pm$ along the projection of $\omega$, as in Figure~\ref{fig:plumbing}.  We call $\T^\pm$ the \emph{positive/negative Hopf $(i,j)$--stabilization} of $\T$ along $\omega$.
\end{definition}

\begin{figure}[h!]
	\centering
	\includegraphics[width=\textwidth]{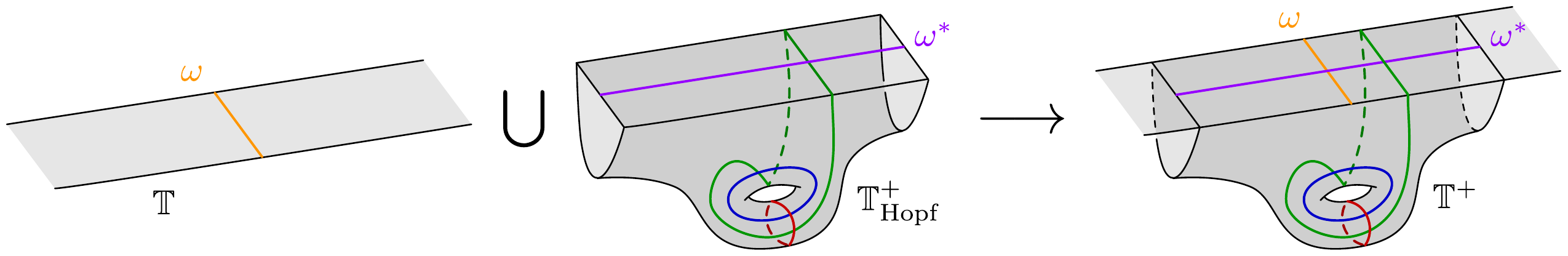}
	\caption{The positive Hopf stabilization $\T^+$ of a trisection $\T$ along an arc $\omega$ in the core of $\T$.}
	\label{fig:plumbing}
\end{figure}

By a \emph{plumbing} of trisections, we mean a plumbing of pages along the projection of arcs to the pages.  Diagrammatically, this is represented by plumbing the relative trisection diagrams along the corresponding arcs in the core surface, as in Figure~\ref{fig:plumbing}. This induces boundary connected sums at the level of the three-dimensional and four-dimensional pieces of the trisections and plumbing at the level of the core surfaces and pages.  Hopf stabilization was first studied in the setting of trisections by Castro~\cite{Cas_16_Relative-trisections-of-smooth} and Castro--Gay--Pinz\'on-Caicedo~\cite{CasGayPin_18_Diagrams-for-relative-trisections}.  We rephrase their main result in the more general setting of the present article.

\begin{proposition}[\cite{CasGayPin_18_Diagrams-for-relative-trisections}, Corollary 17]
	Let $\T$ be a $(g,\bold k;\bold p,\bold f)$--trisection for a four-manifold $X$ inducing an open-book decomposition $(B,\pi)$ on $\partial X$.  Then a positive (resp., negative) Hopf stabilization $\T^\pm$ is a $(g+1,\bold k;\bold p',\bold f')$--trisection of $X$ inducing a positive (resp., negative) Hopf stabilization of $(B,\pi)$, where $\bold f'$ is obtained from $\bold f$ by either increasing or decreasing the value of $f_j$ by one, and $\bold p'$ is obtained from $\bold p$ by either decreasing or increasing the value of $p_j$ by one, according with, in each case, whether or not $\omega$ spans distinct boundary components of $P_i^j$ or not.
\end{proposition}

The upshot of this proposition is that, to the extent that open-book decompositions of three-manifolds are related by Hopf stabilization and destabilization, any two trisections of a compact four-manifold can be related by a sequence of Hopf stabilizations and core stabilizations. Giroux and Goodman proved that two open-book decompositions on a fixed three-manifold have a common Hopf stabilization if and only if the associated plane fields are homotopic~\cite{GirGoo_06_On-the-stable-equivalence-of-open}, answering a question of Harer~\cite{Har_82_How-to-construct-all-fibered-knots}. From this, together with Theorem~\ref{thm:GK_unique}, we can state the following.

\begin{corollary}
\label{coro:rel_tri_unique}
	Suppose that $\T$ and $\T'$ are two trisections of a fixed four-manifold $X$. Assume that $\partial X \not= \emptyset$ and that for each component of $\partial X$, the open-book decompositions induced by $\T$ and $\T'$ have associated plane fields that are homotopic.  Then, $\T$ and $\T'$ become isotopic after they are each core stabilized and Hopf stabilized some number of times.
\end{corollary}

Recently, Piergallini and Zuddas showed there is a complete set of moves that suffice to relate any two open-book decompositions on a given three-manifold~\cite{PieZud_18_Special-moves-for-open}.  By giving trisection-theoretic versions of each move, Castro, Islambouli, Miller, and Tomova were able to prove the following strengthened uniqueness theorem for trisected manifolds with boundary~\cite{CasIslMil_19_The-relative-L-invariant-of-a-compact}.

\subsection{Interior perturbation of bridge trisections}
\label{subsec:interior_perturbation}
\ 

Having overviewed stabilization operations for four-manifold trisections, we now discuss the analogous operations for bridge trisections.  To avoid confusion, we will refer to these analogous operations as \emph{perturbation operations}; they will generally correspond to perturbing the bridge trisected surface relative to the core surface.  Throughout, the obvious inverse operation for a perturbation will be referred to as a \emph{deperturbation}.

We begin by recalling the perturbation operation for bridge trisections first introduced in~\cite{MeiZup_17_Bridge-trisections-of-knotted} and invoked in~\cite{MeiZup_18_Bridge-trisections-of-knotted}.  This perturbation operation requires the existence of a flat disk in $\Dd_i$.  To distinguish this operation from the subsequent one, we append the adjective ``Whitney''.

\begin{definition}[\textbf{\emph{Whitney perturbation}}]
	Let $\Ff$ be a neatly embedded surface in a four-manifold $X$ such that $\Ff$ is in $(b,\bold c;\bold v)$--bridge trisected position with respect to a trisection $\T$ of $X$.  Let $D\subset\Dd_i$ be a flat disk, and let $D_*\subset Y_i$ be a disk that has no critical points with respect to the standard Morse function on $Y_i$ and that is isotopic rel-$\partial$ to $D$, via a three-ball $B$.  Let $\Delta$ be a neatly embedded disk in $B$ that intersects $D_*$ in a vertical strand.  Let $\Ff'$ denote the surface obtained by isotoping $\Ff$ via a Whitney move across $\Delta$.  Then, $\Ff'$ is in $(b+1,\bold c';\bold v)$--bridge trisected position with respect to $\T$, where $\bold c'=\bold c$, except that $c_i'=c_i+1$.  This Whitney move is called an \emph{Whitney $i$--perturbation}.
\end{definition}

See Figures~14 and~23 of~\cite{MeiZup_18_Bridge-trisections-of-knotted} for a visualization of a Whitney pertubation.  The usefulness of Whitney perturbations is made clear by the following result, which was proved in~\cite{MeiZup_17_Bridge-trisections-of-knotted} in the case that $\T$ as genus zero (so $X=S^4$) and in~\cite{HugKimMil_18_Isotopies-of-surfaces-in-4-manifolds} in the general case.

\begin{theorem}[\cite{MeiZup_17_Bridge-trisections-of-knotted,HugKimMil_18_Isotopies-of-surfaces-in-4-manifolds}]
\label{thm:HKM}
	Fix a four-manifold $X$ and a trisection $\T$ of $X$.  Let $\Ff,\Ff'\subset X$ be isotopic closed surfaces, and suppose $\T_\Ff$ and $\T_{\Ff'}$ are bridge trisections of $\Ff$ and $\Ff'$ induced by~$\T$.  Then, there is a sequence of interior (Whitney) perturbations and deperturbations relating $\T_\Ff$ and $\T_{\Ff'}$
\end{theorem}

Even without the presence of a flat disk, there is still a perturbation operation available.  Despite being called a ``finger'' perturbation, the following perturbation is not an inverse to the Whitney perturbation.  The adjective ``Whitney'' and ``finger'' are simply descriptive of how the surface is isotoped relative to the core to achieve the perturbation.  However, it is true that the inverse to a Whitney perturbation (or a finger perturbation) is a finger deperturbation.

\begin{definition}[\textbf{\emph{finger perturbation}}]
	Let $\Ff$ be a neatly embedded surface in a four-manifold $X$ such that $\Ff$ is in $(b,\bold c;\bold v)$--bridge trisected position with respect to a trisection $\T$ of $X$. Fix a bridge point $x\in\bold x$, and let $N$ be a small neighborhood of $x$, so $N\cap\Ff$ is a small disk. Let $\omega\subset \partial N$ be a trivial arc connecting $\Tt_i$ to $\Sigma$.  Perform a finger-move of $\Ff$ along $\omega$, isotoping a small bit of $\Ff$ toward and through $\Sigma$, as in Figure~\ref{fig:finger}.  Let $\Ff'$ denote the resulting surface.  Then, $\Ff'$ is in $(b+1,\bold c';\bold v)$--bridge position with respect to $\T$, where $\bold c'=\bold c$, except that $c_i'=c_i+1$.  This finger move is called an \emph{finger $i$--perturbation}.
\end{definition}

\begin{figure}[h!]
	\centering
	\includegraphics[width=.8\textwidth]{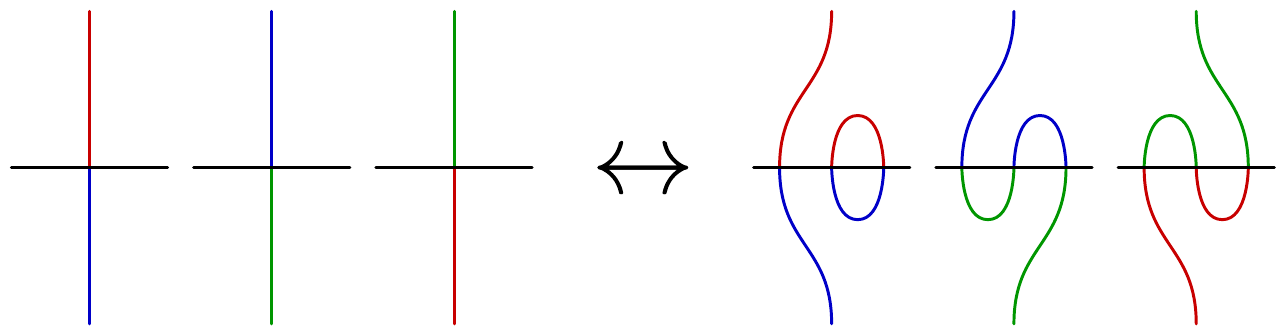}
	\caption{A local picture corresponding to a finger 1--perturbation.}
	\label{fig:finger}
\end{figure}

Note that the disk of the disk-tangle $\Dd_i$ containing the bridge point $x$ is neither required to be flat nor vertical in the definition of a finger perturbation.  However, if this disk \emph{is} flat, then the operation is the simplest form of a Whitney perturbation, corresponding to the case where the vertical strand in $D_*$ is boundary parallel through vertical strands.  The simplicity of the finger perturbation operation is expressed by the following proposition. Let $\T_{S^2}^i$ denote the 2--bridge trisection of the unknotted two-sphere satisfying $c_i = 2$.

\begin{proposition}
	If the bridge trisection $\T_\Ff'$ is obtained from the bridge trisection $\T_\Ff$ via a finger $i$--perturbation, then $\T_\Ff' = \T_\Ff\#\T_{S^2}^i$.
\end{proposition}

The proof is an immediate consequence of how bridge trisections behave under connected sum.  Note that a Whitney perturbation corresponds to a connected sum as in the proposition if and only if it is a finger perturbation; in general, a Whitney perturbation cannot be described as the result of a connected sum of bridge trisections.  For example, the unknotted two-sphere admits a $(4,2)$--bridge trisection that is not a connected sum of (nontrivial) bridge trisections, even though it is (Whitney) perturbed.

\subsection{Markov perturbation of bridge trisections}
\label{subsec:Markov_perturbation}
\ 

Let $\T_{D^2}$ denote the 0--bridge trisection of the unknotted disk $D^2$ in $B^4$.

\begin{definition}[\textbf{\emph{Markov perturbation}}]
	Let $\T'$ be a $(b,\bold c;\bold v)$--bridge trisection of a neatly embedded surface $(X',\Ff')$, and let $\T''$ be the 0--bridge trisection of $(B^4, D^2)$. Choose points $y^\varepsilon\in\Tt_i^\varepsilon\cap P_i^\varepsilon$ for $\varepsilon\in\{',''\}$.
	Let $(X,\Ff)=(X',\Ff')\natural (B^4, D^2)$, and let $\T = \T'\natural\T''$. Then $\T$ is a $(b+1,\bold c;\bold v')$--bridge trisection of $(X,\Ff) = (X',\Ff')$, where $\bold v = \bold v'$, except that $v^j = (v^j)'+1$, where $y^1\in \Ll^j$.  The bridge trisection $\T'$ is called the \emph{Markov $i$--pertubation} of $\T$.
\end{definition}

\begin{figure}[h!]
	\centering
	\includegraphics[width=.6\textwidth]{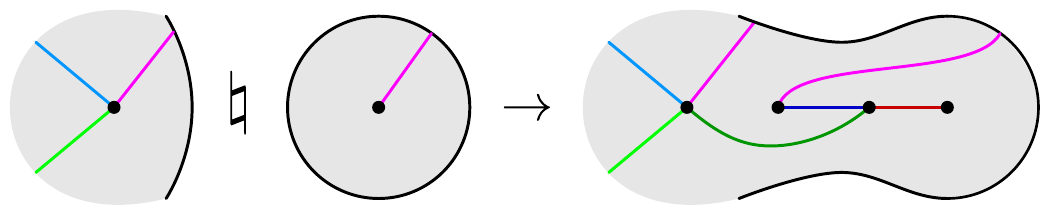}
	\caption{Shadow diagrams depicting the local process of Markov 3--perturbation.}
	\label{fig:markov_shadow}
\end{figure}

In justification of this definition:  That $\T'$ is a new bridge trisection follows from Proposition~\ref{prop:bcs}.  That $\Ff'$ is isotopic to $\Ff$ follows from the fact that we are forming the boundary connected sum with a trivial disk.  That $\Ll'$ is obtained from $\Ll$ via a Markov perturbation follows from our understanding of a Markov perturbation as the trivial connected sum of a braided link with a meridian of a component of the binding -- see Subsection~\ref{subsec:OBD}.  Note that the left-most blue and green arcs are shown in light blue and light green to indicate that they might correspond to flat or vertical strands.  The pink arcs correspond to vertical strands.

The importance of this operation is due to the Generalized Markov Theorem, which states that any two braidings of a given link with respect to a fixed open-book decomposition can be related by an isotopy that preserves the braided structure, except at finitely many points in time at which the braiding is changed by a Markov stabilization or destabilization~\cite{Mar_35_Uber-die-freie-Aquivalenz,Sko_92_Closed-braids-in-3-manifolds,Sun_93_The-Alexander-and-Markov-theorems}.  See Subsection~\ref{subsec:OBD}.

Taken together, the stabilization and perturbation moves described in this section should suffice to relate any two bridge trisections of a fixed four-manifold pair.  

\begin{conjecture}
	Let $\T_1$ and $\T_2$ be bridge trisections of a given surface $(X,\Ff)$ that are diffeomorphic as trisections of $X$. Then, there are diffeomorphic bridge trisections $\T_1'$ and $T_2'$ such that $\T_\varepsilon'$ is obtained from $\T_\varepsilon$ via a sequence of moves, each of which is of one of the following types:
	\begin{enumerate}
		\item core stabilization
		\item Hopf stabilization
		\item relative double twist
		\item interior perturbation/deperturbation
		\item Markov perturbation/deperturbation
	\end{enumerate}
\end{conjecture}

To prove this conjecture, it should suffice to carefully adapt the techniques of~\cite{HugKimMil_18_Isotopies-of-surfaces-in-4-manifolds} from the setting of isotopy of closed four-manifold pairs equipped with Morse functions to the setting of isotopy rel-$\partial$ of four-manifold pairs with boundary.  The following is a diagrammatic analog to this conjecture.

\begin{conjecture}
	Suppose that $\DD^1$ and $\DD^2$ are shadow diagrams for a fixed surface-link $(X,\Ff)$.  Then, $\DD^1$ and $\DD^2$ can be related by a finite sequence of moves, each of which is of one of the following types:
	\begin{enumerate}
		\item core stabilization/destabilization
		\item Hopf stabilization/destabilization
		\item relative double twist
		\item interior perturbation/deperturbation
		\item Markov perturbation/deperturbation
		\item arc and curve slides
		\item isotopy rel-$\partial$
	\end{enumerate}
\end{conjecture}


\bibliographystyle{amsalpha}
\bibliography{RBT.bib}

\end{document}